\PassOptionsToPackage{usenames,dvipsnames}{xcolor}
\documentclass[11pt,final]{book}
\usepackage[a4paper,hmargin={2.4cm,2.4cm},vmargin={3.1cm,2.9cm},headheight=14pt]{geometry}

\usepackage{combelow}
\usepackage{amsthm}
\usepackage{amssymb}
\usepackage[lite]{amsrefs}
\usepackage{stmaryrd}
\usepackage{makeidx}
\usepackage[all]{xy}
\usepackage{epsf}
\usepackage{enumerate}
\usepackage{fancyhdr}
\usepackage{amsmath,amscd}
\usepackage{appendix}
\usepackage[inline]{enumitem}
\usepackage{mathtools}
\usepackage{bm}
\usepackage{tkz-euclide}
\usepackage{tikz,xifthen}
\usepackage{tikz-cd}
\usepackage{mathrsfs}
\usetikzlibrary{arrows}
\usetikzlibrary{decorations.markings}
\usetikzlibrary{backgrounds}
\usepackage{pgfplots}
\pgfplotsset{compat=1.10}
\usepgfplotslibrary{fillbetween}
\usetikzlibrary{patterns}
\usepackage[T1]{fontenc}
\usepackage{tensor}
\usepackage[linktocpage]{hyperref}
\usepackage{hyphenat}
\usepackage{appendix}
\usepackage{float}
\usepackage{titlesec}
\usepackage[textsize=tiny,linecolor=MidnightBlue,backgroundcolor=MidnightBlue!25,textwidth=2.7cm]{todonotes}
\usepackage{cancel}
\usepackage{multicol}
\usepackage{adjustbox}
\usepackage{needspace}
\usepackage{stmaryrd}
\usepackage{ragged2e}
\usepackage{afterpage}

\usepackage{tocloft}
\titleformat{\subsection}{\large\bfseries}{\thesubsection}{1em}{}
\PassOptionsToPackage{usenames,dvipsnames}{xcolor}
\usetikzlibrary{calc,babel} 

\tikzset{
  curve/.style={
    settings={#1},
    to path={
      (\tikztostart)
      .. controls ($(\tikztostart)!\pv{pos}!(\tikztotarget)!\pv{height}!270:(\tikztotarget)$)
      and ($(\tikztostart)!1-\pv{pos}!(\tikztotarget)!\pv{height}!270:(\tikztotarget)$)
      .. (\tikztotarget)\tikztonodes
    },
  },
  settings/.code={%
    \tikzset{quiver/.cd,#1}%
    \def\pv##1{\pgfkeysvalueof{/tikz/quiver/##1}}%
  },
  quiver/.cd,
  pos/.initial=0.35,
  height/.initial=0,
}

\hypersetup{
    colorlinks=true,
    linkcolor=MidnightBlue!90!black,
    citecolor=black!25!red,
    urlcolor=MidnightBlue!90!black,
    pdftitle={Thesis},
    pdfpagemode=FullScreen,
} 

\usetikzlibrary{decorations.markings}
\tikzset{double line with arrow/.style args={#1,#2}{decorate,decoration={markings,%
mark=at position 0 with {\coordinate (ta-base-1) at (0,1pt);
\coordinate (ta-base-2) at (0,-1pt);},
mark=at position 1 with {\draw[#1] (ta-base-1) -- (0,1pt);
\draw[#2] (ta-base-2) -- (0,-1pt);
}}}}
\tikzset{Equals/.style={-,double line with arrow={-,-}}}

\makeatletter
\newsavebox{\@brx}
\newcommand{\llangle}[1][]{\savebox{\@brx}{\(\m@th{#1\langle}\)}%
  \mathopen{\copy\@brx\kern-0.5\wd\@brx\usebox{\@brx}}}
\newcommand{\rrangle}[1][]{\savebox{\@brx}{\(\m@th{#1\rangle}\)}%
  \mathclose{\copy\@brx\kern-0.5\wd\@brx\usebox{\@brx}}}
\makeatother

\newcommand{\plus}{\scalebox{0.6}{$\bm{+}$}}

\makeatletter
\def\blfootnote{\xdef\@thefnmark{}\@footnotetext}
\makeatother
\setlist[enumerate]{leftmargin=25pt, label={(\roman*)}}

\DeclarePairedDelimiterX\set[1]\lbrace\rbrace{\def\given{\;\delimsize\vert\;}#1}
\renewcommand{\d}[1]{\ensuremath{\operatorname{d}\!{#1}}}
\newcommand{\D}[1]{\ensuremath{\operatorname{D}\!{#1}}}

\newcommand{\rra}{\rightrightarrows}
\newcommand{\Ra}{\Rightarrow}
\newcommand{\ra}{\rightarrow}

\newcommand{\R}{\mathbb{R}}
\newcommand{\T}{\mathbb{T}}
\newcommand{\Z}{\mathbb{Z}}
\renewcommand{\AA}{\mathbb{A}}
\newcommand{\GG}{\mathbb{G}}
\newcommand{\MM}{\mathbb{M}}

\newcommand{\rk}{\mathcal{R}}
\newcommand{\C}{\mathcal{C}}
\newcommand{\G}{\mathcal{G}}
\newcommand{\DD}{\mathscr{D}}

\newcommand{\vb}{\mathcal{VB}}
\newcommand{\A}{\mathscr{A}}
\newcommand{\F}{\mathcal{F}}
\newcommand{\ve}{\mathscr{V}\hspace{-0.25em}\mathscr{E}}
\newcommand{\wedgedot}{\:\dot\wedge\:}
\newcommand{\ul}[1]{\underline{\smash{#1}}}

\renewcommand{\O}{\mathcal{O}}
\renewcommand{\S}{\mathscr{S}}
\renewcommand{\sec}{§}
\renewcommand{\L}{\mathscr{L}}
\renewcommand{\frak}{\mathfrak}

\renewcommand{\Im}{\operatorname{Im}}
\newcommand{\End}{\operatorname{End}}
\newcommand{\Der}{\operatorname{Der}}
\newcommand{\Hom}{\operatorname{Hom}}
\newcommand{\Int}{\operatorname{Int}}
\newcommand{\GL}{\operatorname{GL}}
\newcommand{\rank}{\operatorname{rank}}
\newcommand{\rankl}{\operatorname{rank}^L}
\newcommand{\rankr}{\operatorname{rank}^R}
\newcommand{\codim}{\operatorname{codim}}
\newcommand{\sgn}{\operatorname{sgn}}
\newcommand{\ind}{\operatorname{ind}}
\newcommand{\pr}{\mathrm{pr}}
\newcommand{\vol}{\mathrm{vol}}
\newcommand{\Ad}{\mathrm{Ad}}
\newcommand{\Bis}{\mathrm{Bis}}
\newcommand{\ad}{\operatorname{ad}}
\newcommand{\Hor}{\mathrm{Hor}}
\newcommand{\id}{\mathrm{id}}
\newcommand{\lin}{\mathrm{lin}}
\newcommand{\obs}{\mathrm{obs}}
\newcommand{\im}{\operatorname{im}}
\newcommand{\inv}{\mathrm{inv}}
\newcommand{\ext}{\mathrm{ext}}
\newcommand{\ev}{\mathrm{ev}}

\newcommand{\vf}{\mathfrak X}
\newcommand{\spl}[1]{\mathrm{Split}({#1})}
\newcommand{\splc}[1]{\mathrm{Split}_c({#1})}
\newcommand{\deriv}[2]{\frac{d}{d{#1}}\Big|_{{#1}={#2}}}
\newcommand{\smallderiv}[2]{\left.\tfrac{d}{d{#1}}\right|_{{#1}={#2}}}
\newcommand{\arc}{\mathscr{C}}
\newcommand{\cev}[1]{{#1}^L}
\newcommand{\inner}[2]{\left\langle{#1},{#2}\right\rangle}
\newcommand{\innersmall}[2]{\langle{#1},{#2}\rangle}
\newcommand{\innerr}[2]{\llangle #1, #2\rrangle}
\newcommand{\comp}[1]{{#1}^{(2)}}

\renewcommand{\|}{\,|\,}

\DeclareRobustCommand{\gobblefive}[5]{}

\numberwithin{equation}{chapter}

\theoremstyle{plain} 
\newtheorem{thm}{Theorem}
\newtheorem{prop}[thm]{Proposition}

\newtheorem*{cor*}{Corollary}
\newtheorem{theorem}{Theorem}
\numberwithin{theorem}{chapter}
\newtheorem{corollary}[theorem]{Corollary}
\newtheorem{lemma}[theorem]{Lemma}
\newtheorem{proposition}[theorem]{Proposition}

\theoremstyle{definition}
\newtheorem{definition}[theorem]{Definition}

\theoremstyle{remark}
\newtheorem{remark}[theorem]{Remark}
\newtheorem{intermezzo}[theorem]{Intermezzo}
\newtheorem*{remark*}{Remark}
\newtheorem{example}[theorem]{Example}

\newcommand{\authorname}
{Žan Grad}

\makeindex

\begin{document}

\frontmatter

\newgeometry{a4paper,hmargin={2.4cm,2.4cm},vmargin={2.1cm,2.1cm},headheight=14pt}  

\thispagestyle{empty}
\begin{center}
  \begin{figure}[h]
    \includegraphics[width=4cm]{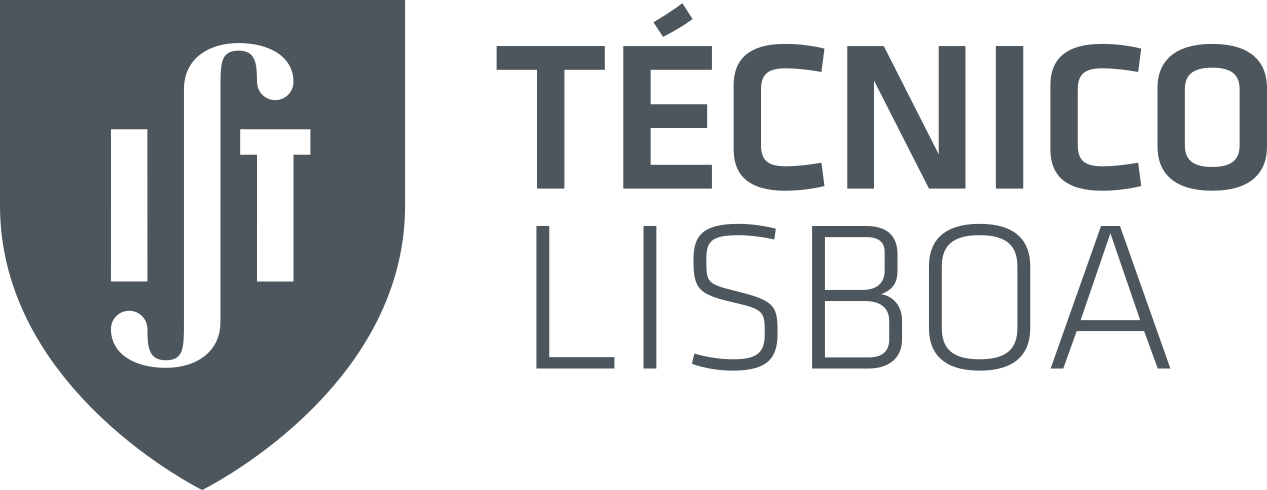}
    \centering
  \end{figure}
\vspace*{1em}

\begin{minipage}{\textwidth}
  \begin{center}
    \large\bfseries
    UNIVERSIDADE DE LISBOA\\[0.3em]
    INSTITUTO SUPERIOR TÉCNICO
\end{center}
\end{minipage}

\vspace{5em}

\begin{minipage}{\textwidth}
\begin{center}
  \renewcommand{\baselinestretch}{1.2}
  \bfseries\LARGE{Fundamentals of Lie categories 
and Yang--Mills theory for multiplicative Ehresmann connections}
\end{center}
\end{minipage}

\vspace{3em}

\begin{minipage}{\textwidth}
\begin{center}
  \renewcommand{\baselinestretch}{1.2}
  \Large\authorname
\end{center}
\end{minipage}

\vspace{3em}

\begin{minipage}{\textwidth}
\begin{center}
  \renewcommand{\baselinestretch}{1.2}
  \Large\begin{tabular}{ r l }
    \textbf{Supervisor:} & Doctor\ Pedro Manuel Agostinho Resende\\
    \textbf{Co-Supervisor:} & Doctor\ Ioan Mărcuț
  \end{tabular}
\end{center}
\end{minipage}

\vspace{3em}

\large{Thesis approved in public session to obtain the PhD Degree in\\[1em] \Large\textit{Mathematics}}

\vspace{1em}
\textbf{Jury final classification:} Pass with Distinction and Honour

\vspace{3em}
\large{\textbf{Jury}}

\vspace{1em}
\begin{minipage}{1\textwidth}
\renewcommand{\baselinestretch}{1.2}
\footnotesize
\textbf{Chairperson:}

\vspace{0.3em}
Doctor João Luís Pimentel Nunes, Instituto Superior Técnico, Universidade de Lisboa

\vspace{1em}
\textbf{Members of the Committee:}

\vspace{0.3em}
Doctor Rui António Loja Fernandes, Department of Mathematics, University of Illinois Urbana-Champaign, EUA \\
Doctor Alejandro Cabrera, Instituto de Matemática, Universidade Federal do Rio de Janeiro, Brasil \\
Doctor Ioan Mărcuț, Faculty of Mathematics and Natural Sciences, University of Cologne, Alemanha\\
Doctor Gonçalo Marques Fernandes de Oliveira, Instituto Superior Técnico, Universidade de Lisboa \\
Doctor Rosa Isabel Sena Neves Gomes Durão Dias, Instituto Superior Técnico, Universidade de Lisboa \\
Doctor João Nuno Mestre Fernandes da Silva, Faculdade de Ciências e Tecnologia, Universidade de Coimbra
\end{minipage}

\vspace{3em}
\textbf{Funding Institution:}

\vspace{0.3em}
Fundação para a Ciência e a Tecnologia (FCT)

\vfill
\large
\textbf{2025}
\end{center}

\clearpage

\thispagestyle{empty}
\vfill

\cleardoublepage

\restoregeometry





\thispagestyle{empty}

\vspace*{4cm}
\begin{center}
 \large\textit{Mojemu dedku Frančku}
\end{center}

\clearpage

\thispagestyle{empty}

\cleardoublepage

\pagestyle{plain}
\chapter*{Acknowledgements}
\fancyhead{}
{
\baselineskip=15pt
First and foremost, I would like to wholeheartedly thank my supervisors, professors Pedro Resende and Ioan Mărcuț. 

\vspace{0.8em}
\noindent Pedro---thank you for supporting my research endeavors and allowing me the freedom to follow my ideas. You kept your faith in me at all times, and my stubbornness could not have been met with greater kindness. 

\vspace{0.8em}
\noindent Ioan---you helped me when I needed you most. Our discussions have always breathed new life into my thinking, and if I found myself out of ideas, you  knew which directions to send me looking for new ones. I admire the force of your mathematical insights, which so often cut straight to the heart of a problem. 

\vspace{0.8em}
\noindent I would like to thank the friendly mathematicians from the field of Poisson geometry who have made me a part of their community---certainly, this beautiful field of mathematics would not have been as exciting without you. I am indebted to Rui Loja Ferndandes, João Nuno Mestre, Marius Crainic, and Alejandro Cabrera for some very insightful conversations that helped this thesis greatly. Moreover, João, I thank you and Raquel Caseiro for organizing the groupoid weeks in Coimbra---to this day, it remains my favorite workshop. I thank Henrique Bursztyn and Daniel Álvarez for the warm reception in Rio, and both Francesco Cattafi and Madeleine Jotz for inviting me to Würzburg.  I am thankful for the inspiring conversations that took place in Manaus with Marco Zambon, Pedro Frejlich, Eckhard Meinrenken, Ivan Struchiner, Marco Gualtieri, and Olivier Brahic. I am grateful for the talks during the beautiful climb of the Vesuvius crater with Karandeep Singh, Luca Accornero, Aldo Witte, and Wilmer Smilde. Without the high-spirited moments with Camilo Angulo and Sylvain Lavau, I would have taken academia too seriously. 

\vspace{0.8em}
\noindent I sincerely thank the professors and colleagues at the Faculty of Mathematics and Physics, University of Ljubljana, who guided and supported me throughout my undergraduate and master's studies. I am deeply grateful to Sašo Strle for being an unwavering pillar of support in my career, and for supervising my master's thesis on Yang--Mills theory, which actually paved the way for this work. Additionally, I would like to thank Janez Mrčun and Jure Kališnik for various discussions, and Matjaž Konvalinka for the warm and motivating conversations. I thank Aleksey Kostenko for some important advice regarding my postdoctoral endeavors.

\vspace{0.8em}
\noindent I gratefully acknowledge the financial support provided by the CAMGSD (Centre for Mathematical Analysis, Geometry, and Dynamical Systems) of Instituto Superior Técnico, which enabled me to attend various conferences and conduct research visits abroad.

\pagebreak

\vspace{0.6em}
\noindent To my friends in Coimbra---Lennart, Sebastian, Sarah, Leo, and Axelle. Lennart and Sebastian, I absolutely adore discussing math with you. I hope that wherever your postdocs take you, you will have the best of times. I would like to thank my friends in Lisbon who have made my time there wonderful, including Pacheco, Candeias, Manel, Xabi, Dino, and Mel. 

\vspace{0.6em}
\noindent
To my dearest, amazing Katarina---thank you for the lemon artichoke pastas, the Amalfi rides, the shark smiles, the sun, and the love. Among the greatest luxuries in my life are also my friends back home---without you, my life would be utterly empty and pointless: Žaži, Ivana, Lari, Soban, Vrabo (with Lisa and Crni), Primož, Rok, Matic, Benjamin, Jakob, Ela, Blaž, Jurij, Žiga, Teja; the list goes on. Thank you for always having my back and keeping my sanity in check.

} 

{
\baselineskip=15pt
\vspace{0.6em}
\noindent
Most importantly, I would like to thank my family. Mami in ati, vselej sta mi stala ob strani in me podpirala, zlasti takrat, ko je moja pot bila vijugasta in neuhojena---za to sem vama hvaležen iz dna srca. Dragi dedi Franček, to delo posvečam tebi, saj si mi v življenju vedno bil v vzgled. Hvala tudi babici Hildi, stricu Tomažu,  Nini, Manji, in Izi, za podporo in ljubezen. Čeprav sem pogosto daleč od vas, se mi zdi, da ste vedno blizu.

}
\clearpage 

\pagestyle{plain}
\cleardoublepage

\pagestyle{fancy}

\fancyhead{}
\fancyhead[CE]{Contents} 
\fancyhead[CO]{Contents}

\tableofcontents

\clearpage \pagestyle{plain}


\mainmatter

\fancyhead{} \setcounter{tocdepth}{1}
\chapter*{Introduction}
\addcontentsline{toc}{chapter}{Introduction} \pagestyle{fancy} 
\fancyhead[CE]{Introduction} 
\fancyhead[CO]{Introduction}

This dissertation consists of two research projects, as already suggested by the title: \emph{Fundamentals of Lie categories} and \emph{Yang--Mills theory for multiplicative Ehresmann connections}. Both projects belong to the field of differential geometry---more precisely, they both concern the foundations and applications of theory of Lie groupoids and algebroids. 
\vspace{0.8em}

\noindent \textit{Fundamentals of Lie categories.} The first and shorter part of this thesis revisits one of the most basic structural assumptions in the theory of Lie groupoids: invertibility. That is, we study Lie categories, which are small categories (objects and arrows form sets) endowed with a compatible smooth structure. 
The motivation for this work comes from theoretical physics, where the states of a physical system and the processes between them are often modelled with smooth manifolds, and physical systems which exhibit irreversible dynamics are ubiquitous. For instance, the second law of thermodynamics states that for a process between two states in an isolated system to be feasible, the entropy $S$ of the initial state must not be greater than that of the final state, that is, $\Delta S\geq 0$. In turn, the invertible processes of a system are those for which the entropy remains unchanged. As the motivating picture of thermodynamics perhaps already suggests, one difference between the theory of Lie categories and Lie groupoids is that in a Lie category, the space of arrows is allowed to possess a boundary, as the reversible arrows are precisely those with $\Delta S=0$. This cannot happen for Lie groupoids, provided the space of objects is boundaryless.  We introduce new examples of Lie categories, explore their structural properties, inspect their differences and similarities with Lie groupoids, and research the notions emerging naturally from the lack of the invertibility assumption.
\vspace{0.8em}

\noindent \textit{Yang--Mills theory for multiplicative Ehresmann connections.} The second and principal part of the thesis is concerned with generalizing Yang--Mills theory---a cornerstone of modern mathematical physics---from principal bundles to possibly \emph{non-integrable} and \emph{non-transitive} Lie algebroids. To provide a motivation for this, let us first revisit the classical Yang--Mills theory on principal bundles. The central mathematical notion which appears at its foundation is that of the curvature of a connection on a principal bundle. Roughly speaking, Yang--Mills theory entails an application of the variational principle to principal bundles. More precisely, the action functional is defined on the space of all principal bundle connections: it inputs a connection and outputs the $L^2$-norm of its curvature. Finding critical points of this action functional provides the means of finding its (local) minima, and morally, minimizing the $L^2$-norm of the curvature may be viewed as the next best scenario to having a flat connection. To provide more background, we now discuss the Euler--Lagrange equation for this action (called the \textit{Yang--Mills equation}), while  simultaneously showing how the theory relates to its origins in physics. Consider the trivial $\mathrm U(1)$-bundle on Minkowski space as the underlying principal bundle; in this case, the Yang--Mills equation corresponds  precisely to two of the four Maxwell equations from the theory of electromagnetism, and the remaining two are captured by the Bianchi identity. They respectively read 
\begin{align*}
\d{}\star F=0\quad\text{and}\quad\d{F}=0,
\end{align*}
where $\star$ denotes the Hodge star operator on Minkowski space $M$, and the curvature $F\in\Omega^2(M)$ is interpreted as the electromagnetic field strength. The simplicity of these equations has to do with the fact that the structure Lie group is abelian---more generally, when $G$ is potentially nonabelian and $P\ra M$ is any principal $G$-bundle, these equations similarly have the form 
\begin{align*}
  \d{}^\nabla\!\star F=0\quad\text{and}\quad\d{}^\nabla{F}=0,
  \end{align*}
where the curvature is now identified with the form $F\in\Omega^2(M;\ad (P))$ with values in the adjoint vector bundle $\ad (P)=(P\times\frak g)/G$, and $\nabla$ is the induced linear connection on $\ad (P)$. Hence, classical Yang--Mills theory on principal bundles can be seen as a far-reaching generalization of the theory of electromagnetism, providing a coordinate-invariant way of writing the differential equations that govern the dynamics of force carriers in physics (called \textit{gauge bosons}), on an arbitrary fixed spacetime. 

On the other hand, the theory of principal bundles has a special place in the theory of Lie groupoids, namely, any principal bundle gives rise to its \textit{gauge groupoid} $\G(P)$, whose Hom-sets are defined as $G$-equivariant maps between the fibres of $P\ra M$; conversely, a given principal bundle can be completely recovered from its gauge groupoid. 
Moreover, principal bundle connections can be viewed as certain connections on $\G(P)$. Specifically, as observed in \cite{gerbes} and \cite{mec}, there is a bijective correspondence:

\noindent\begin{minipage}{\linewidth}
  \begin{align*}
  \left\{
  \parbox{4.5cm}{\centering connections on a\\ principal bundle $P\ra M$}
  \right\}
  \longleftrightarrow
  \left\{
  \parbox{7.2cm}{\centering multiplicative Ehresmann connections\\ for the anchor map $\Phi\colon\G(P)\ra M\times M$}
  \right\}
  \end{align*}
\end{minipage}
\vspace{0.5em}

\noindent The anchor map $\Phi$ simply takes a $G$-equivariant map between two fibres of $P$ and outputs the base points of the two fibres, and a \textit{multiplicative Ehresmann connection} is a usual Ehresmann connection $E$ for $\Phi$, namely, a distribution $E$ on $\G(P)$ satisfying $T(\G(P))=\ker\d\Phi\oplus E$, with the additional requirement that $E\subset T(\G(P))$ is a subgroupoid of the tangent groupoid of $\G(P)$. Since the notion of a multiplicative Ehresmann connection makes sense more generally than for transitive groupoids, the fundamental question of this thesis is: is it possible to construct a Yang--Mills theory for multiplicative Ehresmann connections? 

The main result of this thesis is a positive answer to this question, hence relaxing the transitivity condition. Furthermore, it is already clear in the work by Atiyah and Bott \cite{atiyah-bott} that the formulation of Yang--Mills theory is infinitesimal in nature: the formalism does not depend on the gauge groupoid, but rather on its Lie algebroid---the \textit{Atiyah algebroid} of the principal bundle. This insight enables us to further relax the integrability condition, resulting in a generalized framework that is suitable for general Lie algebroids. The central geometric notion, previously given by principal connections, is  now replaced with \textit{infinitesimal multiplicative Ehresmann connections}. The obtained framework provides a flexible and conceptually clear formulation of Yang--Mills theory, opens a pathway for new interactions between gauge theory and higher geometry, and makes advances to the theory of (infinitesimal) multiplicative Ehresmann connections. To temporarily appease the reader, let us hereby write the obtained Yang--Mills equations, together with the Bianchi identities. Instead of a pair of equations as in the classical theory, we now have two pairs: 
\begin{align*}
\begin{aligned}
    \d{}^\nabla\!\star F&=0,&\qquad \d{}^\nabla F&=G,\\
    \d{}^\nabla\!\star G&=\tfrac 1\mu \star F,& \d{}^\nabla G&=0,
\end{aligned}
\end{align*}
where $F$ is a 2-form on the base that models the curvature of a multiplicative connection, $G$ is the so-called curvature 3-form, and $\mu\in\R$ is a constant. We now take a deeper look at the individual chapters of the thesis and provide a summary of the main results obtained in each of them.

\subsection*{Chapter \ref*{chapter:lie_cats}: Fundamentals of Lie categories}
We begin  by introducing the notion of a Lie category---roughly speaking, it is a small category, internal to the category of smooth manifolds. In our treatment, we allow the space of arrows to possess a boundary, which introduces a few subtleties since the space of composable arrows may be a manifold with corners. Examples include:
\begin{itemize}
  \item \textit{Lie monoids} (with possible boundary) and \textit{Lie groupoids} (necessarily without boundary).
  \item The \textit{endomorphism category} of a vector bundle $E\ra X$: 
  \[\End(E)=\set{\xi\colon E_x\rightarrow E_y\given \xi \text{ is a linear map between fibres of $E$}}.\]
  \item \textit{Bundles of Lie monoids}, such as the exterior bundle $\Lambda(E)=\bigoplus_{i}\Lambda^i(E)$, with wedge product as composition. This is an example of a bundle of associative algebras, i.e., it is a bundle of Lie monoids that is enriched over the category of vector spaces.
  \item The \textit{action category} of an action $\phi\colon M\times X\ra X$ of a Lie monoid $M$ on a manifold $X$:
  \[M\ltimes X = \set{(g,x)\in M\times X\given (g,x) \text{ is a regular point of the action }\phi}.\]
  \item The \textit{order category} $C=\set{(y,x)\in\R\times\R\given x\leq y}$ of $\R$.
  \item The \textit{fat category} of a $\vb$-groupoid.
\end{itemize}

\noindent The first natural question about Lie categories concerns invertible arrows. Namely, we ask whether the invertible arrows $\G(C)\subset C$ form a Lie groupoid. We obtain the following result, generalizing and clarifying the result of Charles Ehresmann from his pioneering paper \cite{ehr1959}.
\begin{thm}
If $C$ has a well-behaved boundary, its invertible arrows $\G(C)\subset C$ form an embedded Lie subgroupoid. More precisely, if $u(X)\subset \Int C$, then $\G(C)$ is open in $\Int C$, and if $\partial C$ is a wide subcategory of $C$, then $\G(C)$ is open in $\partial C$.
\end{thm}

\noindent Next, we recognize that the construction of Lie algebroids carries through to any Lie category, by working with either left-invariant or right-invariant vector fields. Of course, since there is no inversion, the two algebroids may fail to be isomorphic. However, if the units are contained within the interior, $u(X)\subset \Int C$, then both algebroids agree with the algebroid of $\G(C)$.  Furthermore, the absence of inverses leads to a natural generalization of rank from linear algebra: every arrow $g\in C$ has a \textit{left} and a \textit{right rank}, 
\begin{align*}
\rankl(g)=\rank\d(L_g)_{1_{s(g)}}\quad\text{and}\quad \rankr(g)=\rank\d(R_g)_{1_{t(g)}}.
\end{align*}
We inspect several examples and the general properties of this notion of rank. Notably, nice properties hold for arrows which have maximal left and right rank (\textit{regular} arrows).
\begin{prop}
The subset of regular arrows in any Lie category $C$ is open. 
\end{prop}
\noindent The rank of arrows also dictates certain structural properties of a Lie category. For instance, we show that if all arrows have full and constant rank, the composition map is submersive (Corollary \ref{cor:composition_submersion}).
Related to the notion of rank, a Lie category $C\rra X$ comes equipped with natural singular distributions determined by differentials of left and right translations. Focusing on left translations, $D\subset TC$ is defined as
\[D_g=\Im\d(L_g)_{1_{s(g)}}\subset \ker\d t_g\subset T_gC.\]

\begin{prop}
In a Lie category without a boundary, the singular distribution $D\subset TC$ is integrable, and its integral manifold through $g\in C$ has dimension $\rankl(g)$.
\end{prop}
\noindent We then focus on Lie categories which are extendable to Lie groupoids, that is, those categories $C$ for which there exists an injective immersive functor $C\hookrightarrow G$ into a Lie groupoid $G$. The existence of such an extension affects the ranks and algebroids as follows:
\begin{itemize}
  \item All arrows in an extendable category $C$ have full and constant rank, and all left and right translations are injective. In particular, the composition map of $C$ is submersive. 
\item If the extension is such that $\dim C=\dim G$, then both Lie algebroids of $C$ are isomorphic to the Lie algebroid of $G$.
\item The Hom-sets of $C$ are closed embedded submanifolds. In particular, for any $x\in X$, the set $\Hom(x,x)\subset C$ is a Lie monoid.
\item If a Lie monoid $M$ is extendable to a Lie group, it is parallelizable. Whether this holds for arbitrary Lie monoids remains an open question. 
\end{itemize}
\noindent Moving further, we generalize completeness results for invariant vector fields on Lie categories. For instance, it is well-known that on a Lie group, any left-invariant vector field is complete. The analogous result for Lie monoids is slightly more subtle: since we are allowing a nonempty boundary, a vector field may be only half-complete, that is, any integral path can be indefinitely extended in one direction (Theorem \ref{thm:complete}), whereas the empty boundary case is analogous to the case of Lie groups. On the other hand, on Lie groupoids, it is well known that a left-invariant vector field $\alpha^L$ is complete if and only if the vector field $\rho(\alpha)$ on the base is complete. A similar result on completeness of left-invariant vector fields is also obtained for Lie categories with well-behaved boundaries, see Proposition \ref{prop:completeness_characterization}.

Finally, we interpret classical statistical thermodynamics within the framework of Lie categories. In statistical physics, one often studies an isolated physical system consisting of an unknown number of particles, each of which can exist in a superposition of finitely many a priori given \textit{microstates}. Mathematically, such a system is described by tuples of probabilities  $(p_i)_{i=0}^n$ which satisfy $\sum_i p_i=1$. The set of these tuples forms the standard $n$-simplex $\Delta^n$. The transitions between these configurations are represented as arrows, which are constrained by the second law of thermodynamics. Specifically, the set of allowable arrows is limited to
\[D=\set*{(p_i)_i\rightarrow (q_i)_i\given S(q_i)_i- S(p_i)_i\geq 0},\]
where $S(p_i)_i=-\sum_i p_i\log p_i$ is the entropy of a given configuration. In the differentiable setting, however, certain adjustments are necessary for the category  $D\rra \Delta^{n}$, since it is not a Lie category. The key issue lies in the behavior of the entropy function $S\colon \Delta^n\ra \R$ at points of $\Delta^n$ where the differential of $S$ becomes problematic. More precisely, these problematic points include the boundary points of the $n$-simplex, $\partial\Delta^n$, where the derivative of $S$ diverges, and the so-called microcanonical configuration $p^\mu=(\tfrac 1{n+1},\dots, \tfrac 1{n+1})\in\Delta^n$, where $S$ is critical. One then obtains a Lie category by restricting to $X=\Int\Delta^n - p^\mu$. We conclude the discussion by noting that the elementary framework of statistical thermodynamics, as understood by physicists, implicitly involves a Lie category in the background.

\subsection*{Chapter \ref*{chapter:background}: Background for the second part}
This chapter highlights the necessary background for the remainder of the thesis. We start with some well-known elementary properties and examples of Lie groupoids and algebroids and their representations. The chapter continues with a series of the following technical prerequisites.
\begin{itemize}
    \item A precise definition of two important cochain complexes appearing in our research: the \textit{Bott--Shulman--Stasheff complex} of a Lie groupoid $G$ and the \textit{Weil complex} of a Lie algebroid $A$ \cites{diff_cohomology, weil, spencer, homogeneous}. These are certain complexes of representation-valued differential forms on groupoids and algebroids. To be precise, in the case of groupoids, the underlying spaces of the complex are
    \[
    \Omega^q(G^{(p)};V),
    \]
    where $V$ is a representation of the groupoid $G$, and $G^{(p)}$ is the $p$-th level of the nerve of $G$. On the other hand, the Weil complex, denoted $W^{p,q}(A;V)$, represents its infinitesimal analogue; it has a rather involved definition, so we postpone it to Chapter \ref*{chapter:background}. 
    The following two classes of differential forms contained in these complexes are important to our research.
    \begin{enumerate}[label={(\roman*)}]
      \item The cocycles at $p=0$. These are the so-called \textit{invariant} forms on $M$, which are differential forms “constant” along the leaves of the orbit foliation.
      \item The cocycles at $p=1$. In the global picture, these are \textit{multiplicative} forms on a Lie groupoid $G$, that is, forms compatible with the groupoid multiplication. In the infinitesimal picture, these are the so-called \textit{infinitesimal multiplicative} (IM) forms on an algebroid $A$, whose  defining conditions are highly nontrivial (equations \eqref{eq:c1}--\eqref{eq:c3}).
    \end{enumerate}
    \item Next, we recall the definition of the \textit{van Est map} from \cite{homogeneous}, which relates the global and infinitesimal cochain complexes from the previous item, 
    \[
    \ve\colon\Omega^q(G^{(p)};V)\ra W^{p,q}(A;V),
    \]
    where $A$ is the Lie algebroid of $G$. We also recall the \textit{van Est theorem} from \cite{homogeneous}, relating their cohomologies (Theorem \ref{thm:homogeneous}). A simple application thereof is given in Corollary \ref{corollary:van_est_multiplicative}, where we recover the bijective correspondence between multiplicative and infinitesimal multiplicative forms, provided the given groupoid has simply connected $s$-fibres. We note that this result was first obtained in \cite{spencer}, by utilizing jets of bisections of a groupoid.
    \item The remainder of the chapter consists of introducing the notions of \textit{$\vb$-groupoids} and \textit{$\vb$-algebroids}, i.e., groupoids and algebroids in the category of vector bundles \cites{dvb, gracia-saz, vb-algebroid-morphisms}. They are of vital importance for establishing the horizontal exterior covariant derivative, induced by an IM connection on a Lie algebroid (this is obtained in Chapter \ref{chapter:mec}). At last, the idea of proof of the van Est theorem for representation-valued forms is presented.
\end{itemize}

\subsection*{Chapter \ref*{chapter:invariant}: Invariant linear connections on representations}
The motivation for the research in this chapter is provided by real-valued multiplicative forms on Lie groupoids, that is, $V=\R_M$ is the trivial representation. Such forms were first used to study the global counterparts of Poisson and Dirac structures \cites{symplectic_groupoids,twisted_dirac}. Their infinitesimal counterparts, IM forms, received further treatment in \cites{im_forms, linear_mult, local, meinrenken_pike, weil, ve_mein}. In the global case, the underlying sets of the cochain complex are simply $\Omega^q(G^{(p)}).$
Since we are considering the case $V=\R_M$, this is a double complex, endowed with the simplicial differential $\delta$ and de Rham differential $\d{}$, satisfying $\delta^2=0,$ $\d{}^2=0$ and $\delta{\d{}}=\d{}\delta$.
\begin{align*}
  \begin{tikzcd}[ampersand replacement=\&, column sep=large, row sep=large]
    {\Omega^{q+1}(M)} \& {\Omega^{q+1}(G)} \& {\Omega^{q+1}(G^{(2)})} \& \cdots \\
    {\Omega^q(M)} \& {\Omega^{q}(G)} \& {\Omega^q(G^{(2)} )} \& \cdots
    \arrow["{\delta}", from=1-1, to=1-2]
    \arrow["{\delta}", from=1-2, to=1-3]
    \arrow["{\delta}", from=1-3, to=1-4]
    \arrow["{\d{}}", from=2-1, to=1-1]
    \arrow["{\delta}", from=2-1, to=2-2]
    \arrow["{\d{}}", from=2-2, to=1-2]
    \arrow["{\delta}", from=2-2, to=2-3]
    \arrow["{\d{}}", from=2-3, to=1-3]
    \arrow["{\delta}", from=2-3, to=2-4]
  \end{tikzcd}
  \end{align*}
As already mentioned, the 1-cocycles of the simplicial differential correspond precisely to multiplicative forms; that $\d{}$ and $\delta$ commute thus has an important consequence: the de Rham differential preserves multiplicativity. 

A natural question is whether it is possible to obtain a double complex in the case when $V$ is an arbitrary representation. Let us suppose that a linear connection $\nabla$ on $V$ is given, and let $\d{}^\nabla$ denote the exterior covariant derivative on $\Omega^q(G^{(p)};V)$ induced by the pullback connection. The following result establishes the necessary and sufficient condition on $\nabla$ for the exterior covariant derivative $\d{}^\nabla$ to commute with $\delta$.
\begin{thm}
The map $\d{}^\nabla$ commutes with $\delta$ if and only if $\nabla$ is \textit{$G$-invariant}, that is, if there holds:
\begin{align*}
    s^*\nabla=t^*\nabla,
\end{align*}
under the identification $s^*V\cong t^*V$ given by the representation $G\curvearrowright V$.  In particular, if $\nabla$ is $G$-invariant, $\d{}^\nabla$ maps multiplicative forms to multiplicative forms.
\end{thm}
This result was already obtained for the level $p=1$ in \cite{mec}*{Appendix A}. Importantly, our method of proof differs from the one in \cite{mec} in that we have produced an explicit formula for the commutator $[\delta,\d{}^\nabla]$ for the general level $p\geq 0$ (see Lemma \ref{lem:G_invariant}), which is crucial in the next sections. In the infinitesimal realm, the task of obtaining the appropriate formula for $\d{}^\nabla$ on the Weil complex $W^{p,q}(A;V)$ is more difficult (Definition \ref{def:d_nabla_weil}). By computing the commutator $[\delta,\d{}^\nabla]$ in the infinitesimal realm (Lemma \ref{lem:A_invariant}), the analogous result as for groupoids follows:
\begin{thm}
The map $\d{}^\nabla$ commutes with $\delta$ if and only if $\nabla$ is $A$-invariant, that is, if there holds
\[\nabla^A_\alpha=\nabla_{\rho(\alpha)}\quad\text{and}\quad\iota_{\rho(\alpha)}R^\nabla=0,\]
for any $\alpha\in A$. In particular, if $\nabla$ is $A$-invariant, the operator $\d{}^\nabla$ maps IM forms to IM forms. 
\end{thm}
Here, $\nabla^A$ denotes the given representation $A\curvearrowright V$. As in the case of groupoids, the $A$-invariance condition above has already been briefly studied in \cite{mec}, and the discussion there was limited to the level $p=1$; again, we stress that our improvements are crucial for our results in Chapter \ref*{chapter:mec}. We see that the condition of $A$-invariance is very restrictive: it implies the connection $\nabla$ actually induces the representation $\nabla^A$ on $V$ of the Lie algebroid $A$. Globally, $G$-invariance is restrictive as well: it implies that the action by the flow $\phi^{\smash{\alpha^L}}_\lambda(1_x)$ of a left-invariant vector field $\alpha^L$ must equal the parallel transport along  $\lambda\mapsto \phi^{\smash{\alpha^L}}_\lambda(1_x)$ with respect to $s^*\nabla$.

As opposed to the case of real coefficients $V=\R_M$, on a general vector bundle $V$ there is no canonical choice of $\nabla$, and a flat connection may not exist at all. Of course, for the map $\d{}^\nabla$ to induce a double complex, we need both $\delta \d{}^\nabla=\d{}^\nabla\delta$ and $(\d{}^\nabla)^2=0$. As we see from the last theorem above, invariance of $\nabla$ only implies the curvature $R^\nabla$ vanishes along the orbits in $M$, allowing it to be nonvanishing in the transversal directions. For this reason, it is natural for us to instead have in mind the relaxed notion of a \textit{curved double complex}, by which we mean a bigraded vector space $C=(C^{p,q})_{p,q\geq 0}$ equipped with: 
\begin{enumerate}[label={(\roman*)}]
  \item A \textit{differential} $\delta\colon C^{p,q}\ra C^{p+1,q}$, making $(C^{\bullet,q},\delta)$ into a cochain complex at any $q\geq 0$.
  \item A cochain map $C^{\bullet,q}\ra C^{\bullet,q+1}$, called the \textit{vertical map}.
\end{enumerate}
When the vertical map squares to zero, we recover the usual notion of a double complex. For instance, in this language, the penultimate theorem above says that a $G$-invariant connection gives rise to a curved double complex $\Omega^\bullet(G^{(\bullet)};V)$, with the differential $\delta$ and the vertical map $\d{}^\nabla$.
\begin{align*}
    \begin{tikzcd}[ampersand replacement=\&, column sep=large, row sep=large]
        {\Omega^{q+1}(M;V)} \& {\Omega^{q+1}(G;V)} \& {\Omega^{q+1}(G^{(2)};V)} \& \cdots \\
        {\Omega^q(M;V)} \& {\Omega^{q}(G;V)} \& {\Omega^q(G^{(2)};V )} \& \cdots
        \arrow["{\delta}", from=1-1, to=1-2]
        \arrow["{\delta}", from=1-2, to=1-3]
        \arrow["{\delta}", from=1-3, to=1-4]
        \arrow["{\d{}^\nabla}", from=2-1, to=1-1]
        \arrow["{\delta}", from=2-1, to=2-2]
        \arrow["{\d{}^\nabla}", from=2-2, to=1-2]
        \arrow["{\delta}", from=2-2, to=2-3]
        \arrow["{\d{}^\nabla}", from=2-3, to=1-3]
        \arrow["{\delta}", from=2-3, to=2-4]
    \end{tikzcd}
    \end{align*}
We next inspect how the van Est map relates the two exterior covariant derivatives in the global and the infinitesimal realm. Preliminarily, we note that $G$-invariance implies $A$-invariance; the converse also holds  if we assume $G$ has connected $s$-fibres.
\begin{thm}
If $\nabla$ is $G$-invariant, the van Est map commutes with exterior covariant derivatives:
\begin{align*}
    \ve \d{}_G^\nabla=\d{}_A^\nabla \ve.
\end{align*}
Moreover, this equality holds on all normalized forms regardless of $G$-invariance of the connection $\nabla$, so in particular, it holds for multiplicative forms.
\end{thm}
We hereby recall that normalized forms on the nerve of $G$ are those which vanish when pulled back along any degeneracy map. 
To conclude Chapter \ref*{chapter:invariant}, we obtain obstruction classes for the existence of $G$-invariant and $A$-invariant connections on a given representation $V$, see Theorem \ref{thm:obstruction_invariance_G} and Theorem \ref{thm:obstruction_invariance}. Without going into too much detail, we just note that these are cohomological classes which live in the (simplicial) cohomology groups
\[
    H^{1,1}(G;\End V)\quad \text{and}\quad  H^{1,1}(A;\End V).
\]
Here, the representation on $\End V$ is the induced one by $V$; we direct to equalities \eqref{eq:bss_cohomology} and \eqref{eq:weil_cohomology} for the definition of cohomology groups. At last, we show that the van Est map relates these two obstruction classes (Proposition \ref{prop:g_inv_a_inv}).

\subsection*{Chapter \ref*{chapter:mec}: Multiplicative Ehresmann connections}
We begin the discussion of Chapter \ref*{chapter:mec} by first recalling the definition of a \textit{bundle of ideals} of a Lie algebroid---it is a subbundle $\frak k\subset \ker\rho$ of $A$ satisfying
\[
  [\Gamma(A),\Gamma(\frak k)]\subset\Gamma(\frak k).
\]
In other words, $\ad\colon A\curvearrowright\ker\rho$ restricts to a representation of $\frak k$.
Similarly, in the global setting, a bundle of ideals is defined by the requirement that the differential of conjugation, $\Ad\colon G\curvearrowright \ker\rho$, restricts to a representation on $\frak k$. The main class of examples comes from surjective submersive groupoid morphisms $\Phi\colon G\ra H$ covering the identity, where one takes  $\frak k=\ker\d\Phi|_M$; such bundles were first studied in \cites{gerbes, mec, poisson_submanifolds}. At the heart of our research is the notion of a \textit{multiplicative Ehresmann connection} for a bundle of ideals $\frak k\subset \ker\rho$, which can be defined in the following two equivalent ways: 
\begin{itemize}
  \item As a multiplicative distribution $E\subset TG$ which complements the smearing $K\subset TG$ of $\frak k$. 
  \item As a multiplicative 1-form  
  $\omega\in\Omega^1_m(G;\frak k),$
  which additionally satisfies $\omega|_{\frak k}=\id_{\frak k}$.
\end{itemize}
As known from \cite{mec}, such a connection on $G$ induces a linear connection $\nabla$ on $\frak k\ra M$. Importantly, in Proposition \ref{prop:conn}, we find an easier way of constructing this connection which does not require an abstract integration of $\frak k$. Note that $\nabla$ is not an invariant connection (unless $\frak k$ is abelian); instead, on account of our construction of $\nabla$, we obtain the central theme of the chapter:
\begin{thm}
Let $G$ be a Lie groupoid with a multiplicative Ehresmann connection $\omega\in\A(G;\frak k)$. The horizontal exterior covariant derivative $\D{}^\omega=h^*\d{}^\nabla$ is a cochain map:
\begin{align*}
    \delta\D{}^\omega=\D{}^\omega\delta.
\end{align*}
In particular, $\D{}^\omega$ maps multiplicative forms to multiplicative forms.
\end{thm}
\noindent Here, the map $h^*$ denotes the projection of differential forms onto the \textit{horizontal subcomplex},
\[h^*\colon\Omega^{q}(G^{(p)};\frak k)\ra \Omega^{q}(G^{(p)};\frak k)^\Hor.\]
The horizontal subcomplex is a certain intrinsic subcomplex (that is, independent of the choice of a multiplicative connection) of $\Omega^q(G^{(\bullet)};\frak k)$, introduced in Definition \ref{defn:horizontal}. In other words, a multiplicative Ehresmann connection turns forms on the nerve into a curved double complex. As the name suggests, the horizontal exterior covariant derivative above provides a far-reaching generalization of the horizontal exterior covariant derivative from the theory of principal bundles. This result immediately implies that the curvature \[\Omega^\omega=\D{}^\omega\omega\] of a multiplicative Ehresmann connection $\omega$ is a multiplicative horizontal 2-form on $G$. Together with the obtained expression for the connection $\nabla$, this enables us to provide a simpler proof of the properties of $\Omega^\omega$ (Proposition \ref{prop:multiplicative_curvature}) established in \cite{mec}. 

Infinitesimally, one aims to find an analogous operator $\D{}^{(\C,v)}$ on the Weil complex, now induced by an \textit{infinitesimally multiplicative (IM) connection} $(\C,v)\in\Omega^1_{im}(A;\frak k)$ for a bundle of ideals $\frak k$ on a Lie algebroid, defined as an IM 1-form on $A$ with $v|_{\frak k}=\id_{\frak k}$. In other words, $v\colon A\ra \frak k$ is a splitting of the short exact sequence
\begin{align*}
    \begin{tikzcd}[ampersand replacement=\&, column sep=large]
    0 \& {\frak k} \& A \& {A/\frak k} \& 0.
    \arrow[from=1-1, to=1-2]
    \arrow[from=1-2, to=1-3]
    \arrow["v", bend left=30, from=1-3, to=1-2]
    \arrow[from=1-3, to=1-4]
    \arrow[from=1-4, to=1-5]
  \end{tikzcd}
  \end{align*}
However, in striking contrast with the groupoid case, there is now no straightforward and intuitive way of defining $h^*$ for Weil cochains, hence the very definition of the horizontal exterior covariant derivative is evasive in the infinitesimal setting. The discovery of this operator for Weil cochains is thus also considered one of the main achievements of the chapter---see Definition \ref{def:hor_proj} and the defining equation \eqref{eq:D_inf} of $\D{}^{(\C,v)}$ for the level $p=1$. The operator $h^*$ on Weil cochains is obtained by employing the viewpoint of $\vb$-algebroids. Namely, we view an IM connection as a $\vb$-subalgebroid of the tangent algebroid $TA$, and identify Weil cochains with special forms on certain larger $\vb$-algebroid, called \textit{exterior cochains} (Definition \ref{def:exterior_cochains}). This is a somewhat technical procedure, resulting in the derivation of the wanted horizontal projection $h^*$, which is the content of Theorem \ref{thm:derivation_hor_proj}.

The procedure above ultimately yields the wanted horizontal exterior covariant derivative in the infinitesimal setting, whose main property is the following infinitesimal analogue of the last theorem.
\begin{thm}
Let $A$ be a Lie algebroid with an IM connection $(\C,v)\in\A(A;\frak k)$. The horizontal exterior covariant derivative $\D{}^{(\C,v)}=h^*\d{}^\nabla$ is a cochain map, that is,
\begin{align*}
\delta\D{}^{(\C,v)}=\D{}^{(\C,v)}\delta.
\end{align*}
In particular, $\D{}^{(\C,v)}$ maps IM forms to IM forms. 
\end{thm}
\noindent As with groupoids, its importance is central to understanding the curvature of an IM connection, 
\[
    \Omega^{(\C,v)}=\D{}^{(\C,v)}(\C,v).
\]
As discussed in \sec\ref{sec:im_curvature}, for instance, it enables us to establish its important structural properties, such as the \textit{infinitesimal Bianchi identity}:
\[
    \D{}^{(\C,v)}\Omega^{(\C,v)}=0.
\]
It is now natural to ask whether the van Est map commutes with the horizontal exterior covariant derivatives developed in the global and the  infinitesimal realm. As shown in the proof of the following result, this is false in general, however, we still obtain a positive result at the level of multiplicative forms.
\begin{thm}
    The van Est map commutes with the horizontal exterior covariant derivatives at the level of multiplicative forms, that is, the following diagram commutes.
  \begin{align*}
      \begin{tikzcd}[row sep=large,column sep=large,ampersand replacement=\&]
      {\Omega^\bullet_m(G;\frak k)} \& {\Omega^{\bullet+1}_m(G;\frak k)} \\
      {\Omega^\bullet_{im}(A;\frak k)} \& {\Omega^{\bullet+1}_{im}(A;\frak k)}
      \arrow["{\D{}^\omega}", from=1-1, to=1-2]
      \arrow["{\ve}"', from=1-1, to=2-1]
      \arrow["{\ve}", from=1-2, to=2-2]
      \arrow["{\D{}^{(\C,v)}}"', from=2-1, to=2-2]
    \end{tikzcd}
    \end{align*}
  \end{thm}
\noindent The conclusion of Chapter \ref*{chapter:mec} consists of constructing obstruction classes for the existence of (infinitesimal) multiplicative Ehresmann connections for a given bundle of ideals on a Lie groupoid or algebroid. Without going into detail, we note they live in the cohomology groups 
\[
    H^{2,1}(G;\frak k)^\Hor\quad \text{and}\quad  H^{2,1}(A;\frak k)^\Hor,
\]
where the superscript $\Hor$ pertains to the fact that the cohomology is that of the horizontal subcomplexes from Definitions \ref{defn:horizontal} and \ref{defn:horizontal_inf}, respectively. We show that the van Est map relates the two obstruction classes (Proposition \ref{prop:existence_im}).

\subsection*{Chapter \ref*{chapter:ym}: Foliated and multiplicative Yang--Mills theory}
In the last chapter, we present two ways of obtaining the desired generalization of Yang--Mills theory. We start with some motivation. As already mentioned, the formulation of classical Yang--Mills theory for a principal $G$-bundle $P\ra M$ is infinitesimal in nature---instead of defining the action functional on the affine space of connections on a principal $G$-bundle $P\ra M$, one can instead define it on the affine space of splittings of the \textit{Atiyah sequence}, 
\[\begin{tikzcd}
	0 & {\ad(P)} & {\frac{TP}G} & TM & 0.
	\arrow[from=1-1, to=1-2]
	\arrow[from=1-2, to=1-3]
	\arrow[from=1-3, to=1-4]
	\arrow[from=1-4, to=1-5]
\end{tikzcd}\]
Indeed, since $G$-equivariance is built into this sequence, its splittings are in a bijective correspondence with principal bundle connections. This important  viewpoint gives rise to a first idea for obtaining a generalization of Yang--Mills theory from principal bundles to Lie algebroids: simply replace the Atiyah algebroid with any regular Lie algebroid and consider the short exact sequence defined by its anchor. In the integrable case, a splitting of such a sequence is just a smoothly varying family of principal bundle connections, hence we view this as a \textit{foliated Yang--Mills theory}.
\subsubsection{Foliated Yang--Mills theory}
The importance of foliated theory is that it serves as a stepping stone for multiplicative theory. As motivated above, let us consider the splittings of the short exact sequence
\begin{align*}
\begin{tikzcd}[ampersand replacement=\&]
    0 \& \ker\rho \& A \& T\F \& 0,
    \arrow[from=1-1, to=1-2]
    \arrow[from=1-2, to=1-3]
    \arrow["\rho", from=1-3, to=1-4]
    \arrow["\sigma"{pos=0.48}, bend left=30, from=1-4, to=1-3]
    \arrow[from=1-4, to=1-5]
\end{tikzcd}
\end{align*}
One of the main results regarding such splittings is the following.
\begin{thm}
Let $A\Ra M$ be a regular Lie algebroid (with some extra assumptions). A splitting $\sigma$ of the short exact sequence above is a critical point of the action functional
\[
\sigma\mapsto \int_M\inner{F^\sigma}{F^\sigma}_{\frak k}\vol_M
\]
if and only if its curvature $F^{\sigma}$ is a solution to the foliated Yang--Mills equation,
\begin{align*}
\d{}^{\nabla^\sigma}\!\star_\F F^{\sigma}=0.
\end{align*}
\end{thm}
\noindent Here, $F^\sigma\in\Omega^2(T\F;\ker\rho)$ is the curvature of the splitting, $\nabla^\sigma$ denotes the leafwise connection induced by the splitting $\sigma$, and $\star_\F$ is the \textit{foliated Hodge star operator}, defined leafwise. In the theorem, we are assuming an $\ad$-invariant metric $\inner\cdot\cdot_{\frak k}$ on the isotropy $\frak k=\ker\rho$ can be chosen, as well as an orientation on $M$; we also choose a Riemannian metric on $M$. However, the foliated theory requires additional assumptions on the orbit foliation which we postpone until \sec\ref{sec:foliated_ym}---we hereby only note that those assumptions will not be needed in the multiplicative case. Although limited to describing the dynamics along the orbits of $A$, this generalization already gives rise to several interesting examples, discussed in \sec\ref{sec:fym_examples}. The following properties hold.
\begin{itemize}
    \item The foliated Yang--Mills equation is \textit{underdetermined}---we develop a precise criterion on when an affinely deformed critical splitting is again critical (Proposition \ref{prop:underdetermined_fym}). The proof of this statement also enables us to formally compute the tangent space of critical points of the foliated Yang--Mills action.
    \item The foliated Yang--Mills theory is \textit{gauge invariant}, that is, the action functional is invariant under the pullbacks by inner automorphisms of an integrating Lie groupoid of $A$ (Theorem \ref{thm:fym_gauge_invariance}). This can also be made sense of infinitesimally, yielding \textit{infinitesimal gauge invariance}.
\end{itemize}

\subsubsection{Multiplicative Yang--Mills theory}
Equipped with the knowledge of the behavior of foliated Yang--Mills theory, we begin developing the multiplicative Yang--Mills theory in \sec\ref{sec:multiplicative_ym}. As with the foliated case, a crucial part of obtaining such a theory will be in varying the action functional, so we begin there. Namely, we begin by inspecting how the curvature of a multiplicative Ehresmann connection changes under an affine deformation
\[
\omega\rightarrow \omega+\lambda\alpha
\]
for any $\lambda\in \R$ and $\alpha\in\Omega^1_m(G;\frak k)^\Hor$. The relevant result is the following.
\begin{thm}
    On a Lie groupoid $G$, the curvature of a multiplicative Ehresmann connection changes with an affine deformation as
    \begin{align*}
        \Omega^{\omega+\lambda\alpha}=\Omega^\omega+\lambda \D{}^\omega\alpha+\lambda^2\mathbf c_2(\alpha),
    \end{align*}
    where the second-order coefficient is homogeneous of degree two and independent of  $\omega$.
\end{thm}

We see that the horizontal exterior covariant derivative plays an important role in the first-order coefficient of the above expansion. Furthermore, the infinitesimal analogue of this result is made possible by the discovery of $\D{}^{(\C,v)}$. More precisely, we can now inspect how the curvature of an IM connection $(\C,v)$ changes under an affine deformation
\begin{align*}
  (\C,v)\rightarrow (\C,v)+\lambda (L,l)
\end{align*}
where $\lambda\in\R$ and $(L,l)\in\Omega^1_{im}(A,\frak k)^\Hor$ is any horizontal IM form. We similarly obtain:
\begin{thm}
On a Lie algebroid $A$, the curvature of an IM connection changes with an affine deformation as
\begin{align*}
  \Omega^{(\C,v)+\lambda(L,l)}=\Omega^{(\C,v)}+\lambda \D{}^{(\C,v)}(L,l)+\lambda^2\mathbf c_2(L,l),
\end{align*}
where the second-order coefficient is homogeneous of degree two and independent of $(\C,v)$.
\end{thm}
\noindent Since our aim is to formulate the multiplicative Yang--Mills theory without integrability assumptions on $A$ (just as in the foliated case), we henceforth focus on the infinitesimal multiplicative connections.

What follows next is a crucial insight for the process of formulating the desired multiplicative Yang--Mills theory: that the class of IM connections, on which we plan to define an action functional, should be restricted to a certain well-behaved class of IM connections, namely, those with cohomologically trivial curvature. Such IM connections are called \textit{primitive}. More precisely, these are IM connections $(\C,v)$ whose curvature can be expressed as \[\Omega^{(\C,v)}=\delta^0 F,\] for some form $F\in\Omega^2(M;\frak k)$, called the \textit{curving} of $(\C,v)$. 
In other words, the curvature IM form is modelled by a form on the base. Due to the importance of this class of IM connections for our theory,  we research them thoroughly in \sec\ref{sec:primitive}. We gather only a few of the basic results here, needed for this introduction:
\begin{enumerate}[label={(\roman*)}]
  \item In general, the curving of an IM connection is not unique---it is only unique up to an invariant 2-form $\beta\in\Omega^2_\inv(M;\frak k)=\ker\delta^0$. Such forms must necessarily be transversal, that is, $\iota_X\beta=0$ whenever $X\in T\F$, and centre-valued, so $\beta\in\Gamma(\Lambda^2(T\F)^\circ\otimes z(\frak k))$.
  \item The Bianchi identity for $\Omega^{(\C,v)}$ translates to two identities which are satisfied by the so-called \textit{curvature 3-form} $G\coloneq\d{}^\nabla F$, namely, $\d{}^\nabla G=0$ and  $\delta^0 G=0$ ($G$ is an invariant 3-form).
  \item A primitive IM connection $(\C,v)$ with curving $F$ may be deformed by a cohomologically trivial 1-form $\lambda\delta^0\gamma$, where $\gamma\in\Omega^1(M;\frak k)$ and $\lambda\in\R$, and the deformed IM connection $(\C,v)+\lambda\delta^0\gamma$ will again be primitive, with curving
  \[
  \smash{F^{\lambda\gamma}=F+\lambda\d{}^\nabla\gamma-\frac {\lambda^2}2[\gamma,\gamma].}
  \]
  \item On a transitive algebroid, every IM connection is primitive with vanishing 3-curvature. The same is true if the typical fibre of $\frak k$ is semisimple (Proposition \ref{prop:primitive_semisimple}). 
\end{enumerate}

The theory we have developed comes together in the next theorem, which forms the central result of the thesis. Preliminarily, we define the \textit{multiplicative Yang--Mills action functional} as
\[
\S\colon \DD(A;\frak k)\ra \R,\quad \S((\C,v),F)=\int_M\inner FF_{\frak k}+\mu\int_M \inner GG_{\frak k},
\]
where $\mu\in\R$ is the \textit{structure constant} of the theory which can be chosen arbitrarily, and the domain of the action consists of IM connections and their curvings, that is,
\[\DD(A;\frak k)=\set{((\C,v),F)\given \delta^0F =\Omega^{(\C,v)}}.\] 
The second term in the action functional, involving the 3-curvature, will stand out as a source of novel phenomena in the theory.  As in the foliated case, the definition of the action functional assumes an ad-invariant metric $\inner\cdot\cdot_{\frak k}$ on $\frak k$ can be chosen, together with a Riemannian metric and an orientation on $M$.

\begin{samepage}
  \noindent When defining the critical points of the multiplicative Yang--Mills action functional, we run into the following subtlety: there are two natural decompositions of $\DD(A;\frak k)$ into affine spaces (see Remark \ref{rem:decompositions_affine}), giving rise to two distinct notions of criticality of a triple $((\C,v),F)$:
  \begin{itemize}
    \item \emph{Longitudinal criticality}: $\smallderiv\lambda0\S((\C,v)+\lambda \delta^0\gamma, F^{\lambda\gamma})=0$, for all $\gamma\in\Omega^1(M;\frak k)$.
    \item \emph{Transversal criticality}:
      $\smallderiv\lambda0\S((\C,v), F+\lambda\beta)=0,$ 
      for all $\beta\in\Omega_{\inv}^2(M;\frak k)$.
  \end{itemize}
\end{samepage}
Note that these two notions are related precisely to the points (iii) and (i) above, respectively. Importantly, these are precisely the directions in which we vary our action.
\begin{thm}
Let $\frak k$ be a bundle of ideals on a Lie algebroid $A\Ra M$ and let $(\C,v)$ be a primitive IM connection with a curving $F$. Assuming the induced linear connection $\nabla$ is compatible with the metric $\inner\cdot\cdot_{\frak k}$, we have the following equivalences.
\begin{align*}
  \text{$((\C,v),F)$ is longitudinally critical}\quad&\iff\quad \d{}^\nabla\!\star F=0,\\[0.3em]
  \text{$((\C,v),F)$ is transversally critical and adapted}\quad&\iff\quad \d{}^\nabla\! \star G=\tfrac 1\mu\star F.\qquad\qquad
\end{align*}
\end{thm}
\noindent The notion of adaptedness appearing in the theorem is rather strong. For instance, in the case $\mu=-1$, the condition reads
\[
\delta^0(\delta^\nabla G)=\Omega^{(\C,v)}.
\]
In the transitive case, we see that since $G=0$, this condition is actually equivalent to flatness of the given connection, $F=0$. In fact, it is our view that one should carefully consider which notions of criticality to employ in a given context:
\begin{enumerate}[label={(\roman*)}]
  \item Low codimension or trivial centre: $\codim \F\leq 2$ or $z(\frak k)=0$. In this case, only longitudinal criticality is of substance, and one focuses only on solving the first Yang--Mills equation.
  \item High codimension and non-trivial centre: $\codim \F\geq 3$ and $z(\frak k)\neq 0$. In this case, both notions of criticality should be considered, and one focuses on solving both equations. In the abelian case, $z(\frak k)=\frak k$, the second equation implies the first.
\end{enumerate}
Gathering both Yang--Mills equations together with the Bianchi identities and the defining equalities for $F$ and $G$, we find the theory has a certain intricate symmetry:
\begin{align*}
\begin{aligned}
    \delta^\nabla F&=0,&\qquad \d{}^\nabla F&=G,&\qquad \delta^0 F&=\Omega^{(\C,v)},\\
    \delta^\nabla G&=-\tfrac 1\mu F,& \d{}^\nabla G&=0,&\qquad \delta^0 G&=0,
\end{aligned}
\end{align*}
The obtained theory has the following notable features.
\begin{itemize}
    \item Multiplicative Yang--Mills theory should be viewed as a \textit{constrained variational problem}, the constraint being that upon varying a given triple $((\C,v),F)$, the cohomological class \[[(\C,v)]\in H^{1,1}(A;\frak k),\] is kept constant. Within this interpretation, the two notions of criticality are merged, and the notion of adaptedness is viewed as an additional constraint. This interpretation is established in \sec\ref{sec:constrained_variational_problem}; as a consequence, we obtain the formal tangent space to the space of solutions of the Yang--Mills equations (Proposition \ref{prop:formal_tangent_space}).
    \item Combining the two Yang--Mills equations yields
    \[
    \Delta F=(\d{}^\nabla\delta^\nabla+\delta^\nabla\d{}^\nabla)F=-\tfrac 1\mu F,
    \]
    that is, the curving $F$ must be an eigenvector of the Laplacian defined by the connection $\nabla$, with eigenvalue $-\frac 1\mu$. This is in contrast with the classical theory on principal bundles, where satisfying the Yang--Mills equation amounts to harmonicity of $F$, that is, $\Delta F=0$.
    \item As already mentioned, in the transitive case, the 3-curvature automatically vanishes, so only the first term $\innerr FF_{\frak k}$ of the action functional survives. In this case, only longitudinal criticality is considered, and it recovers the classical Yang--Mills theory. 
    \item If working with a pseudo-Riemannian manifold $M$, the Yang--Mills equations allow for \textit{self-dual} and \textit{anti self-dual} solutions. Namely, if $M$ is 5-dimensional, and taking for instance $\mu=1$, the (anti) self-dual solutions read
    \begin{align*}
      G=\pm \star F,\quad F=\pm (-1)^s \star G,
    \end{align*}
    where $s$ denotes the index of the pseudo-Riemannian manifold $M$, i.e., the number of negative components in its signature. Note that this just defines a class of solutions---establishing the criteria for their existence is beyond the scope of this thesis.
    \item Just as with the foliated Yang--Mills theory, the multiplicative Yang--Mills action functional is \textit{gauge invariant}, that is, it is invariant under the pullbacks by inner automorphisms of an integrating Lie groupoid of the algebroid $A$ (Theorem \ref{thm:multiplicative_ym_gauge_invariance}). As in the foliated case, this phenomenon  can also be made sense of infinitesimally, by working with flows of inner derivations instead of inner automorphisms of an integrating groupoid, yielding \textit{infinitesimal gauge invariance}.
\end{itemize}
At last, we discuss an important example of the developed framework in \sec\ref{sec:central_extensions_ym}: $S^1$-bundle gerbes. The example is done in the setting of Lie groupoids rather than that of Lie algebroids, and it reveals that it is sometimes preferential to work in the global instead of the infinitesimal setting, the underlying reason being that $S^1$ is not simply connected (and compact). In turn, this motivates a formulation of multiplicative Yang--Mills theory for multiplicative Ehresmann connections on Lie groupoids in \sec\ref{sec:multiplicative_ym_groupoids}---its relationship with the developed theory on Lie algebroids is also obtained in Proposition \ref{prop:relation_multiplicative_ym_groupoids_algebroids}. We conclude the thesis with a relaxed discussion on the relationship of the obtained framework with other existent generalizations of the classical Yang--Mills theory.

\clearpage \pagestyle{plain}
\chapter{Fundamentals of Lie categories}\label{chapter:lie_cats}
\pagestyle{fancy}
\fancyhead[CE]{Chapter \ref*{chapter:lie_cats}} 
\fancyhead[CO]{Fundamentals of Lie categories} 

The contents of this chapter, presenting the framework of Lie categories, were published in \cite{lie_cats}.

\section{Introduction}
\label{sec:intro}
Since its conception, category theory has proven to provide a unified framework for the language of mathematics, by making use of the observation that objects and morphisms thereof arise regardless of the mathematical field one considers. This paradigm has also been adopted by physicists (see e.g., \cite{catsquantum}), namely that one can interpret physical states (contained in a certain phase space) as objects, and physical processes as morphisms between them, thus forming a category which corresponds to the physical system at hand. In order to provide such a realization of a given physical system, the corresponding category should ideally have the capacity to straightforwardly describe the phenomena which pertain to the given physical system, and also capture the means for describing and calculating relevant physical quantities. For a physicist, the latter is generally done using the basic tools of calculus, in terms of a preferred set of coordinates. 

There is a natural way of obtaining such a unified framework for describing physical processes, by intertwining category theory with the theory of smooth manifolds. That is, we require the set of objects (states) and the set of morphisms (processes) of an abstract category to be a pair of smooth manifolds, and that the composition law $(f,g)\mapsto fg$ is a smooth map (concatenation of two processes); the mathematical structure obtained is that of a category internal to the category $\mathbf{Diff}$ of smooth manifolds. Historically, this kind of mathematical structure, which we will define in more precise terms as a Lie category, was first introduced and briefly studied by Charles Ehresmann in his seminal paper \cite{ehr1959}; in the same paper, Ehresmann further focused on the notion of a Lie groupoid, which additionally imposes that all morphisms are invertible. This latter notion has nowadays been thoroughly researched and continues to have a status of an active field of research, whereas the same cannot be said for Lie categories; apart from the work of Ehresmann, to the best of our knowledge, this is their first systematic treatment. The main reason behind the fact that Lie categories have gained negligible attention compared to Lie groupoids lies in the fact that the assumption on morphisms being invertible implies that all left and right translations are diffeomorphisms, which accounts for certain preferable properties of the structure of a Lie groupoid, that will also be highlighted in our work.

Our motivation for studying Lie categories, and not merely groupoids, stems from the physical interpretation that invertible morphisms correspond to reversible processes, and in physics not all processes in a given system are reversible; this is well known from the theory of thermodynamics, where reversible processes are precisely those where the change of entropy when transiting between two states is zero; we will describe this in more precise terms in \sec \ref{sec:std}. Another, perhaps more notorious example, is given by irreversibility of wave-function collapse in quantum theory.

As we will see, it will turn out to be desirable to allow the space of morphisms of a Lie category to possess a boundary---we will encounter both, mathematical and physical examples where the boundary of the space of morphisms will play a distinguished role. To temporarily appease and motivate the reader to this regard, let us briefly note that Lie monoids are examples of Lie categories with the set of objects a singleton, and that the monoid $[0,\infty)$, either for multiplication or addition, provides a first example of a Lie category with boundary. This simple example already shows certain intriguing qualities: for instance, considering $[0,\infty)$ for multiplication, all its invertible elements are contained in its interior, and considering $[0,\infty)$ for addition, its only invertible element is in its boundary. This phenomenon, as may be expected, is one of the features of Lie categories, namely that the units dictate the behavior of invertibles.

Overall, we aim to convince the reader that the interplay of geometrical and categorical structures alone (without the invertibility assumption) provides new exciting questions which were so far overlooked within the scope of Lie theory. Let us summarize our main objectives.
\begin{enumerate}[label={(\roman*)}]
\item To demonstrate that Lie categories allow for an abundance of interesting examples which have so far been missed in the theory of Lie groupoids.
\item To expose some ideas and constructions which carry through to Lie categories from the theory of Lie groupoids, e.g., the Lie algebroid construction, and to detect which results actually depend on the existence of inverses.
\item To inspect the relation between Lie groupoids and Lie categories, and show that novel notions can be obtained when invertibility is dropped.
\item Last, but not least, to provide a heuristic algorithm for constructing Lie categories that describe physical systems, and to reveal that the mathematical structure implicitly present in the foundations of statistical thermodynamics is that of a Lie category.
\end{enumerate}

\section*{Notation}

All our categories are small, i.e., the classes of objects and morphisms are sets. We will denote a category by $C\rra X$, where $C$ is the set of morphisms, $ X$ is the set of objects, 
and the two arrows indicate the \emph{source map} $s\colon C\rightarrow X$ and \emph{target map} $t\colon C\rightarrow X$, which are defined by 
\[
s(x\xrightarrow g y)=x,\quad t(x\xrightarrow g y)=y,
\]
for any morphism $g\colon x\rightarrow y$ in $C$. Alongside these two maps, a category $C$ comes equipped with the \emph{composition map}
\[
m\colon \comp C\rightarrow C,\quad (g,h)\mapsto gh,
\]
where $\comp C=\set{(g,h)\in C\times C\given s(g)=t(h)}$ is the set of all pairs of \emph{composable morphisms}. Moreover, $C$ comes equipped with the \emph{unit map}
\[
u\colon X\rightarrow C,\quad x\mapsto 1_x.
\]
We also define
\[
 C_x=s^{-1}(x),\quad C^y=t^{-1}(y),\quad C_x^y= C_x\cap C^y,
\]
and call $C_x$ the \emph{source fibre} at $x$, and $C^y$ the \emph{target fibre} at $y$. Note that any morphism $g\in C$ determines the maps
\begin{equation}
\label{eq:trans}
\begin{aligned}
&L_g\colon C^{s(g)}\rightarrow C^{t(g)},\quad L_g(h)=gh,\\
&R_g\colon C_{t(g)}\rightarrow C_{s(g)},\quad R_g(h)=hg,
\end{aligned}
\end{equation}
called the \emph{left translation} and \emph{right translation} by $g$, which are just the pre-composition and post-composition by $g$.

\section{Basic definitions and examples}
\label{sec:basics}

\begin{definition}
A \emph{Lie category} is a small category $C \rra X$, where $C$ is a smooth manifold with or without boundary, the base space $X$ is a smooth manifold without boundary, and there holds:
\begin{enumerate}[label={(\roman*)}]
\item The source and target maps $s,t\colon C\rightarrow X$ are smooth submersions.
\item The unit map $u\colon X\rightarrow C$ and the composition map $m\colon \comp C\rightarrow C$ are smooth.
\end{enumerate}
If $C$ has a boundary, we also assume that $C\rra X$ has a \emph{regular boundary}, that is:
\begin{enumerate}
\item[(iii)] The restrictions $\partial s,\partial t\colon\partial C\rightarrow X$ of $s$ and $t$ are smooth submersions.
\end{enumerate}
\end{definition}
\begin{remark}
Given any $x\in X$, assumptions (i) and (iii) ensure that $C_x$ and $C^x$ are neat submanifolds of $C$ (see \cite{difftop}*{p.\ 60} or \cite{corners}*{Proposition 4.2.9}), that is:
\begin{align}
\label{eq:fibrebdr}
\partial( C_x)= C_x\cap \partial C \quad\text{and}\quad \partial( C^x)= C^x\cap \partial C.
\end{align}
Moreover, assumptions (i) and (iii) ensure that the set
\[
\comp C=(s\times t)^{-1}(\Delta_ X)
\]
of composable morphisms is a neat submanifold of $C\times C$, that is:
\begin{align}
\label{eq:regbdr_c2}
\partial(\comp C)=\comp C\cap \partial( C\times C)=\comp C\cap ( C\times \partial C \cup \partial C\times C),
\end{align}
by transversality theorem, see Corollary \ref{cor:fibred_prod}. This corollary also ensures that if $C$ has a boundary, the corner points of $\comp C$ are precisely the composable pairs in $\partial C\times\partial C$; moreover, the tangent space of $\comp C$ at a composable pair $(g,h)$ equals
\begin{align}
T_{(g,h)} C^{(2)}=\set{(v,w)\in T_g C\oplus T_h C\given \d s(v)=\d t(w)}.
\end{align}
Smoothness of $\comp C$ implies that the requirement of smoothness of the composition map $m\colon \comp C\rightarrow C$ makes sense, and furthermore that left and right translations $L_g,R_g$ are smooth maps between appropriate fibres, as defined by equations \eqref{eq:trans}; we thus obtain a covariant functor $C\rightarrow \mathbf{Diff}$, given on objects as $x\mapsto C^x$ and on morphisms as $g\mapsto L_g$, and a contravariant one given by $x\mapsto C_x$, $g\mapsto R_g$. 
\end{remark}
\begin{remark}
\label{rem:units_embedded}
The unit map $u\colon X\rightarrow C$ of a Lie category $C\rra X$ is an embedding, which is a consequence of the fact that it is a smooth section of the source (and target) map, hence an injective immersion which is a homeomorphism onto its image, whose continuous inverse is given by $s|_{u( X)}$. 
\end{remark}
\begin{definition}
A \emph{morphism} of Lie categories is a smooth functor $F\colon C\rightarrow D$. A Lie category $C$ is said to be a \emph{Lie subcategory} of $D$, if there is an injective immersive morphism $F\colon C\rightarrow D$ of Lie categories.
\end{definition}

By the remark before the definition, any morphism of Lie categories induces a smooth map between the respective object manifolds of $C$ and $D$. We now turn to examples of Lie categories.

\begin{example}\
In the case when the object space $ X$ is a singleton, a Lie category $C\rra \set{*}$ will be called a \emph{Lie monoid}. Simply put, a Lie monoid is a monoid $M$ together with a structure of a smooth manifold with or without boundary, such that the multiplication $m\colon M\times M\rightarrow M$ is smooth.

\pagebreak

Concrete examples of Lie monoids frequently arise as embedded submonoids of Lie groups. For instance, we may consider the closed ray $[0,\infty)\subset \R$ for addition, or more generally, the $n$-dimensional half-space $\mathbb H^n=\R^{n-1}\times [0,\infty)\subset \R^n$, where $n\in\mathbb N$, which is a commutative Lie monoid for the usual addition. On the other hand, the closed ray $[0,\infty)$ for multiplication is not\footnote{This is easily seen since it has a non-cancellative (absorbing) element.} a submonoid of any Lie group. Further examples of Lie monoids that do not arise as submonoids of Lie groups are: 
\begin{enumerate}[label={(\roman*)}]
\item The set $\R^{n\times n}$ of square $n$-dimensional matrices is a Lie monoid for matrix multiplication; more generally, we may consider the Lie monoid $\End(V)$ of endomorphisms of a finite dimensional vector space $V$, under composition. Even more generally, any finite-dimensional unital algebra is a Lie monoid that is enriched over the category $\mathbf{Vect}$ of vector spaces, since the multiplication map is bilinear.
In particular, this includes the real line, the complex plane, and quaternions for multiplication.
\item The closed unit disk $\overline{\mathbb D}\subset \mathbb C$, an abelian Lie monoid for complex multiplication, and the closed unit 4-ball $\bar{\mathbb B}^4$, a non-abelian Lie monoid for quaternionic multiplication.
\end{enumerate}
\end{example}

The following example generalizes the similar notion of triviality from the theory of Lie groupoids.
\begin{example}
Let $X$ be a smooth manifold without boundary and $M$ a Lie monoid. A \emph{trivial Lie category} is defined by $C=X\times M\times X$ and $ X=X$, with $s=\pr_3, t=\pr_1$ and composition as
\[
(z,g,y)(y,h,x)=(z,gh,x).
\]
That composition is smooth follows from the smoothness of multiplication in $M$. In the case when $M=\set e$ is a trivial monoid, we obtain the well-known \emph{pair groupoid}.
\end{example}

The next example reveals the spirit of the notion of a Lie category---that is, it can be thought of as a smooth family of endomorphisms of an abstract structure, parametrized by the base manifold $ X$. This is aligned with the philosophy that a Lie groupoid can be thought of as a smooth collection of automorphisms (symmetries) of a structure parametrized by $ X$.
\begin{example}
\label{ex:endomorphism_cat}
Let $\pi\colon V\rightarrow X$ be a vector bundle over a smooth manifold $X$ without boundary, whose typical fibre is a fixed vector space $W$. The \emph{endomorphism category} of $V\rightarrow X$ is the category $\End(V)\rra X$, where the set of morphisms is defined as the set
\[
\End(V)=\set{\xi\colon V_x\rightarrow V_y\given \xi \text{ is linear},x,y\in X}
\]
of linear homomorphisms between the fibres of $V\rightarrow X$, and the structure maps $s,t,m,u$ are defined in the obvious way. To show that $\End(V)$ admits a structure of a Lie category without boundary, we must define a smooth structure on $\End(V)$; it is induced by local trivializations on $V$ in the following way. Denote by 
\[
\set{U_i\times W\xrightarrow{\psi_i} \pi^{-1}(U_i)\given {i\in I}}
\]
an atlas of local trivializations of $V$ over an open cover $(U_i)_{i\in I}$ of $X$, and denote by $\tau_{ij}\colon U_i\cap U_j\rightarrow \GL(W)$ the respective transition maps, i.e., $\psi_i^{-1}\psi_j(x,w)=(x,\tau_{ij}(x)w)$. For any two indices $i,j\in I$, we define the map
\begin{align*}
&\Psi_i^j\colon U_j\times \End(V)\times U_i\rightarrow \End(V)^{U_j}_{U_i}=s^{-1}(U_i)\cap t^{-1}(U_j)\\
&\Psi_i^j(y,A,x)(v|_x)=\psi_j(y,A\pr_W \psi^{-1}_i(v|_x)),
\end{align*}
whose inverse is 
\[
(\Psi_i^j)^{-1}(\xi\colon V_x\rightarrow V_y)=(y,w\mapsto \pr_W\psi_j^{-1}\xi\psi_i(x,w),x).
\]
The smoothness of transition maps in this atlas is easily checked by computing \begin{align*}
(\Psi_k^l)^{-1}\Psi_i^j(y,A,x)=(y,\tau_{jl}(y)^{-1}A\tau_{ik}(x),x),
\end{align*}
where $x\in U_i\cap U_k$ and $y\in U_j\cap U_l$. From the local charts, it is clear that $s$ and $t$ are submersions, and the smoothness of composition map $m$ follows from smoothness of multiplication in $\End(W)$; finally, this composition is bilinear when restricted to $\End(V)^z_y\times \End(V)^y_x\subset \End(V)\comp{}$, so $\End(V)$ is moreover enriched over the category $\mathbf{Vect}$ of vector spaces.
\end{example}

\begin{example}[Bundles of Lie monoids]
\label{ex:monoid_bundles}
A Lie category with coinciding source and target map $s=t=:p$ is called a \emph{bundle of Lie monoids}. In this case, two morphisms are composable if, and only if, they are in the same fibre of $p$.

A concrete example of such a Lie category is the \emph{endomorphism bundle} $V\otimes V^*\rightarrow X$ of a vector bundle $V\rightarrow X$, with the composition given on simple tensors as $(v_2\otimes \varphi_2)(v_1\otimes\varphi_1)=\varphi_2(v_1)\,\varphi_1\otimes v_2$, and extended by bilinearity; this can easily be identified with the composition of linear maps $V_x\rightarrow V_x$, so $V\otimes V^*\subset \End(V)$ is a subcategory of $\End(V)$. Using Lemma \ref{lem:lie_subcat}, it is not hard to see that it is actually an embedded Lie subcategory of $\End(V)$.
Another concrete example of a bundle of Lie monoids is the \emph{exterior bundle}, 
\[\Lambda(V)=\bigoplus_{k=0}^{\rank V}\Lambda^k(V),\] of an $\mathbb F$-vector bundle $V\rightarrow X$, i.e., $\Lambda(V)$ consists of all multivectors in $V$, and composition of $\alpha\in \Lambda^k(V_x)$ and $\beta\in \Lambda^l(V_x)$ is given as $\alpha\wedge\beta$. The units are given by $1_x=1\in \mathbb F=\Lambda^0(V_x)$, for any $x\in X$. Again, the given composition map is smooth, which follows easily from bilinearity of wedge product. Moreover, $\Lambda(V)$ is again enriched over $\mathbf{Vect}$. In general, bundles of Lie monoids enriched over $\textbf{Vect}$ would rightfully be called \emph{smooth bundles of unital associative algebras}.
\end{example}
\begin{example}[Action categories]
\label{ex:action_cats}
An \emph{action} of a Lie monoid $M$ on a smooth manifold $X$ is a smooth map $\phi\colon M\times X\rightarrow X$, denoted by $\phi(g,x)=gx$, which satisfies $ex=x$ and $g(hx)=(gh)x$, for any $x\in X$ and $g,h\in M$. We observe that contrary to the case of Lie group actions, the map $\phi$ may not be submersive since the action is no longer by automorphisms of $X$, thus the target map in the naïve generalization of the action groupoid would not be a submersion.

	To remedy this, we construct the \emph{action category} of a given Lie monoid action $\phi\colon M\times X\rightarrow X$ as follows. Denote by
	$$
	M\ltimes X = \set{(g,x)\in M\times X\given \phi \text{ is a submersion at }(g,x)}
	$$
	the set of regular points of the action $\phi$.\footnote{Intuitively, the action category accounts only	 for the non-critical dynamics pertaining to the given action.} Defining the structure maps as usual, i.e., $s(g,x)=x$, $t(g,x)=gx$, the units as $u(x)=(e,x)$, and the composition as
	\[
	(g,hx)(h,x)=(gh,x),
	\] 
	we obtain a Lie category $M\ltimes X\rra X$. Indeed, since $M\ltimes X\subset M\times X$ is an open subset, we only need to check that if $\phi$ is a submersion at the points $(h,x)$ and $(g,hx)$, it is also a submersion at $(gh,x)$. To that end, we first notice that we may write the condition $g(hx)=(gh)x$ as an equality of maps $M\times M\times X\rightarrow X$,
	\[
	\phi\circ(\pr_1,\phi\circ(\pr_2,\pr_3))=\phi\circ(m\circ(\pr_1,\pr_2),\pr_3),
	\]
	where $m\colon M\times M\rightarrow M$ denotes the multiplication in the Lie monoid $M$. Differentiating this equality at $(g,h,x)$, we obtain
	\[
	\d\phi_{(g,hx)}\circ (\pr_1,\d\phi_{(g,x)}\circ(\pr_2,\pr_3))=\d\phi_{(gh,x)}\circ (\d m_{(g,h)}\circ (\pr_1,\pr_2),\pr_3).
	\]
	By assumption, the left-hand side is surjective, thus the same holds for the first map on the right-hand side. Hence, $M\ltimes X\rra X$ is a category and thus a Lie category.
\end{example}

\begin{example}
\label{ex:order_cat}
A simple, yet important example of a Lie category is the \emph{order category} of $\R$, which is defined as the wide subcategory of the pair groupoid $G=\R\times\R\rra \R$, given by
\[
 C=\set{(y,x)\in\R\times\R\given x\leq y}.
\]
The space of morphisms is thus the half-space below the diagonal in $\R^2$, and the source and target maps are the projections to the vertical and the horizontal axis, respectively, implying that the boundary of $C$ is regular. Moreover, since the inversion in the pair groupoid $G$ is given by the reflection over the diagonal, the units in $C$ are precisely the elements of the diagonal, and these are the only invertibles. 

\begin{figure}[h]
\centering
\stepcounter{equation}
\includegraphics{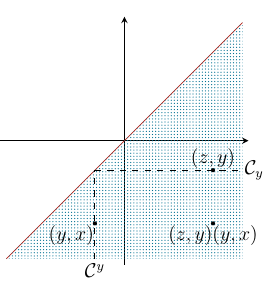}
\caption{The order category of $\R$.}
\end{figure}
\noindent Although simple, this example has an important property: $C$ can be seen as the preimage of the set $[0,\infty)$ under the functor 
\[
f\colon \R\times \R\rightarrow (\R,+), \quad f(y,x)=y-x.
\]
The following example is a generalization of this; we will make use of it when considering applications to statistical thermodynamics in \sec \ref{sec:std}.
\end{example}
\begin{example}
\label{ex:preimage_subcat}
Suppose a Lie category $D\rra X$ without boundary is given, together with a smooth functor $f\colon D\rightarrow (\R,+)$ such that:
\begin{samepage}
\begin{enumerate}[label={(\roman*)}]
\item $0\in\R$ is a regular value of $f$,
\item $s|_{f^{-1}(0)}$ and $t|_{f^{-1}(0)}$ are submersions.
\end{enumerate}
\end{samepage}
Then the preimage $C=f^{-1}\big([0,\infty)\big)$ is a Lie category. Indeed, since $0$ is a regular value of $f$, $C$ is a smooth embedded submanifold in $D$ with boundary $\partial C=f^{-1}(0)$, see e.g., \cite{difftop}*{p.\ 62}. Moreover, since $f^{-1}\big((0,\infty)\big)=\Int C$ is open in $D$, the restrictions $s|_ C,t|_ C$ of the source and target maps to $C$ are submersions. Functoriality of $f$ implies that $C$ is a wide subcategory of $D$ with all invertibles contained within $\partial C$, and moreover it also implies that $\partial C$ is a wide subcategory of $C$. The assumption (ii) enables us to use Lemma \ref{lem:lie_subcat} below to conclude that $C$ is an embedded Lie subcategory of $D$.
\end{example}

\pagebreak

\begin{lemma}
\label{lem:lie_subcat}
Let $C$ be a wide subcategory of a Lie category $D\rra X$. Suppose $C$ is also an embedded submanifold of $D$, such that $s|_ C, t|_ C$ are submersions and either of the following holds:
\begin{enumerate}[label={(\roman*)}]
\item $C$ has no boundary.
\item $C$ has a boundary, and $s|_{\partial C}, t|_{\partial C}$ are submersions.
\end{enumerate}
Then $C$ is an embedded Lie subcategory of $D$.
\end{lemma}
\begin{proof}
The only thing needed to be proven is that the restriction $m|_{\comp C}\colon \comp C\rightarrow C$ of the composition map $m\colon \comp D\rightarrow D$, is smooth. To that end, it is enough to check that $\comp C$ is a submanifold of $\comp D$; notice that $\comp C=\comp D\cap ( C\times C)$, so we will make use of the transversality theorem. 

Suppose first that $C$ and $D$ have no boundary, so that by the usual transversality theorem it is enough to check $\comp D$ and $C\times C$ are transversal in $D\times D$, i.e.,
\[
T_{(g,h)}\comp D+ T_g C\oplus T_h C=T_g D\oplus T_hD,\quad\text{for all }(g,h)\in \comp C.
\]
To show this equality, let $(v,w)\in T_g D\oplus T_h D$. Since $s|_ C$ is a submersion, there is a vector $v'\in T_g C$ with $\d s_g(v')=\d t_h(w)$. Define $v''=v-v'$, and now since $t|_ C$ is a submersion, there is a vector $w'\in T_h C$ such that
\[
\d t_h(w')=\d t_h(w)-\d s_g(v'').
\]
Now define $w''=w-w'$. Clearly, $(v',w')\in T_g C\oplus T_h C$, and on the other hand, the definition of $w'$ and $w''$ imply
\[
\d s_g(v'')=\d t_h(w)- \d t_h(w')=\d t_h(w''),
\]
so that $(v'',w'')\in T_{(g,h)} \comp D$, which concludes our proof for the boundaryless case. If $D$ has a boundary, then the above proof together with Proposition \ref{prop:transversality_char} used on the inclusion \[C\times C\xhookrightarrow{\iota} D\times D,\] ensures that $\iota\pitchfork\comp D$, and since $\comp D\subset D\times D$ is a neat submanifold, $\comp C\subset C\times C\subset D\times D$ is a submanifold by Proposition \ref{prop:transversality}. Hence, $\comp C$ is a submanifold of $\comp D$.

If $C$ has a boundary, then the assumption (ii) ensures that all the strata 
\[
\Int C\times\Int C,\quad (\Int C\times\partial C)\cup(\partial C\times\Int C),\quad \partial C\times\partial C
\]
of $C\times C$ are transversal to $\comp D$ by a similar proof as above, so by Proposition \ref{prop:transversality_char} we again conclude $\iota\pitchfork \comp D$.
\end{proof}
\begin{remark}
In particular, an embedded submonoid of a Lie monoid is its embedded Lie submonoid, however, a direct proof of this is much easier.
\end{remark}

\begin{example}
	Later, we will encounter yet another Lie category (without boundary)---the fat category associated to a $\vb$-groupoid, see Remark \ref{rem:fat_cat}.
\end{example}



\section{Invertibility of morphisms}
Invertible morphisms form an important subclass of morphisms in any category $C$---recall that $g\in C$ is said to be \emph{invertible}, if there is a unique morphism $g^{-1}\in C$ such that $g^{-1}g=1_{s(g)}$ and $gg^{-1}=1_{t(g)}$. We observe that the set
\[
\G (C)=\set{g\in C\given g\text{ is invertible}}
\]
is a groupoid over the same base as $C$. We call $\G (C)$ the \emph{core} of $C$, the study of which begins with the following simple observation.


\begin{proposition}
\label{prop:inv}
For any morphism $g$ in a category $C$, the following are equivalent.
\begin{enumerate}[label={(\roman*)}]
\item $g$ is invertible.
\item Left and right translations by $g$ are bijections.
\item Left and right translations by $g$ are surjections.
\item $g$ has a left and a right inverse, i.e., there exist $g'\in C^{s(g)}, g''\in C_{t(g)}$ with $gg'=1_{t(g)}$ and $g''g=1_{s(g)}$. 
\end{enumerate}
\end{proposition}
\begin{proof}
Implications $ (i) \Rightarrow (ii) \Rightarrow (iii)$ are clear, and $(iii)\Rightarrow (iv)$ follows by observing that $(iv)$ means $1_{t(g)}\in \Im(L_g)$, $1_{s(g)}\in \Im(R_g)$. Lastly, implication $(iv) \Rightarrow (i)$ follows from elementary abstract algebra: $g''=g''1_{t(g)}=g''gg'=1_{s(g)}g'=g'$.
\end{proof}
\begin{remark}
Note that injectivity of left and right translations corresponds to cancellative properties. For example, injectivity of $L_g$ is equivalent to stating that for any two $h,k\in C^{s(g)}$, $gh=gk$ implies $h=k$.
\end{remark}

The regular boundary assumption on a Lie category $C$ has the important consequence that the units dictate where invertible elements can be:
\begin{lemma} 
\label{lem:regbdr}
Let $C$ be a Lie category. For any invertible morphism $g\in \G (C)$, the morphisms $g,g^{-1}, 1_{s(g)},1_{t(g)}$ must either all be contained in the interior $\Int C$, or in the boundary $\partial C$.
\end{lemma}
\begin{proof}[Proof]
If $C$ has no boundary then the lemma holds trivially, so suppose $\partial C\neq \emptyset$. Invertibility of $g$ means $L_g\colon C^{s(g)}\rightarrow C^{t(g)}$ is a diffeomorphism, which maps $1_{s(g)}\mapsto g$, so we must have either $1_{s(g)}\in \partial( C^{s(g)})$ and $g\in \partial( C^{t(g)})$, or $1_{s(g)}\in \Int( C^{s(g)})$ and $g\in \Int( C^{t(g)})$. Regularity of boundary implies $\partial( C^x)= C^x\cap \partial C$ for any $x\in X$, so we must have either $1_{s(g)}\in \partial C$ and $g\in \partial C$, or $1_{s(g)}\in \Int C$ and $g\in \Int C$. A similar result is obtained for $1_{t(g)}$ and $g$ using the right translation $R_g$, and similarly for $g^{-1}$ and the units using $L_{g^{-1}}, R_{g^{-1}}$ since $s(g^{-1})=t(g)$.
\end{proof}

The last lemma implies that any Lie groupoid (a Lie category with all morphisms invertible) must have an empty boundary, provided the base manifold is boundaryless. We also obtain the following two immediate corollaries.

\begin{corollary}
\label{cor:regbdr_inv}
For any Lie category $C\rra X$, there holds:
\begin{align}
u( X)\subset \Int C\ &\text{ implies }\ \G (C)\subset \Int C,\\
u( X)\subset \partial C\ &\text{ implies }\ \G (C)\subset \partial C.
\end{align}
\end{corollary}
\begin{corollary}
The invertible elements of any Lie monoid are either contained in its interior, or in its boundary.
\end{corollary}

An impending question is whether the core $\G (C)$ of $C$ is a Lie groupoid. Without additional assumptions this is false, as shown by the following counterexample.
\begin{example}
A simple example of a Lie category $C$ with a regular boundary, whose core $\G (C)$ is not a manifold, is the disjoint union of the order category on $\R$ and the pair groupoid on $\R$. More concretely, we define 
\begin{align*}
 C&=(-\infty,0)^2\cup \set{(y,x)\in (0,\infty)^2\given x\leq y},\\
 X&=\R\setminus \set 0,
\end{align*}
with the categorical structure inherited by the pair groupoid structure on $\R$. In this case, \[\G (C)=(-\infty,0)^2\cup \set{(x,x)\given x>0},\] which has two components of different dimensions. We have depicted this in Figure \ref{fig:disjoint_cat}, with invertible elements of $C$ in red, and non-invertible in blue; the units in $\Int C$ are depicted with a dashed line.

\begin{figure}[H]
\centering
\stepcounter{equation}
\includegraphics{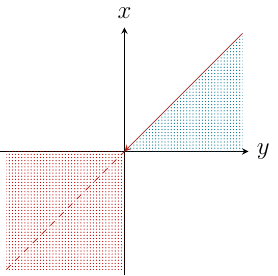}
\caption{The disjoint union of the order category with the pair groupoid.}
\label{fig:disjoint_cat}
\end{figure}
\noindent The culprit in the example above is that the units were allowed to be both in the interior and in the boundary.
\end{example}

\begin{samepage}
\begin{definition}
\label{def:normal_bdry}
A Lie category $C\rra X$ is said to have a \emph{normal boundary}, if either of the following holds:\
\begin{enumerate}[label={(\roman*)}]
\item $u( X)\subset \Int C$.
\item $\partial C$ is a wide subcategory of $C$, i.e., $\partial C$ is a subcategory of $C$ and $u( X)\subset\partial C$.
\end{enumerate}
\end{definition}
\end{samepage}
\begin{remark}
A Lie category without a boundary vacuously has a normal boundary. 
\end{remark}

The following result was proved by Charles Ehresmann in his pioneering paper \cite{ehr1959}, where he introduced Lie categories. Ehresmann proved it by implicitly assuming $\partial C=\emptyset$, and with the method of local coordinates; the proof we present is conceptually clearer since it is coordinate-free, and it holds for categories with a normal boundary.
\begin{theorem}
\label{thm:open_inv}
If $C$ is a Lie category with a normal boundary, then its core $\G (C)$ is an embedded Lie subcategory of $C$. More precisely, if $u( X)\subset \Int C$, then $\G (C)$ is open in $\Int C$, and if $u( X)\subset \partial C$, then $\G (C)$ is open in $\partial C$.
\end{theorem}
\begin{proof}
Consider first the case when $u( X)\subset \Int C$. 
We need to show that any $k\in \G (C)$ has an invertible neighborhood in $C$; we will do so by showing that $k$ admits a left-invertible and a right-invertible neighborhood. With this motive, define the map
\begin{align}
\vartheta\colon \comp C\rightarrow { C\tensor[_t]{\times}{_t}}\, C,\quad \vartheta(g,h)=(g,gh),
\end{align}
so $\vartheta=(\pr_1,m)$. We claim that this is a local diffeomorphism at the point $(k,k^{-1})$. Note that by equation \ref{eq:regbdr_c2} and a similar result for $\partial({ C\tensor[_t]{\times}{_t}}\, C)$, we have
\[
(k,k^{-1})\in\Int(\comp C)\quad\text{ and }\quad(k,1_{t(k)})=\vartheta(k,k^{-1})\in \Int({ C\tensor[_t]{\times}{_t}}\, C),
\]
so that it is enough to show $\d\vartheta_{(k,k^{-1})}$ is an isomorphism, by virtue of the usual inverse map theorem for manifolds without boundary. Due to dimensional reasons, it is enough to check $\vartheta$ is an immersion at $(k,k^{-1})$.

Suppose $\d\vartheta_{(k,k^{-1})}(v,w)=0$ for some $(v,w)\in T_{(k,k^{-1})}(\comp C)$. The identity 
$
\d \vartheta(v,w)=(v,\d m(v,w))
$ implies $v=0$, so we obtain
\begin{align*}
\d m_{(k,k^{-1})}(v,w)=\d (L_k)_{k^{-1}}(w).
\end{align*}
Since $L_k$ is a diffeomorphism, we conclude $w=0$, hence $\vartheta$ is a local diffeomorphism at $(k,k^{-1})$, i.e., there is a neighborhood $U$ of $(k,k^{-1})$ which is mapped diffeomorphically onto a neighborhood $V$ of $(k,1_{t(k)})$. Now note that $C$ embeds into ${ C\tensor[_t]{\times}{_t}}\, C$, by the map
$
\lambda(g)=(g,1_{t(g)}),
$
so the set $\vartheta^{-1}(V\cap\lambda (C))$ is diffeomorphic to $V\cap \lambda (C)$, and consists of pairs $(g,g')\in U$ such that $gg'=1_{t(g)}$. In other words, $\lambda^{-1}(V)$ is a neighborhood of $k$, whose elements all have right inverses. 

Similarly, we can show that the map $\tilde\vartheta\colon \comp C\rightarrow { C\tensor[_s]{\times}{_s}}\, C, (g,h)\mapsto(gh,h)$ is a local diffeomorphism at $(k^{-1},k)$, and use the embedding $\tilde\lambda\colon C\rightarrow { C\tensor[_s]{\times}{_s}}\, C$, $\tilde\lambda(g) = (1_{s(g)},g)$ to obtain a neighborhood $\tilde\lambda^{-1}(\tilde V)$ of $k$, whose elements all have left inverses. To conclude, note that Proposition \ref{prop:inv} guarantees $\lambda^{-1}(V)\cap \tilde\lambda^{-1}(\tilde V)$ is an invertible neighborhood of $k$.

Finally, for the case $u( X)\subset \partial C$, just note that normality of the boundary can be used with Lemma \ref{lem:lie_subcat} to conclude that $\partial C$ is an embedded Lie subcategory of $C$ without boundary, so we can apply the previous case to $\partial C$.
\end{proof}

Given a morphism $F\colon C\rightarrow D$ between Lie categories with normal boundaries, we have $F(\G (C))\subset \G(D)$ by functoriality, so we may define $\G(F)=F|_{\G (C)}$. The map 
\[\G\colon \mathbf{LieCat_\partial}\rightarrow\mathbf{LieGrpd}, \]
defines a functor from the category of Lie categories with a normal boundary to the category of Lie groupoids without boundary, which is easily seen to be right adjoint to the inclusion functor $\mathbf{LieGrpd}\hookrightarrow \mathbf{LieCat_\partial}$. 

Using the last theorem we can also show that the universal property of the core extends to the differentiable setting. We note that more consequences of the Theorem \ref{thm:open_inv} will be explored in upcoming sections.
\begin{corollary}
Let $C$ be a Lie category with a normal boundary, and let $H\subset C$ be a groupoid that is also a Lie subcategory of $C$. If any smooth morphism $F\colon G\rightarrow C$, defined on a Lie groupoid $G$, factors uniquely through $H$, then $H=\G (C)$. 
\end{corollary}
\begin{proof}
The inclusion $H\subset \G (C)$ is an easy consequence of functoriality of $H\hookrightarrow C$, since functors map isomorphisms to isomorphisms. For the converse inclusion, we pick the smooth map $F\colon \G (C)\hookrightarrow C$ and now the existence of $\bar F\colon \G (C)\rightarrow H$ such that $F=\iota\circ \bar F$, ensures that if $g\in \G (C)$, then $g=\bar F(g)\in H$.
\end{proof}

\section{The two Lie algebroids of a Lie category}
We begin this section by stating the following well-known definition.
\begin{definition}
	\label{def:algebroid}
	A \emph{Lie algebroid} is a vector bundle $A\ra X$, equipped with a Lie bracket on its space of sections $\Gamma(A)$ and a vector bundle map $\rho\colon A\ra TX$, such that the \emph{Leibniz rule} is satisfied for any $\alpha,\beta\in \Gamma(A)$:
	\[
		[\alpha,f\beta]=f[\alpha,\beta]+\rho(\alpha)(f)\beta.
	\]
	This axiom implies that $\rho\colon \Gamma(A)\ra\vf(X)$ is a Lie algebra homomorphism: $\rho[\alpha,\beta]=[\rho(\alpha),\rho(\beta)]$.
\end{definition}
The construction of a Lie algebroid using left-invariant (or right-invariant) vector fields on a Lie groupoid readily generalizes to Lie categories; however, we no longer have a canonical isomorphism between the two algebroids, previously given by inversion map. Below, we present the main idea of the construction (\ref{defn:left_inv} through \ref{prop:lie_alg}), mainly to fix notation; we will omit some details which can be found in standard Lie groupoid references.

\begin{definition}
\label{defn:left_inv}
A \emph{left-invariant} vector field on a Lie category $C\rra X$ is a vector field $X\in\vf (C)$, which is tangent to $t$-fibres, i.e., $X\in\Gamma(\ker \d t)$, and left-invariant, i.e., $\d(L_g)_h(X_h)=X_{gh}$ for all $(g,h)\in \comp C$. Denote by $\vf^L (C)$ the vector space of left-invariant vector fields on $C$.
\end{definition}

\begin{lemma}
\label{lem:closed_lie_bracket}
Let $C$ be a Lie category. The vector space $\vf^L (C)$ is closed under the Lie bracket of $TC$, and canonically isomorphic to the vector space $\Gamma(A^L (C))$ of sections of the vector bundle $A^L (C)=u^*(\ker\d t)$ over $ X$.
\end{lemma}
\begin{proof}
Closedness under the Lie bracket is a consequence of the fact that if $X,Y$ are $L_g$-related to themselves when restricted to appropriate fibres, so is $[X,Y]$. The canonical isomorphism $\mathrm{ev}\colon\vf^L (C)\rightarrow \Gamma(A^L (C))$ is given by restriction to the units, and its inverse $\mathrm{ev}^{-1}$ maps any section $\alpha\in \Gamma(A^L (C))$ to its left-invariant extension $\alpha^L$, given as
$
\alpha^L(g)=\d(L_g)_{1_{s(g)}}(\alpha_{s(g)}).
$
This is a smooth section $C\rightarrow T C$ since it can be realized as the composition $\alpha^L=\d m\circ \tau^\alpha$, where $\tau^\alpha\colon C\rightarrow T(\comp C)$ is given by $g\mapsto (0_g,\alpha_{1_{s(g)}}).$
\end{proof}
\begin{definition}
The \emph{left Lie algebroid} of a Lie category $C\rra X$ is defined as the vector bundle $A^L (C)\rightarrow X$, endowed with the Lie bracket $[\cdot,\cdot]$ on its sections as induced by the isomorphism $\mathrm{ev}$, together with the \emph{anchor map} $\rho^L\colon A^L (C)\rightarrow T X$, $\rho^L=\d s|_{A^L (C)}$.
\end{definition}
Similarly, the \emph{right Lie algebroid} of a Lie category $C\rra X$ is defined as the vector bundle $A^R(C)=u^*(\ker \d s)$. Its Lie bracket is induced by the identification of its sections with the space $\vf^R (C)$ of right-invariant vector fields, and its anchor map is defined as $\rho^R=\d t|_{A^R (C)}$. The proof of the next proposition is again the same as with the Lie groupoid case, so we omit it. 
\begin{proposition}
\label{prop:lie_alg}
The left and right Lie algebroids $A^L(C)$ and $A^R(C)$ of a Lie category $C\rra X$ are indeed Lie algebroids.
\end{proposition}

Any morphism $F\colon C\rightarrow D$ over $\id_ X$ of Lie categories over the same object manifold $ X$, induces morphisms between their left and right Lie algebroids (respectively), denoted by
\begin{align*}
F_*^L\colon A^L (C)\rightarrow A^L(D)\quad\text{and}\quad F_*^R\colon A^R (C)\rightarrow A^R(D),
\end{align*}
and defined on respective sections of $A^L (C)$ and $A^R (C)$ in the obvious way:
\begin{align*}
\alpha \mapsto \d F\circ \alpha.
\end{align*}
In the case when the object manifolds of $C$ and $D$ are not equal and the morphism $F$ does not restrict to a diffeomorphism between the units, we encounter the same complications as in the Lie groupoid case; this is resolved in the same manner as for Lie groupoids, and since we will not be needing this more general result, we point the reader to \cite{mackenzie}*{Chapter 4.3} for details. The upshot is that $A^L$ and $A^R$ are functors from the category $\mathbf{LieCat}$ of Lie categories to the category $\mathbf{LieAlgd}$ of Lie algebroids.

As mentioned, we do not have a canonical isomorphism between the left and right Lie algebroid of a Lie category, which is given in a Lie groupoid by the inversion map. However, we have the following result, and we will later encounter a related one when studying extensions of categories to groupoids (see \sec \ref{sec:ext}).
\begin{proposition}
Let $C\rra X$ be a Lie category. If the units of $C$ are contained in the interior of $C$, i.e., $u( X)\subset \Int C$, then the left and right Lie algebroids of $C$ are isomorphic to the Lie algebroid of its core $\G (C)$.
\end{proposition}
\begin{proof}
By Theorem \ref{thm:open_inv}, $\G (C)$ is open in $C$, so we get the following chain of isomorphisms of Lie algebroids:
\begin{align}
\label{eq:chain_iso}
A^L (C)\cong A^L(\G (C))\cong A^R(\G (C))\cong A^R (C),
\end{align}
where the first and last isomorphism are induced by the inclusion $\G (C)\hookrightarrow C$, and the isomorphism in the middle is induced by inversion in the groupoid $\G (C)$.
\end{proof}

\begin{remark}
Note that if $C$ has a normal boundary and $u( X)\subset \partial C$, the Lie algebroid $A(\G (C))$ of the core will always fail to be isomorphic to the two Lie algebroids of $C$, since the rank of the vector bundle $A(\G (C))$ is one less than the rank of $A^L (C)$ and $A^R (C)$. This is demonstrated by the following two examples: 
\begin{enumerate}[label={(\roman*)}]
\item The two Lie algebras of the Lie monoid $M=\mathbb H^n$ are isomorphic (as vector spaces) to $A_L(M)\cong \R^n\cong A_R(M)$, whereas $A(\G(M))\cong\R^{n-1}$ since the core of $M$ is $\G(M)=\R^{n-1}$.
\item Consider the order category $C=\set{(y,x)\in \R^2\given x\leq y}\rra \R$ from Example \ref{ex:order_cat}. Notice that its core $\G (C)$ is just the base groupoid over $\R$, hence its Lie algebroid is the zero bundle $A(\G (C))=\R\times \set 0$. On the other hand, the left and right Lie algebroid of $C$ are both isomorphic to $T\R$.
\end{enumerate}
\end{remark}

\section{Ranks of morphisms}
We now observe that the differentiable structure on a Lie category enables us to generalize the notion of rank from linear algebra.
\begin{definition}
\label{defn:rank}
Let $C\rra X$ be a Lie category and let $g\in C$. The \emph{left rank} and \emph{right rank} of $g$ are defined as
\begin{align*}
\rankl(g)&=\rank\d(L_g)_{1_{s(g)}},\\
\rankr(g)&=\rank\d(R_g)_{1_{t(g)}}.
\end{align*}
If the left and right rank of a morphism $g$ are equal, we just write $\rank(g)=\rankl(g)=\rankr(g)$ and call this integer the \emph{rank} of $g$.
Moreover, we say $g$ has \emph{full rank}, if its left and right ranks are full, that is, if
\[
\rank(g)=\codim_ C( X)=\dim C-\dim X=:\delta.
\]
In this case, we will sometimes say that $g$ is a \emph{regular} morphism. If $g$ is not regular, we will call it \emph{singular}. Finally, we say that $g$ has \emph{constant} left rank, if $\rankl(g)=\rank\d(L_g)_h$ for all $h\in C^{s(g)}$, and similarly that $g$ has \emph{constant} right rank, if $\rankr(g)=\rank\d(R_g)_h$ for all $h\in C_{t(g)}$. To avoid ambiguity, we will sometimes write $\rank_ C$ instead of $\mathrm{rank}.$
\end{definition} 
\begin{example}\
\begin{enumerate}[label={(\roman*)}]
\label{ex:ranks}
\item All invertible morphisms in any Lie category have full and constant rank.
\item In a Lie monoid $M$, the ranks of an element $g\in M$ are just the ranks of $\d(L_g)_e$ and $\d(R_g)_e$. If $M$ is an abelian Lie monoid, then clearly any element of $M$ has equal left and right rank. The example $M=\mathbb H^n$ shows that regularity does not imply invertibility. 

We will later see that regularity and constancy of ranks of all morphisms in a Lie category is ensured whenever dealing with a Lie subcategory of a Lie groupoid, as is with the simple example $\mathbb H^n\hookrightarrow \R^n$ for addition.
\item If $M=\R^{n\times n}$, let us show how the above notion of rank relates with the usual one from linear algebra. If $A\in\R^{n\times n}$, then the usual notion reads $\rank A=\dim \Im A$, when $A$ is seen as the map $\R^n\rightarrow \R^n$. On the other hand, the left rank in Definition \ref{defn:rank} equals:
\[\rankl_M(A)=\dim \Im \d(L_A)_I=\dim \Im L_A,\]
where $L_A\colon \R^{n\times n}\rightarrow \R^{n\times n}$ is the left translation by $A$, and the last equality follows by linearity. Similarly, $\rankr_M(A)=\dim\Im R_A$. Denoting by $E_{ij}$ the matrix with 1 in place $(i,j)$ and zero elsewhere, we have that $AE_{ij}$ has $i$-th column of $A$ in $j$-th column and is zero elsewhere, and $E_{ij}A$ has $j$-th row of $A$ in $i$-th row and is zero elsewhere, so:
\begin{align*}
\Im(L_A)&=\operatorname{Lin}(AE_{ij})_{i,j=1}^n=\set[\Big]{\begin{bmatrix}v_1\dots v_n\end{bmatrix}\given v_i\in \Im(A)},\\
\Im(R_A)&=\operatorname{Lin}(E_{ij}A)_{i,j=1}^n=\set*{\begin{bmatrix}v_1\dots v_n\end{bmatrix}^{\mathsf T}\given v_i\in \Im\left(A^{\mathsf T}\right)}.
\end{align*}
Since $\rank A=\rank A^{\mathsf T}$, it follows that 
\[
\rankl_M(A)=\rankr_M(A)=n\rank(A),
\]
and we see that $A$ is regular if, and only if, $A$ is invertible. 

The above result readily generalizes to arbitrary finite-dimensional vector spaces: if $W$ is a vector space, then
$\rank_{\End(W)}(A)=\dim (W)\rank (A).$ 
Moreover, for any vector bundle $V$, the above result clearly also generalizes to the endomorphism bundle $V\otimes V^*\ra M$, so that for any $A\colon V_x\rightarrow V_x$,
\[
\rank_{V\otimes V^*}A=\rank (V)\rank (A).
\]
\item Examples of Lie monoids that do not have coinciding left and right ranks may be found in the context of finite-dimensional unital algebras. As a concrete example, take the algebra $A\subset \R^{2\times 2}$ of upper-diagonal real $2\times 2$ matrices, the canonical basis of which is given by $a = \bigl[ \begin{smallmatrix}1 & 0\\ 0 & 0\end{smallmatrix}\bigr]$, $b = \bigl[ \begin{smallmatrix}0 & 0\\ 0 & 1\end{smallmatrix}\bigr]$, $c = \bigl[ \begin{smallmatrix}0 & 1\\ 0 & 0\end{smallmatrix}\bigr]$. Identifying $T_e A$ with $A$ and noting that left and right translations are linear maps, we can identify $\d(L_g)_e = L_g$ for any $g\in A$, and similarly for the right translations. It is easy to see $\Im(R_a)=\operatorname{Lin}(a)$ and $\Im(L_a)=\operatorname{Lin}(a,c)$ by computing all the products of $a$ with the canonical basis, so we conclude that $2=\rankl_A(a)\neq\rankr_A(a)=1$.
\end{enumerate}
\end{example}

\begin{proposition}[Properties of ranks]
In a Lie category $C\rra X$, there holds:
\begin{enumerate}[label={(\roman*)}]
\item The ranks of a composition of composable morphisms $g,h\in C$ are bounded from above:
\begin{align}
\rankl(gh)&\leq \rankl(h),\\
\rankr(gh)&\leq \rankr(g).
\end{align}
\item The ranks of $g\in C$ are bounded from below by the ranks of anchors $\rho^L\colon A^L (C)\rightarrow T X$ and $\rho^R\colon A^R (C)\rightarrow T X$ of Lie algebroids of $C$:
\begin{align}
\rankl(g)\geq \rank\rho_{s(g)}^L,\\
\rankr(g)\geq \rank\rho_{t(g)}^R.
\end{align}
\end{enumerate}
\end{proposition}
\begin{proof}
The proof is an application of the chain rule. More precisely, since $L_{gh}=L_g\circ L_h$, we have
\[
\rankl(gh)=\rank(\d(L_g)_h\circ\d(L_h)_{1_{s(h)}})\leq \rankl(h),
\]
and similarly for the right translation by using $R_{gh}=R_h\circ R_g$, so $(i)$ follows. For the property $(ii)$ note that the diagram
\[
\begin{tikzcd}[column sep=small]
 C^{s(g)} \arrow[rr, "L_g"] \arrow[rd, "s|_{ C^{s(g)}}"'] & & C^{t(g)} \arrow[ld, "s|_{ C^{t(g)}}"] \\
& X &
\end{tikzcd}
\]
implies $\rankl(g)\geq \rank \d(s|_{ C^{s(g)}})_{1_{s(g)}}=\rank \rho^L_{s(g)}$.
\end{proof}

In the context of Lie groupoids, it is well known that the composition map is a submersion. This is no longer true for Lie categories in general; a counterexample is provided by the Lie monoid $\R^{n\times n}$ for matrix multiplication, where $\d m_{(0,0)}$ is easily seen to be the zero map. However, the following results ensure that $m$ is a submersion in case all morphisms have full and constant rank; as mentioned, this is the case for Lie categories extendable to Lie groupoids, as we will see later in Lemma \ref{lem:extensions}.

\begin{lemma}
\label{lem:theta}
Let $C \rra X$ be a Lie category and $g\in C$. The map $L_g$ has full rank at $h\in C^{s(g)}$ if, and only if, the map 
$\vartheta\colon \comp C\rightarrow { C\tensor[_t]{\times}{_t}}\, C, (g,h)\mapsto(g,gh)$
has full rank at $(g,h)$. In particular, $g$ has full left rank if, and only if, $\vartheta$ has full rank at $(g,1_{s(g)})$.
\end{lemma}
Recall that we have already encountered the map $\vartheta$ in the proof of Theorem \ref{thm:open_inv}. In a Lie groupoid, this map is a bijection with inverse $(g,h)\mapsto (g,g^{-1}h)$, so above lemma guarantees it is a diffeomorphism, which can be used e.g., to show that the inversion map is automatically smooth, by realizing it as the composition \[g\mapsto (g,1_{t(g)})\xmapsto{\vartheta^{-1}} (g,g^{-1})\mapsto g^{-1}.\]

An analogous result to Lemma \ref{lem:theta} of course holds for ranks of right translations, using the map $\tilde\vartheta\colon \comp C\rightarrow { C\tensor[_s]{\times}{_s}}\, C$, given by $(g,h)\mapsto (gh,h)$. For instance, the second part of above lemma would then read: $g$ has full right rank if, and only if, $\tilde\vartheta$ has full rank at $(1_{t(g)},g)$. 
\begin{proof}
For the forward implication, note that $\d\vartheta_{(g,h)}(v,w)=(v,\d m_{(g,h)}(v,w))$ holds for all $(v,w)\in T_{(g,h)}\comp C$, hence $\d\vartheta_{(g,h)}(v,w)=(0,0)$ first implies $v=0$, and since $\d m_{(g,h)}(0,w)=\d (L_g)_h(w)$, we get $w=0$ since $L_g$ has full rank at $h$. For the other direction, note that $\d(L_g)_h(w)=0$ implies $\d\vartheta_{(g,h)}(0,w)=(0,0)$, so the assumption that $\vartheta$ has full rank at $(g,h)$ yields $w=0$.
\end{proof}
\begin{corollary}
\label{cor:composition_submersion}
Let $C\rra X$ be a Lie category and $g\in C$. If the map $L_g$ has full rank at $h\in C^{s(g)}$, then the composition map $m\colon\comp C\rightarrow C$ is a submersion at $(g,h)$. Hence, if all morphisms have full and constant rank, the composition is submersive.
\end{corollary}
\begin{proof}
Differentiating the identity $m=\pr_2\circ\vartheta$ yields 
$
\d m_{(g,h)}=\d(\pr_2)_{(g,gh)}\circ\d\vartheta_{(g,h)},
$
which is a composition of surjective maps by previous lemma and the fact that $\pr_2\colon { C\tensor[_t]{\times}{_t}}\, C\rightarrow C$ is a submersion.
\end{proof}
\noindent We now direct our attention to the subsets of $C$ of morphisms with full rank. Let us denote by
\begin{align*}
\rk_r^L (C)&=\set{g\in C\given \rankl(g)=r},\\
\rk_r^R (C)&=\set{g\in C\given \rankr(g)=r},
\end{align*}
the sets of morphisms with left and right rank equal to $r$, respectively, and furthermore by $\rk_r (C)=\rk_r^R (C)\cap \rk_r^L (C)$ morphisms whose both ranks equal $r$.


\begin{proposition}
\label{prop:regular_morphisms_open}
In any Lie category $C$, the subset  of regular morphisms $\rk_\delta (C)$ is open. 
\end{proposition}
\begin{proof}
Let us first show that differentials of left translations define a certain morphism of vector bundles. We consider the following vector bundle over $C$:
\[E^L=\coprod_{g\in C}(\ker\d t_{1_{s(g)}})^*\otimes\ker\d t_g=(s^*u^*\ker \d t)^*\otimes \ker \d t.\]
Denote the projection map by $p_L\colon E^L\rightarrow C$, so the fibre $p_L^{-1}(g)=E^L_g$ of $E^L$ consists of all linear maps $\ker \d t_{1_{s(g)}}\rightarrow \ker\d t_g$. Note that the map $\phi_L\colon C\rightarrow E^L$, given as 
\[g\mapsto \d(L_g)_{1_{s(g)}}\colon \ker \d t_{1_{s(g)}}\rightarrow \ker\d t_g,\] is a section of $E^L$, which is smooth since 
\[
\d (L_g)_{1_{s(g)}}=\d m_{(g,1_{s(g)})}(0_g,-).
\]
In other words, $\phi_L$ defines a morphism $s^*u^*\ker\d t\rightarrow \ker \d t$ of vector bundles. Now take an atlas of local trivializations $\psi_i\colon p_L^{-1}(U_i)\rightarrow U_i\times \R^{\delta\times\delta}$ of $E^L$. The set $U_i\cap \rk^L_\delta(C)=\phi_L^{-1}\circ\psi_i^{-1}(U_i\times \mathrm{GL}(\delta,\R))$ is open in $U_i$, hence
\[
\rk_\delta^L (C)=\bigcup_i \big(U_i\cap \rk_\delta^L (C)\big)
\]
is open in $C$. A similar proof works for $\rk^R_\delta (C)$, so the result for $\rk_\delta (C)$ follows.
\end{proof}

\begin{remark}
The smooth map $\phi_L$ (respectively, $\phi_R$) which was used in the last proposition, can also easily be used to show that $\rankl$ (respectively, $\rankr$) is a lower semi-continuous function, i.e., that any morphism $g\in C$ admits a neighborhood on which the left (respectively, right) rank is non-decreasing. Note that this should not be confused with lower semi-continuity of the map $C^{s(g)}\rightarrow \R$, given as $h\mapsto \rank \d(L_g)_h$, where $g\in C$ is a fixed morphism.
\end{remark}

Properties of ranks of morphisms in a Lie category $C$ reflect on its algebraic structure, as illustrated by the following simple observation.

\begin{corollary}
Let $C$ be a Lie category. If the regular morphisms have constant rank, then $\rk_\delta (C)$ is an open Lie subcategory of $C$. Hence, in this case, the left and right algebroids of $\rk_\delta (C)$ are isomorphic to those of $C$, respectively.
\end{corollary}
\begin{proof}
We only have to check that the composition of two regular morphisms is a regular morphism. To that end, just notice that
$
\d(L_{gh})_{1_{s(h)}}=\d(L_g)_h\circ\d(L_h)_{1_{s(h)}}
$
is a composition of maps with full rank, and similar holds for right translations.
\end{proof}


\subsubsection{Action of the core $\G (C)$ on a Lie category $C$}

The natural left and right actions of the core $\G (C)$ on a Lie category $C$ can be used to find properties of Lie categories.\footnote{A reference for actions of Lie groupoids on smooth maps is \cite{mackenzie}*{Chapter 1.6}.} We will only focus on describing the left action of $\G (C)$ on $t\colon C\rightarrow X$; the right action of $\G (C)$ on $s\colon C\rightarrow X$ follows a similar construction. 

Denote by $\G (C) \tensor[_s]{\times}{_t} C=\set{(g,c)\in \G (C)\times C\given s(g)=t(c)}$ the fibred product of the maps $s|_{\G (C)}$ and $t$. This set has a natural structure of a groupoid over $C$. Indeed, the source and target maps are given as $\underline s(g,c)=c$ and $\underline t(g,c)=gc$, the composition is given as $(g,hc)(h,c)=(gh,c)$, the unit map is $\underline u(c)=(1_{t(c)},c)$ and the inverses are given by $(g,c)^{-1}=(g^{-1},gc)$. We leave it to the reader to check that this defines a groupoid over $C$. Let us deal with its smooth structure.

\begin{proposition}
Let $C\rra X$ be a Lie category without boundary. The left action of the core $\G (C)$ on the target map $t\colon C\rightarrow X$ yields a Lie groupoid
\[
\G (C)\tensor[_s]{\times}{_t} C\rra C.
\]
\end{proposition}
\begin{proof}
The core $\G (C)$ is an open Lie subgroupoid of $C$ by Theorem \ref{thm:open_inv}, and so the usual transversality theorem for manifolds without boundary implies that 
\[\underline C=\G (C)\tensor[_s]{\times}{_t} C=(s|_{\G (C)}\times t)^{-1}(\Delta_ X)\]
is a smooth submanifold of $\G (C)\times C$ without boundary. Furthermore, $\underline s=\mathrm{pr}_2$ is a submersion, and Corollary \ref{cor:composition_submersion} tells us that the target map \[\underline t=m|_{\G (C)\tensor[_s]{\times}{_t} C}\] is also a submersion, thus $\underline C{}^{(2)}\subset \underline C\times\underline C$ is also a submanifold by transversality. The smoothness of composition $\underline m$ and unit map $\underline u$ then follows from the smoothness of respective maps in $C$.
\end{proof}
\begin{remark}
\label{rem:base_boundary}
Notice that in the case when $C$ has a normal boundary, the maps $\underline s$ and $\underline t$ are still submersions, but their restrictions to the boundary $\partial\underline C=\G (C)\times \partial C$ are not, which shows the need for amending the definition of a Lie category when the object manifold has a boundary (or corners). 

In this case, $\underline C\subset \G (C)\times C$ is still a submanifold (by Proposition \ref{prop:transversality}) and the maps $\underline s$ and $\underline t$ are topological submersions (see Definition \ref{defn:top_sub}) that also satisfy $\underline s(\partial\underline C)\subset \partial C$ and $\underline t(\partial\underline C)\subset \partial C$, so \cite{corners}*{Proposition 4.2.1} implies that $\underline s$-fibres and $\underline t$-fibres are submanifolds of $\underline C$. This suggests a definition of a Lie category for the case when the object manifold has corners, but we will not pursue this further here.
\end{remark}
We can now use the action Lie groupoids above to prove that it does not matter at which invertible morphism we measure the ranks of a given morphism $g\in C$.
\begin{corollary}
\label{cor:ranks_invertibles}
Let $C\rra X$ be a Lie category with a normal boundary. For any $g\in C$,
\begin{align*}
\rankl(g)&=\rank \d(L_g)_h\ \text{ for any }h\in \G (C)^{s(g)},\\
\rankr(g)&=\rank \d(R_g)_h\ \text{ for any }h\in \G (C)_{t(g)}.
\end{align*}
\end{corollary}
\begin{proof}
Consider first the case when $\partial C=\emptyset$. The restriction of the target map to any source fibre has constant rank, in any Lie groupoid. In our case, $\underline s^{-1}(g)=\G (C)_{t(g)}$ for any $g\in C$, and the map
\[\underline t|_{\underline s^{-1}(g)}=R_g|_{\G (C)_{t(g)}}\]
has constant rank which must thus be equal to $\rankr(g)$. A similar proof works for left translations, using the right action of $\G (C)$ on $s\colon C\rightarrow X$.

Now suppose $\partial C\neq \emptyset$. We may consider the left action groupoids 
\[
\G (C)\tensor[_s]{\times}{_t}\Int C\rra\Int C,\quad \G (C)\tensor[_s]{\times}{_t}\partial C\rra\partial C
\]
of $\G (C)$ on $t|_{\Int C}$ and $t|_{\partial C}$, respectively; note that these are in fact actions of $\G (C)$ since $L_g(\Int C^{s(g)})\subset \Int C^{t(g)}$ and $L_g(\partial C^{s(g)})\subset\partial C^{t(g)}$ for any $g\in\G (C)$. The same technique as above now shows the wanted conclusion.
\end{proof}

\subsubsection{Singular distributions associated to translations}

A Lie category $C\rra X$ (assume it is boundaryless for simplicity) comes equipped with singular distributions determined by differentials of left and right translations. Focusing only on left translations, define the singular distribution $D\subset T C$ as
\[D_g=\Im\d(L_g)_{1_{s(g)}}\leq \ker\d t_g\leq T_g C.\]
In what follows, we show that $D$ is \emph{integrable}, i.e., there is a decomposition $\F(D)$ of $C$ into maximally connected weakly embedded submanifolds,\footnote{Recall that an injective immersion $\varphi\colon M \rightarrow N$ is said to be a \emph{weak embedding}, if any smooth map $f\colon P\rightarrow N$ with the property $f(P)\subset \varphi(M)$, factors through $\varphi$.} called the \emph{leaves} of $\F(D)$, whose tangent spaces coincide with the fibres of $D$. 

Denote the $D$-valued vector fields on $C$ by
\[\Gamma(D)=\set{X\in \vf (C)\given X_g\in D_g\text{ for all }g\in C}.\]
It is not hard to see that $D$ is \emph{locally of finite type}, i.e., for any $g\in C$ there exist finitely many $X_i\in \Gamma(D)$ such that:
\begin{enumerate}[label={(\roman*)}]
\item $D_g=\mathrm{Span}(X_i|_g)_i$,
\item For any $X\in \Gamma(D)$ there exists a neighborhood $U$ of $g$ in $C$ and functions $f_i{}^j\in C^\infty(U)$, such that $[X,X_i]_h=\sum_j f_i{}^j(h)X_j|_h$ for all $h\in U$.
\end{enumerate}
Indeed, we may pick sections $(\alpha_i)_i$ of $A^L (C)$, such that they constitute a local frame on some neighborhood $V$ of $s(g)$ in $ X$, and extend them to left-invariant vector fields $(X_i)_i$. The point (i) is clearly satisfied; to show (ii), denote $U=s^{-1}(V)$, and note that that the set of left-invariant vector fields is closed under the Lie bracket by Lemma \ref{lem:closed_lie_bracket}, so on $U$ there holds
\[
[X_i,X_j]=\textstyle\sum_k f_{ij}{}^k X_k,
\]
for some functions $f_{ij}{}^k\in C^\infty(U)$. Since any $X\in\Gamma(D|_U)$ can be written as a $C^\infty(U)$-linear combination of the tuple $(X_i|_U)_i$, the rest follows by using the Leibniz rule for the Lie bracket. 

Since $D$ is locally of finite type, it is integrable by \cite{sussman}. The following proposition says that the leaves are precisely the connected components of the orbits \[\set*{\underline s(\underline t^{-1}(g))\given g\in C}\] of the groupoid $C \tensor[_s]{\times}{_t} \G (C)\rra C$ corresponding to the right action of $\G (C)$ on $s\colon C\rightarrow X$.

\begin{proposition}
Let $C\rra X$ be a Lie category without boundary. The integral manifold of the singular distribution $D\subset T C$ through $g\in C$ is $L_g(\G (C)^{s(g)})$.
\end{proposition}
\begin{proof}
Corollary \ref{cor:ranks_invertibles} states $\smash{L_g|_{\G (C)^{s(g)}}}$ has constant rank, so rank theorem can be applied to deduce $L_g(\G (C)^{s(g)})$ is an immersed submanifold of $C^{t(g)}$, the tangent space at $gh$ of which is the image of differential $\d(L_g)_h$, for any $h\in \G (C)^{s(g)}$.
\end{proof}

\begin{remark}
	In the case when $C$ has a boundary, the last proposition is in general not true, as is easily seen by considering the order category on $\R$. In general, a natural candidate for the integral manifold of $D$ through $g$ is $L_g(U)$ for an appropriate open neighborhood of $1_{s(g)}$ in $C^{s(g)}$, but proving such a general result is harder since rank theorem is not true for manifolds with boundaries (without additional assumptions on $L_g$).
\end{remark}

\section{Extensions of Lie categories to Lie groupoids}
\label{sec:ext}
We have witnessed interesting examples of Lie categories appear by restricting to an embedded subcategory of a Lie groupoid, in spirit of Lemma \ref{lem:lie_subcat}. As we will see next, in order to ascertain the main properties of such Lie categories, the assumption of being embedded may be weakened to being immersed. 

\begin{definition}
An \emph{extension to a Lie groupoid} of a Lie category $C\rightrightarrows X$ is a Lie groupoid $G\rightrightarrows X$, together with an injective, immersive functor $F\colon C\rightarrow G$ over the identity. In other words, $G$ is a Lie groupoid such that $C$ is its wide Lie subcategory. If such an extension exists, we say that $C$ is \emph{extendable to a Lie groupoid}. Furthermore, we say that an extension to a Lie groupoid is \emph{weakly étale}, if $\dim G=\dim C$.
\end{definition}
\begin{remark}
Note that if $C$ does not have a boundary, then an extension is weakly étale if, and only if, the map $F$ is étale (i.e., a local diffeomorphism), by virtue of inverse map theorem. If $\dim G=\dim C$ and $\partial C\neq \emptyset$, then $F$ cannot be a local diffeomorphism at points from $\partial C$ since $\partial G=\emptyset$ (see e.g., the order category from Example \ref{ex:order_cat}).
\end{remark}

An obvious necessary condition for an arbitrary category $C$ to admit an injective functor $F$ into a groupoid $ G$, is \emph{cancellativity} of all elements in $C$, i.e., all left and right translations in $C$ must be injective. For example, if $gh=gk$ holds for some morphisms in $C$, then $F(g)F(h)=F(g)F(k)$ holds in $F (C)\subset  G$, implying $F(h)=F(k)$, and hence $h=k$ by injectivity of $F$, so $L_g$ is injective.

In the differentiable setting, extensions to groupoids reflect on ranks and algebroids; the following lemma yields necessary conditions on a Lie category $C$ to be extendable to a Lie groupoid.
\begin{lemma}
\label{lem:extensions}
If a Lie category $C$ is extendable to a Lie groupoid $ G$, all its morphisms have full and constant rank, and all left and right translations are injective. Moreover, if the extension is weakly étale, then $A^L (C)\cong A( G)\cong A^R (C)$.
\end{lemma}
\begin{proof}
Let $F\colon C\rightarrow  G$ be the groupoid extension, and let $g,h\in C$ be composable. To prove that $g$ has full and constant rank, denote the left translation in $C$ by $g$ as $L_g^ C$, and the left translation in $ G$ by $F(g)$ as $L^ G_{F(g)}$. By functoriality, the diagram
\[\begin{tikzcd}[column sep=2.5em]
	{ C^{s(g)}} & { C^{t(g)}} \\
	{ G^{s(g)}} & { G^{t(g)}}
	\arrow["{L_g^ C}", from=1-1, to=1-2]
	\arrow["{L_{F(g)}^ G}"', from=2-1, to=2-2]
	\arrow["F"', from=1-1, to=2-1]
	\arrow["F", from=1-2, to=2-2]
\end{tikzcd}\]
commutes, so if $v\in T_h( C^{s(g)})$ is such that $\d(L_g^ C)_h(v)=0$, it implies $\d F_h(v)=0$, and so $v=0$ since $F$ is an immersion. This proves that $\d(L_g^ C)_h$ has full rank, and a similar proof shows an analogous result for the right translation. Hence, $g$ has full and constant rank.

To prove the second part, denote the source and target maps in $C$ and $ G$ as $s^ C, t^ C$ and $s^ G,t^ G$, respectively. Commutativity of the following diagram,
\[\begin{tikzcd}[column sep=1.25em]
	 C &&  G \\
	& X
	\arrow["F", from=1-1, to=1-3]
	\arrow["{t^ C}"', from=1-1, to=2-2]
	\arrow["{t^ G}", from=1-3, to=2-2]
\end{tikzcd}\]
together with the assumption that our extension is weakly étale, implies that for any $x\in X$, $\d F_{1_x}$ maps $\ker \d t^ C_{1_x}$ isomorphically onto $\ker \d t^ G_{1_x}$, so $F$ induces an isomorphism \[F_*^L\colon A^L (C)\rightarrow A^L( G)\] of vector bundles, and similarly between $A^R (C)$ and $A^R( G)$. Since $F$ is a morphism of Lie categories, these are in fact isomorphisms of Lie algebroids, so we yield the wanted chain of isomorphisms.
\end{proof}

\begin{corollary}
If a Lie category is extendable to a Lie groupoid, then the composition map $m\colon\comp C\rightarrow C$ is a submersion.
\end{corollary}
\begin{proof}
Follows directly from Corollary \ref{cor:composition_submersion} and Lemma \ref{lem:extensions}.
\end{proof}

\begin{remark}
To conclude $A^L (C)\cong A^R (C)$, it is enough to replace the assumption on the extension $F\colon C\rightarrow  G$ being weakly étale with the following weaker condition: 
\begin{align}
\label{eq:nice_extension}
\d(\inv)_{1_x}(\d F_{1_x}(\ker \d t_{1_x}^ C))\subset\d F_{1_x}(\ker \d s^ C_{1_x}),
\end{align}
but notice that if $\dim  G\neq \dim C$, then the left and right algebroids of $C$ will not be isomorphic to the algebroid $A( G)$, since their ranks will differ. For simpler notation, assume that $C\subset  G$ and $F\colon C\hookrightarrow G$ is an injective immersion.

To show that $\d(\inv)_{1_x}(\ker \d t_{1_x}^ C)\subset\ker \d s^ C_{1_x}$ implies $A^L (C)\cong A^R (C)$, first observe that this is an equality instead of just an inclusion. Hence, the inversion map on the extension induces a vector bundle isomorphism \[\inv_*\colon A^L (C)\rightarrow A^R (C),\]
and now it is not difficult to see that for any $\alpha\in  \Gamma(A^L (C))$, its left-invariant extension $\alpha^L$ to whole $ G$ is $\inv$-related to the right-invariant extension $(\inv_*\alpha)^R$. Hence, if $\beta\in \Gamma(A^L (C))$ is another section, we obtain that for any $g\in G$,
\[
\d(\inv)_g([\alpha^L,\beta^L]_g)=[(\inv_*\alpha)^R,(\inv_*\beta)^R]_{g^{-1}}.
\]
Taking $g=1_x$ shows that $\inv_*$ preserves the brackets of $A^L (C)$ and $A^R (C)$, so it is an isomorphism of Lie algebroids.
\end{remark}

\begin{remark}
\label{rem:questions}
The importance of Lemma \ref{lem:extensions} is in the fact that it enables us to easily provide positive answers to the following questions, for the case of Lie categories extendable to groupoids:
\begin{enumerate}[label={(\roman*)}]
\item Are Lie monoids parallelizable?
\item For fixed objects $x,y\in X$ of a Lie category $C\rra X$, is the set $C_x^y$ of morphisms from $x$ to $y$ a submanifold of $C$? Equivalently, do Hom-functors map into the category $\mathbf{Diff}$ instead of just $\mathbf{Set}$?
\end{enumerate}
The proofs in the extendable case are the same as the proofs for Lie groupoids (or Lie groups). In fact, these results rely only on the fact that arrows have full rank; the questions above remain open for Lie categories whose arrows do not have full rank.
\end{remark}


\begin{corollary}
If a Lie monoid $M$ is extendable to a Lie group, it is paralellizable.
\end{corollary}
\begin{proof}
Extending a basis of $T_eM$ to a tuple of left-invariant vector fields yields a global frame for $TM$ since all elements of $M$ have full rank by Lemma \ref{lem:extensions}.
\end{proof}

\begin{corollary}
\label{cor:hom_sets_embedded}
If a Lie category $C\rra X$ is extendable to a Lie groupoid, then for any $x,y\in X$ the set $C_x^y$ is a closed embedded submanifold of $C$. Hence, $C_x^x$ is a Lie monoid for any $x\in X$.
\end{corollary}
\begin{proof}
We realize $C_x^y$ as an integral manifold of a certain distribution on $C_x$, namely $D=\ker \d(t|_{ C_x})$, or more instructively, $D_g=\ker \d t_g\cap \ker \d s_g$ for all $g\in C_x$. This is a regular distribution on $C_x$, since $D_g=\d(L_g)_{1_x}(D_{1_x})$ holds---the latter is a consequence of the equality $s|_{ C^{t(g)}}\circ L_g=s|_{ C^x}$ and $g$ having full left rank by Lemma \ref{lem:extensions}, which moreover implies that $D$ is a trivial vector subbundle of $T C_x$. Summarizing, we have shown the map $t|_{C_x}$ has constant rank.

Since $D$ is the kernel of a differential of a smooth map, it is involutive, so by Frobenius' theorem integrable. The leaves of the corresponding foliation are the connected components of subspaces $\set{ C_x^y\given y\in X}$ of $C_x$, so they are initial submanifolds. Since the subspaces $C_x^y$ are also closed in $C_x$, they are embedded.
\end{proof}

\section{Completeness of invariant vector fields}
It is well-known that on any Lie group $G$, left-invariant vector fields are complete. In this section, we generalize this result to Lie monoids with normal boundaries. Furthermore, we generalize the characterization of completeness of left-invariant vector fields on Lie groupoids to Lie categories with normal boundaries. At last, the exponential map for Lie monoids will be discussed.

\begin{definition}
Let $X$ be a vector field on a smooth manifold $M$ with or without boundary, and denote by $J^X_g$ the maximal interval on which the integral path $\gamma^X_g$ of $X$, starting at $g\in M$, is defined. We say that $X$ is \emph{half-complete} if for any $g\in M$ there either holds $[0,\infty)\subset J^X_g$ or $(-\infty,0]\subset J^X_g$.
\end{definition}

\begin{remark}
\label{rem:inward}
In what follows, we will assume the reader is familiar with the notion of inward-pointing and outward-pointing tangent vectors on the boundary; we direct to \cite{lee}*{p.\ 118} for a basic reference. A particularly useful observation is that a vector field $X\in\vf(M)$ which is inward-pointing (or outward-pointing) at a certain point $g\in \partial M$, must remain inward-pointing (or outward pointing) on a neighborhood of $g$ in $\partial M$, from which it follows that $J_g^X$ cannot contain an open neighborhood of zero, but may only contain a half-closed interval $[0,\varepsilon)$, for some $\varepsilon>0$ (or $(-\varepsilon,0]$ in the outward-pointing case).
\end{remark}

Recall from Definition \ref{def:normal_bdry} that a Lie monoid $M$ has a normal boundary, if either $e\in\Int M$, or we have both that $e\in\partial M$ and $\partial M$ is a submonoid of $M$.

\begin{theorem}
\label{thm:complete}
Let $M$ be a Lie monoid with a normal boundary and let $X$ be a left-invariant vector field on $M$. The following holds:
\begin{enumerate}[label={(\roman*)}]
\item Suppose either that $e\in\Int M$, or that $e\in\partial M$ and $X_e$ is tangent to $\partial M$. Then $X|_{\partial M}$ is tangent to $\partial M$, and $X$ is complete.
\item Suppose $e\in\partial M$ and $X_e$ is either inward-pointing or outward-pointing. Then either $J_e^X=[0,\infty)$ or $J_e^X=(-\infty,0]$, respectively, and $X$ is half-complete.
\end{enumerate}
Moreover, the flow of $X$ is given for all $t\in J_e^X$ by $\phi_t^X=R_{\phi_t^X(e)}$.
\end{theorem}
\begin{proof}
We first inspect the assumptions from (i): notice that if $e\in\Int M$, then $J_e^X$ clearly contains an open interval around zero; on the other hand, if $e\in\partial M$ and $X_e$ is tangent to $\partial M$, then $X|_{\partial M}$ must be everywhere tangent to $\partial M$ by left-invariance of $X$ and the fact that $\partial M$ is a Lie submonoid of $M$, so $J_e^X$ must again contain an open interval around zero. Furthermore, the assumption from (ii) that $X_e$ is either inward-pointing or outward-pointing implies that $J_e^X$ contains a half-open interval of the form $[0,\varepsilon)$ or $(-\varepsilon,0]$, respectively.

We next observe that for any $g\in M$, the composition $L_g\circ \gamma_e^X$ is an integral path of $X$ starting at $g$, so maximality of $J_{g}^X$ implies
\begin{align}
\label{eq:eunderg}
J_e^X\subset J_g^X,
\end{align} 
and also $\phi_t^X(g)=g\gamma_e^X(t)=R_{\gamma_e^X(t)}(g)$ for all $t\in J_e^X$. Notice that \eqref{eq:eunderg} now implies that in the case $e\in\Int M$, $J_g^X$ contains an open interval around zero for all $g\in M$, and this holds for any $g\in \partial M$, so we conclude that $X|_{\partial M}$ must be tangent to $\partial M$, by virtue of Remark \ref{rem:inward}. Let $\tau\in J_e^X$ and consider the affinely translated path 
\[\zeta^\tau\colon (J_e^X-\tau)\rightarrow M, \quad \zeta^\tau (t)=\gamma_e^X(t+\tau).\]
Since the maximal domain of an affinely translated integral path is just the affinely translated maximal domain, we get $J_e^X-\tau=J_{\zeta^\tau(0)}^X$.\footnote{We do not need the Lie monoid structure for this fact.} 
Together with \eqref{eq:eunderg}, this implies
\[
J_e^X+\tau\subset J_e^X,\ \text{for all }\tau\in J_e^X.
\]
This implies that if $J_e^X$ contains an open interval around zero, it must equal $\R$, and if $[0,\varepsilon)\subset J_e^X$ for some $\varepsilon>0$, then $[0,\infty)\subset J_e^X$; similarly for the outward-pointing case. Together with equation \eqref{eq:eunderg}, this proves our claims regarding completeness and half-completeness.
\end{proof}

\begin{corollary}
If $M$ is a Lie monoid and $e\in \Int M$, then elements in $\partial M$ do not have full rank. Hence, if also $\partial M\neq \emptyset$, then $M$ is not extendable to a Lie group.
\end{corollary}
\begin{proof}
For a proof by contradiction, assume that there is a $g\in\partial M$ with full rank, and pick any inward-pointing vector $v\in T_g M$. Since $g$ has full rank, the vector $v$ is extendable to a unique left-invariant vector field $X$ on $M$, but now $X|_{\partial M}$ must be tangent to $\partial M$ by Theorem \ref{thm:complete}, contradicting our assumption that $v$ is inward-pointing. The second part of the corollary follows from Lemma \ref{lem:extensions}.
\end{proof}

An easy example of the last corollary in play is the Lie monoid $[0,\infty)$ for multiplication. This result may easily be generalized; in any Lie category $C\rra X$, a similar inclusion as \eqref{eq:eunderg} holds:
\begin{align}
\label{eq:unitunderg}
J_{1_{s(g)}}^X\subset J_g^X,
\end{align}
for any $g\in C$ and any left-invariant vector field $X\in\vf^L (C)$. The same argument as before shows that if $u( X)\subset \Int C$, then $X|_{\partial C}$ is tangent to $\partial C$, so we similarly obtain: 

\begin{corollary}
If $C\rra X$ is a Lie category with $u( X)\subset \Int C$, then morphisms in $\partial C$ do not have full rank. Hence, if $u( X)\subset\Int C$ and $\partial C\neq \emptyset$, then $C$ is not extendable to a Lie groupoid.
\end{corollary}

\begin{proof}
As before, suppose there is a morphism $g\in\partial C$ with full rank. Regularity of the boundary implies $\partial ( C^{t(g)})=\partial C\cap C^{t(g)}$, so we may pick an inward-pointing vector $v\in \ker \d t_g$. Since $g$ has full rank, we can form the vector $\d(L_g)_{1_{s(g)}}^{-1}(v)\in \ker\d t_{1_{s(g)}}$, extend it to a section of $u^*\ker\d t$ using partitions of unity, and finally extend it to a left-invariant vector field $X$ on $C$. Since $X_g=v$, we arrive to a contradiction to the fact that $X|_{\partial C}$ is tangent to $\partial C$. 
\end{proof}



\begin{remark}
\label{rem:gpd_empty_bdry}
In the case when $g$ is an invertible morphism of a Lie category, the inclusion \eqref{eq:unitunderg} is actually an equality, which follows from the fact that $L_{g^{-1}}\circ\gamma^X_g$ is an integral path of $X$ starting at $1_{s(g)}$. 
\end{remark}

We now present the promised characterization of completeness of invariant vector fields on a Lie category. The proof is a small adaptation of the one from the theory of Lie groupoids.

\begin{proposition}
\label{prop:completeness_characterization}
Let $C\rra X$ be a Lie category with a normal boundary, and let $\alpha\in \Gamma(A^L (C))$ be a section of its left Lie algebroid. Suppose either $u( X)\subset \Int C$, or that $u( X)\subset\partial C$ and $\alpha^L|_{u( X)}$ is tangent to $\partial C$. Then $\alpha^L\in\vf^L (C)$ is complete if, and only if, $\rho^L(\alpha)\in\vf( X)$ is complete.
\end{proposition}
\begin{proof}
First note we have already argued that under given assumptions, the restriction $\alpha^L|_{\partial C}$ is tangent to $\partial C$, implying that the maximal domain of any integral path of $\alpha^L$ is an open interval.

For the forward implication, note that $\alpha^L$ and $\rho^L(\alpha)$ are $s$-related, so that if $\phi^{\alpha^L}_t$ is defined for some $t\in\R$, then so is $\phi^{\rho^L(\alpha)}_t$, by the fact that $s$ is a surjective submersion.
\[\begin{tikzcd}
	 C & C \\
	 X & X
	\arrow["s"', from=1-1, to=2-1]
	\arrow["{\phi^{\alpha^L}_t}", from=1-1, to=1-2]
	\arrow["{\phi^{\rho^L(\alpha)}_t}"', from=2-1, to=2-2]
	\arrow["s", from=1-2, to=2-2]
\end{tikzcd}\]
Indeed, take an open cover $(U_i)_i$ of $ X$ by domains of local sections $\sigma_i\colon U_i\rightarrow C$ of $s$, and define $\vartheta_i\colon \R\times U_i\rightarrow X$ as $\vartheta_i(t,x)=(s\circ\phi_t^{\alpha^L}\circ\sigma)(x)$. It is straightforward to check that for any $x\in U_i$, the map $t\mapsto \vartheta_i(t,x)$ is the integral path of $\rho^L(\alpha)$ starting at $x$, so the maps $\vartheta_i$ collate to the global flow $\phi^{\rho^L(\alpha)}\colon\R\times X\rightarrow X$ of $\rho^L(\alpha)$, by uniqueness of integral paths.

For the converse implication, suppose $\rho^L(\alpha)$ is complete, and let $\gamma\colon (a,b)\rightarrow C$ be an integral path of $\alpha^L$; then $s\circ\gamma$ is an integral path of $\rho^L(\alpha)$, thus $s\circ\gamma$ admits a unique extension to an integral path of $\rho^L(\alpha)$ defined on whole $\R$, which we again denote by $s\circ\gamma$, so the expression $(s\circ\gamma)(b)$ is defined. There now exists a path $\delta\colon (b-\varepsilon,b+\varepsilon)\rightarrow C$, which is an integral path of $\alpha^L$, such that $\delta(b)=1_{(s\circ\gamma)(b)}$; importantly, we may assume that $\varepsilon$ is small enough that $\delta$ maps into $\G (C)$, since the latter is either open in $\Int C$ or in $\partial C$ by Theorem \ref{thm:open_inv}. We define our wanted extension $\bar\gamma\colon (a,b+\varepsilon)\rightarrow 
 C$ of $\gamma$ as
\[
\bar\gamma(t)=
\begin{cases}
\gamma(t)&\text{if }t\in(a,b),\\
\gamma(b-\frac{\varepsilon}2)\delta(b-\frac\varepsilon 2)^{-1}\delta(t)&\text{if }t\in(b-\varepsilon,b+\varepsilon).
\end{cases}
\]
Since both $s\circ\gamma$ and $s\circ\delta$ are integral paths of $\rho^L(\alpha)$ valued $(s\circ\gamma)(b)$ at $b$, they coincide on their common domain, so that $s(\gamma(b-\frac\varepsilon 2))=s(\delta(b-\frac\varepsilon 2))$, and moreover $\delta$ must lie in $C^{(s\circ\gamma)(b)}$ since $\alpha^L$ is tangent to $t$-fibres, which altogether implies that the multiplication in the definition of $\bar \gamma$ is well-defined. Finally, $\bar \gamma$ is in fact an integral path of $\alpha^L$ since for any $t\in (b-\varepsilon,b+\varepsilon)$ there holds:
\[
\bar\gamma'(t)=\d{\big(L_{\gamma(b-\frac\varepsilon 2)\delta(b-\frac\varepsilon 2)^{-1}}\big)}_{\delta(t)}(\delta'(t))=\alpha^L_{\gamma(b-\frac\varepsilon 2)\delta(b-\frac\varepsilon 2)^{-1}\delta(t)}=\alpha^L_{\bar\gamma(t)}.
\]
By uniqueness of integral paths, the two partial definitions of $\bar\gamma$ coincide on their common domain, from which we conclude that $\bar\gamma$ indeed extends $\gamma$. A similar trick can be used for the other endpoint, showing that the integral path $\gamma$ can be extended to the whole $\R$.
\end{proof}

Let us now turn back to Lie monoids. As a corollary of Theorem \ref{thm:complete}, we can generalize the exponential map from the theory of Lie groups to Lie monoids. The following corollary says that when $e\in\Int M$, the map $T_eM\rightarrow M$, defined as $v\mapsto \phi^{v^L}_1(e)$ is just the usual exponential map on $\G(M)$.

\begin{corollary}
\label{cor:exp_grp}
Let $M$ be a Lie monoid with $e\in \Int M$. The image of the integral path $\gamma_e^X$ of any left-invariant vector field $X$ on $M$ is contained in the core $\G(M)$.
\end{corollary}
\begin{proof}
By Theorem \ref{thm:complete}, there holds $\phi_{-t}^X(e)\phi_t^X(e)=\phi_t^X(\phi_{-t}^X(e))=e$ for any $t\in \R$, so $\phi_{-t}^X(e)$ is the inverse of $\phi_t^X(e)$.
\end{proof}

On the other hand, Theorem \ref{thm:complete} also enables us to define the exponential map for Lie monoids with boundaries:
\begin{definition}
Let $M$ be a Lie monoid such that $\partial M$ is its submonoid, $e\in\partial M$. Let $T_e^{\plus}M\subset T_eM$ denote the subset consisting of inward-pointing vectors in $T_e M$ and the vectors in $T_e(\partial M)$. The \emph{exponential map} on $M$ is then defined as
\begin{align*}
\exp\colon T_e^{\plus} M\rightarrow M,\quad \exp(v)=\phi^{v^L}_1(e).
\end{align*}
\end{definition}
\begin{remark}
In any boundary chart centered at $e$, $T_e^{\plus}M$ is identified with the closed upper half-space $\mathbb H^{\dim M}$. We observe that $T_e^{\plus}M$ possesses an algebraic structure of a semimodule over a semiring $[0,\infty)$.\footnote{A semiring $R$ satisfies all the axioms of a ring, except the existence of additive inverses. Because of this, we must additionally impose $0\cdot a=0=a\cdot 0$ for all $a\in R$. We amend modules to obtain semimodules in precisely the same way. } 

As before, Corollary \ref{cor:exp_grp} ensures that the restriction of $\exp$ to the vectors tangent to $\partial M$, is precisely the usual exponential map of the Lie group $\G(M)$.
\end{remark}

Similar results to those from the theory of Lie groups can be obtained for the exponential map as defined above, by using an identical approach to the respective proofs, but working with one-sided derivatives:
\begin{enumerate}[label={(\roman*)}]
\item Let $v\in T^{\plus}_eM$ and $t\geq 0$. By rescaling lemma, $\phi^{v^L}_{tr}(e)=\phi_r^{tv^L}(e)$ for all $r\geq 0$, and setting $r=1$ we obtain
\[
\exp(tv)=\phi_t^{v^L}(e).
\]
\item The map $\exp$ is a smooth, and there holds $\d{(\exp)}_e=\id_{T_eM}$. Since $\exp$ maps $\exp(T_e(\partial M))\subset \partial M$ by Theorem \ref{thm:complete} (i), one can apply the inverse map theorem for maps between manifolds with boundaries to conclude that $\exp$ is a local diffeomorphism at the point $0\in T_e^{\plus} M$. 
\item Considering the Lie monoid $[0,\infty)$ for addition, any smooth homomorphism $\alpha\colon [0,\infty)\rightarrow M$ of Lie monoids is called a \emph{one-parametric submonoid} of $M$. For any $v\in T_e^{\plus}M$, the map $t\mapsto \exp(tv)$ is a one-parametric submonoid of $M$; conversely, any one-parametric submonoid $\alpha$ can be written as $\alpha(t)=\exp(t\dot\alpha(0))$, where $\dot\alpha(0)$ denotes the one-sided derivative of $\alpha$ at zero. 
\end{enumerate} 

A consequence of point (iii) above is \emph{naturality} of $\exp$, i.e., if $\phi\colon M\rightarrow N$ is a morphism between Lie monoids with normal boundaries, such that the units of $M$ and $N$ are contained in the respective boundaries, then for any $v\in T_e^{\plus}M$, the map $\alpha(t)= \phi(\exp_M(tv))$ defines a one-parametric submonoid of $N$ with $\dot\alpha(0)=\d\phi(v)$, so we obtain that the following diagram commutes.
\[\begin{tikzcd}
	M & N \\
	{T_e^{\plus}M} & {T_e^{\plus}N}
	\arrow["\phi", from=1-1, to=1-2]
	\arrow["\d\phi", from=2-1, to=2-2]
	\arrow["{\exp_M}", from=2-1, to=1-1]
	\arrow["{\exp_N}"', from=2-2, to=1-2]
\end{tikzcd}\]

\section{Application to physics: Statistical Thermodynamics}
\label{sec:std}
The interpretation that morphisms correspond to physical processes, and objects to physical states, can be applied to yield a rigorous approach to statistical physics, which we will now demonstrate. We will first focus purely on categorical aspects, and then consider differentiability.

Suppose we are given an isolated physical system consisting of an unknown number of particles, each of which can be in one of the $n+1$ a priori given \emph{microstates}, which we will index by
\[
i\in\set{0,\dots,n}.
\]
Since the number of particles in our system is unknown and often large, we need to work with tuples of probabilities $(p_0,\dots,p_n)$, where each $p_i$ is the probability that a particle, chosen at random, is in the $i$-th microstate. Any tuple of probabilities $(p_i)_{i=0}^n$ is subjected to the constraint
\[
\textstyle\sum_i p_i=1,
\]
and we will refer to any such tuple $p=(p_i)_i$ as a \emph{configuration} of the system, which is just a probability distribution on the finite set of microstates above. The set of all configurations of our system is thus the standard $n$-simplex,
\[
\Delta^{n}=\set[\big]{(p_0,\dots,p_n)\in [0,1]^{n+1}\given \textstyle\sum_i p_i=1}.
\]
We associate to any configuration $(p_i)_i$ of our system its expected surprise,
\[
S(p_i)_i=-\sum_i p_i\log p_i,
\]
which is called the \emph{entropy} of the configuration $(p_i)_i$. Letting $f(x)=x\log x$, we find $\lim_{x\rightarrow 0^{\plus}}f(x)=0$, so $f$ may be extended to $[0,\infty)$ by defining $f(0)=0$, implying that $S\colon \Delta^{n}\rightarrow \R$ is defined on whole $\Delta^n$ and continuous. 

The construction of the space of morphisms between different configurations of our system is now an application of the \emph{second law of thermodynamics}: 
\begin{quote}
\vspace{0.5em}
\emph{A process in an isolated physical system is feasible if, and only if, the change of entropy pertaining to the process is non-negative.}
\vspace{0.5em}
\end{quote}
In accord with the second law, we define
\[
\D=\set*{(p_i)_i\rightarrow (q_i)_i\given S(q_i)_i- S(p_i)_i\geq 0}.
\]
In other words, $D$ consists of pairs $(q,p)\in \Delta^{n}\times \Delta^{n}$ of configurations, such that the entropy of the target configuration $q$ is no less than that of the source $p$. That $D\rra \Delta^{n}$ is a category follows from the fact that the map
\[
\delta S\colon \Delta^{n}\times\Delta^{n}\rightarrow \R,\quad \delta S(q,p)=S(q)-S(p),
\]
is a functor\footnote{This is aligned with the moral that entropy should be perceived as a categorical concept, which was first adopted by Baez et al.; $\delta S$ is in fact the only map (up to a multiplicative scalar) which is functorial, convex-linear and continuous, see \cite{baez2011} for details.} from the pair groupoid of $\Delta^{n}$ to the group $\R$ for addition. Notice that the invertible morphisms in $D$ are precisely $\delta S^{-1}(0)$, which are just the processes with zero entropy change.

In the differentiable setting, we need to make certain adjustments to our category $D\rra \Delta^{n}$, since it is not a Lie category. First, we note that $\Delta^n$ is a manifold with corners, and $S$ is not smooth at its boundary $\partial\Delta^{n}$ since $\lim_{x\rightarrow 0^{\plus}}(x\log x)'=-\infty$, which is why we first restrict our attention to the interior $\Int\Delta^{n}.$ 
Secondly, we notice that the category $D\rra\Delta^{n}$ has a terminal object, as shown by the following.

\begin{lemma}
The only critical point of entropy $S|_{\Int\Delta^{n}}$ is given by the so-called microcanonical configuration, i.e., $p_i^{\mu}=\frac 1{n+1}$ for all $i$. This is a maximum of $S$.
\end{lemma}
\begin{proof}
Constraint $\sum_i p_i=1$ implies $\sum_i \d p_i=0$ and $\d p_0=-\sum_{i=1}^{n}\d p_i$, so we have
\[
\d S_{(p_i)_i}=-\sum_i(1+\log p_i)\d p_i=-\sum_i\log p_i \d p_i=-\sum_{i=1}^{n}(\log p_i-\log p_0)\d p_i,
\]
which vanishes if, and only if, $p_i=p_0$ for all $i$, i.e., $p_i=\frac 1{n+1}$. That this is a local maximum is left as an exercise, and it is not hard to see that the value of $S$ at $(p_i^\mu)_i$ is greater than the value of $S$ on $S|_{\partial\Delta^{n}}$.
\end{proof}

The configuration $p^\mu$ has little physical importance---for example, we will see below that it can be interpreted as the configuration that is attained at thermodynamical equilibrium at infinite temperature. We will thus remove it from the interior of our $n$-simplex of objects, and define our space of objects as
\[
 X=\Int\Delta^n- \set{p^\mu},
\]
which is a smooth manifold without boundary. Moreover, we define the space of morphisms over $ X$ as the set 
\[
 C=(\delta S|_{ X\times X})^{-1}([0,\infty))=\set*{(q,p)\in X\times X\given S(q)- S(p)\geq 0}.
\] 
As before, $C\rra X$ is a subcategory of the pair groupoid on $ X$. By virtue of Example \ref{ex:preimage_subcat}, to prove that $C\rra X$ is a Lie category, we first need to check $\delta S|_{ X\times X}$ has a regular value 0. By above claim, the only critical point of the map $\delta S$ is the identity morphism $(p^\mu,p^\mu)\in D$, which is not in $C$.

Secondly, we have to show that $C\rra X$ has a regular boundary, i.e., that $s|_{\partial C},t|_{\partial C}$ are submersions. To this end, first note that for any $(q,p)\in\partial C$, 
\begin{align*}
T_{(q,p)}\partial C&=\ker\d(\delta S)_{(q,p)}\\
&=\set*{(v,w)\in T_q X\oplus T_p X\given \textstyle\sum_{i=1}^{n}(\log q_i-\log q_0)v_i=\sum_{i=1}^{n}(\log p_i-\log p_0)w_i}.
\end{align*}
Let $v\in T_q X$. Since $p_i\neq p_0$ for some $i$, we define
\[
w_i=\frac{\textstyle\sum_{i=1}^{n}(\log q_i-\log q_0)v_i}{\log p_i-\log p_0},
\]
and now we let $w=(w_0,0,\dots,0,w_i,0,\dots,0)$ where $w_0=-w_i$ is set to ensure $w\in T_p X$, thus we obtain the wanted pair $(v,w)\in T_{(q,p)}\partial C$ with $\d t(v,w)=v$; we similarly show that $s|_{\partial C}$ is a submersion. Furthermore, $C\rra X$ satisfies all the conclusions from Lemma \ref{lem:extensions}, since the pair groupoid $ X\times X$ is its weakly étale extension; in particular, the left and right Lie algebroids of $C$ are isomorphic to $T X\approx X\times \R^n$.

The question interesting for physics is: what is the configuration $p^{\epsilon}\in X$ at which the system attains a thermodynamical equilibrium? The answer to this question is obtained using the so-called \emph{Gibbs algorithm}, which we provide here for completeness. To this end, we need additional a priori given data, namely each microstate $i$ of our system has an a priori assigned quantity $E_i$, called the \emph{energy} of the $i$-th microstate. To derive the wanted configuration $p^\epsilon$ we utilize the so-called \emph{principle of maximum entropy}:

\begin{quote}
\vspace{0.5em}
\emph{A state is at thermodynamical equilibrium if, and only if, it maximizes the entropy with respect to the systemic constraints.}
\vspace{0.5em}
\end{quote}
That is, we must find the constrained extrema of $S$ with respect to constraints $\sum_i p_i=1$ and $\sum_i p_i E_i=E(p_i)_i$. Here, $E\colon X\rightarrow \R$ is a function on the object space, determined by the first law of thermodynamics up to an additive constant, as we will see in equation \eqref{eq:first_law}. This will impose thermodynamical considerations onto the statistical description of our system.

To apply the method of Lagrange multipliers, consider the function $\hat S\colon X\rightarrow \R$,
\[
\hat S(p_i)_i=S(p_i)_i-\lambda_1\left(E(p)-\textstyle\sum_i p_i E_i\right)-\lambda_2(1-\textstyle\sum_i p_i).
\]
Requiring $\frac{\partial \hat S}{\partial p_i}=0$ yields 
$
p_i=\frac 1 Ze^{\lambda_1 E_i},
$
where we have defined $Z=e^{-(1+\lambda_2)}$. The constraint 
$
\textstyle\sum_i p_i=1$ then reads 
\[
Z=e^{-(1+\lambda_2)}=\textstyle\sum_i e^{\lambda_1 E_i}.
\]
On the other hand, the first law of thermodynamics for systems where no work is exerted reads
\begin{align}
	\label{eq:first_law}
	\d E=kT\d S
\end{align}
where $T$ is the temperature (an external, macroscopic constant) at which the system is held, and $k$ denotes the Boltzmann constant. Writing out the total differentials $\d E$ and $\d S$ gives $\lambda_1=-\frac 1 {kT}$, and so we finally obtain the \emph{equilibrium configuration}:
\[
p_i^\epsilon=\frac 1 {Z} e^{-\frac{E_i}{kT}},\quad Z=\sum_i e^{-\frac{E_i}{kT}}.
\]
We observe that in this categorical framework, the configurations from which it is possible to attain the equilibrium configuration $p^\epsilon$, are now expressible as 
$
s(t^{-1}({p^\epsilon}))=\set*{p\in X\given S(p^\epsilon)\geq S(p)}.
$
We also observe that $p^\epsilon\ra p^\mu$ as $T\ra \infty$, which provides a physical justification for removing the microcanonical configuration from the space of objects.

\begin{subappendices}

\section{Further research}
Although we hope to have succeeded in portraying the richness of Lie categories and their potential in physics, we have not exhausted all research options regarding them. We state some of them here, and note that they provide possibilities for future research.

\begin{enumerate}[label={(\roman*)}]
\item As stated in Remark \ref{rem:questions}, the question remains whether all Lie monoids are parallelizable, and whether $C_x^y$ is a smooth manifold for given objects $x,y\in X$ of a Lie category $C\rra X$.
\item Does there exist a class of Lie algebroids that cannot be integrated to a Lie groupoid, but can be integrated to a Lie category? 
\item Remark \ref{rem:base_boundary} shows the need for considering Lie categories whose object manifold has a boundary, or more generally, corners. Another example of such a Lie category should be the flow of a smooth vector field on a manifold with boundary, generalizing the flow Lie groupoid of a vector field on a boundaryless manifold. 
\item Infinite-dimensional Lie categories. The exterior bundle $\Lambda(E)$ of a vector bundle $E$ is an example of a bundle of Lie monoids, and it is moreover a subcategory of the tensor bundle $\oplus_{k=0}^\infty\otimes^k E$, which is a bundle of monoids that fails to have finite-dimensional fibres. Moreover, in statistical mechanics, physicists often work with an infinite number of microstates, and in quantum mechanics with infinite\hyp{}dimensional Hilbert spaces. These examples show the need for introducing infinite\hyp{}dimensional Lie categories. Since the theory of infinite\hyp{}dimensional (Banach) manifolds with corners is already well-developed in \cite{corners}, a theory of infinite\hyp{}dimensional Lie categories with corners seems realizable.
\item Multiplicative differential forms on Lie categories. On Lie groupoids, such structures can be used to describe integrated counterparts of Poisson structures \cite{poisson}, and more recently they have been used to provide a natural generalization of connections on principle bundles in \cite{mec}. We suspect that interesting geometric structures can be described with multiplicative differential forms on Lie categories.
\item Haar systems on Lie categories. It is well known that Haar systems on Lie groupoids provide a generalization of the notion of Haar measures on Lie groups, and provide a connection of the theory of Lie groupoids to noncommutative geometry. A possible further generalization to Lie categories should likewise provide a means of equipping the space $C_c (C)$ of compactly supported functions on the space of arrows of a given category with the convolution product. In this fashion, one expects to generalize the construction of a groupoid C*-algebra, but due to the lack of existence of inverses of arrows, there is a priori no natural way of obtaining the involution, hence the construction potentially generalizes only to a \emph{category Banach algebra}. We strongly suspect that a sensible notion of a Haar system on a Lie category will rather be defined in terms of a measure $\mu$ on the space $\comp C$ of composable pairs of arrows, satisfying certain invariance and continuity conditions, instead of defining it in terms of a left-invariant $t$-fibre supported measure on $C$. Roughly speaking, the convolution on $C_c (C)$ would then be defined by 
\[(f_1*f_2)(g)=\int_{m^{-1}(g)}f_1(g_1)f_2(g_2) \d \mu(g_1,g_2),\]
where the integration is done over the set $m^{-1}(g)=\set{(g_1,g_2)\in \comp C\given g_1g_2=g}$ of composable arrows which compose to $g\in C$. We note that for groupoids, $m^{-1}(g)$ is diffeomorphic to the $t$-fibre over $t(g)$, so that Haar systems on groupoids can indeed be equivalently described with a family of $m$-fibre supported measures.
	
\item Smooth sieves on Lie categories. In the context of Lie groupoids, the notion of a sieve is vacuous since all morphisms are invertible---any sieve on an object must equal the whole $t$-fibre over that object. However, a notion of a smooth sieve seems possible for Lie categories, and their properties might be interesting, together with the properties of Grothendieck sites of such sieves.
\item Generalization of Morita equivalence from the context of Lie groupoids to Lie categories.
\end{enumerate}

\section{Transversality on manifolds with corners}
\label{sec:transversality_corners}

In this appendix, we state some basic definitions and results regarding finite-dimensional manifolds with corners that we need, mostly drawing from \cite{corners}; in what follows, $X$ is assumed to be a second-countable topological space. Let $V$ be a real $n$-dimensional Banach space\footnote{Keep in mind that picking a basis of $V$ gives a homeomorphism $V\rightarrow\R^n$.}; denote by $\Lambda=(\lambda_i)_i$ a (possibly empty) set of linearly independent covectors $\lambda_i\in V^*$ and let \[V_\Lambda=\set{v\in V\given \lambda(v)\geq 0\text{ for all }\lambda\in\Lambda}.\]
Vacuously, there holds $V_\Lambda=V$ when $\Lambda = \emptyset$. We will call $\Lambda$ a \emph{corner-defining system} on vector space $V$, and also denote \[V_\Lambda^0=\set{v\in V\given \lambda(v)= 0\text{ for all }\lambda\in\Lambda}=\cap_{\lambda\in\Lambda}\ker\lambda.\]
Given such a corner-defining system, a \emph{chart with corners} on $X$ at $p\in X$ is a map $\varphi\colon U\rightarrow V_\Lambda$ where $U\subset X$ is an open neighborhood of $p$, $\varphi(U)\subset V_\Lambda$ is an open neighborhood of $0$, and $\varphi$ is a homeomorphism onto its image with $\varphi(p)=0$. Such a chart is said to be $n$\emph{-dimensional}, and the point $p$ is said to have \emph{index} $|\Lambda|$, for which we will write $\ind_X(p)=|\Lambda|$. 

We say that $X$ is an $n$-dimensional \emph{manifold with corners} if there is an $n$-dimensional chart with corners around every point, and any two such charts $\varphi$ and $\varphi'$ are \emph{compatible}, i.e., $\varphi(U\cap U')\subset V_\Lambda$ and $\varphi'(U\cap U')\subset V_{\Lambda'}$ are open subsets, and $\varphi'\circ\varphi^{-1}\colon \varphi(U\cap U')\rightarrow\varphi'(U\cap U')$ is a diffeomorphism.\footnote{Denote $\R^n_k=\R^{n-k}\times[0,\infty)^k$. A map $ U\rightarrow V$ between open subsets of $\R^n_k$ and $\R^m_l$ is \emph{smooth} at $p\in U$, if it admits a smooth extension to an open neighborhood of $p$ in $\R^n$.} Such a collection of charts with corners is called an \emph{atlas with corners}. On a manifold with corners, the map $\ind_X\colon X\rightarrow \mathbb{N}_0$ is well-defined by boundary invariance theorem found in \cite{corners}*{Theorem 1.2.12}. 

We will use the following notation regarding the elementary terms for manifolds with corners: the $k$\emph{-boundary} of $X$ is denoted by $\partial^kX=\set{p\in X\given \ind_X(p)\geq k}$, and we say that $X$ is a \emph{manifold with boundary} if $\partial X:=\partial^1 X\neq \emptyset$ and $\partial^2X=\emptyset$. The $k$\emph{-stratum} of $X$ is its subspace $S^k(X)=\set{p\in X\given \ind_X(p)=k}$. The set of connected components of the $k$-stratum is denoted by $\mathcal S^k(X)$, and the set of $k$\emph{-faces} of $X$ is the family 
\[\F^k(X)=\set*{\bar S\given S\in \mathcal S^k(X)}\]
of topological closures in $X$ of the components of the $k$-stratum. We note that the $k$-stratum $S^k(X)$ can be given the following canonical differential structure induced by $X$: if $p\in S^k(X)$ and $\varphi\colon U\rightarrow V_\Lambda$ is a corner chart at $p$, then the restriction 
\[
\varphi|_{U\cap S^k(X)}\colon U\cap S^k(X)\rightarrow V_\Lambda^0
\] is a corner chart at $p$ on $S^k(X)$ since there holds $\varphi(U\cap S^k(X))=\varphi(U)\cap V_\Lambda^0$. With this structure, $S^k(X)$ becomes a manifold without boundary of dimension $n-k$.

A subspace $X'\subset X$ is said to be a \emph{submanifold} of $X$, if for any $p\in X'$ there is a chart with corners $\varphi\colon U\rightarrow V_\Lambda$ at $p$ on $X$, and a linear subspace $V'$ together with a corner-defining system $\Lambda'$ on $V'$, such that $\varphi(U\cap X')=\varphi(U)\cap V'_{\Lambda'}$ and this is an open subset of $V'_{\Lambda'}$. Such a chart is said to be \emph{adapted} to $X''$ by means of $(V',{\Lambda'})$, and gives us a way of defining an intrinsic atlas with corners on $X'$. If there holds $\partial X'=\partial X\cap X'$, we say that $X'\subset X$ is a \emph{neat} submanifold.

In the context of manifolds with corners, it is more useful to use non-infinitesimal notions of submersivity and transversality. 
\begin{definition}
\label{defn:top_sub}
A smooth map $f\colon X\rightarrow X'$ between manifolds with corners is a \emph{topological submersion} at $p\in X$, if there is an open neighborhood $U$ of $f(p)$ in $X'$ and a smooth map $\sigma\colon U\rightarrow X$ such that $\sigma(f(p))=p$ and $f\circ\sigma=\id_{U}$. If this holds for all $p\in X$, we just say that $f\colon X\rightarrow X'$ is a topological submersion.
\end{definition}
The notion of a topological submersion is stronger than the usual infinitesimal one which can be seen by differentiating the equality $f\circ\sigma=\id_U$. Moreover, in the context of boundaryless manifolds they are equivalent, which is an easy consequence of the usual rank theorem.

\begin{definition}
Suppose that $f\colon X\rightarrow X'$ is a smooth map between manifolds with corners, and $X''\subset X'$ is a submanifold. The map $f$ is \emph{topologically transversal} to $X''$ at $p\in X$, written symbolically as \[f\pitchfork_p X'',\] if either $f(p)\notin X''$, or there is a chart $\varphi'\colon U'\rightarrow V'_{\Lambda'}$ on $X'$ at $p$ adapted to $X''$ by means of $(V'',{\Lambda''})$ and an open neighborhood $U$ of $p$ in $X$, such that $f(U)\subset U'$ and
\[\begin{tikzcd}
	\tau\colon U & {U'} & {\varphi'(U')} & {(V''\oplus W'')_{L^*\Lambda'}} & {W''}
	\arrow["{f|_U}", from=1-1, to=1-2]
	\arrow["{\varphi'}", "\approx"', from=1-2, to=1-3]
	\arrow["{L^{-1}}", from=1-3, to=1-4]
	\arrow["{\mathrm{pr}_2}", from=1-4, to=1-5]
\end{tikzcd}\]
is a topological submersion at $p$, where $W''$ is any complementary subspace to $V''$ in $V'$, and $L\colon V''\oplus W''\rightarrow V'$ is the linear isomorphism given by $L(v,w)=v+w$. If for all $p\in X$ we have $f\pitchfork_p X$, we just write $f\pitchfork X''$ and say $f$ is topologically transversal to $X''$. 
\end{definition}
\noindent Note that $f\pitchfork_p X''$ implies the usual infinitesimal transversality condition
\begin{align}
\label{eq:inf_trans}
\d f_p(T_p X)+ T_{f(p)}X''=T_{f(p)}X'.
\end{align}
Indeed, if $v'\in T_{f(p)}X'$, there holds
\[
v'=\d f_x(v)+\d(\varphi')_p^{-1}L\big(\pr_1 L^{-1}\d(\varphi')_p(v'-\d f_x(v)),0\big)
\]
where $v\in T_p X$ with $\d \tau_p (v)=\mathrm{pr}_2 L^{-1}\d(\varphi')_p(v')$ exists since $\tau$ is a submersion at $p$. We will show in Proposition \ref{prop:transversality_char} that \eqref{eq:inf_trans} implies $f\pitchfork_p X''$ in case $X$ has no boundary.

We now state the main transversality theorem for manifolds with corners, which is a generalization of the same result for boundaryless manifolds.
\begin{proposition}
\label{prop:transversality}
Let $f\colon X\rightarrow X'$ be a smooth map between manifolds with corners and let $X''\subset X'$ be a neat submanifold. If $f\pitchfork X''$, then $f^{-1}(X'')\subset X$ is a submanifold with $\codim_X f^{-1}(X)=\codim_{X'}X''$, whose tangent bundle is
\[T(f^{-1}(X''))=(\d f)^{-1}(TX'').\]
Moreover, $f^{-1}(X'')\subset X$ is totally neat, i.e., $S^k(f^{-1}(X))=f^{-1}(X'')\cap S^k(X)$ for any $k\geq 0$. In particular, $f^{-1}(X)$ is a neat submanifold of $X$.
\end{proposition}
\begin{proof}
\cite{corners}*{Proposition 7.1.14}.
\end{proof}
In the context of Lie categories with nonempty regular boundaries, this result enables us to show that the set of composable morphisms has a structure of a smooth manifold. Let us show how.
\begin{lemma}
\label{lem:bdry_topsub}
Let $f\colon X\rightarrow X'$ be a smooth map from a manifold $X$ with boundary to a boundaryless manifold $X'$, such that both $f$ and $\partial f$ are submersions. Then $f$ is a topological submersion.
\end{lemma}
\begin{proof}
We want to show $f$ is a topological submersion at any $p\in X$. If $p\in \Int X$ this follows from the usual rank theorem used on $f|_{\Int X}$, and if $p\in \partial X$ it follows from the usual rank theorem used on $f|_{\partial X}$.
\end{proof}
\begin{lemma}
\label{lem:topsub_transversal}
If $f\colon X\rightarrow X'$ is a smooth topological submersion between manifolds with corners and $\partial X'=\emptyset$, then $f\pitchfork X''$ holds for any submanifold $X''\subset X'$.
\end{lemma}
\begin{proof}
Suppose $f(p)\in X''$ and let $\varphi'\colon U'\rightarrow V'$ be a chart on $X'$ at $f(p)$ adapted to $X''$ by means of $(V'',{\Lambda''})$. Let $W''$ be any complementary subspace to $V''$ in $V$; continuity of $f$ ensures there is a neighborhood $U\subset X$ of $p$ such that $f(U)\subset U'$, and now the composition
\[\begin{tikzcd}
	U & {U'} & {\varphi'(U')} & {V''\oplus W''} & {W''}
	\arrow["{f|_U}", from=1-1, to=1-2]
	\arrow["{\varphi'}", "\approx"', from=1-2, to=1-3]
	\arrow["{L^{-1}}", from=1-3, to=1-4]
	\arrow["{\mathrm{pr}_2}", from=1-4, to=1-5]
\end{tikzcd}\]
is a topological submersion as a composition of topological submersions $f|_U$ and $\pr_2 L^{-1}\varphi'$.
\end{proof}
\begin{corollary}
\label{cor:fibred_prod}
Let $X$ and $Y$ be manifolds with boundaries and $Z$ a manifold without boundary. If $f\colon X\rightarrow Z$ and $g\colon Y\rightarrow Z$ are smooth maps such that $f,f|_{\partial X}$ and g,$g|_{\partial Y}$ are submersions, then $X{\tensor*[_f]{\times}{_g}}Y=(f\times g)^{-1}(\Delta_{Z})\subset X\times Y$ is a submanifold with tangent space at $(p,q)$ equal to
\[
T_{(p,q)}(X{\tensor*[_f]{\times}{_g}}Y)=\set{(v,w)\in T_p X\oplus T_q Y\given \d f(v)=\d g(w)},
\]
and its boundary is $\partial(X{\tensor*[_f]{\times}{_g}}Y)=(X{\tensor*[_f]{\times}{_g}}Y )\cap (\partial X\times Y\cup X\times\partial Y)$.
\end{corollary}
\begin{proof}
Lemma \ref{lem:bdry_topsub} ensures $f$ and $g$ are topological submersions, and it is easy to check that $f\times g\colon X\times Y\rightarrow Z\times Z$ is also a smooth topological submersion. Now Lemma \ref{lem:topsub_transversal} implies $(f\times g)\pitchfork\Delta_Z$ since $Z$ is boundaryless, and Proposition \ref{prop:transversality} finishes the proof since $\Delta_Z\subset Z\times Z$ is trivially a neat submanifold.
\end{proof}
The following proposition shows in particular that in the case of boundaryless manifolds, the usual infinitesimal notion of transversality is equivalent to the one above.
\begin{proposition}
\label{prop:transversality_char}
Let $f\colon X\rightarrow X'$ be a smooth map between manifolds with corners and $X''\subset X'$ a submanifold. For any $p\in f^{-1}(X'')$ with $\ind_X(p)=k$, the following statements are equivalent.
\begin{enumerate}[label={(\roman*)}]
\item $\Im\d (f|_{S^k(X)})_p+ T_{f(p)}X''=T_{f(p)}X'.$
\item $f|_{S^k(X)}\pitchfork_p X''$.
\item $f\pitchfork_p X''$.
\end{enumerate}
\end{proposition}
\begin{proof}
To show $(i)\Rightarrow (ii)$, take a chart $\varphi'\colon U'\rightarrow V'_{\Lambda'}$ on $X'$ at $f(p)$ adapted to $X''$ by means of $(V'',{\Lambda''})$, take a complementary subspace $W''$ to $V''$ of $V'$, and consider the map
\[\begin{tikzcd}
	\tau^k\colon U & {U'} & {\varphi'(U')} & {(V''\oplus W'')_{L^*\Lambda'}} & {W''}
	\arrow["{f|_U}", from=1-1, to=1-2]
	\arrow["{\varphi'}", "\approx"', from=1-2, to=1-3]
	\arrow["{L^{-1}}", from=1-3, to=1-4]
	\arrow["{\mathrm{pr}_2}", from=1-4, to=1-5]
\end{tikzcd}\]
where $U$ is the domain of a chart neighborhood $\varphi\colon U\rightarrow V$ in $S^k(X)$ of $p$; by continuity we may assume $f(U)\subset U'$. Since $\partial S^k(X)$ is boundaryless, it is enough to show that $\tau^k$ is a submersion. To that end, take any $w\in T_0 W''\cong W''$ and then $u:=\d(\varphi')_p^{-1}L(0,w)\in T_{f(p)}X'$, so by assumption there exist $v\in T_p X$ and $v''\in T_{f(p)}X''$ such that $u=\d f_p(v)+v''$. Since $L^{-1}\d(\varphi')_p(v'')\in V''$, we get
\[
\d h_p(v)=\pr_2L^{-1}\d(\varphi')(u-v'')=w.
\]
Conversely, $(ii)\Rightarrow (i)$ follows from the fact that topological transversality implies transversality.

For the implication $(ii)\Rightarrow (iii)$, note that if $Z\xrightarrow{g} Z'\xrightarrow{h} Z''$ are smooth maps such that $h\circ g$ is a topological submersion at $p\in Z$, then it is easy to see $h$ is a topological submersion at $g(p)$. Use this result for the composition $S^k(X)\hookrightarrow X\xrightarrow \tau W''$ which equals $\tau^k$. The converse implication $(iii)\Rightarrow (ii)$ is a consequence of the following lemma used on the map $\tau$.
\end{proof}

\begin{lemma}
If $f\colon X\rightarrow X'$ is a smooth map between manifolds with corners, which is a topological submersion at $p$ and there holds $f(p)\in \Int (X')$, then $f|_{S^k(X)}$ is a submersion at $p$, where $k=\ind_X(p)$.
\end{lemma}
\begin{proof}
Let $\varphi\colon U\rightarrow V_{\Lambda}$ be a chart on $X$ at $p$ and $\varphi'\colon X'\rightarrow V'$ on $X'$ at $f(p)$. Note that $\varphi(U\cap S^k(X))=\varphi(U)\cap V_{\Lambda}^0$.

Let $v'\in T_{f(p)}X'$ be arbitrary. By assumption, there is a neighborhood $U'$ of $f(p)$ and a smooth map $\sigma\colon U'\rightarrow U\subset X$ such that $\sigma(f(p))=p$ and $f\circ \sigma =\id_{U'}$. Since $f(p)\in\Int X'$, there is a smooth path $\gamma\colon(-\varepsilon,\varepsilon)\rightarrow U'$ with $\dot\gamma(0)=v'$, and now consider the path $\delta=\varphi\circ\sigma\circ\gamma\colon(-\varepsilon,\varepsilon)\rightarrow \varphi(U)\subset V_{\Lambda}$. Since $\delta(0)=0$ and $\delta$ is defined on an open interval, we must have that $\delta$ maps into $V_{\Lambda}^0$, so $\dot\delta(0)\in T_0V_{\Lambda}^0\cong V_{\Lambda}^0$. Taking $v=\d(\varphi^{-1})_0(\dot\delta(0))$ hence yields $v\in T_p(S^k(X))$ with $\d f_p(v)=v'$.
\end{proof}

\end{subappendices}

\clearpage \pagestyle{plain}
\chapter{Background for the second part}\label{chapter:background}
\pagestyle{fancy}
\fancyhead[CE]{Chapter \ref*{chapter:background}} 
\fancyhead[CO]{Background for the second part}

In this chapter, we have gathered some prerequisite material that will be used in the remainder of the thesis. From now on, we direct our focus to Lie groupoids instead of Lie categories.

\section{More on Lie groupoids and Lie algebroids}
\subsection*{Lie groupoids}
As already discussed in \sec\ref{chapter:lie_cats}, \emph{Lie groupoids} are Lie categories with all arrows invertible. A Lie groupoid will usually be denoted by $G\rra M$, where $G$ and $M$ denote the manifolds of arrows and objects, respectively. We now recall some of their important structural properties and examples. 

 \begin{proposition}
    \label{prop:properties_lie_groupoids}
    In a Lie groupoid $G\rra M$, the following properties hold, for any $x,y\in M$.
    \begin{enumerate}[label={(\roman*)}]
        \item The Hom-sets $G_x^y$ are closed embedded submanifolds of $G$.
        \item The space $G_x^x$ is a Lie group, called the \emph{isotropy group} at $x$.
        \item The space $\O_x\coloneqq t(G_x)$ is an immersed submanifold of $M$, called the \emph{orbit} through $x$.
        \item The restriction $t\colon G_x\ra \O_x$ is a principal $G_x^x$-bundle over $\O_x$.
    \end{enumerate}
 \end{proposition}
\noindent Roughly speaking, the point (iv) above says that a Lie groupoid can be seen as a collection of principal bundles, though it says nothing about the transverse structure that glues them together. The connected components of the orbits comprise a singular foliation on $M$, denoted $\F$.

\begin{example}\
\begin{enumerate}[label={(\roman*)}]
    \item \textit{Bundles of Lie groups}. A bundle of Lie groups is a Lie groupoid with coinciding source and target map. This is a specific case of Example \ref{ex:monoid_bundles}. 
    \item \emph{Submersion groupoids}. For any submersion $\pi\colon M\ra N$, we may consider the fibred product $M\times_\pi M$, which is a Lie groupoid with the structure inherited from the pair groupoid. If $\pi$ is the identity map, we recover the unit groupoid, and if $N$ is a singleton, we recover the pair groupoid.
    \item \emph{Action groupoids}. Example \ref{ex:action_cats} for the case when a Lie group $G$ is acting on a manifold $M$ yields a Lie groupoid $G\ltimes M\rra M$. Since Lie group actions are always submersive, the arrow space is just $G\times M$. We note that the isotropy groups of $G\ltimes M$ can be identified with the stabilizer groups of the action, and the orbits of $G\ltimes M$ can be identified with the orbits of the action.
    \item \emph{General linear groupoids}. In Example \ref{ex:endomorphism_cat}, if we restrict to linear isomorphisms between the fibres of a vector bundle $V\ra M$, we obtain a groupoid. This is precisely the core of the endomorphism category, which is an open Lie subcategory by Theorem \ref{thm:open_inv}, denoted
    \[
    \operatorname{GL}(V)= \G(\End V).
    \]
    \item \emph{Gauge groupoids}. Given a principal $G$-bundle $\pi\colon P\ra M$, we can construct a Lie groupoid whose Hom-sets are $G$-equivariant maps between the fibres of $\pi$. We encode this by defining the arrow space as the quotient
    \[
    \G(P)=\frac{P\times P}{G},
    \]
    where $G$ acts on $P\times P$ by the diagonal (right) action. With the obvious structure maps, we obtain a Lie groupoid over $M$ with only one orbit: the whole base manifold $M$. Such Lie groupoids are called \emph{transitive}. One can show that any transitive groupoid is isomorphic to the gauge groupoid of the principal bundle from Proposition \ref{prop:properties_lie_groupoids} (iv), for any $x\in M$.
\end{enumerate}
\end{example}
\subsubsection{The nerve of a Lie groupoid}
We have already observed that we can form the set $G^{(2)}$ of composable arrows in $G$. Importantly, we can go even further, and consider the set $G^{(p)}$ of $p$-tuples of composable arrows in $G$. In this way, any Lie groupoid determines a simplicial manifold, called the \emph{nerve} of a Lie groupoid. Pictorially:
\[\begin{tikzcd}
	\cdots & {G^{(3)}} & {G^{(2)}} & G & M
	\arrow[shift left=4, from=1-1, to=1-2]
	\arrow[shift right=4, from=1-1, to=1-2]
	\arrow[from=1-1, to=1-2]
	\arrow[shift left=2, from=1-1, to=1-2]
	\arrow[shift right=2, from=1-1, to=1-2]
	\arrow[shift left=3, from=1-2, to=1-3]
	\arrow[shift right=3, from=1-2, to=1-3]
	\arrow[shift left, from=1-2, to=1-3]
	\arrow[shift right, from=1-2, to=1-3]
	\arrow["{\pr_2}", shift left=2, from=1-3, to=1-4]
	\arrow["m"{description}, from=1-3, to=1-4]
	\arrow["{\pr_1}"', shift right=2, from=1-3, to=1-4]
	\arrow["s", shift left, from=1-4, to=1-5]
	\arrow["t"', shift right, from=1-4, to=1-5]
\end{tikzcd}\]
Here, the face maps $f_i^{(p+1)}\colon G^{(p+1)}\ra G^{(p)}$ are defined as
\[
f_i^{(p+1)}(g_1,\dots,g_{p+1})=
\begin{cases}
(g_2,\dots,g_{p+1})&i=0,\\
(g_1,\dots,g_{i}g_{i+1},\dots,g_{p+1})&1\leq i\leq p,\\
(g_1,\dots,g_{p})&i=p+1,
\end{cases}
\]
Because we are assuming the source and target maps are smooth submersions, the set $G^{(p)}$ is a smooth manifold, and because the multiplication in a Lie groupoid is submersive, all the face maps are smooth submersions. The degeneracy maps are just the smooth embeddings $G^{(p-1)}\ra G^{(p)}$ that insert a unit at a given factor.

\subsection*{Lie algebroids}
The notion of a Lie algebroid, typically denoted $A\Ra M$, has already been defined in Definition \ref{def:algebroid}. We now recall some properties which will be needed later on.
\begin{remark}
    When we say that $A$ is a Lie algebroid of a Lie groupoid $G$, we will always mean the left Lie algebroid of $G$. Importantly, not every Lie algebroid is integrable to a Lie groupoid. Precise obstructions to integrability of Lie algebroids to Lie groupoids have been obtained in \cite{integrability}.
\end{remark}
\begin{proposition}
    Given a Lie algebroid $A\Ra M$, the following holds for any $x\in M$.
\begin{enumerate}[label={(\roman*)}]
    \item The space $\frak g_x\coloneqq \ker\rho_x$ is a Lie algebra, called the \emph{isotropy Lie algebra} at $x$, with the bracket
    \[
    [\xi,\eta]=[\tilde\xi,\tilde\eta]_x,
    \]
    for any vectors $\xi,\eta\in\frak g_x$, where the sections $\tilde\xi,\tilde\eta\in\Gamma(A)$ are their arbitrary extensions.
    \item The image $\im(\rho)\subset TM$ of the anchor of $A$ is an integrable singular distribution. The leaves of the associated singular foliation, denoted $\F$, are called the \emph{orbits} of $A$.
    \item If $A$ is the algebroid of a Lie groupoid $G$, then $\frak g_x$ is the Lie algebra of the isotropy group $G_x^x$, and the leaves of the two orbit foliations coincide.
\end{enumerate}
\end{proposition}
Let us now list some examples of algebroids corresponding to the examples from the world of Lie groupoids, listed above. 
\begin{example}\
\label{ex:algebroid_examples}
\begin{enumerate}[label={(\roman*)}]
    \item \textit{Bundles of Lie algebras}. A bundle of Lie algebras is a Lie algebroid with a vanishing anchor. 
    \item The Lie algebroid of a submersion groupoid for a given submersion $\pi\colon M\ra N$ is the involutive distribution $\ker\d\pi \subset TM$.
    \item \emph{Action algebroids}. An action of a Lie algebra $\frak g$ on a manifold $M$ defines a Lie algebroid $\frak g \ltimes M\Ra M$. As a vector bundle, it is defined as the product $\frak g\times M$; its anchor is given by the Lie algebra action, and the bracket is determined by the Leibniz rule and the condition $[c_\xi,c_\eta]=[\xi,\eta]$ where $c_\xi,c_\eta$ are the constant sections associated to $\xi,\eta\in\frak g$. Every Lie group action induces a Lie algebra action, however, not every Lie algebra action is integrable to a Lie group action.
    \item \emph{General linear algebroids}. 
    Given a vector bundle $V\ra M$, the Lie algebroid of the general linear groupoid $\mathrm{GL}(V)$ is isomorphic to the Lie algebroid $\frak{gl}(V)$ whose sections are the so-called \emph{covariant differential operators} (also called \emph{derivations}) of $V$:
    \[
    \Gamma(\frak{gl}(V))=\set*{\big(D\colon\Gamma(V)\xrightarrow{\text{lin.}} \Gamma(V), X_D\in\vf(M)\big)\given  D(f\sigma)=fD(\sigma)+X_D(f)\sigma}
    \]
    The anchor is given by the assignment $(D,X_D)\mapsto X_D$, and the bracket is just the commutator of differential operators. For details, see \cite{actions}*{Theorem 1.4}. 
    \item \emph{Atiyah algebroid}. Given a principal $G$-bundle $\pi\colon P\ra M$, the Lie algebroid of the gauge groupoid $(P\times P)/G\rra M$ is isomorphic to the quotient vector bundle $TP/G$ where the action is given by the differentials of right translations. It fits into the short exact sequence of Lie algebroids
\[\begin{tikzcd}
	0 & {\ad(P)} & {\frac{TP}G} & TM & 0,
	\arrow[from=1-1, to=1-2]
	\arrow[from=1-2, to=1-3]
	\arrow["\rho",from=1-3, to=1-4]
	\arrow[from=1-4, to=1-5]
\end{tikzcd}\]
called the \emph{Atiyah sequence}, where $\ad(P)$ denotes the adjoint vector bundle $\frac{P\times\frak g}{G}$ with respect to the adjoint representation of $G$ on $\frak g$. The space of sections of $TP/G$ is precisely the Lie algebra of $G$-invariant vector fields on $P$, and the anchor is the map induced by $\d\pi$.

This is an example of a \emph{transitive algebroid}, that is, $\im(\rho)=TM$, but not every transitive algebroid is integrable to a principal bundle. A precise obstruction is given in \cite{integration_transitive}.

As a concrete example, consider the Atiyah algebroid of the frame bundle $\operatorname{Fr}(V)\ra M$ of a vector bundle $V\ra M$. It is isomorphic precisely to the general linear algebroid from the previous example, and the corresponding Atiyah sequence reads
\[\begin{tikzcd}
	0 & {\End(V)} & {\frak{gl}(V)} & TM & 0.
	\arrow[from=1-1, to=1-2]
	\arrow[from=1-2, to=1-3]
	\arrow[from=1-3, to=1-4]
	\arrow[from=1-4, to=1-5]
\end{tikzcd}\]
\end{enumerate}
\end{example}

\subsubsection{Special classes of Lie groupoids and algebroids}\
\begin{itemize}
    \item A Lie groupoid (or algebroid) is said to be \emph{regular}, if the orbit foliation $\F$ is regular, i.e., if all orbits have the same dimension. This is equivalent to saying that the anchor map $\rho\colon A\ra TM$ has constant rank (as a vector bundle map).
    \item A Lie groupoid is said to be \emph{proper}, if $(s,t)\colon G\ra M\times M$ is a proper map. For instance, the gauge groupoid of a principal bundle is proper if and only if the structure group is compact; the action groupoid of a Lie group action is proper if and only if the action map is proper. 
\end{itemize}

\subsubsection{Bisections}
\begin{definition}
  A \emph{bisection} of a Lie groupoid $G\rra M$ is a smooth map $b\colon M\ra G$ such that $s\circ b=\id_M$ and $\varphi\coloneq t\circ b\colon M\ra M$ is a diffeomorphism. A \emph{local} bisection is defined similarly, on an open subset $U\subset M$, with the requirement that $\varphi\colon U\ra \varphi(U)$ is a diffeomorphism between open subsets of $M$. The set of all bisections will be denoted by $\Bis(G)$.
\end{definition}
\begin{remark}
  The asymmetry between the source and target map is only apparent---we can equivalently describe a bisection as a submanifold $B\subset G$ such that both $s|_B$ and $t|_B$ are diffeomorphisms $B\ra M$. Accordingly, one can reverse the role of the source and target map by simply defining 
  \[
    \overline b\colon M\ra G,\quad \overline b=b\circ\varphi^{-1},
  \]
  which now satisfies $t\circ\bar b=\id_M$ and $s\circ \bar b=\varphi^{-1}$ is a diffeomorphism. To avoid confusion, we will sometimes use the term \emph{source bisection} to refer to the notion from the definition above, and \emph{target bisection} to mean the one just discussed.
\end{remark}
\begin{remark}
  The set $\Bis(G)$ is a group, with the multiplication defined as 
  \[
    b_1b_2\colon M\ra G,\quad (b_1b_2)(x)=b_1(t(b_2(x)))b_2(x).
  \]
  The unit element is the unit map $u\colon M\ra G$ and the inverses are given by $b^{-1}=\inv\circ \overline b$.
\end{remark}
\begin{example}
  \label{ex:exp_bisection}
  Given a section $\alpha\in\Gamma(A)$ of the Lie algebroid of $G$, its left-invariant extension $\alpha^L\in\vf(G)$ defines the following (local, target) bisection for any fixed $\lambda\in\R$:
  \begin{align*}
    \exp(\lambda\alpha)\colon U\ra G,\quad \exp(\lambda\alpha)(x)=\phi^{\alpha^L}_\lambda(1_x),
  \end{align*}
  where $U$ is the subset of all $x\in M$ for which the flow of $\alpha^L$ through $1_x$ at time $\lambda$ is defined. If the vector field $\rho(\alpha)\in\vf(M)$ is complete,  $\alpha^L\in\vf(G)$ is also complete by Proposition \ref{prop:completeness_characterization}, hence $\exp(\alpha)$ defines a global (target) bisection.
\end{example}
Any (source) bisection $b$ determines the \emph{global translations},
\begin{align*}
  \begin{aligned}
    &L_b\colon G\ra G,&\quad L_b(g)&=b(t(g))g,\\
    &R_b\colon G\ra G,&\quad R_b(g)&=g \overline b(s(g)).
  \end{aligned}
\end{align*}
When restricted to the source and target fibres (respectively), they read
\begin{align*}
  \begin{aligned}
    &L_b|_{G^x}\colon G^x\ra G^{\varphi(x)},&\quad L_b|_{G^x}&=L_{b(x)},\\
    &R_b|_{G_x}\colon G_x\ra G_{\varphi^{-1}(x)},&\quad R_b|_{G_x}&=R_{\overline b(x)}.
  \end{aligned}
\end{align*}
The \emph{inner automorphism} by a (source) bisection $b$ is defined as the groupoid automorphism
\begin{equation*}
\begin{aligned}
	&I_b\colon G\ra G,\\ 
	&I_b(g)=b(t(g))g b(s(g))^{-1}.
\end{aligned}
\qquad
\vcenter{\hbox{
	\begin{tikzcd}
		G & G \\
		M & M
		\arrow["{I_b}", from=1-1, to=1-2]
		\arrow[shift left, from=1-1, to=2-1]
		\arrow[shift right, from=1-1, to=2-1]
		\arrow[shift left, from=1-2, to=2-2]
		\arrow[shift right, from=1-2, to=2-2]
		\arrow["\varphi", from=2-1, to=2-2]
	\end{tikzcd}
}}
\end{equation*}
We denote the subgroup of all such groupoid automorphisms by $\mathrm{Inn}(G)\subset \mathrm{Aut}(G)$.
\begin{example}
  \label{ex:gauge_transformations_inner_automorphisms}
  On the gauge groupoid $\G(P)$ of a principal $G$-bundle $\pi\colon P\ra M$, bisections are in a bijective correspondence with gauge transformations ($G$-equivariant bundle automorphisms). Indeed, if we are given a gauge transformation $f\colon P\ra P$ covering $\varphi\colon M\ra M$, then the corresponding  (source) bisection reads $b(x)=[f(u),u]$ for an arbitrary $u\in \pi^{-1}(x)$, which is well-defined by $G$-equivariance, and smooth since $\pi$ is a surjective submersion. Conversely, given a source bisection $b\colon M\ra \G(P)$, take any $u\in P$ and observe there holds $b(\pi(u))=[f(u),u]$ for some uniquely determined $f(u)\in P$; it is easy to see that this defines a smooth $G$-equivariant automorphism $f\colon P\ra P$ covering $\varphi=t\circ b$. The inner automorphism corresponding to a gauge transformation $f$ clearly reads $[u_1,u_2]
  \mapsto[f(u_1),f(u_2)]$ for any $u_1,u_2\in P$. Therefore, under this identification, the kernel of the group homomorphism $\Bis(\G(P))\ra \mathrm{Inn}(\G(P))\subset \mathrm{Aut}(\G(P))$, $b\mapsto I_b$, is identified with gauge transformations of the form $f(u)=u\cdot g$, where $g\in Z(G)$ is a fixed element from the centre of the Lie group $G$. Hence, $\mathrm{Inn}(\G(P))$ can be identified with gauge transformations of $P$, modulo transformations by a constant central element of $G$, i.e., \[\mathrm{Inn}(\G(P))\cong \mathrm{Aut}(P)/Z(G).\]
  At last, combining the identification $\Bis(\G(P))\cong \mathrm{Aut}(P)$ with Example \ref{ex:exp_bisection}, we see that any section $\xi\in\Gamma(\ad(P))$ defines a gauge transformation covering the identity, given by
  \[
  f(u)=u\cdot\exp(v_u),
  \]
  for any $u\in P$, where $v_u\in\frak g$ is the unique vector such that $\xi_{\pi(u)}=[u,v_u]\in \ad(P)_{\pi(u)}$.
\end{example}

\section{Representations}
\begin{definition}
    A \textit{representation} of a Lie groupoid $G\rra M$ on a vector bundle $\pi\colon V\ra M$ is a linear action on $V$, that is, smooth map $\Delta\colon G\tensor[_s]{\times}{_\pi}V\ra V$, denoted $(g,v)\mapsto g\cdot v$ satisfying
\begin{align*}
  \pi(g\cdot v)\in V_{t(g)},\quad 1_x\cdot v=v,\quad g\cdot (h\cdot v)=(gh)\cdot v,
\end{align*}
for all $g,h\in G$ and $v\in V$ at which the expressions are defined. Additionally, for any $g\in G$, we require the action map $\Delta_g\colon V_{s(g)}\ra V_{t(g)}$, $v\mapsto g\cdot v$ to be linear. In other words, a representation of $G$ on $V$ is a Lie groupoid morphism $G\ra \mathrm{GL}(V).$ Since any representation is in particular an action of a Lie groupoid, we will sometimes denote a representation by $G\curvearrowright V$.
\end{definition}
Passing to the infinitesimal level, let $A\Ra M$ be any Lie algebroid. Imitating the Lie groupoid case, we may define a representation of $A$ on $V$ as a Lie algebroid morphism $A \ra \frak{gl}(V).$
At the level of sections, this condition just means that we have a map 
\[\Gamma(A)\ra \Gamma(\frak{gl}(V)),\quad \alpha\mapsto \big(\nabla^A_\alpha, \rho(\alpha)\big).\] 
That this is a Lie algebroid morphism then translates to a simple flatness condition,
\[
    \nabla^A_{[\alpha,\beta]}\sigma=\nabla^A_\alpha \nabla^A_\beta\sigma - \nabla^A_\beta \nabla^A_\alpha\sigma.
\]
So, we may state the definition of an algebroid representation in the following way.
\begin{definition}
    A \textit{representation} of a Lie algebroid $A\Ra M$ is a flat $A$-connection on a vector bundle $V\rightarrow M$, that is, a bilinear map
\begin{align*}
  \nabla^A\colon \Gamma(A)\times \Gamma(V)\ra\Gamma(V),\quad (\alpha,\sigma)\mapsto \nabla^A_\alpha \sigma,
\end{align*}
which is $C^\infty(M)$-linear in the first argument, satisfies the Leibniz identity
\begin{align*}
  \nabla^A_\alpha(f\sigma)=f\nabla^A_\alpha \sigma+ \rho(\alpha)(f)\sigma,
\end{align*}
and satisfies the flatness condition $\nabla^A_{[\alpha,\beta]}=[\nabla^A_\alpha,\nabla^A_\beta]$.
\end{definition}
\begin{proposition}
    \label{prop:reps_groupoid_to_algebroid}
    If $A$ is the Lie algebroid of a Lie groupoid $G$, a representation of $G$ on $V$ induces a representation of $A$ on $V$, by setting
\begin{align}
    \label{eq:representation_formula}
    \big(\nabla^A_\alpha \xi\big)(x)=\deriv\lambda 0 \phi^{\alpha^L}_{\lambda}(1_x)\cdot \xi\big(\phi^{\rho(\alpha)}_{\lambda}(x)\big),
\end{align}
where $\alpha^L\in\vf(G)$ denotes the left-invariant extension of $\alpha\in \Gamma(A)$ on $G$, $\phi^{\alpha^L}_{\lambda}$ is its flow at time $\lambda$, and we denoted by $\phi^{\smash{\rho(\alpha)}}_\lambda$ the flow of the vector field $\rho(\alpha)\in\vf(M)$. 

Moreover, if $G$ has simply connected $s$-fibres, this defines a bijective correspondence between representations of $G$ and representations of $A$.
\end{proposition}
\begin{remark}
    The second part of the last proposition follows by virtue of Lie's second fundamental theorem for groupoids \cite{lie2}.
\end{remark}
 \begin{example}
    Every regular Lie groupoid $G$ comes equipped with two canonical representations: 
    \begin{itemize}
        \item The \emph{adjoint representation}:
        \[\Ad\colon G\curvearrowright \ker\rho,\quad \Ad_g(\xi)=\d(C_g)_{1_{s(g)}}(\xi),\]
        where $C_g=L_{g}\circ R_{g^{-1}}\colon G^{s(g)}_{s(g)}\ra G^{t(g)}_{t(g)}$ is the conjugation map, and $\xi\in \ker\rho_{s(g)}$.
        \item The \emph{normal representation}:
        \[
        \nu\colon G\curvearrowright \operatorname{coker}\rho=TM/T\F,\quad \nu_g[v]=[\d t(\overline v)],
        \]
        for any $v\in T_{s(g)}M$, where $\overline v\in T_g G$ is any vector such that $\d s_g(\overline v)=v$.
    \end{itemize}
In the realm of regular Lie algebroids, the two respective flat $A$-connections read
    \[
    \nabla^A_\alpha\xi=[\alpha,\xi],\quad\text{and}\quad \nabla^A_\alpha[X]=[[\rho(\alpha),X]],
    \]
    for any $\alpha\in\Gamma(A)$ and $\xi\in\Gamma(\ker\rho)$, $X\in\vf(M)$.
 \end{example}

\section[Bott--Shulman--Stasheff and Weil complex with coefficients]{Bott--Shulman--Stasheff and Weil complex with coefficients}
\label{sec:bss_weil}
In this section, we recall the definitions from \cite{homogeneous} of the cochain complexes which are fundamental for our work. On the side of groupoids, the cochain complexes consist of differential forms on the nerve, which, importantly, have values in a representation. An analogous infinitesimal notion to representation-valued forms on the nerve is much more involved---these objects are called Weil cochains. The most important feature of these two complexes is that multiplicative forms appear in them as 1-cocycles. This will be further emphasized in each of the respective complexes.

\subsection*{Representation-valued forms on a Lie groupoid}
Let $V$ be a representation of a Lie groupoid $G\rra M$ and consider the set of  $V$-valued differential forms of degree $q$ on the level $p$ of the nerve of $G$,
\begin{align*}
&\Omega^{0,q}(G;V)\coloneqq \Omega^q(M;V),\\
&\Omega^{p,q}(G;V)\coloneqq \Omega^q(G^{(p)};(s\circ \pr_p)^*V),\quad \text{for }p\geq 1,
\end{align*}
where $\pr_p\colon G^{(p)}\ra G$ denotes the projection to the last element. For a fixed degree $q$, this becomes a cochain complex, with the differential defined as follows. At level $p=0$,
\begin{align}
\begin{split}
\label{eq:delta_0}
&\delta^{0}\colon \Omega^q(M;V)\ra \Omega^q(G;s^*V),\quad(\delta^0\gamma)_g=(s^*\gamma)_g-g^{-1}\cdot (t^*\gamma)_g.
\end{split}
\end{align}
At level $p\geq 1$, the differential is given by the alternating sum of pullbacks along the face maps,
\begin{align}
\begin{split}
\label{eq:delta_l}
&\delta^p\colon \Omega^{p,q}(G;V)\ra \Omega^{p+1,q}(G;V),\quad
\delta^p=\sum_{i=0}^{p}(-1)^i\big(f_i^{(p+1)}\big)^* + (-1)^{p+1}\Phi_*\big(f_{p+1}^{(p+1)}\big)^*
\end{split}
\end{align}
where the map $\Phi_*\colon \Omega^q(G^{(p+1)};(t\circ\pr_{p+1})^*V)\ra \Omega^q(G^{(p+1)};(s\circ\pr_{p+1})^*V)$ changes the coefficients of forms---it is induced by the isomorphism of vector bundles
\begin{align*}
  \phi\colon t^*V\ra s^*V,\quad \phi(g,v)=(g,g^{-1}\cdot v)
\end{align*}
or rather its $(p+1)$-level analogue,
\begin{align*}
&\Phi\colon(t\circ \pr_{p+1})^*V\ra (s\circ \pr_{p+1})^*V,\\
&\Phi(g_1,\dots,g_{p+1},v)= (g_1,\dots,g_{p+1},g_{p+1}^{-1}\cdot v).
\end{align*}
It is well known that $\delta^{p+1}\circ\delta^p=0$ for any fixed $q\geq 0$.
\begin{definition}
  Let $G\rra M$ be a Lie groupoid with a representation $V$. The differential $\delta$ is called the \textit{simplicial differential} of differential forms on the nerve. At any fixed degree $q\geq 0$, it defines a cochain complex called the \textit{Bott--Shulman--Stasheff complex}, 
  \[
  (\Omega^{\bullet,q}(G;V),\delta),
  \]
  whose cohomology is called the \textit{simplicial cohomology} of differential forms on the nerve,
\begin{align}
  \label{eq:bss_cohomology}
      H^{p,q}(G;V)\coloneqq H^p(\Omega^{\bullet,q}(G;V),\delta).
\end{align}
\end{definition}
\begin{example}
At level $p=0$, the cocycles $\ker\delta^0$ are called \textit{invariant forms} on $M$, i.e., forms $\gamma\in\Omega^q(M;V)$ which satisfy the following identity for any $g\in G$:
\[(t^*\gamma)_g=g\cdot (s^*\gamma)_g.\]
At level $p=1$, the cocycles $\ker\delta^1$ are called \textit{multiplicative forms} on $G$. These are differential forms $\omega\in\Omega^q(G;s^*V)$ that satisfy
\begin{align}
\label{eq:multiplicative}
  \omega_{gh}(\d m(X_i,Y_i))_i=\omega_h(Y_i)_i+h^{-1}\cdot \omega_g(X_i)_i
\end{align}
for any vectors $(X_i,Y_i)\in T_{(g,h)}G^{(2)}$, where $s(g)=t(h)$, and we denote them by $\Omega^q_m(G;V)$. The cocycles arising as coboundaries $\im \delta^0\subset \Omega^\bullet_m(G;V)$ are called \textit{cohomologically trivial}.
\end{example}
\begin{remark}
    In the example above, we are using the notation
    \[(X_i)_i= (X_i)_{i=1}^q=(X_1,\dots, X_q),\]
    and similarly $(\d m (X_i,Y_i))_i=(\d m (X_1,Y_1),\dots,\d m (X_q,Y_q))$. Such notation will often be used, and the number $q$ of elements of a given tuple will always be clear from context.
\end{remark}

\subsection*{Representation-valued Weil cochains}
Let $V$ be a representation of a Lie algebroid $A\Ra M$. Following \cite{homogeneous}*{\sec 4}, the infinitesimal analogue of the complex $\Omega^{p,q}(G;V)$ is the Weil complex: it is defined as the family of sets $W^{p,q}(A;V)$ of sequences $c=(c_0,\dots,c_p)$ of alternating $\R$-multilinear maps 
\[
c_k\colon\underbrace{\Gamma(A)\times\dots\times \Gamma(A)}_{p-k \text{ copies}}\ra \Omega^{q-k}(M;S^k(A^*)\otimes V),
\]
whose failure at being $C^\infty(M)$-multilinear is controlled by the \textit{Leibniz identity}
\[
c_k(f\alpha_1,\alpha_2,\dots,\alpha_{p-k}\|\cdot)=fc_k(\alpha_1,\dots,\alpha_{p-k}\|\cdot)+\d f\wedge c_{k+1}(\alpha_2,\dots,\alpha_{p-k}\|\alpha_1,\cdot).
\]
Here, the notation reflects the fact that each $c_k$ in total inputs $p$ sections of $A$, accounting also for the $k$ arguments in which $c_k$ is symmetric. More precisely, we write
\[c_k(\alpha_1,\dots,\alpha_{p-k}\|\beta_1,\dots,\beta_k)=c_k(\ul\alpha\|\ul\beta) \in \Omega^{q-k}(M;V)\]
to mean that $c_k$ is $\R$-multilinear and antisymmetric in the first $p-k$ arguments, and  $C^\infty(M)$-multilinear and symmetric in the last $k$ arguments. The elements of $W^{p,q}(A;V)$ will be called \textit{Weil cochains}. The term $c_0$ will be called the \textit{leading term}, while the higher terms will be called the \textit{correction terms}.

\begin{remark}
There is a general principle that we will often use, as in \cite{weil}. Suppose we want to define a map $F\colon C\ra W^{p,q}(A;V)$ to the Weil complex from some set $C$. To do so, we will often prescribe the leading term $(Fc)_0$ and then infer the correction terms $(Fc)_1,\dots,(Fc)_p$ by consecutively applying the Leibniz identity. The leading term will always have a clear conceptual meaning, whereas the correction terms will usually be more complicated. We will usually omit this heuristic procedure of obtaining the correction terms, and will instead provide the proof of well-definedness of $F$. 
\end{remark}
As in the setting of groupoids, the sets $W^{p,q}(A;V)$ form a cochain complex, with the differential $\delta\colon W^{p,q}(A;V)\ra W^{p+1,q}(A;V)$ defined as follows. First  note that for arbitrary $m,n\geq 0$ the space $\Omega^m(M;S^n(A^*)\otimes V)$ is a module for the Lie algebra $\Gamma(A)$, with the action given  by the Lie derivative induced by the representation, defined by the chain rule:
\begin{align*}
(\L^A_\alpha\gamma)(\ul\beta)(\ul X)&=\nabla^A_\alpha\gamma(\ul\beta)(\ul X)\\
&-\textstyle\sum_i\gamma(\beta_1,\dots,[\alpha,\beta_i],\dots,\beta_n)(\ul X)\\
&-\textstyle\sum_j\gamma(\ul\beta)(X_1,\dots,[\rho(\alpha),X_j],\dots, X_m).
\end{align*}
Then, the leading term $(\delta c)_0$ for a given sequence $c=(c_0,\dots,c_p)\in W^{p,q}(A;V)$ is defined as the \textit{Koszul differential} of the leading term $c_0$, with respect to this module structure:
\begin{align*}
(\delta c)_0 (\alpha_0,\dots,\alpha_p)&=\textstyle\sum_i(-1)^{i}\L^A_{\alpha_i} \big(c_0(\alpha_0,\dots,\widehat{\alpha_i},\dots,\alpha_{p})\big)\\
&+\textstyle\sum_{i<j}(-1)^{i+j} c_0([\alpha_i,\alpha_j],\alpha_0,\dots,\widehat{\alpha_i},\dots,\widehat{\alpha_j},\dots,\alpha_{p}).
\end{align*}
The correction terms can be found by consecutively applying the Leibniz rule; as obtained in \cite{homogeneous}*{equation 4.2}, they read
\begin{align}
\begin{split}
\label{eq:delta_inf}
(-1)^k(\delta c)_k &(\alpha_0,\dots,\alpha_{p-k}\|\beta_1,\dots,\beta_k)\\
&=\textstyle\sum_{i=0}^{p-k}(-1)^{i}\L^A_{\alpha_i} \big(c_k(\alpha_0,\dots,\widehat{\alpha_i},\dots,\alpha_{p-k}\|\cdot)\big)(\ul\beta)\\
&+\textstyle\sum_{i<j}^{p-k}(-1)^{i+j} c_k([\alpha_i,\alpha_j],\alpha_0,\dots,\widehat{\alpha_i},\dots,\widehat{\alpha_j},\dots,\alpha_{p-k}\|\ul\beta)\\
&-\textstyle\sum_{j=1}^k \iota_{\rho(\beta_j)} c_{k-1}(\ul\alpha\|\beta_1,\dots,\widehat{\beta_j},\dots,\beta_k).
\end{split}
\end{align}
Notably, the $(p+1)$-th term simply reads
\[
(\delta c)_{p+1}(\ul\beta)=(-1)^p\textstyle\sum_j \iota_{\rho(\beta_j)} c_p(\beta_1,\dots,\widehat{\beta_j},\dots,\beta_p).
\]
\begin{definition}
  Let $A\Ra M$ be a Lie algebroid with a representation $V$. The differential $\delta$ is called the \textit{simplicial differential} of Weil cochains. At any fixed degree $q\geq 0$, it defines a cochain complex called the \textit{Weil complex}, 
  \[
  (W^{\bullet,q}(A;V),\delta),
  \]
  whose cohomology is called the \textit{simplicial cohomology} of Weil cochains,
\begin{align}
  \label{eq:weil_cohomology}
      H^{p,q}(A;V)\coloneqq H^p(W^{\bullet,q}(A;V),\delta).
\end{align}
\end{definition}
\begin{example}
At $p=0$, we have $W^{0,q}(A;V)=\Omega^q(M;V)$ and $\delta$ is given by
\begin{align}
\label{eq:delta0_im}
(\delta c)_0(\alpha)=\L^A_\alpha c, \qquad(\delta c)_1(\beta)=\iota_{\rho(\beta)}c.
\end{align}
The cocycles at $p=0$ are called \textit{invariant forms}, denoted $\Omega^q_{\inv}(M;V)$. At level $p=1$, the simplicial differential of a Weil cochain $c=(c_0,c_1)$ reads 
\begin{align}
\label{eq:delta_p=1}
\begin{split}
  (\delta c)_0(\alpha_1,\alpha_2)&=\L^A_{\alpha_1} c_0(\alpha_2)-\L^A_{\alpha_2} c_0(\alpha_1)-c_0[\alpha_1,\alpha_2],\\
  (\delta c)_1(\alpha,\beta)&=-(\L^A_{\alpha} c_1)(\beta)+\iota_{\rho(\beta)}c_0(\alpha)=-\L^A_{\alpha} (c_1(\beta))+ c_1[\alpha,\beta]+\iota_{\rho(\beta)}c_0(\alpha),\\
  (\delta c)_2(\beta_1,\beta_2)&=-\iota_{\rho(\beta_1)} c_1(\beta_2)-\iota_{\rho(\beta_2)} c_1(\beta_1).
\end{split}
\end{align}
The cocycles at $p=1$ are called \textit{infinitesimal multiplicative forms} (more briefly, \textit{IM forms}) or \textit{Spencer operators}, and we denote $\Omega^\bullet_{im}(A;V)=\ker\delta^1$. By the equations above, a cocycle $c=(c_0,c_1)$ satisfies
\begin{align}
  c_0[\alpha_1,\alpha_2]&=\L^A_{\alpha_1} c_0(\alpha_2)-\L^A_{\alpha_2} c_0(\alpha_1),\label{eq:c1}\tag{C.1}\\
  c_1[\alpha,\beta]&=\L^A_{\alpha} (c_1(\beta))-\iota_{\rho(\beta)}c_0(\alpha),\label{eq:c2}\tag{C.2}\\
  \iota_{\rho(\beta_1)} c_1(\beta_2)&=-\iota_{\rho(\beta_2)} c_1(\beta_1).\label{eq:c3}\tag{C.3}
  \end{align}
The equations \eqref{eq:c1}--\eqref{eq:c3} will be referred to as the \textit{compatibility conditions} for $(c_0,c_1)$; as noted in \cite{spencer}*{Remark 2.7}, condition \eqref{eq:c2} follows from the Leibniz identity and \eqref{eq:c1} unless $\dim M =q+1$, however, it nonetheless often turns out to be useful to keep it in mind. The cocycles $\im\delta^0\subset \Omega^\bullet_{im}(A;V)$ arising as coboundaries are called \textit{cohomologically trivial} IM forms. 
\end{example}

\section{Van Est theorem for representation--valued forms}
The relationship between the Bott--Shulman--Stasheff and Weil complex is provided by the so-called \emph{van Est map}. The central property of this map is that under certain connectedness assumptions, it is a quasi-isomorphism---this is called the \emph{van Est theorem}. Let us first state it and then provide an explicit defining formula; the idea of proof is given later in \sec\ref{sec:idea_van_est}.
\begin{theorem}[\cite{homogeneous}*{Theorem 4.4}]
  \label{thm:homogeneous}
    Let $G$ be a Lie groupoid with Lie algebroid $A$, and let $V$ be its representation. The van Est map is a cochain map $\ve\colon \Omega^{\bullet,q}(G;V)\ra W^{\bullet,q}(A;V)$ at each fixed degree $q$, that is,
    \begin{align*}
        \ve\circ\delta_G=\delta_A\circ\ve.
    \end{align*}
    Moreover, if the Lie groupoid $G$ has $p_0$-connected $s$-fibres, then the induced map in cohomology,
    \[
        H^{p,q}(G;V)\ra H^{p,q}(A;V),
    \] 
    is an isomorphism for all $p\leq p_0$ and injective for $p=p_0+1$, for any fixed degree $q\geq 0$.
    \end{theorem}
The van Est map is defined as follows. Given a section $\alpha\in\Gamma(A)$, we will denote the flow of the left-invariant vector field $\cev\alpha$ at time $\lambda$ by $\phi^{\smash{\cev\alpha}}_\lambda$. Furthermore, $\cev\alpha$ defines a vector field on the $p$-th level of the nerve $G^{(p)}$, which we will denote by 
\[
\alpha^{(p)}|_{(g_1,\dots,g_p)}\coloneqq (0_{g_1},\dots,0_{g_{p-1}},\cev\alpha_{g_p}),
\]
and whose flow is given by $\phi^{\alpha^{(p)}}_\lambda(g_1,\dots,g_p)=(g_1,\dots,g_{p-1},\phi^{\cev\alpha}_\lambda(g_p))$. We now define the operator, used to define the leading term of $\ve$, which lowers the level of a given form:
\begin{align*}
&R_\alpha\colon \Omega^{p,q}(G;V)\ra \Omega^{p-1,q}(G;V),\\
&R_\alpha \omega|_{(g_1,\dots,g_{p-1})}=j_p^*\left(\deriv\lambda 0\phi^{\cev\alpha}_\lambda \big(1_{s(g_{p-1})}\big)\cdot \big(\phi^{\alpha^{(p)}}_\lambda\big)^*\omega\right),
\end{align*}
where $j_p\colon G^{(p-1)}\ra G^{(p)}$ is the degeneracy map that inserts the identity into the last factor (for an even more explicit expression of $R_\alpha$ see Example \ref{ex:RJ} below). This enables us to define the leading term of the van Est map:
\begin{align}
\label{eq:ve_leading}
\ve (\omega)_0(\alpha_1,\dots,\alpha_p)=\sum_{\sigma\in S_p}(\sgn\sigma) R_{\alpha_{\sigma(1)}}\dots R_{\alpha_{\sigma(p)}}\omega.
\end{align}
The correction terms are obtained from equation \eqref{eq:ve_leading} using the Leibniz rule. To write them, one shows that $R_{f\alpha}=f R_\alpha+\d f\wedge J_\alpha$, where $J_\alpha\colon \Omega^{p,q}(G;V)\ra \Omega^{p-1,q-1}(G;V)$ reads
\begin{align}
\label{eq:J_alpha}
  J_\alpha\omega=j_p^*(\iota_{\alpha^{(p)}}\omega).
\end{align}
As obtained in \cite{homogeneous}*{equation 4.4}, the correction terms then read
\begin{align*}
\ve(\omega)_k(\alpha_1,\dots,\alpha_{p-k}\|\beta_1,\dots,\beta_k)=(-1)^{\frac{k(k+1)}2}\sum_{\smash[b]{\sigma\in S_p}}(\sgn\sigma) (-1)^{\epsilon(\sigma,k)} D_{{\sigma(1)}}\dots D_{{\sigma(p)}}\omega,
\end{align*}
where we are denoting
\begin{align*}
D_j&=\begin{cases}
J_{\beta_j},&\text{if }1\leq j\leq k,\\
R_{\alpha_{k-j}},&\text{if }k+1\leq j\leq p,
\end{cases}\\
\epsilon(\sigma,k)&=\#\set{(i,j)\in\set{1,\dots,k}^2\given i<j\text{ and }\sigma^{-1}(i)>\sigma^{-1}(j)}.
\end{align*}
The sign that $\epsilon$ induces is sometimes also called the \textit{Koszul sign}.
\begin{example}
\label{ex:RJ}
At level $p=1$, the van Est map on a form $\omega\in \Omega^q(G;s^*V)$ reads
\begin{align}
\label{eq:van_est_p=1}
\begin{split}
    \ve(\omega)_0(\alpha)(X_i)_{i=1}^q&=(R_\alpha\omega)_x(X_i)_i=\deriv\lambda 0\phi^{\alpha^L}_\lambda(1_x)\cdot \omega\big({\d(\phi^{\cev\alpha}_\lambda)}_{1_{x}}\d u_{x} X_i\big)_i,\\
    \ve(\omega)_1(\beta)(X_i)_{i=1}^{q-1}&=(J_\beta\omega)_x(X_i)_i=\omega(\beta_x,\d u_x (X_1),\dots, \d u_x (X_{q-1}))
\end{split}
\end{align}
for any $X_i\in T_x M$ and $\alpha,\beta\in\Gamma(A)$. At level $p=2$, the maps $R_\alpha$ and $J_\alpha$ on $\omega\in \Omega^{2,q}(G;V)$ read
\begin{align*}
(R_\alpha\omega)_g(X_i)_{i=1}^q&=\deriv\lambda 0\phi^{\cev\alpha}_\lambda (1_{s(g)})\cdot \omega\big(X_i,\d(\phi^{\cev\alpha}_\lambda)_{1_{s(g)}}\d u_{s(g)} \d s_g X_i\big)_i,\\
(J_\alpha\omega)_g(X_i)_{i=1}^{q-1}&=\omega_{(g,1_{s(g)})}\big((0_g,\alpha_{s(g)}),(X_i,\d u_{s(g)} \d s_g X_i)_i\big)
\end{align*}
for any vectors $X_i\in T_g G$. We hereby provide some intuition regarding the signs appearing in the van Est map. The first correction term of $\ve(\omega)$ for $\omega$ as above reads
\[
\ve(\omega)_1(\alpha\|\beta)=-(J_\beta R_\alpha- R_\alpha J_\beta)\omega.
\]
As another example, the second correction term $\ve(\omega)_2(\alpha\|\beta_1,\beta_2)$ for $\omega\in\Omega^{3,q}(G;V)$ equals
\begin{align*}
  -(J_{\beta_1}J_{\beta_2}R_\alpha+J_{\beta_2}J_{\beta_1}R_\alpha-J_{\beta_1}R_\alpha J_{\beta_2}-J_{\beta_2}R_\alpha J_{\beta_1}+R_\alpha J_{\beta_1}J_{\beta_2}+R_\alpha J_{\beta_2}J_{\beta_1})\omega\in \Omega^{q-2}(M;V).
\end{align*}

\noindent A pictorial way of interpreting the signs in this expression, ignoring the initial factor $(-1)^{k(k+1)/2}$, is provided by the diagram below. Consider all paths $\Omega^{3,q}(G;V)$ to $\Omega^{q-2}(M;V)$ following horizontal and diagonal lines. Starting with the path that first follows the maximal number of horizontal segments, we obtain $JJR$. Take a positive sign there, symmetrize the expression in the indices of $J$'s and antisymmetrize in the indices of $R$'s. For other possible paths, we follow a similar procedure, with the difference that for any change of path $JR\ra RJ$, we obtain an additional minus sign; this uncovers the role of the Koszul sign.
{
\[\begin{tikzcd}
	{\Omega^{q}(M;V)} & {\Omega^{q}(G;s^*V)} & {\Omega^{2,q}(G;V)} & {\Omega^{3,q}(G;V)} \\
	{\Omega^{q-1}(M;V)} & {\Omega^{q-1}(G;s^*V)} & {\Omega^{2,q-1}(G;V)} & {\Omega^{3,q-1}(G;V)} \\
	{\Omega^{q-2}(M;V)} & {\Omega^{q-2}(G;s^*V)} & {\Omega^{2,q-2}(G;V)} & {\Omega^{3,q-2}(G;V)}
	\arrow[no head, from=1-1, to=1-2]
	\arrow[no head, from=1-2, to=1-3]
	\arrow["J"', color={rgb,255:red,32;green,121;blue,63}, from=1-3, to=2-2]
	\arrow["R"', color={rgb,255:red,32;green,121;blue,63}, from=1-4, to=1-3]
	\arrow["J"', color={rgb,255:red,186;green,28;blue,30}, from=1-4, to=2-3]
	\arrow["J", shift left, color={rgb,255:red,14;green,72;blue,139}, from=1-4, to=2-3]
	\arrow[no head, from=2-1, to=1-1]
	\arrow[no head, from=2-1, to=2-2]
	\arrow[no head, from=2-2, to=1-2]
	\arrow["J"', shift right, color={rgb,255:red,32;green,121;blue,63}, from=2-2, to=3-1]
	\arrow["J", color={rgb,255:red,186;green,28;blue,30}, from=2-2, to=3-1]
	\arrow[no head, from=2-3, to=1-3]
	\arrow["R", color={rgb,255:red,186;green,28;blue,30}, from=2-3, to=2-2]
	\arrow[no head, from=2-3, to=2-4]
	\arrow["J", color={rgb,255:red,14;green,72;blue,139}, from=2-3, to=3-2]
	\arrow[no head, from=2-4, to=1-4]
	\arrow[no head, from=3-1, to=2-1]
	\arrow[no head, from=3-2, to=2-2]
	\arrow["R", color={rgb,255:red,14;green,72;blue,139}, from=3-2, to=3-1]
	\arrow[no head, from=3-2, to=3-3]
	\arrow[no head, from=3-3, to=2-3]
	\arrow[no head, from=3-3, to=3-4]
	\arrow[no head, from=3-4, to=2-4]
\end{tikzcd}\]
}
\end{example}
\subsection*{Van Est theorem for multiplicative forms with coefficients}
As already stated, the most important level is $p=1$. It turns out that  the van Est theorem there actually already holds at the level of cocycles, i.e., on multiplicative forms. This first appeared as one of the main results in \cite{spencer}, cf.\ Theorem 1.
\begin{corollary}
  \label{corollary:van_est_multiplicative}
    Let $G$ be a Lie groupoid with Lie algebroid $A$, and suppose $V$ is its representation. The formula \eqref{eq:van_est_p=1} maps  multiplicative forms on $G$ with values in $V$ to IM forms on $A$ with values in $V$, that is, for any $q\geq 0$, \[\ve\colon \Omega_m^q(G;s^*V)\ra \Omega_{im}^q(A;V).\]
    Moreover, if $G$ has simply connected source fibres, this is an isomorphism.
\end{corollary}
\begin{remark}
    The proof presented in \cite{spencer} utilizes jets of bisections, which is different from the proof of Theorem \ref{thm:homogeneous} established in \cite{homogeneous}, which uses $\vb$-groupoids and $\vb$-algebroids. 
\end{remark}
\begin{proof}
    The first part clearly follows from the fact that $\ve$ is a cochain map. For the second part, fix the degree $q\geq 0$, and first note that the assumption in particular says $G$ has connected $s$-fibres, so it ensures that $G$-invariant forms coincide with $A$-invariant forms on $M$, i.e., $\ker\delta^0_G=\ker\delta^0_A$. Hence, by the first isomorphism theorem and the fact that $\ve\circ\delta^0_G=\delta^0_A$, the van Est map restricts to an isomorphism $\im\delta^0_G\cong \im\delta^0_A$. Using Zorn's Lemma, we can obtain a complementary vector subspace to $\im\delta^0_G\subset \Omega_m^q(G;V)$, the image of which under $\ve$ must be complementary to $\im\delta^0_A\subset \Omega_{im}^q(A;V)$ since $\ve$ induces an isomorphism $H^{1,q}(G;V)\cong H^{1,q}(A;V)$ by Theorem \ref{thm:homogeneous}. Hence, $\ve$ induces an isomorphism between $\Omega_m^q(G;s^*V)$ and $\Omega_{im}^q(A;V)$  as well.
\end{proof}

\section{Groupoids and algebroids in the category of vector bundles}
\label{sec:vb}
In this section, we recall the fundamental notions related to the theory of $\vb$-groupoids and their infinitesimal counterparts, $\vb$-algebroids. We will review some constructions and examples that will turn out to be useful for our work in \sec\ref{sec:weil_boi}. 

\subsection{\texorpdfstring{$\vb$}{VB}-groupoids}
The results presented here follow the references \cite{mackenzie}*{\sec 11.2}, \cite{gracia-saz}*{\sec 3} and \cite{bundles_over_gpds}*{\sec 3.1}.
\begin{definition}
    A \emph{$\vb$-groupoid} consists of two Lie groupoids $\Gamma\rra E$, $G\rra M$, such that $\Gamma\ra G$ and $E\ra M$ are vector bundles, satisfying the following compatibility conditions. 
\[\begin{tikzcd}[row sep=large, column sep=large]
	\Gamma & G \\
	E & M
	\arrow["{\ul q}", from=1-1, to=1-2]
	\arrow["{\ul s}"'{pos=0.6}, shift right, from=1-1, to=2-1]
	\arrow["{\ul t}"{pos=0.6}, shift left, from=1-1, to=2-1]
	\arrow["t"{pos=0.5}, shift left, from=1-2, to=2-2]
	\arrow["s"'{pos=0.52}, shift right, from=1-2, to=2-2]
	\arrow["q", from=2-1, to=2-2]
\end{tikzcd}\]
    \begin{enumerate}[label={(\roman*)}]
      \item The source and target maps $\ul s, \ul t\colon \Gamma\ra E$ are vector bundle maps covering $s,t\colon G\ra M$.
      \item The bundle projection $\ul q\colon \Gamma\ra G$ is a Lie groupoid morphism covering $q\colon E\ra M$.
      \item The interchange law holds for any $(\gamma_1,\gamma_2), (\gamma_3,\gamma_4)\in \Gamma^{(2)}$ with $\ul q(\gamma_1)=\ul q(\gamma_3), \ul q(\gamma_2)=\ul q(\gamma_4)$:
      \[
      (\gamma_1+\gamma_3)(\gamma_2+\gamma_4)=\gamma_1\gamma_2+\gamma_3\gamma_4.
      \]
    \end{enumerate}
\end{definition}
\begin{remark}
  There exist several equivalent definitions of $\vb$-groupoids, see \cite{gracia-saz}*{Proposition 3.5}. Roughly, apart from the one above, we can also view a $\vb$-groupoid as a \emph{Lie groupoid internal to the category of vector bundles}, or as a \emph{vector bundle internal to the category of Lie groupoids}. These three definitions are conveniently terse, however, in practice it helps to keep in mind all the identities regarding the compatibility of the structures on $\Gamma$, found in \cite{mackenzie}*{\sec 11.2, p.\ 415}.
\end{remark}
\begin{definition}
  Given a $\vb$-groupoid as above, its (left) \emph{core} is the vector bundle
  \[
  C=u^* \ker\ul t\longrightarrow M.
  \]
  Denoting the zero section by $\ul 0\colon G\ra \Gamma$, the left multiplication by $\ul 0_g$ gives an isomorphism $C_{s(g)}\cong \ker \ul t_g$ for any $g\in G$, so we have a short exact sequence of vector bundles over $G$,
\begin{align}
  \label{eq:core_ses}
  \begin{tikzcd}[ampersand replacement=\&]
  	0 \& {s^*C} \& \Gamma \& {t^*E} \& 0,
  	\arrow[from=1-1, to=1-2]
  	\arrow[from=1-2, to=1-3]
  	\arrow["{\ul t}", from=1-3, to=1-4]
  	\arrow[from=1-4, to=1-5]
  \end{tikzcd}
\end{align}
called the (left) \emph{core exact sequence}. It is customary to denote $\partial = \ul s|_C\colon C\ra E$.
\end{definition}
\begin{remark}
  \label{rem:fat_cat}
  As an aside, we note that the splittings of the sequence \eqref{eq:core_ses} which additionally satisfy the natural requirement that they coincide with the canonical splitting $\Gamma|_M=E\oplus C$ at the units of $G$, are precisely the so-called \emph{representations up to homotopy} of $G$ on $\partial\colon C\ra E$ \cite{ruth}.
  
There is a relationship of representations up to homotopy with Lie categories from \sec\ref{chapter:lie_cats}. Namely, the splittings of sequence \eqref{eq:core_ses} can be viewed using the so-called \emph{fat category}, 
  \[
    \hat C(\Gamma)=\set{(g,H)\given g\in G,\ H\subset \Gamma_g\text{ is a subspace complementary to}\,\ker\ul t_g}.
  \]
  This is a Lie category $\hat C(\Gamma)\rra M$ with the obvious source and target map, and multiplication given by $(g_1,H_1)(g_2,H_2)=(g_1g_2,H_1H_2)$, where $H_1H_2=\set{v_1v_2\given v_i\in H_i,\ \ul s(v_1)=\ul t(v_2)}$. The differentiable structure is obtained by observing that the fibre at $g\in G$ of the canonical projection $\hat C(\Gamma)\ra G$ is an affine space modelled on $\smash{\Hom(E_{t(g)},C_{s(g)})}$.
  We note that the invertibles of $\hat C(\Gamma)$ constitute an open wide subgroupoid $\hat G(\Gamma)$ by Theorem \ref{thm:open_inv}, called the \textit{fat groupoid}, whose elements are precisely the pairs $(g,H)$ such that $H$ is also complementary to $\ker\ul s_g$. 

  The splittings of the sequence \eqref{eq:core_ses} are then equivalent to the sections of the surjective submersive functor $\hat C(\Gamma)\ra G$. The natural requirement $\Gamma|_M=E\oplus C$ on a splitting, yielding representations up to homotopy, corresponds to the requirement on a section $G\ra \hat C(\Gamma)$ that it restricts at units to $1_x\mapsto (1_x, \ul u(E_x))$. If $\Gamma$ is the tangent groupoid (see Example \ref{example:vb_groupoids} (i) below), the sections $G\ra \hat C(TG)$ which are moreover morphisms of Lie categories are precisely the so-called \emph{Cartan connections} \cite{cartan_connections}.
\end{remark}
\begin{example}\
  \label{example:vb_groupoids}
  \begin{enumerate}[label={(\roman*)}]
    \item Given a Lie groupoid $G\rra M$, differentiating the structure maps yields the \emph{tangent groupoid} of $G$, which is a $\vb$-groupoid. Its core is precisely the Lie algebroid $A\ra M$, with $\partial=\rho$.
    \[\begin{tikzcd}
      TG & G \\
      TM & M
      \arrow[from=1-1, to=1-2]
      \arrow[shift left, from=1-1, to=2-1]
      \arrow[shift right, from=1-1, to=2-1]
      \arrow[shift left, from=1-2, to=2-2]
      \arrow[shift right, from=1-2, to=2-2]
      \arrow[from=2-1, to=2-2]
    \end{tikzcd}\]
    \item Given a representation $p\colon V\ra M$ of $G\rra M$, it defines a \emph{prolongation} $G\tensor[_s]{\times}{_p}V\rra M$ of $G$ by $V$, which is a $\vb$-groupoid, with the structure maps given by $\ul s=s\circ\pr_1$, $\ul t=t\circ\pr_1$, and
      $(g,v)\cdot (h,w)=(gh, h^{-1}\cdot v+w).$
    Its core is the representation $V\ra M$.
    \[\begin{tikzcd}
      {G\tensor[_s]{\times}{_p}V} & G \\
      {0_M} & M
      \arrow[from=1-1, to=1-2]
      \arrow[shift left, from=1-1, to=2-1]
      \arrow[shift right, from=1-1, to=2-1]
      \arrow[shift right, from=1-2, to=2-2]
      \arrow[shift left, from=1-2, to=2-2]
      \arrow[from=2-1, to=2-2]
    \end{tikzcd}\]
  \item Alternatively, since any representation $G\curvearrowright V$ is in particular an action of $G$, one can also consider the \emph{action $\vb$-groupoid} $G\ltimes V$, defined as the groupoid $G\tensor[_s]{\times}{_p}V\rra V$ with the structure maps given by $\ul s(g,v)=v$, $\ul t(g,v)=g\cdot v$, and $(g,h\cdot w)\cdot (h,w)=(gh, w)$. Its core is the zero vector bundle $0_M\ra M$.
  \[\begin{tikzcd}
      {G\tensor[_s]{\times}{_p}V} & G \\
      V & M
      \arrow[from=1-1, to=1-2]
      \arrow[shift left, from=1-1, to=2-1]
      \arrow[shift right, from=1-1, to=2-1]
      \arrow[shift right, from=1-2, to=2-2]
      \arrow[shift left, from=1-2, to=2-2]
      \arrow[from=2-1, to=2-2]
    \end{tikzcd}\]
\end{enumerate}
\end{example}
\noindent We will need the following two natural constructions of $\vb$-groupoids.
\subsubsection{The dual of a $\vb$-groupoid}
The \emph{dual} of a $\vb$-groupoid $\Gamma\rra E$ is the $\vb$-groupoid, defined in such a way that its (right) core short exact sequence is the dual of the (left) core short exact sequence \eqref{eq:core_ses}. More precisely, it is defined as the $\vb$-groupoid
\[\begin{tikzcd}
	{\Gamma^*} & G \\
	{C^*} & M
	\arrow[from=1-1, to=1-2]
	\arrow[shift left, from=1-1, to=2-1]
	\arrow[shift right, from=1-1, to=2-1]
	\arrow[shift left, from=1-2, to=2-2]
	\arrow[shift right, from=1-2, to=2-2]
	\arrow[from=2-1, to=2-2]
\end{tikzcd}\]
with the source and target maps defined on any $\psi\in \Gamma_g^*$  as
\begin{align*}
  \begin{aligned}
    \ul s^*(\psi)&\in C_{s(g)}^*,& \quad\inner{\ul s^*(\psi)}{c}&=\inner{\psi}{0_g\cdot c},\\
    \ul t^*(\psi)&\in C_{t(g)}^*,& \inner{\ul t^*(\psi)}{d}&=-\innersmall{\psi}{d^{-1}\cdot 0_g},
  \end{aligned}
\end{align*}
for all $c\in C_{s(g)}$, $d\in C_{t(g)}$, where the bracket denotes the canonical pairing. We also note that $d^{-1}=\ul 1_{\partial d}-d$ by the compatibility condition. For the composition, let $\psi_1\in \Gamma_g^*$,  $\psi_2\in \Gamma^*_h$ be two covectors such that $\ul s^*(\psi_1)=\ul t^*(\psi_2)$. Any element of $\Gamma_{gh}$ can be written as a product $\gamma_1\cdot\gamma_2$ for some $\gamma_1\in \Gamma_g$, $\gamma_2\in \Gamma_h$, and now the compatibility condition ensures that
\[
\inner{\psi_1\cdot\psi_2}{\gamma_1\cdot\gamma_2}=\inner{\psi_1}{\gamma_1}+\inner{\psi_2}{\gamma_2}
\]
is a well-defined covector $\psi_1\cdot\psi_2\in\Gamma_{gh}^*$. The unit element $\ul 1^*_{\theta}$ for a covector $\theta\in C^*_x$ is defined as the covector $\ul 1_\theta^*\in\Gamma^*_{1_x}$ given by
\[
\inner{\ul 1_\theta^*}{\ul 1_v+c}=\inner{\theta}{c},
\]
for any $v\in E_x$, $c\in C_x$, where we have noted there is a canonical decomposition $\Gamma_{1_x}=E_x\oplus C_x$. The inverses are then given by $\innersmall{\psi^{-1}}{\gamma^{-1}}=-\inner{\psi}{\gamma}$ for any $\gamma\in \Gamma_g$.
\begin{example}\
  \label{ex:vb_groupoid_dual}
  \begin{enumerate}[label={(\roman*)}]
    \item The dual of the tangent groupoid $TG\rra TM$ is the \textit{cotangent groupoid} $T^*G\rra A^*$. This Lie groupoid is ubiquitous in Poisson geometry; one of the reasons for this is the fact that the canonical symplectic form on $T^*G$ is automatically multiplicative---see \cite{multiplicative_tensors} for details.
    \item Given a representation of $G$ on $V$, the dual of the prolongation $\vb$-groupoid from Example \ref{example:vb_groupoids} (ii) is precisely the action $\vb$-groupoid of the induced action of $G$ on $V^*$, defined as
    \[
    \inner{g\cdot\zeta}{\xi}=\innersmall{\zeta}{g^{-1}\cdot \xi},\ \text{for any $\zeta\in V^*_{s(g)}$ and $\xi\in V_{t(g)}$.}
    \]
    
\end{enumerate}
\end{example}

\subsubsection{Whitney sum of $\vb$-groupoids}
Given two $\vb$-groupoids $\Gamma_1\rra E_1$ and $\Gamma_2\rra E_2$ over the same groupoid $G\rra M$, we can take the Whitney sum of the underlying vector bundles. In this way, we again obtain a $\vb$-groupoid, with the groupoid operations defined componentwise. It is called the \emph{Whitney sum} of $\vb$-groupoids.
\[\begin{tikzcd}
	{\Gamma_1\oplus\Gamma_2} & G \\
	{E_1\oplus E_2} & M
	\arrow[from=1-1, to=1-2]
	\arrow[shift left, from=1-1, to=2-1]
	\arrow[shift right, from=1-1, to=2-1]
	\arrow[shift left, from=1-2, to=2-2]
	\arrow[shift right, from=1-2, to=2-2]
	\arrow[from=2-1, to=2-2]
\end{tikzcd}\]

\subsection{Double vector bundles}
Before defining the notion of a $\vb$-algebroid, we preliminarily recall the underlying notion of a double vector bundle, along with some properties and constructions. 
\begin{definition}
A \emph{double vector bundle} (DVB) consists of two vector bundles $D\ra E$ and $D\ra A$ with the same total space $D$, whose base spaces are vector bundles over the same base manifold $M$, such that the two vector bundles with total space $D$ are \emph{compatible}, in the following sense. Let us denote the two sets of structure maps (addition, zero map, homogeneous structure) of the vector bundles $D\ra A$ and $D\ra E$ as follows:
\begin{equation}
  \label{eq:dvb}
  \vcenter{\hbox{
    \begin{tikzcd}
      D & A \\
      E & M
      \arrow["{q_A^D}", from=1-1, to=1-2]
      \arrow["{q_E^D}"', from=1-1, to=2-1]
      \arrow["{q_A}", from=1-2, to=2-2]
      \arrow["{q_E}"', from=2-1, to=2-2]
    \end{tikzcd}
  }}
  \qquad
  \begin{aligned}
    &+_A\colon D\oplus_A D\ra D,\quad 0_A\colon A\ra D,\quad  h_A^\lambda\colon D\ra D,
  \\
  &+_E\colon D\oplus_E D\ra D,\quad 0_E\colon A\ra D,\quad  h_E^\lambda\colon D\ra D.
  \end{aligned}
  \end{equation}
Both vector bundle structures on $A\ra M$, $E\ra M$ will be denoted by the same symbols: $+$, $0$, $h^\lambda$.
Compatibility conditions read: for any $(d_1,d_2),(d_3,d_4)\in D\oplus_A D$ and $(d_1,d_3),(d_2,d_4)\in D\oplus_E D$,
\begin{align}
    q^D_E(d_1+_A d_2)&=q^D_E(d_1)+q^D_E(d_2),\nonumber\\
    q^D_A(d_1+_E d_3)&=q^D_A(d_1)+q^D_A(d_3),\nonumber\\
    (d_1+_A d_2)+_E(d_3+_A d_4)&=(d_1+_E d_3)+_A(d_2+_E d_4).\label{eq:dvb_interchange_law}
\end{align}
Equation \eqref{eq:dvb_interchange_law} is also known as the \emph{interchange law}, and the vector bundles $A\ra M$ and $E\ra M$ are called the \emph{side bundles} of $D$.
\end{definition}
\begin{remark}
  As with $\vb$-groupoids, there are several ways of describing DVB's, see \cite{dvb}*{Proposition 2.1}. In practice, it is also useful to keep in mind all the other interchange laws which follow from \eqref{eq:dvb_interchange_law}, found in \cite{mackenzie_duality}*{Equations (2)--(7)}. For instance, 
  \[
  h^\lambda_A \circ h_E^\mu=h_E^\mu \circ h^\lambda_A,\quad 0_A (\lambda a)=h^\lambda_E (0_A(a)), \quad h^\lambda_E(d_1+_A d_2)=h^\lambda_E (d_1)+_A h^\lambda_E (d_2),
  \]
  for any $a\in A$, $(d_1,d_2)\in D\oplus_A D$ and $\lambda,\mu\in \R$. 
\end{remark}
\begin{definition}
  The \emph{core} of a DVB as above is the bundle \[C=\ker q^D_A\cap \ker q^D_E\longrightarrow M.\]
  Importantly, we note that the interchange law implies that the restrictions to $C$ of the two vector bundle structures on $D$ coincide, so $C\ra M$ is naturally a vector bundle \cite{mackenzie_duality}*{\sec 2}.
  The natural inclusion $C\hookrightarrow D$ yields the \emph{core exact sequence}, which is a short exact sequence of double vector bundles, where $C$ is seen as a DVB with side bundles $0_M\ra M$, and $A\oplus_M E$ is the DVB with side bundles $A\ra M$ and $E\ra M$.
\[\begin{tikzcd}
	0 & C & D & {A\oplus_M E} & 0
	\arrow[from=1-1, to=1-2]
	\arrow[hookrightarrow, from=1-2, to=1-3]
	\arrow[from=1-3, to=1-4]
	\arrow[from=1-4, to=1-5]
\end{tikzcd}\] 
\end{definition}
\begin{remark}
  Apart from the sequence above, there are two other short exact sequences of vector bundles (in the usual sense), sometimes called the \emph{core exact sequences over $A$ and $E$}. For instance, the one over $E$ reads
\begin{align}
  \label{eq:core_ses_over_E}
  \begin{tikzcd}[ampersand replacement=\&]
  	0 \& (q_E)^!C \& {D_E} \& {(q_E^D)^!A} \& 0,
  	\arrow[from=1-1, to=1-2]
  	\arrow["{\tau_E}", from=1-2, to=1-3]
  	\arrow["{(q^D_A)^!}", from=1-3, to=1-4]
  	\arrow[from=1-4, to=1-5]
  \end{tikzcd}
\end{align}
where $D_E$ simply denotes the vector bundle $D\ra E$, the shriek denotes the usual pullback of a vector bundle along a smooth map, and the two maps read 
\[\tau_E(e,c)=0_E(e)+_A c, \quad (q_A^D)^!(d)=(q_E^D(d),q_A^D(d)).\]
\end{remark}
\subsubsection{Linear and core sections}
For a double vector bundle $D$, there are now two distinct spaces of sections of $D$, namely, the sections of the two projections $q_A^D$ and $q_E^D$. We will denote them by $\Gamma(A,D)$ and $\Gamma(E,D)$, respectively. By the symmetry of the definition of a DVB, let us focus solely on $\Gamma(E,D)$. As the title of this subsection suggests, there are two important special types of sections in $\Gamma(E,D)$. 
\begin{definition}
  A \emph{linear section} of $D\ra E$ is a vector bundle morphism $\sigma\colon E\ra D$ from $E\ra M$ to $D\ra A$, covering a section $\alpha\colon M\ra A$. Linear sections are denoted by $\Gamma_\lin(E,D)$.
\[\begin{tikzcd}
	D & A \\
	E & M
	\arrow["{q_A^D}", from=1-1, to=1-2]
	\arrow[from=1-1, to=2-1]
	\arrow[from=1-2, to=2-2]
	\arrow["\sigma", bend left=20, shift left=1, from=2-1, to=1-1]
	\arrow["{q_E}"', from=2-1, to=2-2]
	\arrow["\alpha"', bend right=20, shift right=1, from=2-2, to=1-2]
\end{tikzcd}\]
A \emph{core section} is any section of $D\ra E$ obtained in the following way: suppose $\beta\in \Gamma(C)$ is a section of the core $C\subset D$, and define $\beta_c\colon E\ra D$ as
\[
\beta_c(e)=\beta(q_E(e))+_A 0_E(e),
\]
where the second term is necessary to ensure $\beta_c$ is a section of $D\ra E$. The space of core sections is denoted by $\Gamma_c(E,D)$.
\end{definition}
The importance of these sections is in the following fact: as a $C^\infty(E)$-module, the space of sections $\Gamma(E,D)$ is generated by the linear and core sections. This is proved in \cite{mackenzie_doubles}*{Proposition 2.2} and will often be used without reference.
\subsubsection{The two duals of a double vector bundle}
As the title of this subsection suggests, it is possible to dualize a DVB in two distinct ways, with respect to either $D\ra A$ or $D\ra E$. Let us focus on dualizing with respect to $D\ra A$.

\newcommand{\ds}[1]{{D_{\!{#1}}^*}}
The \emph{dual} of a DVB with respect to $D\ra A$ is the DVB denoted $\ds{A}$. Before writing out the structure maps, we depict it in the following diagram. 
\begin{align}
  \label{eq:dual_dvb}
  \begin{tikzcd}[ampersand replacement=\&]
  	{\ds{A}} \& A \\
  	{C^*} \& M
  	\arrow["{\overline q_{A}}", from=1-1, to=1-2]
  	\arrow["{\overline q_{C^*}}"', from=1-1, to=2-1]
  	\arrow["{q_A}", from=1-2, to=2-2]
  	\arrow["{q_{C^*}}"', from=2-1, to=2-2]
  \end{tikzcd}
\end{align}
Here, all the maps are known but $\overline q_{C^*}$, and we have to define the vector bundle structure $\ds{A} \ra C^*$.
\begin{itemize}
  \item The bundle projection $\overline q_{C^*}$ is defined for any $\psi\in \ds{A}$ over $\overline q_A(\psi)=a$ and $c\in C_{q_A(a)}$ as
  \[
  \inner{\overline q_{C^*}(\psi)}{c}=\inner{\psi}{0_A(a)+_E c}.
  \]
  \item Addition over $C^*$: for any $\psi_1,\psi_2\in\ds{A}$ with $\overline q_{C^*}(\psi_1)=\overline q_{C^*}(\psi_2)$ and $\overline q_A(\psi_i)=a_i\in A$, define
  \[
  \inner{\psi_1+_C\psi_2}{d}=\inner{\psi_1}{d_1}+\inner{\psi_2}{d_2},
  \]
  where $d\in (q_A^D)^{-1}(a_1+a_2)$ is decomposed as $d=d_1+_E d_2$ for some vectors $d_i\in (q_A^D)^{-1}(a_i)$. Well-definedness of this map is straightforwardly shown using $\overline q_{C^*}(\psi_1)=\overline q_{C^*}(\psi_2)$ and the compatibility conditions.
  \item Homogeneous structure over $C^*$: for any $\psi\in \ds{A}$ and $d\in D$ with $\overline q_A(\psi)=q_A^D(d)$, define
  \[
  \innersmall{h^\lambda_{C^*}(\psi)}{h^\lambda_{E}(d)}=\lambda\inner{\psi}{d},\ \text{for all }\lambda\in\R.
  \]
  \item Zero map over $C^*$: it is defined as the map $0_{C^*}\colon C^*\ra \ds{A}$, given as
  \[
  \innersmall{0_{C^*}(\zeta)}{0_E(e)+_A c}=\inner{\zeta}{c},
  \]
  for any $\zeta\in C_x^*$, $e\in E_x$ and $c\in C_x$. By the short exact sequence \eqref{eq:core_ses_over_E} over $E$, it is enough to define $0_{C^*}(\zeta)$ on such vectors.
\end{itemize}
We note the core of $\ds{A}$ can be identified with $E^*\ra M$, using the linear isomorphism
\begin{align*}
  E_x^*\ni\vartheta\mapsto \overline\vartheta\in (\overline q_{C^*})^{-1}(0_x)\cap (\overline q_{A})^{-1}(0_x),\quad \innersmall{\overline\vartheta}{0_E(e)+_A c}=\inner{\vartheta}{e},
\end{align*}
for any $e\in E_x, c\in C_x$. 
\begin{remark}
  The two duals we obtain in this way fit into a triple vector bundle, called the \emph{cotangent cube}---for details see \cite{dvb}*{Remark 3.12} or \cite{bundles_over_gpds}*{Diagram 2.8}. 
\end{remark}
\subsection{\texorpdfstring{$\vb$}{VB}-algebroids}
\begin{definition}
  A \emph{$\vb$-algebroid} is a double vector bundle $D$ as in \eqref{eq:dvb}, equipped with a Lie algebroid structure on $D\ra E$, i.e., an anchor $\rho_D\colon D\ra TE$ and a Lie bracket $[\cdot,\cdot]_D$ on the sections $\Gamma(E,D)$, with the following additional properties.
  \begin{enumerate}[label={(\roman*)}]
    \item The anchor $\rho_D$ is a vector bundle morphism over a bundle map $\rho_A\colon A\ra TM$ covering $\id_M$.
\[\begin{tikzcd}[row sep=tiny]
	D & A \\
	&&[-1em] M \\
	TE & TM
	\arrow["{q_A^D}", from=1-1, to=1-2]
	\arrow["{\rho_D}"', from=1-1, to=3-1]
	\arrow["{q_A}", from=1-2, to=2-3]
	\arrow["{\rho_A}", from=1-2, to=3-2]
	\arrow["{\d{q_E}}"', from=3-1, to=3-2]
	\arrow[from=3-2, to=2-3]
\end{tikzcd}\]
    \item The bracket $[\cdot,\cdot]_D$ satisfies the following properties on the linear and core sections:
    \begin{align}
      \label{eq:vb_algebroid_bracket}
      \begin{split}
        [\Gamma_\lin(E,D),\Gamma_\lin(E,D)]&\subset \Gamma_\lin(E,D),\\
        [\Gamma_\lin(E,D),\Gamma_c(E,D)]&\subset \Gamma_c(E,D),\\
        [\Gamma_c(E,D),\Gamma_c(E,D)]&=0.
      \end{split}
    \end{align} 
  \end{enumerate}
\end{definition}
\begin{remark}
  The vector bundle $A\ra M$ is automatically a Lie algebroid with anchor $\rho_A$, and the bracket induced by the bracket on $D\Ra E$, using the first of the three equalities above.
\end{remark}
\noindent We now consider the infinitesimal analogues of the $\vb$-groupoids from Example \ref{example:vb_groupoids}; for a general integrability result of $\vb$-algebroids to $\vb$-groupoids, we direct the reader to \cite{vb_integrability}.
\begin{example}\
  \label{example:vb_algebroids}
  \begin{enumerate}[label={(\roman*)}]
    \item The \emph{tangent algebroid} of a Lie algebroid $A\Ra M$ is the following $\vb$-algebroid. 
\[\begin{tikzcd}
	TA & A \\
	TM & M
	\arrow[from=1-1, to=1-2]
	\arrow[Rightarrow, from=1-1, to=2-1]
	\arrow[Rightarrow, from=1-2, to=2-2]
	\arrow[from=2-1, to=2-2]
\end{tikzcd}\]
    The structure of $TA\ra A$ is given by the usual tangent bundle construction, whereas the structure of $TA\ra TM$ is obtained by differentiating the structure maps of $A\ra M$. The bracket of $[\cdot,\cdot]_{TA}$ is essentially defined by \eqref{eq:vb_algebroid_bracket}---that is, for any section $\alpha\in\Gamma(A)$, we note that $\d \alpha\colon TM\ra TA$ is a linear section over $\alpha\colon M\ra A$, and we thus define
    \begin{align*}
      [\d\alpha,\d\beta]_{TA}=\d[\alpha,\beta],\quad 
      [\d\alpha,\beta_c]_{TA}=[\alpha,\beta]_c,\quad
      [\alpha_c,\beta_c]_{TA}=0,
    \end{align*}
    where $\beta\in\Gamma(A)$ is another section. Extending by bilinearity and the Leibniz rule, this determines the bracket on all sections. The anchor of $TA$ is defined as $\rho_{TA}=J\circ \d\rho$, where $J\colon T(TM)\ra T(TM)$ is the canonical involution, see \cite{michor}*{\sec 8.13}. With this structure, $TA$ becomes a $\vb$-algebroid with the property that the bundle projection $TA\ra A$ is a Lie algebroid morphism over $TM\ra M$. Moreover, if $A$ is the algebroid of $G$, then $TA\Ra TM$ is the Lie algebroid of the tangent groupoid $TG\rra TM$ from Example \ref{example:vb_groupoids} (i). For more details, we direct the reader to \cite{bialgebroids}*{Theorems 5.1 and 7.1}.
    \item Let $V$ be a representation of the Lie algebroid $A\Ra M$. The \emph{prolongation $\vb$-algebroid} of $A$ by $V$ is the following $\vb$-algebroid. 
\[\begin{tikzcd}
	{A\oplus V} & A \\
	{0_M} & M
	\arrow[from=1-1, to=1-2]
	\arrow[Rightarrow, from=1-1, to=2-1]
	\arrow[Rightarrow, from=1-2, to=2-2]
	\arrow[from=2-1, to=2-2]
\end{tikzcd}\]
The bundle structure $A\oplus V\ra A$ is just the usual pullback structure, whereas the bundle structure over $0_M$ is just the Whitney sum structure. The anchor and the bracket are given for any $\alpha,\beta\in\Gamma(A)$ and $\xi,\eta\in\Gamma(V)$ by
\begin{align*}
  [(\alpha,\xi),(\beta,\eta)]&=([\alpha,\beta],\nabla^A_\alpha\eta-\nabla^A_\beta\xi),\\
  \rho(\alpha,\xi)&=\rho(\alpha).
\end{align*}
Assuming the algebroid and the representation can be integrated, this is the Lie algebroid of the $\vb$-groupoid from example \ref{example:vb_groupoids} (ii).
\item Alternatively, given any representation of $A\Ra M$ on $V$, we can form the \emph{action Lie algebroid} $A\ltimes V$, which is a $\vb$-algebroid on account of linearity of the action $\Gamma(A)\ra \vf(V)$.
\[\begin{tikzcd}
	{A\oplus V} & A \\
	V & M
	\arrow[from=1-1, to=1-2]
	\arrow[Rightarrow, from=1-1, to=2-1]
	\arrow[Rightarrow, from=1-2, to=2-2]
	\arrow[from=2-1, to=2-2]
\end{tikzcd}\]
Here, the anchor is given by the action, and the bracket is induced by the bracket on $A$, see Remark \ref{rem:action_trick} for more details on this construction.
  \end{enumerate}
\end{example}

\subsubsection{The dual of a $\vb$-algebroid}
Given a $\vb$-algebroid $D$, its \emph{dual $\vb$-algebroid} is defined as follows. First off, the underlying DVB is the dual of the DVB $D$ with respect to the fibration over $A$, just as in diagram \eqref{eq:dual_dvb}.
\[\begin{tikzcd}
	{\ds{A}} & A \\
	{C^*} & M
	\arrow[from=1-1, to=1-2]
	\arrow[Rightarrow, from=1-1, to=2-1]
	\arrow[Rightarrow, from=1-2, to=2-2]
	\arrow[from=2-1, to=2-2]
\end{tikzcd}\]
To see how $\ds{A}\ra C^*$ becomes a Lie algebroid, note that since $D\Ra E$ is a Lie algebroid, $\ds{E}\ra E$ carries a (fibrewise) linear Poisson structure. Moreover, this Poisson structure on $D_E^*$ is also linear with respect to the fibration over $C^*$, when $\ds{E}$ is viewed as a DVB (this is called a \emph{double linear Poisson structure})---see \cite{mackenzie_duality}*{Theorem 2.4} or \cite{dvb}*{Theorem 3.10}. Hence, $(\ds{E})_{C^*}^*$ also has a Lie algebroid structure. Now, the canonical pairing $\ds{A}\oplus_{C^*} \ds{E}\ra \R$ induces an isomorphism $(\ds{E})_{C^*}^*\cong \ds{A}$ of DVB's, see \cite{mackenzie}*{Theorems 9.2.2 and 9.2.4}. This yields the wanted $\vb$-algebroid structure on the double vector bundle $\ds{A}$.
\begin{example}\
  \begin{enumerate}[label={(\roman*)}]
    \item The dual of the tangent algebroid $TA\Ra TM$ is the \emph{cotangent algebroid} $T^*A\Ra A^*$.
    \item Analogously as in Example \ref{ex:vb_groupoid_dual}, the dual of the Example \ref{example:vb_algebroids} (ii) is precisely the action $\vb$-algebroid $A\ltimes V^*$ of the induced action of $A$ on $V^*$, defined by the Leibniz rule
    \[
    \rho(\alpha)\inner{\zeta}{\xi}=\inner{\nabla^A_\alpha\zeta}{\xi}+\inner{\zeta}{\nabla^A_\alpha\xi}
    \]
    for any $\alpha\in\Gamma(A)$, $\zeta\in\Gamma(V^*)$ and $\xi\in\Gamma(V)$.
  \end{enumerate}
\end{example}

\subsubsection{Whitney sum of $\vb$-algebroids}
Just as in the case of $\vb$-groupoids, the Whitney sum is a well-behaved operation in the world of $\vb$-algebroids. That is, given two $\vb$-algebroids $D_1\Ra E_1$ and $D_2\Ra E_2$ covering the same Lie algebroid $A\Ra M$, we can define their Whitney sum with respect to the fibration over $A$.
\[\begin{tikzcd}
	{D_1\oplus_AD_2} & A \\
	{\hspace{0.25em} E_1\oplus_M E_2} & M
	\arrow[from=1-1, to=1-2]
	\arrow[shift right=1.3, Rightarrow, from=1-1, to=2-1]
	\arrow[Rightarrow, from=1-2, to=2-2]
	\arrow[from=2-1, to=2-2]
\end{tikzcd}\]
Here, the vector bundle structure fibred over $E_1\oplus_M E_2$, the anchor, and the bracket are all defined componentwise.

\section{Alternative model for the Weil complex}
\label{sec:alternative_model_weil}
In this section, we provide an alternative model for representation-valued Weil cochains we have seen in \sec\ref{sec:bss_weil}. This model was introduced in \cite{homogeneous}, and will be crucial in \sec\ref{sec:derivation_horproj_weil}. Given a representation $V$ of an algebroid $A\Ra M$, the idea is to view the cochain complex $W^{\bullet,q}(A;V)$ using the usual complex underlying the algebroid cohomology, but of a larger algebroid. It turns out the right algebroid to consider is the $\vb$-algebroid, obtained as the Whitney sum of $q$ copies of $TA\Ra TM$ and the action algebroid of the induced dual representation of $A$ on $V^*$:
\begin{equation}
  \label{eq:big_vb_algebroid}
  \vcenter{\hbox{
    \begin{tikzcd}
      {\AA_q} & A \\
      {\MM_q} & M
      \arrow[from=1-1, to=1-2]
      \arrow[Rightarrow, from=1-1, to=2-1]
      \arrow[Rightarrow, from=1-2, to=2-2]
      \arrow[from=2-1, to=2-2]
    \end{tikzcd}
  }}
  \qquad
  \begin{aligned}
  \AA_q &= \oplus^q_A TA\oplus_A \pi^* V^*,\quad (\pi\colon A\ra M)
  \\
  \MM_q &= \oplus^q_M TM\oplus_M V^*.
  \end{aligned}
  \end{equation}
Let us illuminate its $\vb$-algebroid structure.
\begin{itemize}
  \item The vector bundle structure fibred over $A$ is just the usual Whitney sum structure, while the one fibred over $\MM_q$ is induced componentwise by $TA\ra TM$ and $\pi^*V^*\ra V^*$, that is,
  \begin{align*}
    (X_1,\dots,X_q,(a,\zeta))+_{\MM_q}(X_1',\dots,X_q',(a',\zeta))&=(X_1 +_{TM} X_1',\dots,X_q +_{TM} X_q',(a+a',\zeta)),\\
    0_{\MM_q}(w_1,\dots,w_q,\zeta)&=(0_{TM}(w_1),\dots,0_{TM}(w_q),(0_x,\zeta)),
  \end{align*}
  where the vectors $X_i\in T_a A$, $X_i'\in T_{a'}A$ satisfy $\d\pi(X_i)=\d\pi(X_i')$, and $\zeta\in V_{\pi(a)}^*$.
  \item As usual, the space of sections $\Gamma(\MM_q,\AA_q)$ is generated as a $C^\infty(\MM_q)$-module by the linear and core sections. In this case, these are sections of the following form. For $\alpha\in\Gamma(A)$, 
  \begin{align}
    \label{eq:linear_and_core_sections}
  \begin{split}
      \T\alpha{(w_1,\dots,w_q,\zeta)}&=(\d\alpha(w_1),\dots,\d\alpha(w_q),\chi_\alpha(\zeta)),\\
      \Z_i\alpha{(w_1,\dots,w_q,\zeta)}&=\Big(0_{TM}(w_1),\dots,0_{TM}(w_i)+_A\deriv\lambda 0 \lambda \alpha_x,\dots,0_{TM}(w_q),0_\zeta\Big).
  \end{split}
  \end{align}
  Here, $w_j\in T_x M, \zeta\in V^*_x$ and we have introduced the following vectors in $\pi^*V^*$,
  \[
    \chi_\alpha(\zeta)=(\alpha_x,\zeta),\quad\text{and}\quad 0_\zeta=(0_x,\zeta).
  \]
  \item On the linear and core sections, for any $\alpha,\beta\in\Gamma(A)$, the bracket reads
\begin{align}
  \label{eq:bracket_Z_T}
  [\T\alpha,\T\beta]_{\AA_q}=\T[\alpha,\beta],\quad  [\T\alpha,\Z_i\beta]_{\AA_q}=\Z_i[\alpha,\beta], \quad  [\Z_i\alpha,\Z_j\beta]_{\AA_q}=0.
\end{align}
The anchor is defined componentwise, using the anchors of $TA\Ra TM$ and $\pi^*V^*\Ra V^*$.
\end{itemize}
\subsection{Exterior cochains}
With the $\vb$-algebroid $\AA_q\Ra\MM_q$ in mind, we now define the alternative model to Weil cochains. 

\begin{definition}
  \label{def:exterior_cochains}
  Given any section $\omega\in\Gamma(\MM_q,\Lambda^p\AA_q^*)$,\footnote{The dual and the wedge are with respect to the bundle structure fibred over $\MM_q$.} we say that:
  \begin{enumerate}[label={(\roman*)}]
    \item $\omega$ is \textit{skew-symmertic} with respect to $\AA_q\ra A$, if for any $\sigma\in S_q$,
    \[
    (\sigma_A)^*\omega=\sgn(\sigma)\omega,
    \]
    where $\sigma_A\colon\AA_q\ra\AA_q$ permutes the $q$ components in $\oplus^q_A TA$ according to $\sigma$.
    \item $\omega$ is \textit{multilinear} with respect to $\AA_q\ra A$, if it is
    \begin{align*}
      &\textit{$(q+1)$-homogeneous}\text{:}\hspace{-6em}&(h^\lambda_A)^*\omega&=\lambda^{q+1}\omega,\\
      &\text{and }\textit{simple}\text{:}&(0^i_A)^*\omega&=0,\quad (i=1,\dots,q+1),
    \end{align*}
    where $h_\lambda^A\colon \AA_q\ra\AA_q$ is the homogeneous structure of $\AA_q\ra A$, and the maps $0^i_A\colon\AA_{q-1}\ra\AA_q$ and $0^{q+1}_A\colon \oplus^q_A TA\ra \AA_q$ insert a zero at the $i$-th and $(q+1)$-th factor, respectively.
  \end{enumerate}
  A skew-symmetric and multilinear section  with respect to $\AA_q\ra A$ will be called an \textit{exterior cochain}, and the set of all exterior cochains will be denoted by
  \[
  \Gamma_{\ext}(\MM_q,\Lambda^p\AA_q^*).
  \]
\end{definition}
It is shown in \cite{homogeneous}*{Proposition 4.10} that $\Gamma_{\ext}(\MM_q,\Lambda^\bullet\AA_q^*)\subset \Gamma(\MM_q,\Lambda^\bullet\AA_q^*)$ is a subcomplex of the usual cochain complex used to define the algebroid cohomology of $\AA_q\Ra \MM_q$. The differential there is just the standard one: for any sections $X^i\in\Gamma(\MM_q,\AA_q)$,
\begin{align}
  \label{eq:delta_vb}
\begin{split}
    \delta\omega(X^0,\dots, X^p)&=\textstyle\sum_i(-1)^i\L_{\rho_{\AA_q}(X^i)}\omega(X^0,\dots,\widehat{X^i},\dots,X^p)\\
    &+\textstyle\sum_{i<j}\omega([X^i,X^j]_{\AA_q},X^0,\dots,X^i,\dots,X^j,\dots,X^p).
\end{split}
\end{align}
Here, $\L$ denotes the usual directional derivative of a function. 
\begin{remark}
  To demystify the abstract definition of multilinearity with respect to $\AA_q\ra A$, let us try to see it as multilinearity in the usual sense. Denoting the exterior cochains at the level $p=0$ by $C^\infty_\ext(\MM_q)\coloneqq\Gamma_{\ext}(\MM_q,\Lambda^0\AA_q^*)$, they form a subset of $C^\infty(\oplus^q_M TM\oplus_M V^*)$. If $\omega\in C^\infty_\ext(\MM_q)$ is $(q+1)$-homogeneous and simple, it is also linear in each component by a standard argument as in Euler's homogeneous function theorem. Hence, it is multilinear in the usual sense, and we conclude 
  \[C^\infty_\ext(\MM_q)=\Gamma(\Lambda^q(T^*M)\otimes (V^*)^*)\cong\Omega^q(M;V). \]
  Similarly, at level $p=1$, multilinearity of $\omega\in \Gamma_\ext(\MM_q,\Lambda^1\AA_q^*)$ with respect to $\AA_q\ra A$ is just a terse way of expressing
  \[
  \omega(X_1+_A h_A^\lambda X_1',X_2\dots,X_q,(a,\zeta))=\omega(X_1,\dots,X_q,(a,\zeta))+\lambda \omega( X_1',\dots,X_q,(a,\zeta)),
  \]
  for any $\lambda\in \R$ and vectors $X_1',X_i\in T_a A, \zeta\in V^*_{\pi(a)}$, and similarly for the other arguments and for linearity in $(a,\zeta)$. Going one level further, at $p=2$ multilinearity with respect to $\AA_q\ra A$ reads
\begin{align*}
  &\omega\big((X_1+_A h_A^\lambda X_1',X_2\dots,X_q,(a,\zeta)),(Y_1+_A h_A^\lambda Y_1',Y_2\dots,Y_q,(b,\zeta))\big)\\
  &=\omega\big((X_1,\dots,X_q,(a,\zeta)),(Y_1,Y_2\dots,Y_q,(b,\zeta))\big)+\lambda \omega\big( (X_1',\dots,X_q,(a,\zeta)),(Y_1',Y_2\dots,Y_q,(b,\zeta))\big),
\end{align*}
where now the vectors $Y_1',Y_i\in T_b A$ and $b\in\pi^{-1}(a)$ satisfy $\d\pi(Y_1')=\d\pi(X_1')$ and $\d\pi(Y_i)=\d\pi(X_i)$.
\end{remark}
\subsection{The isomorphism of the two models}
We now describe the canonical identification of the model of exterior cochains with the model of Weil cochains, called the \textit{evaluation map} and denoted by
\begin{align}
  \label{eq:evaluation_map}
  \ev\colon\Gamma_{\ext}(\MM_q,\Lambda^p\AA_q^*)\xrightarrow{} W^{p,q}(A;V).
\end{align}
The wanted map $\ev$ is defined simply by evaluating $\omega$ on the generating sections; this is done  via auxiliary maps. More precisely, for any $\omega\in\Gamma_{\ext}(\MM_q,\Lambda^p\AA_q^*)$, we define
\begin{align}
  \label{eq:def_exterior_weil_generators}
\begin{split}
    &\tilde c_k(\omega)\colon{\times}^p\Gamma(A)\ra C^\infty(\MM_q),\\
    &\tilde c_k(\omega)(\alpha_1,\dots,\alpha_{p-k}\|\beta_1,\dots,\beta_k)=\omega(\Z_1\beta_1,\dots,\Z_k\beta_k,\T\alpha_1,\dots,\T\alpha_{p-k}),
\end{split}
\end{align}
if $k\leq q$, and zero for all $k>q$. In turn, these maps define functions
\begin{align}
  \label{eq:def_exterior_weil_ck}
  \begin{split}
    &c_k(\omega)(\ul\alpha\|\ul\beta)\in C^\infty(\oplus^{q-k}_MTM\oplus_M V^*)\\
    &c_k(\omega)(\ul\alpha\|\ul\beta)(w_1,\dots,w_{q-k},\zeta)=\tilde c_k(\omega)(\ul\alpha\|\ul\beta)(0_x,\dots,0_x,w_1,\dots,w_{q-k},\zeta),
  \end{split}
\end{align}
which are shown to be $(q-k+1)$-homogeneous, simple and skew-symmetric, so they correspond to differential forms $c_k(\omega)(\ul\alpha\|\ul\beta)\in\Omega^{q-k}(M;V)$.
The maps $c_k(\omega)$ can thus be stacked in a sequence as terms of the wanted Weil cochain corresponding to $\omega$,
\[\ev(\omega)\coloneqq\big(c_0(\omega),\dots,c_p(\omega)\big)\in W^{p,q}(A;V).\]
For any fixed $q$, this map provides the wanted isomorphism of the complex of exterior cochains with the Weil complex, as shown in \cite{homogeneous}*{Proposition A.3}. We observe that the complexity of the simplicial differential \eqref{eq:delta_inf} of Weil cochains is now encoded entirely in the Lie algebroid $\AA_q\Ra \MM_q$, since the simplicial differential on exterior cochains \eqref{eq:delta_vb} is just the standard one.
\begin{remark}
  \label{rem:any_vectors_ck}
  It is not hard to see that we can actually pick any vectors instead of $0_x$'s in the defining equation \eqref{eq:def_exterior_weil_ck}, by multilinearity of $\omega$ over $\AA_q\ra A$. That is,
  \[
    c_k(\omega)(\ul\alpha\|\ul\beta)(w_1,\dots,w_{q-k},\zeta)=\tilde c_k(\omega)(\ul\alpha\|\ul\beta)(\tilde w_1,\dots,\tilde w_k,w_1,\dots,w_{q-k},\zeta),
  \]
  for any vectors $\tilde w_i\in T_x M$. For example, in the case $p=1$, we have
  \begin{align*}
    \tilde c_1(\omega)(\beta)(\tilde w,w_1,\dots,w_{q-1},\zeta)&=\omega\big(0_{TM}(\tilde w)+_A\smallderiv\lambda 0\lambda \beta_x,0_{TM}(w_1),\dots 0_{TM}(w_{q-1}),0_\zeta\big)\\
    &=\omega\big(0_{\MM_q}(\tilde w,w_1,\dots,w_{q-1},\zeta)\big)+\tilde c_1(\omega)(\beta)(0_x,w_1,\dots,w_{q-1},\zeta),
  \end{align*}
  where multilinearity over $A$ was used in the second equality. The first term then vanishes by multilinearity over $\MM_q$, since $\omega\in\Gamma(\MM_q,\Lambda^p\AA_q^*)$.
\end{remark}
\begin{remark}
  Although isomorphic, the two models are in practice very different to work with---the Weil cochain model offers a more hands-on, computational approach, while the alternative model offers more conceptual clarity.
\end{remark}

\subsection{Idea of proof of van Est theorem for representation-valued forms}
\label{sec:idea_van_est}
Equipped with the model of exterior cochains, we now portray the idea of proof of Theorem \ref{thm:homogeneous}.
\begin{itemize}
  \item Analogously as with algebroids, one views the complex of forms $\Omega^{\bullet,q}(G;V)$ using the complex for the usual differentiable cohomology of the following $\vb$-groupoid:
  \begin{equation}
    \label{eq:big_vb_groupoid}
    \vcenter{\hbox{
      \begin{tikzcd}
        {\GG_q} & G \\
        {\MM_q} & M
        \arrow[from=1-1, to=1-2]
        \arrow[shift left, from=1-1, to=2-1]
        \arrow[shift right, from=1-1, to=2-1]
        \arrow[shift left, from=1-2, to=2-2]
        \arrow[shift right, from=1-2, to=2-2]
        \arrow[from=2-1, to=2-2]
      \end{tikzcd}
    }}
    \qquad
    \begin{aligned}
    \GG_q &= \oplus^q_G TG\oplus_G s^* V^*,
    \\
    \MM_q &= \oplus^q_M TM\oplus_M V^*.
    \end{aligned}
    \end{equation}
    Namely, one considers $(q+1)$-homogeneous, simple, skew-symmetric functions on the nerve of $\GG_q$. As shown in \cite{homogeneous}*{Proposition 4.8}, they form a subcomplex \[C^\infty_\ext(\GG_q^{(\bullet)})\subset C^\infty(\GG_q^{(\bullet)}).\]
    The isomorphism of complexes $\mathfrak F\colon \Omega^{\bullet,q}(G;V)\ra C^\infty_\ext(\smash{\GG_q^{(\bullet)}})$ is given by evaluation.
    \item Next, one shows that the following diagram commutes \cite{homogeneous}*{Appendix A.3}.
\[\begin{tikzcd}[row sep=large]
	{\Omega^{p,q}(G;V)} & {C^\infty_\ext(\GG_q^{(p)})} \\
	{W^{p,q}(G;V)} & {\Gamma_{\ext}(\MM_q,\Lambda^p\AA_q^*)}
	\arrow["{\frak F}", from=1-1, to=1-2]
	\arrow["\ve"', from=1-1, to=2-1]
	\arrow["{\ve_{\ext}}", from=1-2, to=2-2]
	\arrow["\ev", from=2-2, to=2-1]
\end{tikzcd}\]
Here, $\ve_\ext$ denotes the restriction of the usual van Est map 
\begin{align}
  \label{eq:usual_van_est_GG_AA}
  C^\infty(\smash{\GG_q^{(\bullet)}})\ra \Gamma(\MM_q,\Lambda^\bullet\AA_q^*),
\end{align} 
which indeed maps to exterior cochains, since the canonical projections to the “ext” subcomplexes commute with the van Est map (\cite{homogeneous}*{Proposition 4.13}). 
\item One uses the ordinary van Est theorem for the van Est map $C^\infty(\smash{\GG_q^{(\bullet)}})\ra \Omega^\bullet(\AA_q)$, together with the following fact: the source fibre of a $\vb$-groupoid is an affine bundle over the corresponding source fibre of the base groupoid, hence $\GG_q\rra \MM_q$ is source $p_0$-connected if $G\rra M$ is. Hence, \eqref{eq:usual_van_est_GG_AA} is a quasi-isomorphism for all $p\leq p_0$ and quasi-injective for $p=p_0+1$, and thus so is $\ve_\ext$.
\end{itemize}

\clearpage \pagestyle{plain}
\chapter{Invariant linear connections on representations}\label{chapter:invariant}
\pagestyle{fancy}
\fancyhead[CE]{Chapter \ref*{chapter:invariant}} 
\fancyhead[CO]{Invariant linear connections}

In this chapter, we present new results regarding invariant connections on general representations of Lie groupoids and Lie algebroids. Invariant connections provide an important step towards fully answering the question if representation-valued forms on the nerve of a Lie groupoid form a double complex, as is the case with trivial coefficients. They were first considered in \cite{mec}*{Appendix A}, where the invariance condition was obtained and briefly studied. We hereby further develop the theory, and provide certain improvements which will turn out to be crucial in \sec\ref{chapter:mec}. Hence, this chapter also serves as an important stepping stone for advancing the theory of multiplicative Ehresmann connections. The results obtained here can be found in the preprint \cite{covariant_derivatives}.

\section{The global picture}

Let $V$ be a representation of a Lie groupoid $G\rra M$. We now assume that a connection $\nabla$ on the vector bundle $V\ra M$ is given, with no a priori additional assumptions regarding the compatibility of $\nabla$ with the groupoid action $G\curvearrowright V$. 
\begin{definition}
  The \textit{exterior covariant derivative} of $V$-valued forms $\Omega^{p,q}(G;V)$ is given by
  \[
  \d{}^{\nabla^{s\circ\pr_p}}\colon \Omega^{p,q}(G;V)\ra \Omega^{p,q+1}(G;V),
  \]
where $\nabla^{s\circ\pr_p}$ denotes the pullback along $s\circ \pr_p\colon G^{(p)}\ra M$ of a given connection $\nabla$ on $V$. When no confusion arises, we will simply denote it by $\d{}^\nabla$. 
\end{definition}

The following result tells us precisely when $\d{}^\nabla$ commutes with $\delta$. It was already obtained in \cite{mec}*{Proposition A.7} for $p=1$, however, we hereby prove it by computing the explicit formula for the commutator 
\begin{align}
  \label{eq:commutator}
  [\d{}^\nabla,\delta]\colon \Omega^{p,q}(G;V)\ra \Omega^{p+1,q+1}(G;V)
\end{align}
which was not obtained there, and will be needed in \sec\ref{chapter:mec}.
\begin{theorem}
\label{thm:G_invariant}
Let $\nabla$ be a connection on a representation $V$ of a Lie groupoid $G\rra M$. The map $\d{}^\nabla$ is a cochain map if and only if $\nabla$ is $G$-invariant, that is, if the following tensor vanishes:
\begin{align}
  \label{eq:invariance_form_G}
  \Theta\in \Omega^1(G;\Hom(t^*V,s^*V)),\quad \Theta(X)\xi=\phi(\nabla_X^t\xi)-\nabla_X^s \phi (\xi),
\end{align}
for any $X\in \vf(G)$ and $\xi\in\Gamma(t^*V)$.  Hence, if $\nabla$ is $G$-invariant, $\d{}^\nabla$ preserves multiplicativity.
\end{theorem}
\begin{lemma}
\label{lem:G_invariant}
Let $\nabla$ be a connection on a representation $V$ of a Lie groupoid $G\rra M$. At level $p=0$, the commutator of $\delta$ and $\d{}^\nabla$ reads
\begin{align}
\label{eq:commutator_zero}
  (\d{}^{\nabla^{s}}\delta^0-\delta^0\d{}^{\nabla})\omega=\Theta\wedge t^*\omega
\end{align}
for any $\omega\in \Omega^q(M;V)$. At any higher level $p\geq 1$, there holds
\begin{align}
\label{eq:commutator_higher}
  (\d{}^{\nabla^{s\circ\pr_{p+1}}}\delta^p-\delta^p\d{}^{\nabla^{s\circ\pr_p}})\omega=(-1)^{p}(\pr_{p+1})^*\Theta\wedge\big (f^{(p+1)}_{p+1}\big)^*\omega
\end{align}
for any $\omega\in\Omega^{p,q}(G;V)$. 
Explicitly, the $(q+1)$-form on $G^{(p+1)}$ on the right-hand side reads
\begin{align*}
  \Big((\pr_{p+1})^*\Theta&\wedge\big (f^{(p+1)}_{p+1}\big)^*\omega\Big)(X_1,\dots,X_{q+1})\\
  &=\sum_{i=1}^{q+1}(-1)^{i+1}\Theta(X_i^{p+1})\cdot\omega((X_1^1,\dots, X_1^p),\dots,\widehat{(X_i^1,\dots,X_i^p)},\dots (X_{q+1}^1,\dots, X_{q+1}^p)),
\end{align*}
for any tangent vectors $X_i=(X_i^1,\dots, X_i^{p+1})\in TG^{(p+1)}$ from the same fibre.
\end{lemma}
\begin{proof}
We will make use of the basic fact that if $E\ra B$ is any vector bundle equipped with a connection $\nabla$, and if $\pi\colon N\ra B$ is any smooth map, then the pullback connection $\nabla^\pi$ on $\Omega^\bullet(N;\pi^*E)$ is characterized by the following identity involving its exterior covariant derivative,
\begin{align}
\label{eq:pullback}
  \d{}^{\nabla^\pi}\pi^*=\pi^*\d{}^\nabla.
\end{align}
Using it on the face maps $\pi=f_i^{(p+1)}$ for $i\leq p$, with connection $\nabla^{s\circ\pr_p}$  in place of $\nabla$, the identity $s\circ\pr_p\circ \smash{f_i^{(p+1)}}=s\circ\pr_{p+1}$ yields
\[
\d{}^{\nabla^{s\circ\pr_{p+1}}}(f_i^{(p+1)})^*=(f_i^{(p+1)})^*\d{}^{\nabla^{s\circ\pr_p}}.
\]
Moreover, using \eqref{eq:pullback} on $\pi=f_{p+1}^{(p+1)}$, the identity $s\circ\pr_p\circ f^{(p+1)}_{p+1}=s\circ\pr_p=t\circ\pr_{p+1}$ yields
\[
  \d{}^{\nabla^{t\circ\pr_{p+1}}}(f_{p+1}^{(p+1)})^*=(f_{p+1}^{(p+1)})^*\d{}^{\nabla^{s\circ\pr_p}}.
\]
Using these two equalities and the defining equation \eqref{eq:delta_l} of $\delta$, we compute
\begin{align*}
  \d{}^{\nabla^{s\circ\pr_{p+1}}}\delta^p-\delta^p\d{}^{\nabla^{s\circ\pr_p}} &=(-1)^{p+1}\big(\d{}^{\nabla^{s\circ\pr_{p+1}}}\Phi_* (f_{p+1}^{(p+1)})^* -\Phi_*(f_{p+1}^{(p+1)})^*\d{}^{\nabla^{s\circ\pr_p}}\big)\\
  &=(-1)^{p+1}\big({\d{}^{\nabla^{s\circ\pr_{p+1}}}}\Phi_*-\Phi_*\d{}^{\nabla^{t\circ\pr_{p+1}}}\big)(f_{p+1}^{(p+1)})^*.
\end{align*}
Let us expand the differential operator on the right-hand side of the last expression, on an arbitrary form $\gamma\in\Omega^q(G^{(p+1)};(t\circ\pr_{p+1})^*V)$:
\begin{align*}
(\d{}^{\nabla^{s\circ\pr_{p+1}}}\Phi_*\gamma&-\Phi_*\d{}^{\nabla^{t\circ\pr_{p+1}}}\gamma)(X_1,\dots,X_{q+1})
\\
&=\textstyle\sum\limits_i(-1)^{i+1}\big(\nabla_{X_i}^{s\circ\pr_{p+1}}\circ\Phi-\Phi\circ\nabla_{X_i}^{t\circ\pr_{p+1}}\big)\gamma(X_1,\dots,\widehat X_i,\dots, X_{q+1}),
\end{align*}
for any vectors $X_i\in T G^{(p+1)}$ over the same point, where the terms with sums over double indices have cancelled out. But it is not hard to see that 
\[
\nabla^{s\circ\pr_{p+1}}\circ\Phi-\Phi\circ\nabla^{t\circ\pr_{p+1}}=-\pr_{p+1}^*\Theta,
\]
which is enough to be checked on pullback sections $(t\circ\pr_{p+1})^*\xi$ for $\xi\in \Gamma(V)$. Indeed, the stated equality is immediate once we use the relation $\Phi\circ\pr_{p+1}^*=\pr_{p+1}^*\circ\phi$. At last, taking $\smash{\gamma=(f_{p+1}^{(p+1)})^*}\omega$ concludes the proof of \eqref{eq:commutator_higher}. Identity \eqref{eq:commutator_zero} is proved similarly.
\end{proof}

Theorem \ref{thm:G_invariant} now follows by surjectivity and submersivity of face maps and projections, since the lemma above implies that $\Theta=0$ if and only if at any level $p\geq 0$ (hence, at all levels) the exterior covariant derivative commutes with the simplicial differential.
In conclusion, any $G$-invariant connection $\nabla$ on $V$ yields the columns of a curved double complex,
\[
(\Omega^{\bullet,\bullet}(G;V),\delta,\d{}^\nabla).
\]
We once again emphasize the notable feature that $\d{}^\nabla$ does not square to zero unless $\nabla$ is flat, in which case we obtain a flat double complex.
\begin{align*}
\begin{tikzcd}[ampersand replacement=\&, column sep=large, row sep=large]
	{\Omega^{q+1}(M;V)} \& {\Omega^{q+1}(G;s^*V)} \& {\Omega^{q+1}(G^{(2)};(s\circ\pr_2)^*V)} \& \cdots \\
	{\Omega^q(M;V)} \& {\Omega^{q}(G;s^*V)} \& {\Omega^q(G^{(2)};(s\circ\pr_2)^*V )} \& \cdots
	\arrow["{\delta}", from=1-1, to=1-2]
	\arrow["{\delta}", from=1-2, to=1-3]
	\arrow["{\delta}", from=1-3, to=1-4]
	\arrow["{\d{}^\nabla}", from=2-1, to=1-1]
	\arrow["{\delta}", from=2-1, to=2-2]
	\arrow["{\d{}^\nabla}", from=2-2, to=1-2]
	\arrow["{\delta}", from=2-2, to=2-3]
	\arrow["{\d{}^\nabla}", from=2-3, to=1-3]
	\arrow["{\delta}", from=2-3, to=2-4]
\end{tikzcd}
\end{align*}

\section{The infinitesimal picture}
Now suppose $V$ is a representation of a Lie algebroid $A\Ra M$. As before, we now again assume that a connection $\nabla$ on the representation $V\ra M$ is given, with no a priori assumptions on its compatibility with the algebroid action $A\curvearrowright V$.
\begin{definition}
  \label{def:d_nabla_weil}
  Let $\nabla$ be a connection on a representation $V$ of a Lie algebroid $A\Ra M$. The \textit{exterior covariant derivative} of $V$-valued Weil cochains is the map
  \[\d{}^\nabla\colon W^{p,q}(A;V)\ra W^{p,q+1}(A;V),\]
  whose leading term is defined on any $c\in W^{p,q}(A;V)$ as the exterior covariant derivative of its leading term $c_0$, that is,
  \[
  (\d{}^\nabla c)_0(\ul\alpha)=\d{}^\nabla c_0(\ul\alpha).
  \]
  The correction coefficients $(\d{}^\nabla c)_k$ are defined by
  \begin{align}
    \label{eq:dnabla}
    (-1)^k(\d{}^\nabla c)_k&(\ul\alpha\|\ul\beta)=\d{}^\nabla c_k (\ul\alpha\|\ul\beta)-\textstyle\sum_{i=1}^k c_{k-1}(\beta_i,\ul\alpha\|\beta_1,\dots,\widehat{\beta_i},\dots,\beta_k).
    \end{align}
\end{definition}
In the definition above, the correction terms were obtained using the general principle. As always, one needs to check the obtained map is well-defined, which is a straightforward computation that we have provided in Lemma \ref{lemma:wd_dnabla}. 
\begin{example}
At $p=1$, the exterior covariant derivative of $c=(c_0,c_1)\in W^{1,q}(A;V)$ reads
\begin{align}
\label{eq:ext_cov_der_p=1}
(\d{}^\nabla c)_0(\alpha)=\d{}^\nabla c_0(\alpha),\quad (\d{}^\nabla c)_1(\beta)=c_0(\beta)-\d{}^\nabla c_1(\beta).
\end{align}
This coincides with the operator defined in \cite{mec}*{Proposition A.10}. Going one level further, at $p=2$ we obtain:
\begin{align}
\label{eq:ext_cov_der_p=2}
\begin{split}
(\d{}^\nabla c)_0(\alpha_1,\alpha_2)&=\d{}^\nabla c_0(\alpha_1,\alpha_2),\\
(\d{}^\nabla c)_1(\alpha\|\beta)&=c_0(\beta,\alpha)-\d{}^\nabla c_1(\alpha\|\beta),\\
(\d{}^\nabla c)_2(\beta_1,\beta_2)&=\d{}^\nabla c_2(\beta_1,\beta_2)-c_1(\beta_1\|\beta_2)-c_1(\beta_2\|\beta_1),
\end{split}
\end{align}
for any $c=(c_0,c_1,c_2)\in W^{2,q}(A;V)$.
\end{example}

We now prove the infinitesimal analogue of Theorem \ref{thm:G_invariant}.
\begin{theorem}
\label{thm:A_invariant}
Let $\nabla$ be a connection on a representation $V$ of a Lie algebroid $A\Ra M$. The map $\d{}^\nabla$ is a cochain map if and only if $\nabla$ is $A$-invariant, that is, if there holds
\[\nabla^A_\alpha=\nabla_{\rho(\alpha)}\quad\text{and}\quad\iota_{\rho(\alpha)}R^\nabla=0,\]
for any $\alpha\in A$. In particular, the operator defined by \eqref{eq:ext_cov_der_p=1} maps IM forms to IM forms. 
\end{theorem}
\begin{remark}
  \label{rem:G_inv_implies_A_inv}
  In \cite{mec}*{Proposition A.8}, it was shown that $G$-invariance implies $A$-invariance, and that the converse also holds if $G$ is source-connected. A new proof of this is given in Proposition \ref{prop:g_inv_a_inv}, where the statement is seen as a consequence of the properties of the van Est map.
\end{remark}
\noindent As with groupoids, we prove the theorem by computing the formula for the commutator 
\[
[\d{}^\nabla,\delta]\colon W^{p,q}(A;V)\ra W^{p+1,q+1}(A;V),
\]
which will turn out to be vital later.
\begin{lemma}
\label{lem:A_invariant}
Let $V$ be a representation of a Lie algebroid $A\Ra M$ and let $\nabla$ be a  connection on $V$. The commutator $[\d{}^\nabla,\delta]$ evaluated on a cochain $c=(c_0,\dots,c_p)\in W^{p,q}(A;V)$ reads
\begin{align}
(\d{}^\nabla\delta c)_k(\alpha_0,&\dots,\alpha_{p-k}\|\ul\beta)-(\delta\d{}^\nabla c)_k(\alpha_0,\dots,\alpha_{p-k}\|\ul\beta)\nonumber\\
&=\textstyle\sum_{i=0}^{p-k} (-1)^i T(\alpha_i)\wedge c_k(\alpha_0,\dots,\widehat{\alpha_i},\dots,\alpha_{p-k}\|\ul\beta)\label{eq:A_invariant_commutator}\\
&+\textstyle\sum_{j=1}^k\theta(\beta_j)\cdot c_{k-1}(\ul\alpha\|\beta_1,\dots,\widehat{\beta_j},\dots,\beta_k),\nonumber
\end{align}
where the tensor $\theta\colon A\rightarrow \End(V)$ and the map $T\colon \Gamma(A)\ra \Omega^1(M;\End V)$ are given by
\begin{align}
\label{eq:invariance_form}
\begin{split}
\theta(\alpha)&=\nabla^A_\alpha-\nabla_{\rho(\alpha)},\\
T(\alpha)&=\d{}^\nabla\theta(\alpha)-\iota_{\rho(\alpha)}R^\nabla.
\end{split}
\end{align}
\end{lemma}
\begin{remark}
The exterior covariant derivative $\d{}^\nabla$ in the definition of the map $T$ is with respect to the induced connection on $\End V$. We also note that $(T,\theta)\in W^{1,1}(A;\End V)$ is a Weil cochain with values in the induced representation of $A$ on $\End V$, since it clearly satisfies the Leibniz rule \[T(f\alpha)=f T(\alpha)+\d f\otimes \theta(\alpha).\]
In \sec \ref{sec:obstruction_invariance}, we will see that this is actually an IM form with values in $\End V$, and we will use it to define an obstruction class to the existence of $A$-invariant connections. The pair $(T,\theta)$ will henceforth be called the \textit{$A$-invariance form} of the connection $\nabla$.
\end{remark}
\begin{proof}
Let us first check the theorem holds at the level of leading terms. By the definition of maps $\d{}^\nabla$ and $\delta$, there holds
\begin{align*}
(\d{}^\nabla\delta c)_0(\ul\alpha)&-(\delta\d{}^\nabla c)_0(\ul\alpha)=\textstyle\sum_i (-1)^i \big({\d{}^\nabla}\L^A_{\alpha_i}-\L^A_{\alpha_i}\d{}^\nabla\big) c_0(\alpha_0,\dots,\widehat{\alpha_i},\dots,\alpha_{p}),
\end{align*}
so we must compute $({\d{}^\nabla}\L^A_\alpha-\L^A_\alpha\d{}^\nabla)\gamma$ for any $\gamma\in\Omega^q(M;V)$. Notice that 
\[
\L^A_\alpha\gamma=\L^\nabla_{\rho(\alpha)}\gamma+\theta(\alpha) \gamma,
\]
and using Cartan's magic formula, we get
\begin{align}
\label{eq:d_lie}
\d{}^\nabla\L^\nabla_X\gamma-\L^\nabla_X\d{}^\nabla\gamma =(\d{}^\nabla)^2\iota_X\gamma-\iota_X(\d{}^\nabla)^2\gamma=-(\iota_X R^\nabla)\wedge \gamma
\end{align}
for any $X\in \vf(M)$. A simple computation using these two identities shows
\begin{align*}
{\d{}^\nabla}\L^A_\alpha\gamma-\L^A_\alpha\d{}^\nabla\gamma=T(\alpha)\wedge\gamma,
\end{align*}
proving our claim for the leading coefficient.

We now claim that it was actually enough to check that the formula \eqref{eq:A_invariant_commutator} holds at the level of leading coefficients. This follows from a similar argument as in Remark \ref{rem:action_trick} after this proof, and was inspired by the technique in the proof of \cite{weil}*{Proposition 4.1}. To be precise, first observe that the right-hand side of equation \eqref{eq:A_invariant_commutator} defines an operator 
\[c\mapsto (T,\theta)\wedge c,\]
which maps Weil cochains to Weil cochains by Lemma \ref{lem:T_theta_wedge}, raising both the degree and the level by 1. Hence, the case $q\leq\dim M-1$ is proved. If $q\geq\dim M$, the idea is to apply the same proof for the leading coefficient, but on a larger algebroid. Take any action of $A$ on a surjective submersion $\mu\colon P\ra M$ with $\dim P\geq q+1$ and form the action algebroid $A\ltimes P\Ra P$. The surjective algebroid morphism $p\colon A\ltimes P\ra A$ then induces a pullback of Weil cochains, denoted $p^*$, defined by \eqref{eq:weil_pullback}. The following naturality relations of the operators $\d{}^\nabla$ and $(T,\theta)\wedge \cdot$ then  hold, as a consequence of the characterizing relation \eqref{eq:pullback} of the pullback of forms with coefficients:
\begin{align*}
  p^*\d{}^\nabla&=\d{}^{\nabla^\mu}p^*,\\ p^*\big((T,\theta)\wedge c\big)&=p^*(T,\theta)\wedge p^*c,
\end{align*}
where $p^*(T,\theta)\in W^{1,1}(A\ltimes P;\mu^*V)$ is precisely the $A$-invariance form of the pullback connection $\nabla^\mu$ on the representation $\mu^*V$ of $A\ltimes P$. Since $p^*$ is injective, we are done.
\end{proof}
\begin{remark}
  \label{rem:action_trick}
  Suppose we are given two Weil cochains $c,c'\in W^{p,q}(A;V)$, and we want to show $c=c'$. Observe that for any $k\in \set{0,\dots,p-1}$, $c_k=c_k'$ implies $c_{k+1}=c'_{k+1}$ by the Leibniz rule, provided $q-k\leq \dim M$. Hence, $c_0=c_0'$ implies $c=c'$ provided $q\leq \dim M$, so in this case it is enough to check that the leading terms coincide. On the other hand, if $q>\dim M$, then the first nontrivial component of a nonzero cochain $c$ is $c_{q-\dim M}$, and we would like to see it as the leading term. The following trick is a precise way of doing so.
  
  Take any action of $A\Ra M$ on a surjective submersion $\mu\colon P\ra M$, as in \cite{actions}*{Definition 3.1} and consider the action algebroid $A\ltimes P\Ra P$. For clarity, let us briefly recall its construction. To begin with, its underlying vector bundle is the pullback bundle $\mu^* A=P\times_M A\ra P$. As such, its space of sections  is canonically isomorphic as a $C^\infty(P)$-module to 
  \[\Gamma(\mu^*A)\cong C^\infty(P)\otimes_{C^\infty(M)}\Gamma(A),\]
  and it is thus generated by $\Gamma(A)$.\footnote{The isomorphism $C^\infty(P)\otimes\Gamma(A)\ra \Gamma(\mu^*A)$ is given by  $h\otimes\alpha\mapsto (p\mapsto(p,h(p)\alpha_{\mu(p)}))$, and the tensor product over $C^\infty(M)$ identifies $h\otimes f\alpha=(f\circ \mu)h\otimes \alpha$, where $h\in C^\infty(P), \alpha\in \Gamma(A)$ and $f\in C^\infty(M)$.} The algebroid structure on $A\ltimes P$ is determined by
  \begin{align*}
    \rho_{A\ltimes P}(1\otimes \alpha)&=X^\alpha,\\
    [1\otimes\alpha,1\otimes\beta]_{A\ltimes P}&=1\otimes[\alpha,\beta]
  \end{align*}
  on generators $\alpha,\beta\in\Gamma(A)$, and extended to the whole space by $C^\infty(P)$-linearity of the anchor and the Leibniz rule of the bracket. Here, $\alpha\mapsto X^\alpha$ denotes the action $\Gamma(A)\ra \vf(P)$. We thus obtain the following natural surjective submersion of Lie algebroids.
\begin{align}
  \label{eq:p_mu}
  \begin{split}
    \begin{tikzcd}[ampersand replacement=\&]
    	{A\ltimes P} \& A \\
    	P \& M
    	\arrow["p", from=1-1, to=1-2]
    	\arrow[Rightarrow, from=1-1, to=2-1]
    	\arrow[Rightarrow, from=1-2, to=2-2]
    	\arrow["\mu"', from=2-1, to=2-2]
    \end{tikzcd}
  \end{split}
\end{align}
Moreover, a representation $V$ of $A$ induces a representation $\mu^* V$ of $A\ltimes P$, determined by
\[\nabla^{A\ltimes P}_{1\otimes \alpha}(1\otimes\xi)=1\otimes \nabla^A_\alpha\xi,\] 
for any $\xi\in \Gamma(V)$, and extended by $C^\infty(P)$-linearity and the Leibniz rule. 
Now observe that the pullback along the algebroid map \eqref{eq:p_mu} induces a monomorphism of Weil complexes,
\begin{align}
  \label{eq:weil_pullback}
  \begin{split}
    &p^*\colon W^{p,q}(A;V)\ra W^{p,q}(A\ltimes P;\mu^*V),\\
    &(p^* c)_k(1\otimes \alpha_1,\dots,1\otimes \alpha_{p-k}\|1\otimes \beta_1, \dots, 1\otimes \beta_k)=\mu^* c_k(\ul\alpha\|\ul\beta),
  \end{split}
\end{align} 
defined on the generators and extended by the Leibniz rule in the antisymmetric arguments and by $C^\infty(P)$-linearity in the symmetric ones. Since $\dim P\geq\dim M$, $p^*$ is well-defined. The upshot now is that in the case $q>\dim M$, the first nontrivial term of a nonzero Weil cochain $c\in W^{p,q}(A;V)$ is $c_{q-\dim M}$, but choosing any space $P$ with $\dim P\geq q$ translates this term into $(p^* c)_0\neq 0$. Hence, showing $c=c'$ is equivalent to showing $(p^*c)_0=(p^*c')_0$.
\end{remark}

Theorem \ref{thm:A_invariant} follows directly from the last lemma. In conclusion, an $A$-invariant connection $\nabla$ on $V$ yields the columns of a curved double complex, 
\[
(W^{\bullet,\bullet}(A;V),\delta,\d{}^\nabla).
\]
As in the case of groupoids, it has the obvious feature that $\d{}^\nabla$ does not square to zero unless $\nabla$ is flat, in which case we obtain a flat double complex.
\begin{align}
\begin{tikzcd}[ampersand replacement=\&, column sep=large, row sep=large]
	{\Omega^{q+1}(M;V)} \& {W^{1,q}(A;V)} \& {W^{2,q+1}(A;V)} \& \cdots \\
	{\Omega^q(M;V)} \& {W^{1,q}(A;V)} \& {W^{2,q}(A;V)} \& \cdots
	\arrow["{\delta}", from=1-1, to=1-2]
	\arrow["{\delta}", from=1-2, to=1-3]
	\arrow["{\delta}", from=1-3, to=1-4]
	\arrow["{\d{}^\nabla}", from=2-1, to=1-1]
	\arrow["{\delta}", from=2-1, to=2-2]
	\arrow["{\d{}^\nabla}", from=2-2, to=1-2]
	\arrow["{\delta}", from=2-2, to=2-3]
	\arrow["{\d{}^\nabla}", from=2-3, to=1-3]
	\arrow["{\delta}", from=2-3, to=2-4]
\end{tikzcd}
\end{align}
\begin{remark}
  If $\nabla$ is an $A$-invariant connection, its curvature tensor is an invariant form,
  \[
  R^\nabla\in \Omega^2_\inv(M;\End V),
  \]
  with respect to the induced representation of $A$ on $\End V$. This follows from computing
  \[
  \L^A_\alpha R^\nabla = {\underbrace{\L^\nabla_{\rho(\alpha)} R^\nabla}_{\mathclap{\d{}^{\nabla^{\End V}}\iota_{\rho(\alpha)} R^\nabla}}} + \theta(\alpha)\circ R^\nabla-R^\nabla\circ \theta(\alpha)=0,
  \]
  where the under-brace is due to the Bianchi identity $\d{}^{\nabla^{\End V}} R^\nabla=0$.
\end{remark}

\section{Van Est map versus the exterior covariant derivative}
We now establish the relationship of the van Est map with the exterior covariant derivative.

\begin{theorem}
    \label{thm:van_est_G_A}
    Let $\nabla$ be a connection on a representation $V$ of a Lie groupoid $G\rra M$ with Lie algebroid $A\Ra M$. If $\nabla$ is $G$-invariant, then the van Est map commutes with the exterior covariant derivatives:
    \begin{align}
      \label{eq:ve_d}
      \ve \d{}_G^\nabla=\d{}_A^\nabla \ve.
    \end{align}
    Moreover, this equality holds on all normalized forms regardless of $G$-invariance of the connection $\nabla$, so in particular, it holds for multiplicative forms.
    \end{theorem}
      We recall from \cite{weil} that a form $\omega\in\Omega^{p,q}(G;V)$ is said to be \textit{normalized}, if its pullbacks along all the degeneracy maps $G^{(p-1)}\ra G^{(p)}$ vanish. Any multiplicative form $\omega\in\Omega_m^q(G;V)$ necessarily satisfies $u^*\omega=0$, so it is normalized.
    \begin{lemma}
      Let $\nabla$ be a connection on a representation $V$ of a Lie groupoid $G\rra M$. If either the connection $\nabla$ is $G$-invariant, or a given form $\omega\in\Omega^{p,q}(G;V)$ satisfies $j_p^*\omega=0$, then
    \begin{align}
      \label{eq:R_jL}
      R_\alpha\omega=j_p^* \L^{\nabla^{s\circ\pr_p}}_{\alpha^{(p)}}\omega,
    \end{align}
    for any $\alpha\in\Gamma(A)$.
    \end{lemma}
    \begin{remark}
    Given a connection $\nabla$ on a vector bundle $V\ra M$, the Lie derivative of a $V$-valued differential form $\omega\in\Omega^q(M;V)$ is defined by Cartan's formula, $\L^\nabla_X\omega=\d{}^\nabla\iota_X\omega+\iota_X\d{}^\nabla\omega$ for any $X\in \vf(M)$. Equivalently, we can use parallel transport:
      \[
      (\L^\nabla_X\omega)_x(Y_i)_i=\deriv \lambda 0 \tau(\gamma^X_x)^\nabla_{\lambda,0}\big(\omega(\d(\phi^X_\lambda)_x(Y_i))_i\big),
      \]
      where $\gamma^X_x$ denotes the (maximal) integral path of $X$ starting at $x$, and \[\tau(\gamma^X_x)^\nabla_{\lambda,0}\colon V_{\gamma^X_x(\lambda)}\ra V_x\] denotes the parallel transport along $\gamma^X_x$ with respect to the connection $\nabla$ from time $\lambda$ to $0$.
    \end{remark}
    \begin{proof}
      Consider first the level $p=1$. Let us compute the difference
      \begin{align*}
        \big(R_\alpha\omega-u^*\L^{\nabla^s}_{\alpha^L}\omega\big)_x(X_i)_i=\deriv \lambda 0\bigg(\Delta_{\phi^{\alpha^L}_\lambda(1_x)}-\tau\big(\gamma^{\alpha^L}_{1_x}\big)^{\nabla^s}_{\lambda,0}\bigg)\underbrace{\omega\big({\d(\phi^{\alpha^L}_\lambda)_{1_x}\d u (X_i)}\big)_i}_{\xi_{\gamma^{\rho\alpha}_{x}(\lambda)}}
      \end{align*}
      where $\Delta$ denotes the action of $G$ on $V$, and the expression denoted by $\xi$ defines a section of $V$ along the path $\gamma^{\rho\alpha}_x$. However, for any section $\xi$ along the path $\gamma^{\rho\alpha}_x$, there holds
      \begin{align*}
        \deriv \lambda 0\bigg(\Delta_{\phi^{\alpha^L}_\lambda(1_x)}-\tau\big(\gamma^{\alpha^L}_{1_x}\big)^{\nabla^s}_{\lambda,0}\bigg)\cdot \xi_{\gamma^{\rho\alpha}_{x}(\lambda)}=\nabla^A_{\alpha_x}\xi-\nabla_{\rho\alpha_x}\xi=\theta(\alpha_x)\cdot \xi_x,
      \end{align*}
      where we have observed there holds $\tau\big(\gamma^{\alpha^L}_{1_x}\big)^{\nabla^s}=\tau(s\circ\gamma^{\alpha^L}_{1_x})^\nabla=\tau(\gamma^{\rho\alpha}_x)^\nabla$ since $\nabla^s$ is a pullback connection and $s_*\alpha^L=\rho\alpha$. This implies
      \begin{align}
        \label{eq:R_uL_theta}
        R_\alpha\omega-u^*\L^{\nabla^s}_{\alpha^L}\omega=\theta(\alpha)\cdot u^*\omega.
      \end{align}
      More generally, at any higher level $p>1$, for a given $\omega\in \Omega^{p,q}(G;V)$ we similarly obtain
      \begin{align}
      \big(R_\alpha\omega-j_p^*\L^{\nabla^{s\circ \pr_p}}_{\alpha^{(p)}}\omega\big)_{(g_1,\dots,g_{p-1})}=\theta(\alpha_{s(g_{p-1})})\cdot (j_p^*\omega)_{(g_1,\dots,g_{p-1})},
      \end{align}
      proving the lemma since $G$-invariance implies $A$-invariance.
    \end{proof}
    \begin{proof}[Proof of Theorem \ref{thm:van_est_G_A}]
    Let us first check the theorem holds at the level of leading terms, 
    \begin{align}
      \label{eq:ve_dnabla_0}
      (\ve \d{}^\nabla\omega)_0=(\d{}^\nabla\ve\omega)_0.
    \end{align}
    We do so by first establishing the following identity for the case when either $\nabla$ is $G$-invariant, or the form $\omega$ satisfies $j_p^*\omega=0$:
    \begin{align}
    \label{eq:R_d}
    R_\alpha \d{}^{\nabla^{s\circ\pr_p}}\omega=\d{}^{\nabla^{s\circ\pr_{p-1}}}R_\alpha\omega.
    \end{align}
    Similarly to equation \eqref{eq:d_lie}, the commutator of the $V$-valued Lie derivative with $\d{}^\nabla$ reads
    \[
    \L^{\nabla^{s\circ\pr_p}}_{\alpha^{(p)}}\d{}^{\nabla^{s\circ\pr_p}}-\d{}^{\nabla^{s\circ\pr_p}}\L^{\nabla^{s\circ\pr_p}}_{\alpha^{(p)}}=\iota_{\alpha^{(p)}}R^{\nabla^{s\circ\pr_p}}\wedge \cdot=(s\circ\pr_p)^*(\iota_{\rho\alpha}R^\nabla)\wedge \cdot
    \]
    Moreover, using the identity \eqref{eq:pullback} on $\pi=j_p$ with $\nabla^{s\circ\pr_p}$ as the connection, we get
    \[
    \d{}^{\nabla^{s\circ\pr_{p-1}}}j_p^*=j_p^*\d{}^{\nabla^{s\circ\pr_p}},
    \]
    since $s\circ\pr_p\circ j_p=s\circ\pr_{p-1}$. Hence, by the last lemma, we obtain
    \[
      R_\alpha \d{}^{\nabla^{s\circ\pr_p}}\omega-\d{}^{\nabla^{s\circ\pr_{p-1}}} R_\alpha \omega=(s\circ \pr_{p-1})^* (\iota_{\rho\alpha} R^\nabla)\wedge j_p^*\omega,
    \]
    This proves the identity \eqref{eq:R_d} and thus also \eqref{eq:ve_dnabla_0} for the case when $\nabla$ is $G$-invariant. Dropping the assumption of invariance and instead assuming $\omega$ is normalizable, identity \eqref{eq:R_d} now follows by the fact that $R_\alpha$ preserves normalizability---more precisely, for any $\omega\in\Omega^{p,q}(G;V)$ there holds
    \[
    (j_{p-1}^k)^*R_\alpha\omega=R_\alpha(j_{p}^k)^*\omega,
    \] 
    for any $k\leq p-1$, where $j_{p}^k\colon G^{(p-1)}\ra G^{(p)}$ denotes the degeneracy map that inserts the unit into the $k$-th place (hence $j_p=j_p^p$). Hence, for any $\alpha_1,\dots,\alpha_p\in\Gamma(A)$,
    \begin{align*}
      R_{\alpha_1}\dots R_{\alpha_p}\d{}^{\nabla^{s\circ\pr_p}}\omega&=R_{\alpha_1}\dots R_{\alpha_{p-1}}\d{}^{\nabla^{s\circ\pr_{p-1}}}R_{\alpha_p}\omega&\text{(since $j_p^*\omega=0$)}\\
      &=R_{\alpha_1}\dots R_{\alpha_{p-2}}\d{}^{\nabla^{s\circ\pr_{p-1}}}R_{\alpha_{p-2}}R_{\alpha_p}\omega &\text{(since $(j_p^{p-1})^*\omega=0$)}\\
      &\vdotswithin{=}\\
      &=\d{}^\nabla R_{\alpha_1}\dots R_{\alpha_p}\omega &\text{(since $(j_p^{1})^*\omega=0$)}
    \end{align*}
    From this the desired result for normalized forms follows.
    
    To see that equation \eqref{eq:ve_dnabla_0} was actually enough to show, we again use the trick from Remark \ref{rem:action_trick}. More precisely, we now note that if $G\rra M$ acts on a surjective submersion $\mu\colon P\ra M$, the action differentiates to an action  $\Gamma(A)\ra \vf(P)$ of its Lie algebroid, and the obtained action algebroid $A\ltimes P$ is now just the algebroid of the action groupoid $G\ltimes P$. Denoting the respective surjective submersive Lie groupoid and algebroid morphisms by $p_G\colon G\ltimes P\ra G$ and $p_A\colon A\ltimes P\ra A$, there now clearly holds 
    \[R_{1\otimes\alpha}(p_G)^*=(p_G)^*R_\alpha,\quad J_{1\otimes \alpha}(p_G)^*=(p_G)^* J_\alpha.\]
    Moreover, at level $p=0$, both pullbacks $(p_G)^*$ and $(p_A)^*$ restrict simply to $\mu^*$. Hence, the naturality of van Est map follows:
    \[\begin{tikzcd}[row sep=large]
        {\Omega^{p,q}(G;V)} & {\Omega^{p,q}(G\ltimes P;\mu^*V)} \\
        {W^{p,q}(A;V)} & {W^{p,q}(A\ltimes P;\mu^*V)}
        \arrow["{(p_G)^*}", from=1-1, to=1-2]
        \arrow["\ve"', from=1-1, to=2-1]
        \arrow["\ve", from=1-2, to=2-2]
        \arrow["{(p_A)^*}"', from=2-1, to=2-2]
    \end{tikzcd}\]
    This concludes the proof.
    \end{proof}

\section{Obstruction to existence of invariant connections}
\label{sec:obstruction_invariance}
Given a representation $V$ of either a Lie groupoid or an algebroid, we now construct the cohomological class which controls the existence of invariant connections on $V$. Let
\[\A_\inv(G;V)\quad\text{and}\quad\A_\inv(A;V)\]
denote the sets of all invariant linear connections on $V$ for the global and the infinitesimal case, respectively. We start with the global case.

\subsection*{The global case}
We first observe that the form \eqref{eq:invariance_form_G} controlling the $G$-invariance of a linear connection $\nabla$ on a representation $V$ can be seen as an $\End(V)$-valued form,
\begin{align}
  \label{eq:invariance_form_theta}
  \Theta\in\Omega^1(G;s^*{\End V}),\quad \Theta(X)\xi=\phi\big(\nabla^t_X(\phi^{-1}\xi)\big)-\nabla_X^s\xi,
\end{align}
for any $\xi\in\Gamma(s^*V)$. It is called the \textit{$G$-invariance form} of the connection $\nabla$. Without further ado:
\begin{theorem}
  \label{thm:obstruction_invariance_G}
  Let $V\ra M$ be a representation of a Lie groupoid $G\rra M$. 
  \begin{enumerate}[label={(\roman*)}]
    \item Given any connection $\nabla$ on $V\ra M$, its $G$-invariance form \eqref{eq:invariance_form_theta} is multiplicative,
    \[\Theta\in\Omega_m^1(G;\End V).\]
    \item The set $\A_\inv(G;V)$ is an affine space over  invariant, endomorphism-valued 1-forms \[\Omega^1_\inv(M;\End V).\]
    \item The cohomological class of the $G$-invariance form $\Theta$ of any connection $\nabla$ on $V$ is independent of the choice of connection $\nabla$. We denote it by
    \begin{align*}
      \obs_{\A_\inv(G;V)}\coloneq [\Theta]\in H^{1,1}(G;\End V).
    \end{align*}
    \item A $G$-invariant connection on $V$ exists if and only if $\obs_{\A_\inv(G;V)}=0$.
  \end{enumerate}
\end{theorem}
\begin{proof}
  The proof of (i) essentially consists of understanding what multiplicativity means for $\End(V)$-valued forms. By definition of the induced representation on $\End(V)$, we need to check that for any $X\in\vf(G^{(2)})$ and $\xi\in\Gamma((s\circ\pr_2)^* V)$ there holds
  \begin{align*}
    (m^*\Theta)(X)\cdot\xi=(\pr_2^*\Theta)(X)\cdot\xi+\Phi\big((\pr_1^*\Theta)(X)\cdot(\Phi^{-1}\xi)\big),
  \end{align*}
  where $\Phi\colon(s\circ \pr_1)^*V\ra (s\circ\pr_2)^*V$ is given by $(g,h,\xi)\mapsto (g,h,h^{-1}\cdot\xi)$. By definition of $\Theta$, the terms on the right-hand side above equal
  \begin{align}
    \label{eq:intermed_twoterms_theta}
    (\pr_2^*\Theta)(X)\cdot\xi&=\Phi_2\big(\nabla^{t\circ\pr_2}_X(\Phi_2^{-1}\xi)\big)-\nabla^{s\circ\pr_2}_X \xi\\
    \Phi\big((\pr_1^*\Theta)(X)\cdot(\Phi^{-1}\xi)\big)&=\Phi_2\Phi_1\big(\nabla^{t\circ\pr_1}_X(\Phi_1^{-1}\Phi_2^{-1}\xi)\big)-\Phi_2\big(\nabla^{s\circ\pr_1}_X(\Phi_2^{-1}\xi)\big),\label{eq:intermed_twoterms_theta2}
  \end{align}
  where we are denoting by $\Phi_i$ ($i=1,2$) the vector bundle morphisms covering $\id_{G^{(2)}}$, determined by the identities $\Phi_i (\pr_i^*\eta)=\pr_i^*\phi(\eta)$ for any $\eta\in\Gamma(t^* V)$. They are given by:
\[\begin{tikzcd}[row sep=-2pt]
	{(t\circ\pr_1)^*V} & {(s\circ\pr_1)^*V} & {(s\circ\pr_2)^*V} & \hspace{-1.5em}{\pr_i^* (t^*V)} & {\pr_i^* (s^*V)} \\
	{(g,h,\xi_{t(g)})} & {(g,h,g^{-1}\cdot\xi_{t(g)})} \\
	& {(g,h,\xi_{s(g)})} & {(g,h,h^{-1}\cdot \xi_{s(g)})} & {t^*V} & {s^*V}
	\arrow["{\Phi_1}", from=1-1, to=1-2]
	\arrow["{\Phi_2}", from=1-2, to=1-3]
	\arrow["{\Phi_i}", from=1-4, to=1-5]
	\arrow[from=1-4, to=3-4]
	\arrow[from=1-5, to=3-5]
	\arrow[maps to, from=2-1, to=2-2]
	\arrow[maps to, from=3-2, to=3-3]
	\arrow["\phi"', from=3-4, to=3-5]
\end{tikzcd}\]
Identities  \eqref{eq:intermed_twoterms_theta} and \eqref{eq:intermed_twoterms_theta2} are easily shown on pullback sections, i.e., those of the form $\xi=\pr_2^*\eta$ and $\Phi^{-1}\xi=\pr_1^*\eta$ for $\eta\in\Gamma(s^*V)$, respectively, by observing there holds $\Phi=\Phi_2$.  Notice that the first term of \eqref{eq:intermed_twoterms_theta} and the second term of \eqref{eq:intermed_twoterms_theta2} cancel out since $s\circ\pr_1=t\circ \pr_2$, hence we just need to check $\Theta$ satisfies
\begin{align*}
  (m^*\Theta)(X)\cdot\xi=\Phi_2\Phi_1\big(\nabla^{t\circ\pr_1}_X(\Phi_1^{-1}\Phi_2^{-1}\xi)\big)-\nabla^{s\circ\pr_2}_X \xi,
\end{align*}
This is clear since $t\circ m=t\circ \pr_1$, $s\circ m=s\circ \pr_2$ and the map $\Phi_m\colon (t\circ\pr_1)^*V\ra (s\circ\pr_2)^*V$ determined by $\Phi_m m^*=m^*\phi$ clearly equals $\Phi_m=\Phi_2\Phi_1$, hence (i) is proved.

The points (ii), (iii) and (iv) are direct consequences of the  following observation: if $\tilde\nabla$ and $\nabla$ are two connections on $V$ with respective invariance forms $\tilde\Theta$ and $\Theta$, then
\[
\Tilde\Theta-\Theta=-\delta^0\gamma
\]
where $\gamma=\tilde\nabla-\nabla\in\Omega^1(M;\End V)$. This easily follows from the identity $s^*\gamma=\tilde\nabla^s-\nabla^s$ and a similar one for $t^*\gamma$, concluding the proof.
\end{proof}

\subsection*{The infinitesimal case}
We now prove the infinitesimal analogue of Theorem \ref{thm:obstruction_invariance_G}.
\begin{theorem}
  \label{thm:obstruction_invariance}
  Let $V\ra M$ be a representation of a Lie algebroid $A\Ra M$. 
  \begin{enumerate}[label={(\roman*)}]
    \item Given any connection $\nabla$ on $V\ra M$, its $A$-invariance form \eqref{eq:invariance_form} is multiplicative: \[(T,\theta)\in\Omega^1_{im}(A; \End V).\]
    \item The set $\A_\inv(A;V)$ is an affine space over  invariant, endomorphism-valued 1-forms \[\Omega^1_\inv(M;\End V).\]
    \item The cohomological class of the $A$-invariance form $(T,\theta)$ of any connection $\nabla$ on $V$ is independent of the choice of connection $\nabla$. We denote it by
    \begin{align*}
      \obs_{\A_\inv(A;V)}\coloneq [(T,\theta)]\in H^{1,1}(A;\End V).
    \end{align*}
    \item An $A$-invariant connection on $V$ exists if and only if $\obs_{\A_\inv(A;V)}=0$.
  \end{enumerate}
\end{theorem}
\begin{proof}
  The proof of (i) consists of a straightforward computation verifying the condition \eqref{eq:c1} for the map $T$, which we will now carry out. Writing out the map $T$ in full, it has a formula similar to the usual curvature tensor,
  \begin{align}
    \label{eq:T_explicit}
    T(\alpha)(X)\xi=\nabla_X\nabla^A_\alpha\xi-\nabla^A_\alpha\nabla_X\xi+\nabla_{[\rho\alpha,X]}\xi,
  \end{align}
  for any $\alpha\in\Gamma(A)$, $X\in \vf(M)$ and $\xi\in\Gamma(V)$. If $\beta\in\Gamma(A)$ is another section, then
  \begin{align*}
    \L^A_\alpha (T\beta)(X)\xi&=\nabla^A_\alpha(T(\beta)(X)\xi)-T(\beta)(X)\nabla^A_\alpha\xi-T(\beta)([\rho\alpha,X])\xi\\
    &=\nabla^A_\alpha(\nabla_X\nabla^A_\beta\xi-\nabla^A_\beta\nabla_X\xi+\nabla_{[\rho\beta,X]}\xi)\\
    &-(\nabla_X\nabla^A_\beta-\nabla^A_\beta\nabla_X+\nabla_{[\rho\beta,X]})\nabla^A_\alpha\xi\\
    &-(\nabla_{[\rho\alpha,X]}\nabla^A_\beta\xi-\nabla^A_\beta\nabla_{[\rho\alpha,X]}\xi+\nabla_{[\rho\beta,[\rho\alpha,X]]}\xi)
  \end{align*}
  In this expression, the following terms come in pairs with respect to interchanging $\alpha$ and $\beta$: first and fifth, third and eighth, sixth and seventh. Hence, subtracting $\L^A_\beta (T\alpha)(X)\xi$ from this expression kills these terms. We are left with  
  \begin{align*}
    \L^A_\alpha (T\beta)(X)\xi-\L^A_\beta (T\alpha)(X)\xi=\nabla_X \nabla^A_{[\alpha,\beta]}\xi-\nabla^A_{[\alpha,\beta]}\nabla_X\xi+\nabla_{[[\rho\alpha,\rho\beta],X]}\xi=T[\alpha,\beta](X)\xi,
  \end{align*}
  where we have used the flatness $\smash{\nabla^A_{[\alpha,\beta]}=[\nabla_\alpha^A,\nabla_\beta^A]}$ of the representation $\nabla^A$ and the Jacobi identity on $\vf(M)$. This proves (i).

  The remaining points (ii), (iii) and (iv) are direct consequences of the following more general observation. If $\tilde \nabla$ and $\nabla$ are two connections on $V$ with respective invariance forms $(\tilde T,\tilde \theta)$ and $(T,\theta)$, then there holds
  \begin{align}
    (\tilde T,\tilde \theta)-(T,\theta) = -\delta^0\gamma,
  \end{align}
  where $\gamma=\tilde\nabla-\nabla\in \Omega^1(M;\End V)$. To see this, we first note that on the level of symbols, we clearly have $\tilde\theta(\alpha)-\theta(\alpha)=-\gamma(\rho\alpha)$. For the leading terms, we insert $\nabla+\gamma$ in place of connection $\nabla$ in the equation \eqref{eq:T_explicit} to obtain
  \begin{align*}
    \tilde T(\alpha)(X)\xi&=T(\alpha)(X)\xi+\gamma(X)\nabla^A_\alpha\xi-\nabla^A_\alpha(\gamma(X)\xi)+\gamma([\rho\alpha,X])\xi\\
    &=T(\alpha)(X)\xi-(\L^A_\alpha\gamma)(X)\xi,
  \end{align*}
  that is, $\tilde T(\alpha)=T(\alpha)-\L^A_\alpha \gamma$. This concludes the proof.
\end{proof}
\begin{example}
  As an ill-behaved example, we consider the general linear algebroid $\frak{gl}(V)\Ra M$ of a vector bundle $V\ra M$. It has a natural representation on the vector bundle $V$ by simply applying the derivation:
  \[
  \nabla^A_{(D,X)}\xi=D\xi,
  \]
  for any derivation $(D,X)\in\Gamma(\frak{gl}(V))$ and $\xi\in\Gamma(V)$. This representation does not admit an invariant connection on $V$. Indeed, the existence of such a connection $\nabla$ would mean there holds $\nabla_X=D$ for any derivation $(D,X)$, which is clearly false since we can take $D=\tilde \nabla_X$ for some connection $\tilde \nabla\neq \nabla$. Later, in Proposition \ref{prop:invariant_abelian}, we will see this is a manifestation of the fact that $\End(V)$ is nonabelian as a Lie algebra bundle with the fibrewise commutator.
\end{example}
Finally, we show that the van Est map relates the global and infinitesimal invariance forms.
\begin{proposition}
  \label{prop:g_inv_a_inv}
  Let $A$ be the Lie algebroid of $G$. For any connection $\nabla$ on a representation $V$ of $G$, the van Est map sends the $G$-invariance form of $\nabla$ to its $A$-invariance form:
  \begin{align}
    \label{eq:ve_Theta}
    \ve(\Theta)=(T,\theta).
  \end{align}
  In particular, $G$-invariance implies $A$-invariance of $\nabla$, and if $G$ is source-connected, the converse also holds. Moreover, the van Est map relates the obstruction classes:
    \[\ve(\obs_{\A_{\inv}(G;V)})=\obs_{\A_{\inv}(A;V)}.\]
\end{proposition}
\begin{proof}
  The second part follows from the first part and Theorem \ref{thm:obstruction_invariance_G} (i), since the restriction of the van Est map to multiplicative forms is injective when $G$ is source-connected. That the van Est map relates the obstruction classes is also a direct consequence of \eqref{eq:ve_Theta}.

  Let us prove the identity \eqref{eq:ve_Theta}. We begin by proving that the symbols coincide: for any $\alpha\in\Gamma(A)$ and $\xi\in\Gamma(V)$, there holds
  \begin{align*}
    u^*(\Theta(\alpha^L))\cdot\xi=u^*(\Theta(\alpha^L)\cdot s^*\xi)=u^*\big(\phi\big(\nabla^t_{\alpha^L}(\phi^{-1}s^*\xi)\big)\big)-u^*(\nabla^s_{\alpha^L}s^*\xi).
  \end{align*}
  The second term clearly equals $\nabla_{\rho\alpha}\xi$ since $s_*\alpha^L=\rho\alpha$. Using parallel transport, we can write the first term at any $g\in G$ as 
  \begin{align*}
    \phi\big(\nabla^t_{\alpha^L}(\phi^{-1}s^*\xi)\big)_g &= g^{-1}\cdot\deriv\lambda 0 \underbrace{\tau(\gamma^{\alpha^L}_g)^{\nabla^t}_{\lambda,0}}_{\id_{V_{t(g)}}}\Big(\underbrace{\phi^{\alpha^L}_\lambda(g)}_{\mathclap{g\phi^{\alpha^L}_\lambda(1_{s(g)})}}\cdot \xi_{\phi^{\rho\alpha}_\lambda(s(g))}\Big)=\deriv\lambda 0 \phi^{\alpha^L}_\lambda(1_{s(g)})\cdot\xi_{\phi^{\rho\alpha}_\lambda(s(g))},
  \end{align*}
  which equals $(s^*\nabla^A_\alpha\xi)_g$. We have used here that the action is linear and $t_*\alpha^L=0$. By tensoriality of $\Theta$, we have shown $\smash{\Theta(\alpha^L)=s^*\theta(\alpha)}$ and thus proved that the symbols coincide.

  For the leading terms, we need to show $R_\alpha\Theta=T(\alpha)$. To this end, we first use multiplicativity of $\Theta$ with equation \eqref{eq:R_jL} to express the left-hand side as $R_\alpha\Theta=u^*\L_{\alpha^L}^{\nabla^s}\Theta$ where $\nabla^s$ is the induced connection on $s^*\End(V)$. That is, for any $X\in\vf(M)$ and $\xi\in\Gamma(V)$,
  \begin{align*}
    (R_\alpha\Theta)(X)\cdot\xi&=u^*(\L^{\nabla^s}_{\alpha^L}\Theta)(X)\cdot\xi=u^*\big((\L^{\nabla^s}_{\alpha^L}\Theta)(Y)\cdot s^*\xi\big)\\
    &=u^*\big(\nabla^s_{\alpha^L}(\Theta(Y)\cdot s^*\xi)-\Theta(Y)\cdot \nabla^s_{\alpha^L}s^*\xi-\Theta[\alpha^L,Y]\cdot s^*\xi\big)
  \end{align*}
  where $Y\in \vf(G)$ extends $u_*X$ on $u(M)$ and can be chosen $s$-projectable to $s_*Y=X$. Expanding the three terms on the right-hand side and denoting by $\smash{\overline\nabla^t=\phi\nabla^t\phi^{-1}}$ the pullback connection along $t$ on the bundle $s^*V\ra G$, we get
  \begin{align*}
    \nabla^s_{\alpha^L}\big(\Theta(Y)\cdot s^*\xi\big)&=\nabla^s_{\alpha^L}\overline\nabla^t_Ys^*\xi-s^*\nabla_{\rho\alpha}\nabla_X\xi,\\
    \Theta(Y)\cdot \nabla^s_{\alpha^L}s^*\xi&=\overline\nabla^t_Y\nabla^s_{\alpha^L}s^*\xi-s^*\nabla_X\nabla_{\rho\alpha}\xi,\\
    \Theta[\alpha^L,Y]\cdot s^*\xi&=\overline\nabla^t_{[\alpha^L,Y]}s^*\xi-s^*\nabla_{[\rho\alpha,X]}\xi,
  \end{align*}
  where we have observed that $s_*[\alpha^L,Y]=[\rho\alpha,X]$. Hence, we obtain
  \[
    (R_\alpha\Theta)(X)\cdot\xi = u^*\big(\underbrace{\big(\nabla^s_{\alpha^L}\overline\nabla^t_Y-\overline\nabla^t_Y\nabla^s_{\alpha^L}-\overline\nabla^t_{[\alpha^L,Y]}\big)s^*\xi}_{\eqqcolon S(\alpha)(Y)\cdot s^*\xi}\big)-R^\nabla(\rho\alpha,X)\cdot \xi.
  \]
  We observe that the expression denoted by $S$ is tensorial in both $Y$ and $s^*\xi$ (but not in $\alpha$), so it defines a map $S\colon\Gamma(A)\ra \Omega^1(G;s^*\End V)$. Moreover, we can express it as
  \[
  S(\alpha)(Y)\cdot s^*\xi=R^{\overline \nabla^t}(\alpha^L,Y)\cdot s^*\xi+\overline \nabla^t_Y\big(\Theta(\alpha^L)\cdot s^*\xi \big)-\Theta(\alpha^L)\cdot\overline\nabla^t_Y s^*\xi.
  \]
  Since $R^{\overline\nabla^t}=\phi R^{\nabla^t}\phi^{-1}$ and $R^{\nabla^t}=t^* R^\nabla$, the first term vanishes by $t_*\alpha^L=0$. Hence, it is enough to show that 
  \[
    u^*\Big(\overline \nabla^t_Y\big(\Theta(\alpha^L)\cdot s^*\xi \big)-\Theta(\alpha^L)\cdot\overline\nabla^t_Y s^*\xi\Big)=\nabla_X(\theta(\alpha)\cdot\xi)-\theta(\alpha)\cdot\nabla_X\xi,
  \]
  but this is clear since $u^*\Theta(\alpha^L)=\theta(\alpha)$ and $u^*\overline\nabla^t_Y=\nabla_X$ since $Y$ extends $u_*X$.
\end{proof}

\begin{subappendices}
    \section{Appendix: Auxiliary computations}
    \begin{lemma}
        \label{lemma:wd_dnabla}
        Let $\nabla$ be a connection on a representation $V$ of a Lie algebroid $A\Ra M$. The map $\d{}^\nabla\colon W^{p,q}(A;V)\ra W^{p,q+1}(A;V)$ given by \eqref{eq:dnabla} is well-defined.
        \end{lemma}
        \begin{proof}
          Given $c=(c_0,\dots,c_p)\in W^{p,q}(A;V)$, we would like to show that \[\d{}^\nabla c=((\d{}^\nabla c)_0,\dots,(\d{}^\nabla c)_p)\] defines a Weil cochain. Letting $\ul\alpha=(\alpha_1,\dots,\alpha_{p-k})$, $\ul\beta=(\beta_1,\dots,\beta_k)$ and $f\in C^\infty(M)$, we straightforwardly compute:
        \begin{align*}
          &(-1)^k (\d{}^\nabla c)_k(f\alpha_1,\dots,\alpha_{p-k}\|\ul\beta)=\d{}^\nabla\big(fc_k(\ul\alpha\|\ul\beta)+\d f\wedge c_{k+1}(\alpha_2,\dots,\alpha_{p-k}\|\alpha_1,\beta)\big)\\
          &-\textstyle\sum_i\big(f c_{k-1}(\beta_i,\ul\alpha\|\beta_1,\dots,\widehat{\beta_i},\dots,\beta_k)-\d f\wedge c_k(\beta_i,\alpha_2,\dots,\alpha_{p-k}\|\alpha_1,\beta_1,\dots,\widehat{\beta_i},\dots,\beta_k)\big)\\
          &=f\big({\d{}}^\nabla c_k(\ul\alpha\|\ul\beta)-\textstyle\sum_i c_{k-1}(\beta_i,\ul\alpha\|\beta_1,\dots,\widehat{\beta_i},\dots,\beta_k)\big)\\
          &+\d f\wedge \big(c_k(\ul\alpha\|\ul\beta)+\textstyle\sum_i c_k(\beta_i,\alpha_2,\dots,\alpha_{p-k}\|\alpha_1,\beta_1,\dots,\widehat{\beta_i},\dots,\beta_k)-\d{}^\nabla c_{k+1}(\alpha_2,\dots,\alpha_{p-k}\|\alpha_1,\ul\beta)\big)\\
          &=f (-1)^k (\d{}^\nabla c)_k(\ul\alpha\|\ul\beta)-\d f\wedge (-1)^{k+1}(\d{}^\nabla c)_{k+1}(\alpha_2,\dots,\alpha_{p-k}\|\alpha_1,\ul\beta),
        \end{align*}
        where we have used the Leibniz identity for $c$ in the first and second line, and the graded Leibniz rule for $\d{}^\nabla$ in the third and fourth. In the last line, we recognized that the expressions in the parentheses in the third and fourth line are precisely the defining expressions for $(-1)^k(\d{}^\nabla c)_k$ and $-(-1)^{k+1}(\d{}^\nabla c)_{k+1}$, respectively. Cancelling the factor $(-1)^k$ on both sides finishes the proof.
        \end{proof}

\begin{lemma}
    \label{lem:T_theta_wedge}
    Let $V$ be a representation of a Lie algebroid $A\Ra M$. For any $c\in W^{p,q}(A;V)$ and $(T,\theta)\in W^{1,1}(A;\End V)$, the following expression defines the coefficients of a Weil cochain $(T,\theta)\wedge c\in W^{p+1,q+1}(A;V)$.
  \begin{align*}
  \big((T,\theta)\wedge c\big)_k(\alpha_0,\dots,\alpha_{p-k}\|\ul\beta)&=\textstyle\sum_{i=0}^{p-k} (-1)^i T(\alpha_i)\wedge c_k(\alpha_0,\dots,\widehat{\alpha_i},\dots,\alpha_{p-k}\|\ul\beta)\\
  &+\textstyle\sum_{j=1}^k\theta(\beta_j)\cdot c_{k-1}(\ul\alpha\|\beta_1,\dots,\widehat{\beta_j},\dots,\beta_k),
  \end{align*}
  \end{lemma}
  \begin{proof}
    For any $\ul\alpha=(\alpha_0,\dots,\alpha_{p-k})$, $\ul\beta=(\beta_1,\dots,\beta_k)$ and $f\in C^\infty(M)$, we compute:
  \begin{align*}
  \big((T&,\theta)\wedge c\big)_k(f\alpha_0,\dots,\alpha_{p-k}\|\ul\beta)=\underbrace{T(f\alpha_0)}_{\mathclap{fT(\alpha_0)+\d f\otimes\theta(\alpha_0)}}\wedge\, c_k(\alpha_1,\dots,\alpha_{p-k}\|\ul\beta)\\
  &+\textstyle\sum_{i=1}^{p-k}(-1)^i T(\alpha_i)\wedge \underbrace{c_k(f\alpha_0,\alpha_1,\dots,\widehat{\alpha_i},\dots,\alpha_{p-k}\|\ul\beta)}_{\mathclap{fc_k(\alpha_0,\dots,\widehat{\alpha_i},\dots,\alpha_{p-k}\|\ul\beta)+\d f\wedge c_{k+1}(\alpha_1,\dots,\widehat{\alpha_i},\dots,\alpha_{p-k}\|\alpha_0,\ul\beta)}}\\
  &+\textstyle\sum_{j=1}^k\theta(\beta_j)\cdot \underbrace{c_{k-1}(f\alpha_0,\dots,\alpha_{p-k}\|\beta_1,\dots\widehat{\beta_i},\dots,\beta_k)}_{\mathclap{f c_{k-1}(\alpha_0,\dots,\alpha_{p-k}\|\beta_1,\dots\widehat{\beta_i},\dots,\beta_k)+\d f\wedge c_k(\alpha_1,\dots,\alpha_{p-k}\|\alpha_0,\beta_1,\dots\widehat{\beta_i},\dots,\beta_k)}}
  \end{align*}
  Exchanging $\d f$ with $T(\alpha_i)$ yields an additional minus, and separating the terms with $f$ and $\d f$ yields
  \begin{align*}
    \big((T&,\theta)\wedge c\big)_k(f\alpha_0,\dots,\alpha_{p-k}\|\ul\beta)=f\big((T,\theta)\wedge c\big)_k(\alpha_0,\dots,\alpha_{p-k}\|\beta)\\
    &+\d f\wedge \big({-}\textstyle\sum_{i=1}^{p-k}(-1)^i T(\alpha_i)\wedge c_{k+1}(\alpha_1,\dots,\widehat{\alpha_i},\dots,\alpha_{p-k}\|\alpha_0,\ul\beta)\\
    &\phantom{+\d f\wedge \big(}+\theta(\alpha_0)c_k(\alpha_1,\dots,\alpha_{p-k}\|\ul\beta)+\textstyle\sum_{j=1}^k\theta(\beta_j)c_k(\alpha_1,\dots,\alpha_{p-k}\|\alpha_0,\beta_1,\dots,\widehat{\beta_j},\dots,\beta_k)\big).
    \end{align*}
  Now just observe that the terms in the parentheses in the second and third line add up to 
    \begin{align}
    \big((T,\theta)\wedge c\big)_{k+1}(\alpha_1,\dots,\alpha_{p-k}\|\alpha_0,\ul\beta).\tag*\qedhere
    \end{align}
  \end{proof}
\end{subappendices}

\addtocontents{toc}{\vfill\protect\pagebreak}

\clearpage \pagestyle{plain}
\chapter{Multiplicative Ehresmann connections}\label{chapter:mec}
\pagestyle{fancy}
\fancyhead[CE]{Chapter \ref*{chapter:mec}} 
\fancyhead[CO]{Multiplicative Ehresmann connections}

In this chapter, we present our contributions to the theory of multiplicative Ehresmann connections on Lie groupoids and Lie algebroids. These developments will be crucial for our desired generalization of Yang--Mills theory, however, we believe they are interesting and useful in their own light. The results presented here can be found in the preprint \cite{covariant_derivatives}.

\section{Bundles of ideals}
In the remainder of this thesis, we will consider a special class of representations: subrepresentations of the adjoint representation $\Ad\colon G\curvearrowright \ker\rho$. Note that the latter is only a representation in the set-theoretic sense unless $\ker\rho$ is a vector bundle (that is, unless $G$ is regular), so we instead phrase this in the following way.
\begin{definition}
\label{defn:boi}
On a Lie groupoid $G\rra M$, a vector bundle $\frak k\subset \ker \rho$ is called a \textit{bundle of ideals}, if for any $g\in G$ the map $\Ad\colon G\curvearrowright \ker\rho$ restricts to a map
\[\Ad_g\colon \frak k_{s(g)}\ra \frak k_{t(g)}.\] 
This means precisely that $\Ad\colon G\curvearrowright\frak k$ is a representation of $G$.

On a Lie algebroid $A\Ra M$, a vector bundle $\frak k\subset \ker \rho$ is said to be a \textit{bundle of ideals} if for any $\alpha\in \Gamma(A)$ and $\xi\in\Gamma(\frak k)$ there holds
\begin{align}
  \label{eq:boi_inf}
  [\alpha,\xi]\in\Gamma(\frak k).
\end{align}
\end{definition}
\begin{remark}
  The terminology comes from \cite{mec}, however, it is not standard. In \cite{ideals}, a vector bundle $\frak k\subset\ker\rho$ satisfying \eqref{eq:boi_inf} is instead called a \textit{naïve ideal} of a Lie algebroid.
\end{remark}
\noindent If $A$ is the Lie algebroid of $G$, then 
\[[\alpha,\xi]_x=\deriv\lambda0\Ad_{\phi^{\alpha^L}_\lambda(1_x)}\big(\xi_{\phi^{\rho\alpha}_\lambda(x)}\big),\]
 holds at every $x\in M$, so that if $\frak k$ is a bundle of ideals on $G$, it is also a bundle of ideals on $A$. In other words, the map $(\alpha,\xi)\mapsto [\alpha,\xi]$ is just the infinitesimal version of the adjoint representation of $G\rra M$ on $\frak k$, i.e., a flat $A$-connection on $\frak k\ra M$, which we will denote by $\nabla^A_\alpha\xi=[\alpha,\xi]$. 
 
 \begin{samepage}
 \begin{example}
  \label{ex:boi_groupoid_morphism}
  The main class of examples of bundles of ideals on Lie groupoids comes from surjective submersive groupoid morphisms $\Phi\colon G\ra H$ covering the identity, where one takes $\frak k$ to be the kernel of the associated Lie algebroid morphism, $\frak k=\ker\d\Phi|_M$. Not all bundles of ideals arise in this way---given a bundle of ideals $\frak k$ on $G\rra M$, we define its \textit{smearing} as the distribution $K\subset TG$,
  \begin{align}
    \label{eq:smearing}
    K_g\coloneqq \d (L_g)_{1_{s(g)}}(\frak k_{s(g)})=\d(R_g)_{1_{t(g)}}(\frak k_{t(g)}).
  \end{align}
  It is easy to see $K$ is involutive. If the corresponding foliation $\F(K)$ on $G$ is simple, the natural projection $\Phi\colon G\ra G/\F(K)$ is a surjective submersive Lie groupoid morphism with $\ker\d\Phi|_M=\frak k$. On the other hand, on Lie algebroids, every bundle of ideals $\frak k$ is clearly the kernel of the fibrewise surjective Lie algebroid morphism $\phi\colon A\ra A/\frak k$ that covers the identity.
 \end{example}
\end{samepage}

\section{The global picture}

\subsection{Basic definitions}
In the study of bundles of ideals, the notion of a multiplicative connection offers a richer framework than that of an invariant connection. To lay the groundwork, we begin by recalling the definition and highlighting some fundamental properties.
\begin{definition}
Let $\frak k$ be a bundle of ideals on a Lie groupoid $G\rra M$. A \textit{multiplicative Ehresmann connection} for $\frak k$ is a distribution $E\subset TG$ which is also a wide Lie subgroupoid of $TG\rra TM$, satisfying
\[
TG=E\oplus K,
\]
where $K$ is the smearing of $\frak k$, defined by \eqref{eq:smearing}. 
\end{definition}
\begin{example}
  If $\frak k$ comes from a surjective submersive Lie groupoid morphism $\Phi\colon G\ra H$, there holds $K=\ker\d\Phi$, hence $E$ is an Ehresmann connection in the usual sense, with the additional property of being multiplicative.
 \end{example}

A multiplicative Ehresmann connection can alternatively be viewed as a splitting of the following short exact sequence of $\vb$-groupoids covering $G\rra M$.
\[\begin{tikzcd}[column sep=large]
	0 & K & TG & {TG/K} & 0 \\
	0 & {0_M} & TM & TM & 0
	\arrow[from=1-1, to=1-2]
	\arrow[from=1-2, to=1-3]
	\arrow[shift left, from=1-2, to=2-2]
	\arrow[shift right, from=1-2, to=2-2]
	\arrow[from=1-3, to=1-4]
	\arrow[shift left, from=1-3, to=2-3]
	\arrow[shift right, from=1-3, to=2-3]
	\arrow[from=1-4, to=1-5]
	\arrow[shift left, from=1-4, to=2-4]
	\arrow[shift right, from=1-4, to=2-4]
	\arrow[from=2-1, to=2-2]
	\arrow[from=2-2, to=2-3]
	\arrow[from=2-3, to=2-4]
	\arrow[from=2-4, to=2-5]
\end{tikzcd}\]
In fact, there are several equivalent descriptions of multiplicative Ehresmann connections, among which we particularly favor the description with differential forms. Given a multiplicative Ehresmann connection $E\subset TG$, we can define a differential form $\omega \in \Omega^1(G;s^*\frak k)$ on $G$ with values in the pullback bundle $s^*\frak k\ra G$, 
\begin{align}
\label{eq:omega}
\omega(X)=\d(L_{g^{-1}})_g(v(X))\in \frak k_{s(g)},
\end{align}
for any $X\in T_gG$, where $v(X)$ is the so-called \textit{vertical} component in the unique decomposition $X=h(X)+v(X)\in E_g\oplus K_g=T_gG$. Here, the vector $h(X)$ is called the \textit{horizontal} component of $X$. The differential form just defined favors two characteristic properties. 

\pagebreak
\begin{proposition}[\cite{mec}*{Proposition 2.8}]
Let $G\rra M$ be a Lie groupoid. Given a multiplicative Ehresmann connection $E\subset TG$, the corresponding 1-form $\omega\in \Omega^1(G;s^*\frak k)$ satisfies:
\begin{enumerate}[label={(\roman*)}]
\item $\omega$ is \textit{multiplicative}: for any composable pair $(X,Y)\in T_{(g,h)}G^{(2)}$, there holds
\begin{align*}
\omega_{gh}(\d m(X,Y))=\Ad_{h^{-1}}\circ \omega_g(X) + \omega_h(Y).
\end{align*}
\item $\omega$ restricts to identity on $\frak k$: 
\[\omega|_{\frak k}=\id_{\frak k}.\]
\end{enumerate}
Conversely, given such a form $\omega$, $E=\ker \omega$ defines a multiplicative Ehresmann connection.
\end{proposition}

\begin{remark}
  Another way of expressing the differential form $\omega$ is by means of the \textit{Maurer--Cartan form} on $G$; this is a vector bundle isomorphism, defined on any  tangent vector $X\in \ker\d t_g$ by 
  $
  \Theta_{MC}(X)=\d{(L_{g^{-1}})}_g(X).
  $
  \[\begin{tikzcd}[column sep=small]
	{\ker\d t} && {s^* A} \\
	& G
	\arrow["{\Theta_{MC}}", from=1-1, to=1-3]
	\arrow[from=1-1, to=2-2]
	\arrow[from=1-3, to=2-2]
\end{tikzcd}\]
  Given a multiplicative Ehresmann connection $E$, its connection form $\omega$ is defined in terms of the Maurer--Cartan form simply by $\omega=\Theta_{MC}\circ v$. Moreover, we note that by multiplicativity, property (ii) above is equivalent to $\omega|_K=\Theta_{MC}|_K$, or in other words, $\omega(\xi^L)=s^*\xi$ for any section $\xi\in\Gamma(\frak k)$.
\end{remark}
We will denote the set of forms satisfying the properties from the last proposition as 
\begin{align*}
  \A(G;\frak k)=\set{\omega\in \Omega^1_m(G;\frak k)\given \omega|_{\frak k}=\id_{\frak k}},
\end{align*}
and sometimes simply call them  \textit{multiplicative connections} on $G$. 
The following observation will be needed in the next section: a multiplicative connection defines a distribution $E^s\subset\ker\d s\subset TG$, 
\[
E^s=E\cap \ker\d s.
\]
Clearly, this distribution satisfies the following two properties:
\begin{align}
  \ker \d s&= E^s\oplus K,\\
 \hspace{3.75em}\d(R_h)_g(E^s_g)&=E^s_{gh},\quad\text{for any }(g,h)\in \comp G.\label{eq:right_inv}
\end{align}
The second property will be referred to as \textit{right-invariance} of $E^s$. In terms of the connection 1-form $\omega$, it reads: \[(R_h)^*(\omega|_{\ker\d s_g})=\Ad_{h^{-1}}\circ\omega|_{\ker\d s_{gh}},\] for any composable pair 	$(g,h)\in\comp G$. 
\begin{remark}
  In the particular case when $G$ is regular and $\frak k=\ker \rho$, these two properties imply that for any $x\in M$, $E^s|_{G_x}$ defines a principal connection on the principal $G_x^x$-bundle $t\colon G_x\ra\O_x$. If $G$ is the gauge groupoid of a principal bundle $P\ra M$, we recover the classical notion of a principal connection on $P$.
\end{remark}

\subsection{Induced linear connection on bundles of ideals}

In this section, we show that any multiplicative Ehresmann connection $\omega$ for a bundle of ideals $\frak k$, induces a linear connection $\nabla$ on $\frak k\rightarrow M$. Later, in Proposition \ref{prop:conn_global_inf}, we will see that $\nabla$  is actually a part of the infinitesimal data contained within $\omega$. In what follows, $\nabla$ is constructed using a global approach, in terms of the given Lie groupoid $G\rra M$, since this will turn out to be useful for the subsequent sections. 

To begin, let us denote by $v\colon TG\ra K$ and $h\colon TG\ra E$ the vertical and horizontal projection (both are morphisms of vector bundles over $G$) with respect to a given multiplicative Ehresmann connection that complements $\frak k$. 
\begin{proposition}
\label{prop:conn}
Let $G\rra M$ be a Lie groupoid with a multiplicative Ehresmann connection $\omega\in\A(G;\frak k)$. Let the map $\nabla\colon \vf(M)\times \Gamma(\frak k)\ra  \Gamma(\frak k)$ be given by:
\begin{align}
\label{eq:nabla}
\nabla_X\xi=v[h(Y),\xi^L]|_M
\end{align}
where $Y\in \vf(G)$ is any lift of $X$ along the source map, i.e., $s_*Y=X$, and the bracket denotes the Lie bracket on $TG$. 
The following holds.
\begin{enumerate}[label={(\roman*)}]
  \item $\nabla$ defines a linear connection on $\frak k$.
  \item The left-invariant extension of $\nabla_X\xi$ reads $(\nabla_X\xi)^L=v[h(Y),\xi^L]$.
  \item $\nabla$ preserves the Lie bracket on the bundle of ideals:
\begin{align}
\label{eq:nabla_preserves_bracket}
  \nabla_X[\xi,\eta]_{\frak k}=[\nabla_X\xi,\eta]_{\frak k}+[\xi,\nabla_X\eta]_{\frak k}.
\end{align}
In particular, $\frak k$ is a locally trivial Lie algebra bundle.
\end{enumerate}
\end{proposition}
\begin{remark}
The defining equation \eqref{eq:nabla} above is similar to \cite{mec}*{Equation 2.4}, where $\nabla$ is obtained by first integrating $\frak k$ to a bundle of simply-connected Lie groups. In fact, our definition of $\nabla$ was motivated by the referenced equation, and our aim was to construct it entirely in terms of the groupoid $G$, which is of paramount importance in the next section.
\end{remark}
\begin{proof}
First, we prove that our definition \eqref{eq:nabla} is independent of the choice of the lift $Y\in\vf(G)$ of $X$ along $s\colon G\ra M$. Pick another such lift $\tilde Y\in \vf(G)$, so $Y-\tilde Y\in\Gamma(\ker \d s)$, and hence $Z\coloneqq h(Y-\tilde Y)\in \Gamma(E^s)$. Note that the flow $\phi_\lambda^{\smash{\xi^L}}=R_{\exp(\lambda \xi)}$ of a left-invariant vector field $\smash{\xi^L}$ is given by the right translation along the (target) bisection $\exp(\lambda \xi)$, hence by \eqref{eq:right_inv}, 
	\begin{align}
	\label{eq:principaltrick}
	\big((\phi_\lambda^{\xi^L})_*Z\big)_{g}\in E^s_{g},
	\end{align}
for all $\lambda\in\R$. This implies $\smash{[Z,\xi^L]=[h(Y),\xi^L]-[h(\tilde Y),\xi^L]}\in \Gamma(E^s)$. Since this difference is horizontal, the vertical parts of the two terms in the difference coincide, proving the expression $v[h(Y),\xi^L]$ is independent of the lift $Y$, and hence the same thus holds for $\nabla$. Let us now show $\nabla$ defines a connection. For $C^\infty(M)$-linearity in the first argument of the map $\nabla$, pick any $f\in C^\infty(M)$ and observe that if $Y$ is a lift of $X$, then $(f\circ s)Y$ is a lift of $fX$. Now use the Leibniz rule of the Lie bracket on $TG$ to obtain 
\[[h((f\circ s)Y),\xi^L]=(f\circ s)[h(Y),\xi^L]-\d(f\circ s)(\xi^L)h(Y).\]
The second term vanishes since $\frak k\subset \ker\d s$ and the first term gives us the wanted $C^\infty(M)$-linearity when its vertical part is evaluated at the units. 
For the Leibniz rule of $\nabla$, observe there holds $(f\xi)^L=(f\circ s)\xi^L$, so the Leibniz rule of the Lie bracket on $TG$ yields
\[
[h(Y),(f\circ s)\xi^L]=(f\circ s)[h(Y),\xi^L]+\d (f\circ s)(h(Y))\xi^L.
\]
Since $\d s_{1_x}(h(Y))=\d s_{1_x}(Y)=X_x$, vertically project and evaluate at the units to conclude.

To prove (ii), we must show that $v[h(Y),\xi^L]\in \Gamma(K)$ is left-invariant. For brevity, we will assume that $Y$ is a horizontal lift of $X$ along the source map, so $h(Y)=Y$. Our technique of proving (ii) utilizes bisections on $G$: fix $g\in G$, denote $s(g)=x$, and pick any local (source) bisection $\sigma\colon U\ra G$ with $\sigma(x)=g$ such that\footnote{Such a choice of a local bisection is always possible, by a similar proof as in \cite{mackenzie}*{Proposition 1.4.9}.} 
\begin{align}
\label{eq:bisection_in_E}
  \im\d \sigma_{x}\subset E_g.
\end{align}
The left translation by $\sigma$, given as the diffeomorphism \[
L_\sigma\colon t^{-1}(U)\ra t^{-1}(t\circ\sigma(U))
,\quad L_\sigma(h)=\sigma(t(h))h,
\]
then preserves the splitting $T_hG=E_h\oplus K_h$ for any $h\in G^x$, i.e., its differential restricts to
\begin{align}
  \label{eq:L_sigma_E}
    \d(L_\sigma)_{h}\colon E_{h}\ra E_{gh}.
\end{align}
This follows from a simple computation: differentiate the equation $L_\sigma=m\circ(\sigma\circ t,\id)$ which defines  $L_\sigma$ and use it on an arbitrary horizontal vector $v\in E_h$ to obtain
\[\d(L_\sigma)_h(v)=\d m_{(g,h)}(\underbrace{\d\sigma_x (\d t_h (v))}_{\in E_g},v)\in E_{gh},
\]
where we have used \eqref{eq:bisection_in_E} and the assumption that $E$ is multiplicative. We now use the left translation $L_\sigma$ to locally write $(L_\sigma)_*[Y,\xi^L]=[(L_\sigma)_*Y,\xi^L]$ and thus 
\begin{align*}
	\d(L_\sigma)_{1_{x}}[Y,\xi^L]-[Y,\xi^L]_g&=[(L_\sigma)_*Y-Y,\xi^L]_g=\deriv\lambda0\big((\phi_\lambda^{\xi^L})_*((L_\sigma)_*Y-Y)\big)_g.
\end{align*}
For convenience, let us denote $k_\lambda\coloneqq \exp(-\lambda \xi)_x=\phi^{\xi^L}_{-\lambda}(1_x)\in G^x_x$. We have
$$
\big((\phi_\lambda^{\xi^L})_*((L_\sigma)_*Y-Y)\big)_g=\d(\phi_\lambda^{\xi^L})_{gk_\lambda}\big({\d(L_\sigma)_{k_\lambda}(Y)}-Y_{gk_\lambda}\big).
$$
By \eqref{eq:L_sigma_E}, the difference $\smash{\d(L_\sigma)_{k_\lambda}(Y)-Y_{gk_\lambda}}$ is horizontal, and since $\d s$ annihilates it by the fact that $s\circ L_\sigma=s$, it must be contained in $\smash{E^s_{gk_\lambda}}$ for all $\lambda\in\R$. But since the flow of $\xi^L$ is given by right translations, \eqref{eq:right_inv} now implies
$$
\d(L_\sigma)_{1_{x}}[Y,\xi^L]-[Y,\xi^L]_g\in E^s_g,
$$
and taking the vertical part of this expression yields
\[
v[Y,\xi^L]_g=v(\d(L_\sigma)_{1_{x}}[Y,\xi^L])=v(\d(L_\sigma)_{1_{x}}(h[Y,\xi^L]_{1_x}+v[Y,\xi^L]_{1_x}))=\d(L_g)_{1_{x}}(v[Y,\xi^L]_{1_x}),
\]
where we have again used \eqref{eq:L_sigma_E} on the third equality, together with the fact that $L_\sigma$ restricts on $t$-fibres to the usual left-translation. This proves (ii).

Finally, preservation of the Lie bracket in (iii) easily follows from the Jacobi identity of the Lie bracket on $TG$. That is, supposing $Y$ is a horizontal lift of $X$ along the source map,
\[
(\nabla_X[\xi,\eta]_{\frak k})^L=v[Y,[\xi,\eta]^L]=v[Y,[\xi^L,\eta^L]]=v[[Y,\xi^L],\eta^L]+v[\xi^L,[Y,\eta^L]]
\]
and now observe that in the first term on the right-hand side, $h[Y,\xi^L]$ is horizontal and $s$-projectable to zero, hence it is a section of $E^s$, so by \eqref{eq:right_inv} the bracket $[h[Y,\xi^L],\eta^L]$ is again a section of $E^s$ and vanishes when vertically projected. Now use involutivity of $K$.
\end{proof}

\begin{remark}
Instead of working with lifts along $s\colon G\ra M$ and left-invariant extensions, we can work with lifts along $t\colon G\ra M$ and right-invariant extensions. In fact, there holds
\[
(\nabla_X\xi)^R=v[h(W),\xi^R],
\]
for any lift $W\in \vf(G)$ of $X\in\vf(M)$ along $t\colon G\ra M$, and any $\xi\in \Gamma(\frak k)$. This follows from the following two facts: firstly, if $Y$ is a lift of $X$ along the source map, then $\inv_* Y$ is a lift of $X$ along the target map; secondly, for any $\xi\in \Gamma(\frak k)$, there holds $\xi^L=-\inv_*(\xi^R)$, which is a consequence of the equality $\d(\inv)|_{\frak k}=-\id_{\frak k}$ that comes from the theory of Lie groups. Hence:
\begin{align*}
v[h(W),\xi^R]&=v[\inv_*(h(Y)),-\inv_*(\xi^L)]=-v(\inv_*[h(Y),\xi^L])\\
&=-\inv_*(v[h(Y),\xi^L])=-\inv_*(\nabla_X\xi)^L=(\nabla_X\xi)^R,
\end{align*}
where we have observed that $\inv_*$ commutes with both $v$ and $h$.
\end{remark}

The linear connection $\nabla$ on $\frak k$ obtained in Proposition \ref{prop:conn} has a particularly nice expression when differentiating with respect to vectors tangent to the orbit foliation, as seen below.

\begin{corollary}
\label{cor:conn_orb}
Let $G\rra M$ be a Lie groupoid with a multiplicative Ehresmann connection $\omega\in\A(G;\frak k)$. For any $\alpha\in\Gamma(A)$ and $\xi\in\Gamma(\frak k)$, there holds
\begin{align}
  \nabla_{\rho(\alpha)} \xi=[h(\alpha),\xi],
\end{align}
where the bracket denotes the Lie bracket on the algebroid $A$.
\end{corollary}
\begin{proof}
The vector field $h(\alpha)^L\in\vf(G)$ is a horizontal lift of $\rho(\alpha)$ along the source map, hence we can use Proposition \ref{prop:conn} (ii), to write $(\nabla_{\rho(\alpha)}\xi)^L=v[h(\alpha)^L,\xi^L]=v([h(\alpha),\xi]^L)$. Containment $[\Gamma(A),\Gamma(\frak k)]\subset\Gamma(\frak k)$ ensures that $[h(\alpha),\xi]^L$ is already vertical.
\end{proof}

\subsection{Horizontal exterior covariant derivative}
\label{sec:hor_ext_cov_der}
We now show that any multiplicative Ehresmann connection $\omega$ gives rise to the horizontal exterior covariant derivative on $\frak k$-valued forms. Just like the operator $\d{}^\nabla$ for an invariant connection $\nabla$, this operator will turn out to be a cochain map. Before defining it, we observe that a bundle of ideals $\frak k$ determines an  intrinsic subcomplex of $(\Omega^{\bullet,q}(G;\frak k),\delta)$. 
\begin{definition}
\label{defn:horizontal}
A form $\alpha\in\Omega^{p,q}(G;\frak k)$ is said to be \textit{horizontal}, if it vanishes when evaluated on a $p$-tuple of composable vectors from $K$, i.e., $\iota_X\alpha=0$ for any $X\in K^{(p)}$, where
\[
  K^{(p)}= (\underbrace{K\times\dots\times K}_{p \text{ copies}})\cap T G^{(p)}.
\]
The set of such forms will be denoted by \[\Omega^{p,q}(G;\frak k)^\Hor=\Gamma(\Lambda^q(K^{(p)})^\circ\otimes  (s\circ \pr_p)^*\frak k).\]
This definition applies when $p\geq 1$; at level $p=0$, all forms are defined to be horizontal.
At any fixed $q$, we obtain a subcomplex $(\Omega^{\bullet,q}(G;\frak k)^\Hor,\delta)\subset(\Omega^{\bullet,q}(G;\frak k),\delta)$. This follows from Definition \eqref{eq:delta_0} since $K\subset\ker\d s\cap \ker\d t$ and Definition \eqref{eq:delta_l} since $K\rra 0_M$ is a subgroupoid of $TG\rra TM$. It is called the \textit{horizontal subcomplex} of $\Omega^{\bullet,q}(G;\frak k)$.
\end{definition}
A multiplicative Ehresmann connection $\omega\in\A(G;\frak k)$ induces a \textit{horizontal projection} of tangent vectors in $TG^{(p)}$, by projecting all the components. That is,
\[
h\colon TG^{(p)}\ra E^{(p)}=(\underbrace{E\times\dots\times E}_{p\text{ copies}})\cap TG^{(p)},\quad h(X_1,\dots,X_p)=(h(X_1),\dots,h(X_p)).
\]
In turn, we obtain a horizontal projection of differential forms, given by the precomposition with the horizontal projection $\smash{TG^{(p)}\stackrel h\ra E^{(p)}\hookrightarrow TG^{(p)}}$ in all arguments. We denote it by
\begin{align*}
&h^*\colon \Omega^{p,q}(G;\frak k)\ra \Omega^{p,q}(G;\frak k)^\Hor.
\end{align*}
By multiplicativity of $E=\ker\omega$, it is also clear from the defining equation \eqref{eq:delta_l} of the simplicial differential $\delta$ that $h^*$ is a cochain map, 
\[h^*\delta=\delta h^*.\]
\begin{definition}
\label{defn:D}
Let $G\rra M$ be a Lie groupoid with a multiplicative Ehresmann connection $\omega\in\A(G;\frak k)$. The \textit{horizontal exterior covariant derivative} is defined as the map 
\[
\D{}^\omega=h^*\d{}^\nabla\colon \Omega^{p,q}(G;\frak k)\ra \Omega^{p,q+1}(G;\frak k),
\]
where $\nabla$ is the induced connection by $\omega$ on $\frak k$. Namely, for any $\vartheta\in\Omega^{p,q}(G;\frak k)$, 
\[
(\D{}^\omega\vartheta) (X_0,\dots,X_q)=(\d{}^{\nabla^{s\circ\pr_p}}\vartheta)(h(X_0),\dots,h(X_q)),
\]
for any given vector fields $X_i\in \vf(G^{(p)})$.
\end{definition}

The task at hand is to show that $\D{}^\omega$ is a cochain map; we will also see that in general, the connection $\nabla$ induced by $\omega$ is not $G$-invariant, so that $\d{}^\nabla$ is not a cochain map.\footnote{More precisely, if $\nabla$ is $G$-invariant, then $\frak k$ is abelian; the converse holds if $G$ is source-connected. This is a direct consequence of Proposition \ref{prop:g_inv_a_inv} and the fact that infinitesimally, invariance of $\nabla$ is equivalent to $\frak k$ being abelian, as shown in Proposition \ref{prop:invariant_abelian}.} To this end, we need to compute the tensor $\Theta$ appearing in Theorem \ref{thm:G_invariant}. We note that the means of defining the connection $\nabla$ entirely in terms of the groupoid $G$ as in equation \eqref{eq:nabla} turns out to be crucial for this purpose, together with the left-invariance property in Proposition \ref{prop:conn} (ii). 

The tensor $\Theta$ relates the pullback connections $\nabla^s$ and $\nabla^t$ on the pullback bundles $s^*\frak k\ra G$ and $t^*\frak k\ra G$. Note that these vector bundles are now both canonically isomorphic to $K\ra G$, with the isomorphisms given by restricting the Maurer--Cartan forms, that is,
\begin{align}
\begin{split}
  s^*\frak k\ra K,\ (g,\xi_{s(g)})&\mapsto \d(L_g)_{1_{s(g)}}(\xi_{s(g)}),\label{eq:iso_pullback_k}\\
t^*\frak k\ra K,\ (g,\xi_{t(g)})&\mapsto \d(R_g)_{1_{t(g)}}(\xi_{t(g)}).
\end{split}
\end{align}
This means that $\nabla^s$ and $\nabla^t$ can both be seen as connections on the vector bundle $K\ra G$. In what follows, these identifications will be employed, in order to relate the two pullback connections  simply by their difference $\nabla^t-\nabla^s$, thus avoiding notational complications. Importantly, we will show that $\nabla_X^t-\nabla_X^s$ depends only on the vertical component of $X$. We begin with a few simple observations. 
\begin{lemma}
\label{lem:nabla_st}
Let $G\rra M$ be a Lie groupoid with a multiplicative Ehresmann connection $\omega\in\A(G;\frak k)$. For any $\xi\in\Gamma(\frak k)$ and $X\in \vf(G)$ we have:
\begin{align}
\nabla^s_X(\xi^L)&=v[h(X),\xi^L],\label{eq:nablascinfty}\\
\nabla^t_X(\xi^R)&=v[h(X),\xi^R].\label{eq:nablascinfty2}
\end{align}
\end{lemma}
\begin{proof}
We only prove \eqref{eq:nablascinfty}; the proof of \eqref{eq:nablascinfty2} is similar. By definition of $\nabla^s$ and Proposition \ref{prop:conn} (ii), \eqref{eq:nablascinfty} already holds whenever $X\in\vf(G)$ is an $s$-projectable vector field, since $\xi^L$ is identified under \eqref{eq:iso_pullback_k} with $s^*\xi$. However, for any fixed $\xi\in \Gamma(\frak k)$, the map $\vf(G)\ra \Gamma(K)$ given as $X\mapsto v[h(X),\xi^L]$ is $C^\infty(G)$-linear. This implies that the right-hand side of \eqref{eq:nablascinfty} depends on $X$ pointwise, making the requirement of $s$-projectability redundant.
\end{proof}
\begin{remark}[Restriction of $\nabla^s$ to a $K$-connection is independent of $\omega$]
\label{rem:nabla_s}
Importantly, note that the last lemma implies $\nabla^s_X(\xi^L)=0$ for any vertical vector $X\in K$ and any $\xi\in\Gamma(\frak k)$. By virtue of Leibniz rule, the restricted $K$-connection on $K$ (again denoted by $\nabla^s$) is completely determined by this condition, i.e., that it vanishes on left-invariant sections of $K$. This shows that the $K$-connection $\nabla^s$ is intrinsic---it is independent of the choice of a multiplicative connection. Under the isomorphism \eqref{eq:iso_pullback_k}, it coincides with the trivial pullback connection on $s^*\frak k\ra G$ restricted to $K$, since the latter is defined precisely by the property that it vanishes on the pullback sections, and these are identified with left-invariant sections of $K$. More explicitly, in a local frame $(b_i)_i$ of $\frak k$ over an open subset $U\subset M$, any $Y\in\Gamma(K)$ may be expressed as $Y=Y^ib_i^L$ for some functions $Y^i\in C^\infty(s^{-1}(U))$, so that
\[
\nabla^s_X Y=X(Y^i)b_i^L
\]
for any $X\in K_g$, $g\in s^{-1}(U)$. 

\end{remark}

\begin{corollary}
\label{cor:diff_nabla_ts}
Let $G\rra M$ be a Lie groupoid with a multiplicative Ehresmann connection $\omega\in\A(G;\frak k)$. For any $X\in \vf(G)$ and $Y\in\Gamma(K)$ there holds:
\begin{align}
\label{eq:diff_nabla_ts}
  \nabla^t_XY-\nabla^s_XY=\nabla^t_{v(X)}Y-\nabla^s_{v(X)}Y.
\end{align}
\end{corollary}
\begin{proof}
Since both sides of equation \eqref{eq:diff_nabla_ts} are $C^\infty(G)$-linear in $Y$, it is enough to show that this equality holds for any left-invariant section $Y\in \Gamma(K)$. Pick  any local frame $(b_i)_i$  of $\frak k$ over an open subset $U\subset M$, and expand the left-invariant section $Y$ with respect to the right-invariant extension of the local frame $(b_i)_i$. That is, write $Y=Y^ib_i^R$ for some functions $Y^i\in C^\infty(t^{-1}(U))$, using the summation convention. The left-hand side reads:
\begin{align*}
\nabla_X^t(Y^ib_i^R)-v[h(X),Y^ib_i^R]&=Y^i\nabla_X^t(b_i^R)+X(Y^i)b_i^R-Y^iv[h(X),b_i^R]-(h(X)Y^i)b_i^R\\
&=(v(X)Y^i)b_i^R=\nabla^t_{v(X)}Y
\end{align*}
where we have used the Leibniz rules for $\nabla^t$ and for the Lie bracket on $K$ in the second equality, and Lemma \ref{lem:nabla_st} in the third. The obtained expression clearly equals the right-hand side of equation \eqref{eq:diff_nabla_ts} since $\nabla^s_{v(X)}Y=0$ due to left-invariance of $Y$.
\end{proof}
\begin{theorem}
\label{thm:deltaD}
Let $G\rra M$ be a Lie groupoid with a multiplicative Ehresmann connection $\omega\in\A(G;\frak k)$. The horizontal exterior covariant derivative $\D{}^\omega$ is a cochain map, i.e., 
\begin{align}
\label{eq:deltaD}
  \delta\D{}^\omega=\D{}^\omega\delta.
\end{align}
In particular, $\D{}^\omega$ maps multiplicative forms to multiplicative forms.
\end{theorem}
\begin{proof}
  Corollary \ref{cor:diff_nabla_ts} states that the tensor $\Theta$ equals 
  \[
  \Theta(X)\xi=\phi(\nabla^t_{v(X)}\xi)-\nabla^s_{v(X)}\phi(\xi)
  \]
  for any $X\in \vf(G)$ and $\xi\in \Gamma(t^*\frak k)$,
  where $\phi\colon t^*\frak k\ra s^*\frak k$ is the bundle isomorphism given by $(g,v)\mapsto (g,\Ad_{g^{-1}}(v))$ and $\nabla^s$ and $\nabla^t$ denote the trivial connections on the pullback bundles (as in Remark \ref{rem:nabla_s}). Since $h^*$ is a cochain map, we have 
  \[
  [\D{}^\omega,\delta]= h^*{\d{}^\nabla}\delta-\delta h^*{\d{}^\nabla}=h^*[\d{}^\nabla,\delta],
  \]
  which vanishes by the expression for $[\d{}^\nabla,\delta]$ from Lemma \eqref{lem:G_invariant} since $\Theta(h(X))=0$.
  \end{proof}
In conclusion, a multiplicative Ehresmann connection yields the columns of a curved double complex, depicted in the diagram below, with the feature that $\D{}^\omega$ does not square to zero unless $E=\ker \omega$ is involutive.
\begin{align}
\label{eq:bss_ideals}
\begin{tikzcd}[ampersand replacement=\&, column sep=large, row sep=large]
	{\Omega^{q+1}(M;\frak k)} \& {\Omega^{q+1}(G;s^*\frak k)} \& {\Omega^{q+1}(G^{(2)};(s\circ\pr_2)^*\frak k)} \& \cdots \\
	{\Omega^q(M;\frak k)} \& {\Omega^{q}(G;s^*\frak k)} \& {\Omega^q(G^{(2)};(s\circ\pr_2)^*\frak k )} \& \cdots
	\arrow["{\delta}", from=1-1, to=1-2]
	\arrow["{\delta}", from=1-2, to=1-3]
	\arrow["{\delta}", from=1-3, to=1-4]
	\arrow["{\d{}^\nabla}", from=2-1, to=1-1]
	\arrow["{\delta}", from=2-1, to=2-2]
	\arrow["{\D{}^\omega}", from=2-2, to=1-2]
	\arrow["{\delta}", from=2-2, to=2-3]
	\arrow["{\D{}^\omega}", from=2-3, to=1-3]
	\arrow["{\delta}", from=2-3, to=2-4]
\end{tikzcd}
\end{align}
In the case when $E$ is involutive, a short computation shows we have
\[
(\D{}^\omega)^2\vartheta=h^*(R^{\nabla^{s\circ\pr_p}}\!\wedge\vartheta),
\]
for any $\vartheta\in\Omega^{p,q}(G;\frak k)$. Since involutivity of $E$ also implies $\nabla$ is flat, $(\D{}^\omega)^2=0$ follows.
\begin{remark}
  We observe that the proof generalizes to the following setting: suppose $V$ is an arbitrary representation of $G\rra M$ equipped with a linear connection $\nabla$, and $E$ is a multiplicative Ehresmann connection for a fixed bundle of ideals $\frak k$. Then $\D{}=h^*\circ\d{}^\nabla$ is a cochain map $\Omega^{\bullet,q}(G;V)\ra \Omega^{\bullet,q+1}(G;V)$ if and only if there holds
  $h^*\Theta=0$, where $\Theta$ is the $G$-invariance form of $\nabla$ from \eqref{eq:invariance_form_G}. Furthermore, taking $\frak k=0_M$ recovers the framework from \sec \ref{chapter:invariant}, so the usual exterior covariant derivative can be seen as a special case of the horizontal exterior covariant derivative.
\end{remark}

\subsection{Curvature}
\label{sec:curvature}
\begin{definition}
    The \textit{curvature} $\Omega^\omega\in\Omega^2(G;s^*\frak k)$ of a multiplicative Ehresmann connection $\omega\in \A(G;\frak k)$ on a Lie groupoid $G\rra M$ is given by \[\Omega^\omega=\D{}^\omega\omega.\]
    \end{definition}
    In what follows we state several important properties of the curvature form, generalizing the ones already known from the theory of principal bundles. The following result is already established in \cite{mec}*{Propositions 2.22 and 2.24}, however, a direct proof of (i) below is now possible due to our results from \sec\ref{sec:hor_ext_cov_der}, and also a simpler proof of the structure equation in (iii) on account of our construction of $\nabla$ in Proposition \ref{prop:conn}. 
    
    \begin{samepage}
      \begin{proposition}
        \label{prop:multiplicative_curvature}
          Let $\omega\in\A(G;\frak k)$ be a multiplicative Ehresmann connection on $G\rra M$. Its curvature satisfies the following properties.
          \begin{enumerate}[label={(\roman*)}]
              \item $\Omega^\omega$ is a multiplicative form.
              \item For any horizontal vector fields $X,Y\in \Gamma(E)$ there holds
              \begin{align}
              \label{eq:curv_horizontal}
                  \d (L_g)_{1_{s(g)}}\Omega^\omega_g(X,Y)=h([X,Y]_g)-[X,Y]_g.
              \end{align}
              In particular, the curvature $\Omega^\omega$ vanishes if and only if $E$ is involutive.
              \item The structure equation for $\Omega^\omega$ holds:
              \begin{align}
              \label{eq:structure}
              \Omega^\omega=\d{}^{\nabla^s}\omega+\frac 12 [\omega,\omega]_{s^*\frak k}.
              \end{align}
              \item The \textit{Bianchi identity} for $\Omega^\omega$ holds:
              \begin{align}
              \label{eq:bianchi}
                  \D{}^\omega \Omega^\omega = 0.
              \end{align}
          \end{enumerate}
      \end{proposition}
    \end{samepage}
    \begin{remark}
        The bracket on $\Omega^\bullet(G;s^*\frak k)$ which appears in equation \eqref{eq:structure} is the bracket of forms induced by the Lie bracket on the bundle of Lie algebras $s^*\frak k\ra G$. Skew-symmetry of the
         latter and graded commutativity of the wedge product ensure there holds
       \[[\alpha,\beta]_{s^*\frak k}=(-1)^{kl+1}[\beta,\alpha]_{s^*\frak k}.\] Moreover, the graded Leibniz rule applies: 
    \begin{align}
    \label{eq:d_omega_graded_leibniz}
           \d{}^{\nabla^s}[\alpha,\beta]_{s^*\frak k}=[\d{}^{\nabla^s}\alpha,\beta]_{s^*\frak k}+(-1)^{k}[\alpha,\d{}^{\nabla^s}\beta]_{s^*\frak k},
    \end{align} 
    which follows from the fact that the induced connection $\nabla$ preserves the Lie bracket on $\frak k$. 
    \end{remark}
    \begin{proof}
    Multiplicativity of $\Omega^\omega$ is a direct consequence of Theorem \ref{thm:deltaD}. To show (ii), we note that if $X,Y\in \Gamma(E)$, then there holds
    \begin{align*}
    \Omega^\omega(X,Y)&=\nabla^s_X\omega(Y)-\nabla^s_Y\omega(X)-\omega([X,Y])=-\omega([X,Y])
    \end{align*}
    and now use the defining equation $\eqref{eq:omega}$ of $\omega$ to conclude (ii). 
    
    To prove (iii), first write out the right-hand side of equation \eqref{eq:structure}:
    \begin{align*}
    \Big({\d{}}^{\nabla^s}\omega&+\frac 12[\omega,\omega]_{s^*\frak k}\Big)(X,Y)=\nabla^s_X\omega(Y)-\nabla^s_Y\omega(X)-\omega([X,Y])+[\omega(X),\omega(Y)]_{s^*\frak k},
    \end{align*}
    for any vector fields $X,Y\in \vf (G)$. Note that both sides of \eqref{eq:structure} are $C^\infty(G)$-linear and moreover, any vertical vector can be extended to a vertical left-invariant vector field, and any horizontal vector can be extended to a horizontal $s$-projectable vector field on $G$. Hence, it is enough to consider all possible combinations of $X,Y$ being either vertical and left-invariant, or horizontal and $s$-projectable. If both $X$ and $Y$ are horizontal and $s$-projectable, both sides of \eqref{eq:structure} evaluate to $-\omega([X,Y])$, as already seen above. Furthermore, if both are vertical and left-invariant, i.e., $X=\xi^L$ and $Y=\eta^L$ for some $\xi,\eta\in \Gamma(\frak k)$, then both sides of \eqref{eq:structure} vanish since $[X,Y]=[\xi,\eta]^L$ and so $\omega([X,Y])=s^*[\xi,\eta]_{s^*\frak k}=[\omega(X),\omega(Y)]_{s^*\frak k}$. Finally, suppose that $X=\xi^L$ for some $\xi\in \Gamma(\frak k)$ and $Y\in \Gamma(E)$ is $s$-projectable to some vector field $s_*Y=U\in \vf(M)$. We need to show that the following identity holds:
    \begin{align*}
    \nabla^s_Y\omega(X)=\omega([Y,X]).
    \end{align*}	
    Notice that at any $g\in G$, we have
    \begin{align*}
        \nabla^s_{Y}\omega(X)|_g&=\nabla^s_Y(s^*\xi)|_g=s^*(\nabla_U\xi)|_g=(g,v[Y,\xi^L]_{1_{s(g)}}),\\
        \omega_g([Y,X])&=(g,\d(L_{g^{-1}})_g(v[Y,X]_g)),
    \end{align*}
    where we have used the defining equation \eqref{eq:nabla} of $\nabla$ in the first line, and the defining equation \eqref{eq:omega} of $\omega$ in the second. These expressions coincide by left-invariance of $v[Y,\xi^L]$ from Proposition \ref{prop:conn} (ii).
    
    The proof of the Bianchi identity is a matter of applying the structure equation:
    \begin{align}
    \label{eq:bianchi_intermediate}
      \d{}^{\nabla^s}\Omega^\omega=\d{}^{\nabla^s} (\d{}^{\nabla^s} \omega+\frac 12[\omega,\omega]_{s^*\frak k})=R^{\nabla^s}\wedge \omega+ [\d{}^{\nabla^s}\omega,\omega]_{s^*\frak k}
    \end{align}
    where we have denoted by $R^{\nabla^s}\in\Gamma(\Lambda^2(T^*G)\otimes \End(s^*\frak k))$ the usual curvature tensor of the pullback connection $\nabla^s$ on $s^*\frak k\ra G$. Inserting any three horizontal vectors into this equation now yields zero since $E=\ker \omega$.
    \end{proof}
    \begin{remark}
    If we insert the structure equation into \eqref{eq:bianchi_intermediate} once more, we see that the Bianchi identity can also be written in the alternative form
    \[
     \d{}^{\nabla^s}\Omega^\omega+[\omega,\Omega^\omega]_{s^*\frak k}=R^{\nabla^s}\wedge\omega.
    \]
    \end{remark}

\section{The infinitesimal picture}
\label{sec:weil_boi}
As observed in \cite{mec}, any multiplicative Ehresmann connection is mapped with the van Est map to an infinitesimal multiplicative form whose symbol restricts to the identity map on the bundle of ideals $\frak k\subset \ker\rho$. Importantly, this infinitesimal notion of a multiplicative connection makes sense on its own, without the need for a given algebroid to be integrable. 

\subsection{Infinitesimal multiplicative connections}
\label{sec:mec_inf}

As in the case of groupoids, we start this section by briefly recalling some definitions and results regarding infinitesimal multiplicative connections.
\begin{definition}
Let $\frak k\subset A$ be a bundle of ideals on a Lie algebroid $A\Rightarrow M$. An \textit{infinitesimal multiplicative connection} (more briefly, an \textit{IM connection}) for $\frak k$, is a $\frak k$-valued IM form $(\C,v)\in \Omega^1_{im}(A;\frak k)$, whose symbol $v\colon A\rightarrow \frak k$ restricts on $\frak k$ to the identity:
\begin{align}
\label{eq:symbol_id}
  v|_{\frak k}=\id_{\frak k}.
\end{align}
The set of all IM connections on $A$ for $\frak k$ is denoted $\A(A,\frak k)$.
\end{definition}
\begin{remark}
  If a Lie groupoid $G$ integrates $A$, it is clear that the van Est map takes multiplicative Ehresmann connections to IM connections, since the symbol of $(\C,v)=\ve(\omega)$ for any form $\omega\in \Omega^1_m(G;s^*V)$ is given by $v=\omega|_A$. Together with this observation, Corollary \ref{corollary:van_est_multiplicative} implies that if $G$ has simply-connected source fibres, the van Est map restricts to a bijective correspondence between multiplicative Ehresmann connections on $G$ and IM connections on $A$.
\end{remark}

It is important to recognize that the defining property  \eqref{eq:symbol_id} of an IM connection $(\C,v)$ means precisely that its symbol $v$ is a splitting of the following short exact sequence of vector bundles:
\begin{align}
\label{eq:splitting}
  \begin{tikzcd}[ampersand replacement=\&]
	0 \& {\frak k} \& A \& {A/\frak k} \& 0
	\arrow[from=1-1, to=1-2]
	\arrow[from=1-2, to=1-3]
	\arrow["v", bend left=30, from=1-3, to=1-2]
	\arrow[from=1-3, to=1-4]
	\arrow[from=1-4, to=1-5]
\end{tikzcd}
\end{align}
So, let us denote by $H=\ker (v)$ the \textit{horizontal subbundle} of $A$, and denote by $\alpha=v(\alpha)+h(\alpha)$ the unique decomposition of any vector $\alpha\in A$, pertaining to the splitting $A=\frak k\oplus H$ given by the symbol. We will call $v(\alpha)\in \frak k$ and $h(\alpha)\in H$ the \textit{vertical} and \textit{horizontal} component of $\alpha$, respectively. 

\subsubsection{Associated coupling data of an IM connection}
The splitting of \eqref{eq:splitting}, determined by the symbol of an IM connection $(\C,v)$, enables us to split the information which is contained within the leading term $\C$ into the following two objects.
\begin{enumerate}[label={(\roman*)}]
  \item A linear connection $\nabla$ on the bundle of ideals $\frak k$, given by $\nabla=\C|_{\frak k}$, that is 
\begin{align}
\label{eq:nabla2}
  \nabla_X\xi=\C(\xi)(X).
\end{align}
  \item A tensor field $U\in \Gamma(H^*\otimes T^*M\otimes \frak k)$, given by
\begin{align}
\label{eq:defU}
  U(\alpha)(X)=-\C(\alpha)(X).
\end{align}
\end{enumerate}
The pair $(\nabla,U)$ constructed above will be called the \textit{coupling data} of an IM connection $(\C,v)$, which, owing to the Leibniz identity, indeed consists of a connection and a tensor field. In other words, we can write
\begin{align}
\label{eq:split_C}
  \C(\alpha)=\nabla(v\alpha)-U(h\alpha),
\end{align}
for any $\alpha\in\Gamma(A)$. 
To see the meaning behind the tensor $U$, note that compatibility condition \eqref{eq:c2} implies
\begin{align}
\label{eq:U_along_orbits}
  v[h(\alpha),h(\beta)]=U(h\alpha)(\rho\beta),
\end{align}
so that the orbital part of $U$ measures the failure of the splitting \eqref{eq:splitting} to be a splitting of Lie algebroids, i.e, $U$ vanishes on $T\F$ precisely when $H$ is a Lie subalgebroid of $A$. However, as we will see, $U$ does not encode the whole information about the curvature of $(\C,v)$, since we also need to account for the curvature $R^\nabla$ of $\nabla$. This will become clearer in \sec\ref{sec:im_curvature}.


Turning to the linear connection $\nabla$, we see that its construction is  more direct than in the groupoid case (Proposition \ref{prop:conn}). If $A$ is the Lie algebroid of a Lie groupoid $G$ equipped with the IM connection $(\C,v)=\ve(\omega)$ for a multiplicative Ehresmann connection $\omega$, then the two linear connections on $\frak k$ coincide---this is shown in Proposition \ref{prop:conn_global_inf} below. Before proving it, let us establish the infinitesimal version of Corollary \ref{cor:conn_orb}.
\begin{proposition}
\label{prop:conn2}
Let $A$ be a Lie algebroid with an IM connection $(\C,v)\in \A(A;\frak k)$. For any $\alpha\in\Gamma(A)$ and $\xi\in\Gamma(\frak k)$, there holds
\begin{align}
\label{eq:conn_orb2}
  \nabla_{\rho(\alpha)}\xi=[h(\alpha),\xi].
\end{align}
\end{proposition}
\begin{proof}
Equation \eqref{eq:conn_orb2} is shown with a straightforward computation:
\begin{align*}
  \nabla_{\rho(\alpha)}\xi=\C(\xi)(\rho\alpha)=\L^A_\xi v(\alpha)-v[\xi,\alpha]=[\xi,v(\alpha)]+[\alpha,\xi]=[h(\alpha),\xi],
\end{align*}
where we have used the definition of $\nabla$ and condition \eqref{eq:c2}. 
\end{proof}

\begin{proposition}
  \label{prop:conn_global_inf}
  Let $\omega\in\A(G;\frak k)$ be a multiplicative Ehresmann connection on a Lie groupoid $G\rra M$. Let $A$ be its Lie algebroid, endowed with the IM connection $(\C,v)=\ve(\omega)$. The two linear connections induced by $\omega$ and $(\C,v)$, respectively defined by equations \eqref{eq:nabla} and \eqref{eq:nabla2}, coincide.
\end{proposition}

\begin{proof}
  We need to show that for any $\xi\in\Gamma(\frak k)$ and $X\in\vf(M),$ there holds
  \[
  \C(\xi)_x(X)=v[Y,\xi^L]_{1_x}
  \]
  for any $x\in M$, where $Y$ is any horizontal $s$-lift of $X$. Letting $g_\lambda=\phi^{\xi^L}_{-\lambda}(1_x)$, the right side equals
  \begin{align*}
    v[Y,\xi^L]_{1_x}=\deriv\lambda 0 v\big({\d(\phi^{\xi^L}_\lambda)_{g_{\lambda}}}(Y_{g_{\lambda}})\big).
  \end{align*}
  The flow of a left-invariant vector field is given by the right translation along the bisection $\exp(\lambda\xi)$, i.e., $\phi^{\xi^L}_\lambda=R_{\exp(\lambda\xi)}$ and differentiating 
  $R_{\exp(\lambda\xi)}=m\circ(\id,\phi^{\xi^L}_\lambda\circ u\circ s),$ we get
  \[
  \d{}\big(R_{\exp(\lambda\xi)}\big)_{g_\lambda}(Y)=\d m \big(Y_{g_\lambda},\d{}(\phi^{\xi^L}_\lambda)_{1_x}\d u(X_x)\big)
  \]
  Now observe that since $E$ is multiplicative, $v\colon TG\ra K$ is a groupoid morphism. But since $Y$ is horizontal, we obtain
   \begin{align}
    \label{eq:aux_v_Y_xiL}
     v[Y,\xi^L]_{1_x}=\deriv\lambda 0 \d{}({L_{g_{\lambda}}}) v\big({\d(\phi^{\xi^L}_\lambda)_{1_x}\d u(X_x)}\big).
   \end{align}
  On the other hand, we have
  \begin{align*}
    \C(\xi)_x(X)=\deriv\lambda 0 \underbrace{\Ad_{g_{\lambda}^{-1}}}_{\frak k_x\ra\frak k_x}\big(\underbrace{\omega(\d(\phi^{\xi^L}_\lambda)\d u(X_x))}_{\text{in }\frak k_x\text{ for all }\lambda}\big),
  \end{align*}
  so we can apply the chain rule to differentiate this expression. But since $u^*\omega=0$, this yields precisely the expression \eqref{eq:aux_v_Y_xiL}, concluding the proof.
\end{proof}

To conclude this section, we note that the compatibility conditions \eqref{eq:c1}--\eqref{eq:c3} for an IM connection $(\C,v)$ can actually be rewritten  entirely in terms of $\nabla$ and $U$, as shown in \cite{mec}*{Propositions 5.11 and 5.13}. We record these rewritten conditions here for our convenience. The following holds for all $\alpha,\beta\in\Gamma(A)$ and $\xi,\eta\in\Gamma(\frak k)$: 
\begin{enumerate}[label={(\roman*)}]
  \item The connection $\nabla$ preserves the Lie bracket on $\frak k$:
  \begin{align}
  \nabla[\xi,\eta]_{\frak k}&=[\nabla\xi,\eta]_{\frak k}+[\xi,\nabla\eta]_{\frak k}.\tag{S.1}\label{eq:s1}
  \end{align}
In particular, $\frak k$ is a locally trivial bundle of Lie algebras.
	\item The curvature tensor $R^\nabla$ of $\nabla$ relates to $U$ as follows:
  \begin{align}
	\iota_{\rho(\alpha)}R^\nabla\cdot\xi=[U(h\alpha),\xi]_{\frak k}.\tag{S.2}\label{eq:s2}
  \end{align}
  \item The tensor $U$ acts on the bracket of sections as:
\begin{align}
U(h[\alpha,\beta])=\L^\nabla_{\rho(\alpha)}U(h\beta)-\L^\nabla_{\rho(\beta)}U(h\alpha)+\nabla U(h\alpha)(\rho\beta).
\tag{S.3}\label{eq:s3}
\end{align}
\end{enumerate}
Using the point (ii) above, we can now show that $A$-invariance of the connection $\nabla$ is equivalent to the bundle of ideals $\frak k$ being abelian.
\begin{proposition}
\label{prop:invariant_abelian}
Let $A$ be a Lie algebroid with an IM connection $(\C,v)\in \A(A;\frak k)$. The connection $\nabla=\C|_\frak k$ is $A$-invariant if and only if the bundle of ideals $\frak k$ is abelian.
\end{proposition}
\begin{proof}
Equation \eqref{eq:conn_orb2} implies that for any $\alpha\in\Gamma(A)$, we have
\begin{align}
\label{eq:diff_a_conn}
  \nabla^A_\alpha-\nabla_{\rho(\alpha)}=[v(\alpha),\cdot],
\end{align}
so it is clear that $\nabla^A_\alpha=\nabla_{\rho(\alpha)}$ holds if and only if $\frak k$ is abelian. In this case, $\iota_{\rho(\alpha)}R^\nabla=0$ is implied by the identity \eqref{eq:s2}.
\end{proof}

\subsection{Horizontal projection of Weil cochains}
Following the same approach as in the Lie groupoid picture, we begin by providing the infinitesimal analogue of horizontal forms on the nerve (Definition \ref{defn:horizontal}). 
\begin{definition}
\label{defn:horizontal_inf}
A Weil cochain $c=(c_0,\dots,c_p)\in W^{p,q}(A;\frak k)$ is said to be \textit{horizontal} if 
\begin{align*}
c_i(\cdot\|\xi,\cdot)=0, \quad\text{for all }i\geq 1\text{ and } \xi\in\Gamma(\frak k).
\end{align*}
Note this is a condition on correction terms only. It is clear from equation \eqref{eq:delta_inf} that $\delta$ maps horizontal cochains to horizontal cochains, so at each fixed $q\geq 0$ we obtain the \textit{horizontal subcomplex} 
\[
W^{\bullet,q}(A;\frak k)^\Hor\leq W^{\bullet,q}(A;\frak k).
\]
Equation \eqref{eq:J_alpha} ensures the van Est map restricts to a map between horizontal subcomplexes.
\end{definition}
\begin{example}
\label{ex:horizontal_p=1}
At level $p=1$, horizontal cochains $c=(c_0,c_1)$ are simply the ones whose symbol restricts on $\frak k$ to zero: \[c_1|_{\frak k}=0.\] It is clear from definition \eqref{eq:delta0_im} of $\delta^0$ that any cohomologically trivial form is horizontal (and multiplicative), that is, $\im\delta^0\subset \Omega^\bullet_{im}(A;\frak k)^\Hor$. On a transitive algebroid, every IM form of degree $q\geq 2$ is horizontal by condition \eqref{eq:c3}, that is, $\Omega^q_{im}(A;\frak k)=\Omega^q_{im}(A;\frak k)^\Hor$.
\end{example}
Just as with multiplicative connections on Lie groupoids, we expect an IM connection for $\frak k$ to induce a horizontal projection of Weil cochains,
\[
h^*\colon W^{p,q}(A;\frak k)\ra W^{p,q}(A;\frak k)^\Hor.
\] 
However, in contrast with the case of groupoids, there is now no straightforward and intuitive way of defining $h^*$. The issue, essentially, lies in the model $W^{p,q}(A;\frak k)$ that we are using to describe forms in the infinitesimal setting. We overcome this obstacle in \sec\ref{sec:im_connections_as_vb_algebroids}, by employing the alternative viewpoint of exterior cochains from \sec\ref{sec:alternative_model_weil}. The idea is the following.
\begin{itemize}
  \item Firstly, note that an IM connection $(\C,v)$ can equivalently be described in terms of an Ehresmann connection $E\subset TA$ for the projection $A\ra A/\frak k$. Infinitesimal multiplicativity of $(\C,v)$ is equivalent to $E$ being a $\vb$-subalgebroid of $TA$. 
  \item Secondly, observe that on exterior cochains, the horizontal projection has a simple definition. Namely, as with the groupoid case, the map $h^*$ on this alternative model is given by the precomposition with the map $h\colon TA\ra E$ (Definition \ref{def:hor_proj_ext}). 

  \item At last, use the isomorphism between the two models to derive the wanted formula for $h^*$ on Weil cochains and infer its elementary properties (Theorem \ref{thm:derivation_hor_proj}). Importantly, establishing the properties of $h^*$ is significantly easier if done in the realm of exterior cochains---it is more conceptual and less combinatorial/computational. 
\end{itemize}
Ultimately, this procedure yields the map $h^*$ for Weil cochains, which we will now describe. Preliminarily, note that since $V=\frak k$, we may introduce the pairing
\begin{align}
  \label{eq:pairing}
\begin{split}
  &\Omega^k(M;S^j(A^*)\otimes \frak k)\times \Omega^\ell(M;\frak k) \xlongrightarrow{\wedgedot}\Omega^{k+\ell}(M;S^{j-1}(A^*)\otimes \frak k),\\
  &(\gamma\wedgedot \vartheta)(\beta_1,\dots,\beta_{j-1})(X_1,\dots,X_{k+\ell})\\
  &\phantom{(\gamma\wedgedot}=\frac 1{k!\ell!}\smashoperator{\sum_{\quad\sigma\in S_{k+\ell}}}(\sgn\sigma)\gamma\big(\vartheta(X_{\sigma(1)},\dots,X_{\sigma(\ell)}),\beta_1,\dots,\beta_{j-1}\big)(X_{\sigma(\ell+1)},\dots,X_{\sigma(\ell+k)}).
\end{split}
\end{align}
We will actually only need the case $\ell=1$, when $\vartheta$ is a 1-form. The form we obtain by pairing $\gamma$ consecutively with 1-forms $\vartheta_1,\dots,\vartheta_\ell\in\Omega^1(M;\frak k)$, where $\ell\leq j$, will be denoted
\begin{align}
\label{eq:wedgedot_one_by_one}
\gamma\wedgedot (\vartheta_1,\dots,\vartheta_\ell)\coloneqq \gamma\wedgedot \vartheta_\ell\wedgedot\vartheta_{\ell-1}\wedgedot\cdots\wedgedot\vartheta_1 \in \Omega^{k+\ell}(M;S^{j-\ell}(A^*)\otimes\frak k).
\end{align}
It will turn out that we only need the case when $j=\ell$, i.e., when we use up all the symmetric arguments by pairing with 1-forms. A short computation shows that \eqref{eq:wedgedot_one_by_one}, evaluated on vector fields $X_1,\dots, X_{k+\ell}\in\vf(M)$, reads
\begin{align}
  \label{eq:wedgedot_multiple}
\begin{split}
  \big(\gamma\wedgedot (&\vartheta_1,\dots,\vartheta_\ell)\big)(X_1,\dots, X_{k+\ell})\\&=\frac 1{k!}\smashoperator{\sum_{\quad\sigma\in S_{k+\ell}}}(\sgn\sigma)\gamma\big(\vartheta_1(X_{\sigma(1)}),\dots,\vartheta_\ell(X_{\sigma(\ell)})\big)(X_{\sigma(\ell+1)},\dots,X_{\sigma(\ell+k)}).
\end{split}
\end{align}

\begin{definition}
\label{def:hor_proj}
The \textit{horizontal projection} of Weil cochains, induced by an IM connection $(\C,v)\in\A(A;\frak k)$ on a Lie algebroid $A$, is defined as the map
\[
h^*\colon W^{p,q}(A;\frak k)\ra W^{p,q}(A;\frak k)^\Hor,
\]
whose leading term for a given $c\in W^{p,q}(A;\frak k)$ is given by
\begin{align*}
\begin{split}
(h^*c)_0(\alpha_1,\dots,\alpha_p)=\sum_{j=0}^p\,(-1)^j\smashoperator{\sum_{\hspace{2em}\sigma\in S_{(j,p-j)}}}(\sgn\sigma)c_j(\alpha_{\sigma(j+1)},\dots,\alpha_{\sigma(p)})\wedgedot (\C\alpha_{\sigma(1)}, \dots, \C\alpha_{\sigma(j)}),
\end{split}
\end{align*}
where $S_{(j,p-j)}\subset S_p$ denotes the set of $(j,p-j)$-shuffles. 
The correction terms are given as
\begin{align*}
\begin{split}
(h^*c)_k&(\alpha_1,\dots \alpha_{p-k}\|\beta_1,\dots,\beta_k)\\
&=\sum_{j=k}^p (-1)^{j-k}\smashoperator{\sum_{\hspace{3em}\sigma\in S_{(j-k,p-j)}}}(\sgn\sigma)c_j(\alpha_{\sigma(j-k+1)},\dots,\alpha_{\sigma(p-k)}\| h\ul\beta,\cdot)\wedgedot (\C\alpha_{\sigma(1)}, \dots, \C\alpha_{\sigma(j-k)}).
\end{split}
\end{align*}
The last correction term here simply reads $(h^*c)_p(\ul\beta)=c_p(h\ul\beta).$
\end{definition}
That the correction terms of $h^*$ indeed satisfy the Leibniz identity is a direct consequence of Theorem \ref{thm:derivation_hor_proj}. One can also show well-definedness of $h^*$ directly, using Lemma \ref{lem:wedgedot} (ii) and (v).
\begin{example}
\label{ex:low_levels_h}
We now write down the formula for the horizontal projection for low levels $p$. The most important case is $p=1$, where the above definition reads
\begin{align}
\begin{split}
\label{eq:h_p=1}
(h^*c)_0(\alpha)=c_0(\alpha)-c_1\wedgedot \C\alpha,\quad(h^*c)_1(\beta)=c_1(h\beta).
\end{split}
\end{align}
In other words, $(\alpha\mapsto c_1\wedgedot \C\alpha,\beta\mapsto c_1(v\beta))$ is the vertical part of $c$. We remark that by definition of $\wedgedot$, the second term in the leading coefficient reads
\[
(c_1\wedgedot\C\alpha)(X_1,\dots,X_q)=\textstyle\sum_i(-1)^{i+1}c_1(\C\alpha(X_i))(X_1,\dots,\widehat{X_i},\dots,X_q).
\]
At level $p=2$, we obtain
\begin{align*}
(h^*c)_0(\alpha_1,\alpha_2)&=c_0(\alpha_1,\alpha_2)-(c_1(\alpha_2)\wedgedot \C\alpha_1-c_1(\alpha_1)\wedgedot \C\alpha_2)+c_2\wedgedot (\C\alpha_1,\C\alpha_2),\hspace{0.3em}\\
(h^*c)_1(\alpha\|\beta)&=c_1(\alpha\| h\beta)-c_2(h\beta,\cdot)\wedgedot\C\alpha,\\
(h^*c)_2(\beta_1,\beta_2)&=c_2(h\beta_1,h\beta_2).
\end{align*}
We observe that the $k$-th correction term $(h^*c)_k$ of the horizontal projection of a Weil cochain $c$ contains all the higher correction terms $c_{\ell}$, for $\ell\geq k$.
\end{example}
\begin{remark}
\label{rem:wedges}
The pairing \eqref{eq:pairing} should be regarded as follows. Restricting any one of the symmetric arguments of $\gamma\in \Omega^k(M;S^j(A^*)\otimes \frak k)$ to $\frak k$, we obtain a tensor from $\Omega^k(M;S^{j-1}(A^*)\otimes \End\frak k)$, which will again for simplicity just be denoted $\gamma$. The pairing $\wedgedot$ is then related to the following natural pairing: for any vector bundle $V\ra M$, we can define
\begin{align*}
  &\wedge\colon\Omega^k(M;\End V)\times\Omega^\ell(M;V)\rightarrow \Omega^{k+\ell}(M;V),
\\
  &(\gamma\wedge\vartheta)(X_1,\dots,X_{k+\ell})=\frac 1{k!\ell!}\sum_{\mathclap{\hspace{1em}\sigma\in S_{k+\ell}}}\sgn (\sigma)\,\gamma(X_{\sigma(1)},\dots,X_{\sigma(k)})\cdot\vartheta(X_{\sigma(k+1)},\dots,X_{\sigma(k+\ell)}),
\end{align*}
and by an easy combinatorial exercise, the two pairings are related by
\begin{align}
\label{eq:wedge_wedgedot}
\gamma\wedgedot\vartheta=(-1)^{k\ell}\gamma\wedge \vartheta,
\end{align}
where we have omitted the arguments $\beta_1,\dots,\beta_{j-1}$. The usual wedge $\wedge$ is, in a way, more natural to work with, but it would introduce some additional complications regarding the signs if we used it in place of $\wedgedot$ in the definition of $h^*$. In the following lemma, we use this relation to infer some important properties of $\wedgedot$.
\end{remark}

\begin{samepage}
  \begin{lemma}
  \label{lem:wedgedot}
  The pairing \eqref{eq:pairing} satisfies the following properties, for any $\gamma\in \Omega^k(M;S^j(A^*)\otimes \frak k)$.
  \begin{enumerate}[label={(\roman*)}]
  \item $\gamma\wedgedot (\vartheta_1,\dots,\vartheta_j)$ is alternating and $C^\infty(M)$-multilinear in $\vartheta_1,\dots,\vartheta_j\in\Omega^1(M;\frak k)$.
  \item For a simple tensor $\vartheta\otimes\xi$, where $\vartheta\in\Omega^1(M)$ and $\xi\in\Gamma(\frak k)$, there holds \[\gamma\wedgedot (\vartheta\otimes\xi)=\vartheta\wedge \gamma(\xi,\cdot).\]
  More generally, if 1-forms $\vartheta_2,\dots,\vartheta_\ell\in\Omega^1(M;\frak k)$ are given, then
  \[
  \gamma\wedgedot(\vartheta\otimes\xi, \vartheta_2,\dots,\vartheta_\ell)=\vartheta\wedge\big(\gamma(\xi,\cdot)\wedgedot (\vartheta_2,\dots,\vartheta_\ell)\big).
  \]
  \item For any $\vartheta\in\Omega^\ell(M;\frak k)$ and $X\in\vf(M)$, there holds
  \[\iota_X(\gamma\wedgedot\vartheta)=\gamma\wedgedot(\iota_X\vartheta)+(-1)^\ell (\iota_X\gamma)\wedgedot \vartheta.\]
  \item For any $\vartheta\in\Omega^\ell(M;\frak k)$ and $\alpha\in\Gamma(A)$, there holds
  \[
  \L^A_\alpha(\gamma\wedgedot\vartheta)
  =(\L^A_\alpha\gamma)\wedgedot\vartheta+\gamma\wedgedot(\L^A_\alpha\vartheta). 
  \]
  \item If $\gamma=c_j$ is the $j$-th correction term of $c=(c_0,\dots,c_p)\in W^{p,q}(A;\frak k)$, then there holds
  \begin{align*}
  c_j(f\alpha_1,\alpha_2,&\dots,\alpha_{p-j})\wedgedot (\vartheta_1,\dots,\vartheta_i)\\
  &=f c_j(\ul\alpha)\wedgedot\ul\vartheta+(-1)^i\d f\wedge c_{j+1}(\alpha_2,\dots,\alpha_{p-j}\|\alpha_1,\cdot)\wedgedot \ul\vartheta,
  \end{align*}
  for any sections $\alpha_1,\dots,\alpha_{p-j}\in\Gamma(A)$, forms $\vartheta_1,\dots,\vartheta_i\in\Omega^1(M;\frak k)$ where $i\leq j$, and any function $f\in C^\infty(M)$. 
  \end{enumerate}
  \end{lemma}
\end{samepage}
\begin{proof}
Properties (i) and (ii) are clear from the definition. Property (iii) is easily inferred from relation \eqref{eq:wedge_wedgedot} by noting that the identity
\[
\iota_X(T\wedge\vartheta)=(\iota_X T)\wedge\vartheta+(-1)^k T\wedge(\iota_X\vartheta)
\]
holds for any $T\in\Omega^k(M;\End V)$ and $\vartheta\in\Omega^\ell(M;V)$, for any given vector bundle $V\ra M$. The statement (iv) is a consequence 
of the following general fact:  if an $A$-connection $\nabla^A$ on $V$ is given, then $\L^A$ is distributive over $\wedge$, that is, 
\begin{align}
\label{eq:L_A_wedge}
  \L^A_\alpha(T\wedge\vartheta)=(\L^A_\alpha T)\wedge\vartheta + T\wedge (\L^A_\alpha \vartheta),
\end{align}
where $\L^A_\alpha T$ is defined, as usual, by the chain rule $(\L^A_\alpha T)\cdot\xi=\L^A_\alpha(T\cdot \xi)-T\cdot \nabla_\alpha^A\xi$, for any $\xi\in\Gamma(V)$. It then follows from relation \eqref{eq:wedge_wedgedot} that $\L^A$ is distributive over $\wedgedot$ as well. Finally, item (v) is a straightforward consequence of item (ii), combined with the Leibniz rule for $c$.
\end{proof}
We now arrive to the crucial property of the horizontal projection $h^*$.
\begin{proposition}
\label{prop:h_cochain}
Let $A$ be a Lie algebroid with an IM connection $(\C,v)\in\A(A;\frak k)$ for a bundle of ideals $\frak k$. The horizontal projection of Weil cochains is a cochain map, that is,
\[
h^*\delta=\delta h^*.
\]
\end{proposition}
This is a direct consequence of Proposition \ref{prop:h_cochain_ext} and Theorem \ref{thm:derivation_hor_proj} since the two models are isomorphic, however, the last lemma enables us to prove it directly. The computation for the general case is tedious, so we only provide a direct proof for the level $p=1$, revealing the key ingredient is infinitesimal multiplicativity of $(\C,v)$.
\begin{proof}
Let $c=(c_0,c_1)\in W^{1,q}(A;\frak k)$ be a Weil 1-cochain; we begin with the leading term. Using the definition \eqref{eq:delta_p=1} of $\delta$ together with the definition of $h^*$, we obtain
\begin{align*}
\begin{split}
(h^*\delta c)_0(\alpha_1,\alpha_2)&=(\delta c)_0(\alpha_1,\alpha_2)-\big((\delta c)_1(\alpha_2\|\cdot)\wedgedot \C\alpha_1- (\delta c)_1(\alpha_1\|\cdot)\wedgedot \C\alpha_2\big)+\cancel{(\delta c)_2\wedgedot (\C\alpha_1,\C\alpha_2)}\\
&=\L^A_{\alpha_1}c_0(\alpha_2)-\L^A_{\alpha_2}c_0(\alpha_1)-c_0[\alpha_1,\alpha_2]-(-\L^A_{\alpha_2}c_1\wedgedot \C\alpha_1+\L^A_{\alpha_1}c_1\wedgedot \C\alpha_2),
\end{split}
\end{align*}
where we have used that $\frak k\subset\ker\rho$. We now apply Lemma \ref{lem:wedgedot} (iv) on the last two terms, and combine with the first two terms to obtain
\begin{align*}
(h^*\delta c)_0&(\alpha_1,\alpha_2)=\L^A_{\alpha_1}(h^*c)_0(\alpha_2)-\L^A_{\alpha_2}(h^*c)_0(\alpha_1)-\big(c_0[\alpha_1,\alpha_2]-c_1\wedgedot(\L^A_{\alpha_1}\C\alpha_2-\L^A_{\alpha_2}\C\alpha_1)\big).
\end{align*}
Using the condition \eqref{eq:c1} for the IM connection $(\C,v)$ shows this is equal to $(\delta h^* c)_0(\alpha_1,\alpha_2)$. We next inspect the first correction term:
\begin{align*}
(h^*\delta c)_1(\alpha\|\beta)&=(\delta c)_1(\alpha\|h\beta)-(\delta c)_2(h\beta,\cdot)\wedgedot \C\alpha\\
&=-\L^A_\alpha (c_1 (h\beta))+c_1[\alpha,h\beta]+\iota_{\rho(\beta)} c_0(\alpha)+(\iota_{\rho(\beta)}c_1)\wedgedot \C\alpha,
\end{align*}
where we have observed that $\rho (h\beta)=\rho(\beta)$. On the other hand,
\begin{align*}
(\delta h^* c)_1(\alpha\|\beta)&=-\L^A_\alpha((h^*c)_1(\beta))+(h^*c)_1[\alpha,\beta]+\iota_{\rho(\beta)}(h^* c)_0(\alpha)\\
&=-\L^A_\alpha (c_1 (h\beta))+c_1 (h[\alpha,\beta])+\iota_{\rho(\beta)}c_0(\alpha)-\iota_{\rho(\beta)}(c_1\wedgedot \C\alpha).
\end{align*}
Using the condition \eqref{eq:c2} on the IM connection, we get 
\[
h[\alpha,\beta]=[\alpha,h\beta]+\iota_{\rho(\beta)}\C\alpha,
\]
and now use Lemma \ref{lem:wedgedot} (iii) to conclude $(h^* \delta c)_1=(\delta h^* c)_1$. The equality for the second correction term is simple and left to the reader.
\end{proof}

\subsection{Horizontal exterior covariant derivative}

\begin{definition}
    Let $A$ be a Lie algebroid with an IM connection $(\C,v)\in\A(A;\frak k)$. The \textit{horizontal exterior covariant derivative} of Weil cochains is defined as the map
    \begin{align*}
    \D{}^{(\C,v)}=h^*\d{}^\nabla\colon W^{p,q}(A;\frak k)\ra W^{p,q+1}(A;\frak k)^\Hor,
    \end{align*}
    where $\nabla=\C|_\frak k$ is the induced connection on $\frak k$.
    \end{definition}
    \begin{example}
    The most important case is $p=1$, where equations \eqref{eq:ext_cov_der_p=1} and \eqref{eq:h_p=1} yield
    \begin{align}
    \label{eq:D_inf}
    \begin{split}
      (\D{}^{(\C,v)}c)_0(\alpha)&=\d{}^\nabla c_0(\alpha)-(\d{}^\nabla c)_1\wedgedot \C\alpha,\\
      (\D{}^{(\C,v)}c)_1(\beta)&=(\d{}^\nabla c)_1(h\beta),
    \end{split}
    \end{align}
    where we recall that $(\d{}^\nabla c)_1(\beta)=c_0(\beta)-\d{}^\nabla (c_1(\beta))$ for any $\beta\in\Gamma(A)$, hence the second term of the leading coefficient, for any vectors $X_i\in \vf(M)$, reads
    \begin{align*}
      \big((\d{}^\nabla c)_1\wedgedot \C\alpha\big)(X_0,\dots,X_q)=\textstyle\sum_i (-1)^i \big(c_0(\C\alpha (X_i))-\d{}^\nabla c_1(\C\alpha (X_i))\big)(X_0,\dots,\widehat{X_i},\dots,X_q).
    \end{align*}
    \end{example}
    We now prove the infinitesimal analogue of Theorem \ref{thm:deltaD}.
    \begin{theorem}
    \label{thm:deltaD_inf}
    Let $A$ be a Lie algebroid with an IM connection $(\C,v)\in\A(A;\frak k)$. The horizontal exterior covariant derivative is a cochain map, that is,
    \begin{align*}
    \delta\D{}^{(\C,v)}=\D{}^{(\C,v)}\delta.
    \end{align*}
    In particular, $\D{}^{(\C,v)}$ maps IM forms to IM forms. 
    \end{theorem}
    In other words, an IM connection yields the columns of a curved double complex, depicted in the diagram below, with the feature that $\D{}^{(\C,v)}$ does not square to zero unless $(\C,v)$ is flat, which we will see is equivalent to $U=0$ and $R^\nabla=0$ in the next section.
    \begin{align}
    \label{eq:weil_ideals}
    \begin{tikzcd}[ampersand replacement=\&, column sep=large,row sep=large]
        {\Omega^{q+1}(M;\frak k)} \& {W^{1,q}(A;\frak k)} \& {W^{2,q+1}(A;\frak k)} \& \cdots \\
        {\Omega^q(M;\frak k)} \& {W^{1,q}(A;\frak k)} \& {W^{2,q}(A;\frak k)} \& \cdots
        \arrow["{\delta}", from=1-1, to=1-2]
        \arrow["{\delta}", from=1-2, to=1-3]
        \arrow["{\delta}", from=1-3, to=1-4]
        \arrow["{\d{}^\nabla}", from=2-1, to=1-1]
        \arrow["{\delta}", from=2-1, to=2-2]
        \arrow["{\D{}^{(\C,v)}}", from=2-2, to=1-2]
        \arrow["{\delta}", from=2-2, to=2-3]
        \arrow["{\D{}^{(\C,v)}}", from=2-3, to=1-3]
        \arrow["{\delta}", from=2-3, to=2-4]
    \end{tikzcd}
    \end{align}
    \begin{proof}
    First observe that by Proposition \ref{prop:h_cochain}, $[\D{}^{(\C,v)},\delta]=h^*[\d{}^\nabla,\delta]$. As in the case of groupoids, we will use the explicit expression for $[\d{}^\nabla,\delta]$ from Lemma \ref{lem:A_invariant}, and to do so, we first have to compute the invariance form $(T,\theta)$ of connection $\nabla=\C|_{\frak k}$, given by \eqref{eq:invariance_form}. Equation \eqref{eq:diff_a_conn} already states that the tensor $\theta\colon A\ra \End\frak k$ reads
    \[
    \theta(\alpha)=[v\alpha,\cdot],
    \]
    and on the other hand, the map $T\colon\Gamma(A)\ra \Omega^1(M;\End \frak k)$ equals
    \begin{align*}
    T(\alpha)\cdot\xi&=\nabla(\theta(\alpha)\cdot\xi)-\theta(\alpha)\cdot\nabla\xi-\iota_{\rho(\alpha)}R^\nabla=\nabla[v\alpha,\xi]-[v\alpha,\nabla\xi]-[U(h\alpha),\xi]\\&=[\nabla (v\alpha)-U(h\alpha),\xi]=[\C\alpha,\xi],
    \end{align*}
    where we have used the properties \eqref{eq:s1} and \eqref{eq:s2} of the IM connection $(\C,v)$. 
    Hence, we need to show $h^*(T,\theta)=0$. This is a simple computation: for the leading term,
    \begin{align*}
      (h^*(T,\theta))_0(\alpha)(X)\xi=T(\alpha)(X)\xi-\theta(\C\alpha(X))\xi=[\C\alpha(X),\xi]-[\C\alpha(X),\xi]=0,
    \end{align*}
    and for the symbol, $(h^*(T,\theta))_1(\alpha)=\theta(h\alpha)=0$.
    \end{proof}

\subsection{Curvature}
\label{sec:im_curvature}
From now on, we will sometimes use the simplified notation 
\[c(\alpha)=(c_0(\alpha),c_1(\alpha))\in \Omega^q(M;\frak k)\times \Omega^{q-1}(M;\frak k),\] for any Weil cochain $c=(c_0,c_1)\in W^{1,q}(A;\frak k)$, where $\alpha\in \Gamma(A)$.
\begin{definition}
  Given a Lie algebroid $A$ with an IM connection $(\C,v)\in\A(A;\frak k)$, its \textit{curvature} is the horizontal IM form $\Omega^{(\C,v)}\in\Omega^2_{im}(A;\frak k)^\Hor$, defined as 
  \[\Omega^{(\C,v)}=\D{}^{(\C,v)}(\C,v).\]
  \end{definition} 
  To obtain an explicit formula for the curvature, simply note that $(\d{}^\nabla(\C,v))_1=\C-\d{}^\nabla v$ restricts on the bundle of ideals $\frak k$ to zero, hence we get
  \begin{align}
  \begin{split}
    \label{eq:im_curv}
    \Omega^{(\C,v)}\alpha&=(\d{}^\nabla\C\alpha,\C h\alpha)=(R^\nabla\cdot v\alpha-\d{}^\nabla U(h\alpha),-U(h\alpha))
  \end{split}
  \end{align}
  for any $\alpha\in\Gamma(A)$, where we have used the expression \eqref{eq:split_C} for $\C$ in terms of $\nabla$ and $U$. We observe that this coincides with \cite{mec}*{equation 5.9}, where a different approach for obtaining the formula for the curvature was used.
  
  \begin{remark}
  By definition, all the compatibility conditions for $\Omega^{(\C,v)}$ must follow from those for $(\C,v)$. More precisely, one can show that the condition \eqref{eq:c2} for $\Omega^{(\C,v)}$ is equivalent to \eqref{eq:s3} for $(\C,v)$; moreover, the condition \eqref{eq:c1} for $\Omega^{(\C,v)}$ may be split into three pieces, by considering $\Omega^{(\C,v)}[\alpha,\beta]$ for $\alpha$ and $\beta$ either horizontal or vertical sections:
  \begin{itemize}[]
  \item When both are vertical, the condition translates to: the curvature tensor $R^\nabla$ takes values in the bundle $\Der(\frak k)\subset\End(\frak k)$ consisting of fibrewise derivations of the bracket on $\frak k$. In other words, for any $\xi,\eta\in\frak k$ there holds
  \begin{align}
  \label{eq:R_derivation}
  R^\nabla\cdot [\xi,\eta]=[R^\nabla\cdot\xi,\eta]+[\xi,R^\nabla\cdot\eta].
  \end{align}
  This is a direct consequence of condition \eqref{eq:s1}.
  \item  When one is horizontal and one vertical, the condition reads $\iota_{\rho\alpha}\d{}^{\nabla^{\End\frak k}} R^\nabla=0$, where $\nabla^{\End\frak k}$ denotes the induced connection on $\End \frak k\ra M$. However, the Bianchi identity $\smash{\d{}^{\nabla^{\End\frak k}}} R^\nabla=0$ already holds for any connection $\nabla$, so this case is trivial.
  \item When both are horizontal, one straightforwardly checks that the obtained condition is equivalent to the equation one gets when applying  $\d{}^\nabla$ to both sides of identity \eqref{eq:s3}.
  \end{itemize}
  \end{remark}
  
  Since we have discovered the horizontal exterior covariant derivative of IM forms in equation \eqref{eq:D_inf}, it now becomes easy to establish the infinitesimal version of the Bianchi identity \eqref{eq:bianchi} for the curvature of an IM connection.
  \begin{theorem}[Infinitesimal Bianchi identity] 
    \label{thm:bianchi_inf}
  The horizontal exterior covariant derivative of the curvature of any IM connection $(\C,v)\in \A(A;\frak k)$ on a Lie algebroid $A$ vanishes:
  \begin{align}
  \label{eq:bianchi_inf}
    \D{}^{(\C,v)}\Omega^{(\C,v)}=0.
  \end{align}
  \end{theorem}
  \begin{proof}
  Using the explicit expression \eqref{eq:im_curv} for the curvature, we first compute:
  \begin{align*}
    (\d{}^\nabla\Omega^{(\C,v)})_1(\alpha)=\d{}^\nabla\C\alpha-\d{}^\nabla\C (h\alpha)=\d{}^\nabla \C(v\alpha)=R^\nabla\cdot v\alpha,
  \end{align*}
  for any $\alpha\in \Gamma(A)$. Hence,
  \begin{align*}
    \D{}^{(\C,v)}\Omega^{(\C,v)}\alpha=(\underbrace{\d{}^\nabla\d{}^\nabla \C\alpha}_{R^\nabla\wedge\:\C\alpha}-R^\nabla\wedgedot\C\alpha,R^\nabla\cdot v(h\alpha))=0,
  \end{align*}
  where we have used the relation \eqref{eq:wedge_wedgedot} between $\dot\wedge$ and $\wedge$.
  \end{proof}

  \subsection{Van Est map and the horizontal exterior covariant derivative}
  \label{sec:van_est_D}
  In this section we inspect the relationship between the van Est map and the horizontal exterior covariant derivative in the global and infinitesimal realm. The following theorem states they commute at the level of multiplicative forms, and the proof will also reveal that they do not commute on general  cochains.
  \begin{theorem}
    \label{thm:van_est_D}
    Let $\omega\in\A(G;\frak k)$ be a multiplicative Ehresmann connection on a Lie groupoid $G$. Let $A$ be its Lie algebroid, endowed with the IM connection $(\C,v)=\ve(\omega)$. The van Est map commutes with the horizontal exterior covariant derivatives at the level of multiplicative forms, that is, the following diagram commutes.
  
  \begin{align}
    \label{eq:square_D}
      \begin{tikzcd}[row sep=large,column sep=large,ampersand replacement=\&]
      {\Omega^\bullet_m(G;\frak k)} \& {\Omega^{\bullet+1}_m(G;\frak k)} \\
      {\Omega^\bullet_{im}(A;\frak k)} \& {\Omega^{\bullet+1}_{im}(A;\frak k)}
      \arrow["{\D{}^\omega}", from=1-1, to=1-2]
      \arrow["{\ve}"', from=1-1, to=2-1]
      \arrow["{\ve}", from=1-2, to=2-2]
      \arrow["{\D{}^{(\C,v)}}"', from=2-1, to=2-2]
    \end{tikzcd}
    \end{align}
  \end{theorem}
  \begin{lemma}
    \label{lem:cochain_homotopy}
    Let $\omega\in \A(G;\frak k)$ be a multiplicative connection on a Lie groupoid $G\rra M$. The commutator $[\ve,h^*]\colon \Omega^{p,q}(G;\frak k)\ra W^{p,q}(A;\frak k)^\Hor$ at level $p=1$ has a vanishing symbol, and its leading term equals
    \begin{align*}
      \ve (h^*\eta)_0(\alpha)(X_i)_i-(h^*\ve \eta)_0(\alpha)(X_i)_i=\deriv \lambda 0 (v^*\delta \eta)_{(g_\lambda,g_{\lambda}^{\smash{-1}})}\big(X^\lambda_{i}, \d(\inv)h X^\lambda_{i}\big)_i
      \end{align*}
      for any $\eta\in\Omega^q(G;s^*\frak k)$, $\alpha\in \Gamma(A)$ and $X_i\in T_x M$. Here, we have denoted $g_\lambda=\phi^{\alpha^L}_\lambda(1_x)$ and $X_i^\lambda=\d(\phi^{\alpha^L}_{\lambda_{\phantom L}})\d u(X_i)$, and $v^*=\id-h^*$ denotes the vertical projection of differential forms. 
  \end{lemma}
  \begin{remark}
    \label{rem:cochain_homotopy}
    The lemma suggests that $\ve$ and $h^*$ commute only up to a cochain homotopy $\Psi$, as portrayed in  the (noncommutative) diagram below. Specifically, the formula for $\Psi^2$ can be read from the equation above, whereas $\Psi^1=0$. However, this more general statement will not be further explored here.
  \[\begin{tikzcd}[row sep=large, column sep=large]
    \cdots & {\Omega^{p-1,q}(G;\frak k)} & {\Omega^{p,q}(G;\frak k)} & {\Omega^{p+1,q}(G;\frak k)} & \cdots \\
    \cdots & {W^{p-1,q}(G;\frak k)} & {W^{p,q}(G;\frak k)} & {W^{p+1,q}(G;\frak k)} & \cdots
    \arrow[from=1-1, to=1-2]
    \arrow["{\delta^{p-1}}", from=1-2, to=1-3]
    \arrow["{\delta^p}", from=1-3, to=1-4]
    \arrow["{\Psi^{p}}"{description}, from=1-3, to=2-2]
    \arrow["{\ve h^*}"', shift right, from=1-3, to=2-3]
    \arrow["{h^*\ve}", shift left, from=1-3, to=2-3]
    \arrow[from=1-4, to=1-5]
    \arrow["{\Psi^{p+1}}"{description}, from=1-4, to=2-3]
    \arrow[from=2-1, to=2-2]
    \arrow["{\delta^{p-1}}"', from=2-2, to=2-3]
    \arrow["{\delta^p}"', from=2-3, to=2-4]
    \arrow[from=2-4, to=2-5]
  \end{tikzcd}\]
  \end{remark}
  \begin{proof}
    We begin with the symbol, where the computation is simple---for any section $\beta\in\Gamma(A)$,
    \begin{align*}
    \ve (h^*\eta)_1(\beta)=J_\beta (h^*\eta)=u^*\iota_{\beta^L} (h^*\eta)=u^* h^*\iota_{h\beta^L}\eta=u^*\iota_{h\beta^L}\eta=(h^*\ve\eta)_1(\beta),
    \end{align*}
    where we have used that $E\subset TG$ is a wide subgroupoid. For the leading term, first observe that the vertical projection $v^*= \id-h^*$ enables us to write
  \begin{align*}
  (R_\alpha h^*\eta)_x(X_i)_i=(R_\alpha \eta)_x(X_i)_i-(R_\alpha v^*\eta)_x(X_i)_i,
  \end{align*}
  for any vectors $X_i\in T_x M$. Observe that 
  \begin{align*}
    (v^*\eta)(X_i)_i&=\eta(vX_i+hX_i)_i-\eta(hX_i)_i\\
    &=\sum\limits_{k=1}^q\sum\limits_{\mathclap{\hspace{2.5em}\sigma\in S_{(k,q-k)}}}\sgn(\sigma)\eta(vX_{\sigma(1)},\dots,vX_{\sigma(k)},hX_{\sigma(k+1)},\dots,h X_{\sigma(q)}),
  \end{align*}
  hence $(R_\alpha v^*\eta)_x(X_i)_i$ may be expressed as a sum of terms of the form (up to a sign)
  \begin{align}
    \label{eq:term_in_Rv}
    \deriv\lambda 0 \Ad_{\phi^{\alpha^L}_\lambda(1_x)}\cdot\eta(vY_1^\lambda,\dots, vY_k^\lambda,hY_{k+1}^\lambda,\dots,hY_{q}^\lambda)
  \end{align}
  for some $k\in\set{1,\dots, q}$, where we are denoting $Y^\lambda_i=X^\lambda_{\sigma(i)}$ for a fixed permutation $\sigma\in S_{(k,q-k)}$, and $X^\lambda_i=\d(\phi^{\alpha^L}_\lambda)\d u(X_i)$. 
  We now use the formula for the simplicial differential,
  \begin{align*}
    \delta\eta_{(g,h)}(v_i,w_i)_i=\eta_h(w_i)_i-\eta_{gh}(\d m_{(g,h)}(v_i,w_i))_i+\Ad_{h^{-1}}\cdot\eta_g(v_i)_i,
  \end{align*}
  with arrows $g=g_\lambda\coloneqq \phi^{\alpha^L}_\lambda(1_x)$, $h=g^{-1}$ and vectors
  \begin{align*}
    v_i=\begin{cases}
      vY^\lambda_i & \text{if }i\leq k,\\
      h Y^\lambda_i & \text{if }i>k,
    \end{cases}
    \qquad\quad
    w_i=\begin{cases}
      0 & \text{if }i\leq k,\\
      \d(\inv) h Y^\lambda_i & \text{if }i>k.
    \end{cases}
  \end{align*}
  Using these vectors allows us to write \eqref{eq:term_in_Rv} as
  \begin{align*}
    \eqref{eq:term_in_Rv}&= \eta_{gh}(\d m_{(g,h)}(v_i,w_i))_i+\delta\eta_{(g,h)}(v_i,w_i)_i\eqqcolon(\star)+(\star\star)
  \end{align*}
  It is not hard to see there holds 
  \begin{align*}
    \d m(v_i,w_i)=\begin{cases}
      \Ad_{g_\lambda}\cdot\omega(Y^\lambda_i) & \text{if }i\leq k,\\
      \d u (Y^0_i) & \text{if }i>k,
    \end{cases}
  \end{align*}
  hence the first term $(\star)$ becomes
  \begin{align*}
    (\star)&=\deriv\lambda 0 \eta_{1_x}\big(\Ad_{g_\lambda}\cdot\omega(Y_1^\lambda),\dots, \Ad_{g_\lambda}\cdot\omega(Y_k^\lambda),\d u (Y^0_{k+1}),\dots, \d u (Y^0_{q})\big)\\
    &=\sum_{i=1}^k\eta_{1_x}\bigg({\omega(\d u(Y^0_1))},\dots,\underbrace{\deriv\lambda 0\Ad_{g_\lambda}\cdot\omega(Y_i^\lambda)}_{\C(\alpha)(Y^0_i)},\dots,{\omega(\d u(Y^0_k))},\d u(Y^0_{k+1}),\dots, \d u(Y^0_{q})\bigg).
  \end{align*}
  Since $u^*\omega=0$ by multiplicativity of the connection $\omega$, the expression above is nonzero only in the case when $k=1$. Assuming for a moment that $\eta$ is multiplicative, we thus obtain
  \begin{align*}
    (R_\alpha v^*\eta)_x(X_i)_i&=\smallderiv\lambda 0\textstyle\sum_i  \phi^{\alpha^L}_\lambda(1_x)\cdot\eta(hX_1^\lambda,\dots, vX_i^\lambda,\dots,hX_{q}^\lambda)\\
    &=\textstyle\sum_i (-1)^{i+1}\eta_{1_x}(\C\alpha(X_i),\d u(X_1),\dots,\widehat{\d u(X_i)},\dots, \d u(X_q))\\
    &=(c_1\wedgedot \C\alpha)(X_i)_i,
  \end{align*}
  where we are denoting by $c_1$ the symbol of $\ve (\eta)$. Since the leading term is $c_0(\alpha)=R_\alpha\eta$, this proves the lemma for the case when $\eta$ is multiplicative. If $\eta$ is not multiplicative, the term $(\star\star)$ becomes
  {\begin{align*}
    (\star\star)=\deriv \lambda 0 \delta\eta_{(g_\lambda,g_{\lambda}^{\smash{-1}})}\big(&(v Y_1^\lambda, 0),\dots ,(v Y_k^\lambda, 0),(h Y^\lambda_{k+1}, \d(\inv)h Y^\lambda_{k+1}),\dots,(h Y^\lambda_{q}, \d(\inv)h Y^\lambda_{q})\big).
  \end{align*}
  }Finally, observe that on $TG^{(2)}$ there holds
  \begin{align*}
  h(Y^\lambda_{i}, \d(\inv)h Y^\lambda_{i})&=(hY^\lambda_{i}, \d(\inv)h Y^\lambda_{i}),\\
  v(Y^\lambda_{i}, \d(\inv)h Y^\lambda_{i})&=(vY^\lambda_i,0).
  \end{align*}
  This establishes the wanted formula for the leading terms. 
  \end{proof}
  \begin{proof}[Proof of Theorem \ref{thm:van_est_D}]
    We need to check $\ve (h^* \d{}^{\nabla}\eta) =h^*\d{}^\nabla \ve(\eta)$ holds for any multiplicative form $\eta\in\Omega^q_m(G;s^*\frak k)$. By the lemma above, we have
    \begin{align*}
      \ve(h^*\d{}^{\nabla}\eta)_0(\alpha)(X_i)&=(h^*\d{}^\nabla\ve \eta)_0(\alpha)(X_i)_i+\deriv \lambda 0 (v^*\delta \d{}^\nabla\eta)_{(g_\lambda,g_{\lambda}^{\smash{-1}})}\big(X^\lambda_{i}, \d(\inv)h X^\lambda_{i}\big)_i
    \end{align*}
    where we have already used on the first term that $\ve$ commutes with $\d{}^\nabla$ on multiplicative forms, as established in Theorem \ref{thm:van_est_G_A}. To show that the second term vanishes, first note that multiplicativity of the form $\eta$ implies
    \[
      v^*\delta\d{}^\nabla\eta=-v^*[\d{}^\nabla,\delta]\eta.
    \]
    As already observed in the lemma, $v^*$ acts on individual arguments of a given form by either the horizontal or the vertical projection, so the second component $\d(\inv)h X^\lambda_{i}$ of any vector is either untouched by $v^*$ or sent to zero. This is important since by the formula for the commutator \eqref{eq:commutator_higher}, the second components are inserted into $\Theta$, and so by the fact that $h^*\Theta=0$ from Corollary \ref{cor:diff_nabla_ts}, the second term on the right-hand side above vanishes. This proves the theorem for leading terms; that the symbols coincide follows directly from the last lemma.
  \end{proof}
  \begin{corollary}
    Let $\omega\in\A(G;\frak k)$ be a multiplicative Ehresmann connection on a Lie groupoid $G$, and let $A$ be its Lie algebroid, endowed with the IM connection $(\C,v)=\ve(\omega)$. The van Est map takes the curvature of $\omega$ to the curvature of $(\C,v)$, that is,
    \[
    \Omega^{(\C,v)}=\ve(\Omega^\omega).
    \]
  \end{corollary}

\subsection{IM connections as \texorpdfstring{$\vb$}{VB}-algebroids}
\label{sec:im_connections_as_vb_algebroids}
In this subsection, we will first see that the viewpoint of IM connections as $\vb$-subalgebroids of $TA$, known from \cite{mec}, is just a manifestation of the model isomorphism \eqref{eq:evaluation_map} between Weil and exterior cochains. Secondly, the main purpose of this section is to derive the desired horizontal projection $h^*$ of Weil cochains, which is done by employing the viewpoint of exterior cochains.

First note that the differential of the projection $\phi\colon A\ra B=A/\frak k$, with $\phi$ viewed as a surjective submersion between smooth manifolds, defines a short exact sequence of $\vb$-algebroids.
\begin{align}
  \label{eq:ses_vb_algebroids}
  \begin{tikzcd}[ampersand replacement=\&, column sep=large]
  	0 \& \ker\d\phi \& TA \& {TB} \& 0 \\
  	0 \& {0_M} \& TM \& TM \& 0
  	\arrow[from=1-1, to=1-2]
  	\arrow[hook, from=1-2, to=1-3]
  	\arrow[Rightarrow,from=1-2, to=2-2]
  	\arrow["{\d\phi}", from=1-3, to=1-4]
  	\arrow[Rightarrow,from=1-3, to=2-3]
  	\arrow[from=1-4, to=1-5]
  	\arrow[Rightarrow,from=1-4, to=2-4]
  	\arrow[from=2-1, to=2-2]
  	\arrow[hook, from=2-2, to=2-3]
  	\arrow[from=2-3, to=2-4]
  	\arrow[from=2-4, to=2-5]
  \end{tikzcd}
\end{align}
As was observed in \cite{mec}*{Lemma 5.2}, there is a canonical identification of $K\coloneqq \ker\d\phi$ with $A\oplus_M\frak k=\pi^*\frak k$, given by the injective map
\begin{align}
  \label{eq:identification_i_K}
  &i\colon A\oplus_M\frak k\ra TA,\quad i(\alpha,\xi)=\deriv\lambda 0 (\alpha+\lambda\xi)
\end{align}
whose image is precisely $K$ due to dimensional reasons and $\frak k=\ker\phi$. It is easy to see that the formula for $i$ can be rewritten in the following way, which will be more useful later:
\begin{align}
  \label{eq:identification_i_plus_tm}
  i(\alpha,\xi)=0_A(\alpha)+_{TM}\deriv\lambda 0 \lambda\xi.
\end{align}

Now, to obtain the correspondence of IM connections $\A(A;\frak k)$ with $\vb$-subalgebroids of $TA$ complementary to $K$, suppose we are given a splitting $v^{TA}\colon TA\ra K$ of the sequence \eqref{eq:ses_vb_algebroids}, covering the identity on $A\Ra M$. Such a splitting can be identified via $i\colon A\oplus_M\frak k\ra K$ with a Lie algebroid morphism $\nu\colon TA\ra A\oplus_M\frak k$ satisfying $\nu\circ i=\id_{\pi^*\frak k}$.
\[\begin{tikzcd}[row sep=small, column sep=small]
	TA && {A\oplus_M\frak k} \\
	& A && A \\
	TM && {0_M} \\
	& M && M
	\arrow["{\nu}", from=1-1, to=1-3]
	\arrow[from=1-1, to=2-2]
	\arrow[Rightarrow, from=1-1, to=3-1]
	\arrow[from=1-3, to=2-4]
	\arrow[Rightarrow, from=1-3, to=3-3]
	\arrow["{\id_A}"{pos=0.2}, from=2-2, to=2-4, crossing over]
	\arrow[Rightarrow, from=2-4, to=4-4]
	\arrow[from=3-1, to=3-3]
	\arrow[from=3-1, to=4-2]
	\arrow[from=3-3, to=4-4]
	\arrow[from=4-2, to=4-4]
  \arrow[Rightarrow, from=2-2, to=4-2, crossing over]
\end{tikzcd}\]
Since it is $C^\infty(A)$-linear, $\nu$ can thus be viewed as an element $\nu\in \Gamma_\ext(\MM_1,\Lambda^1\AA_1^*)$. Using the identification of models \eqref{eq:evaluation_map}, one then obtains an IM form
\[
(\C,v)=\ev(\nu)\in \Omega_{im}^1(A;\frak k),
\]
whose multiplicativity corresponds to $\nu$ being a $\vb$-algebroid morphism, and  the condition $v|_{\frak k}=\id_{\frak k}$ corresponds to $\nu\circ i=\id_{\pi^*\frak k}$. Explicitly, the pair $(\C,v)$ reads
\begin{align}
  \label{eq:C_vb_picture}
  \C(\alpha)(w)&=\pr_{\frak k}\big(\nu(\d\alpha(w))\big),\quad (w\in T_xM)\\
  v(\alpha)&=\pr_{\frak k}(\nu(\alpha_c)),\label{eq:v_vb_picture}
\end{align}
where $\alpha_c\in\Gamma(TM,TA)$ is the core section induced by $\alpha\in\Gamma(A)$. It is implicit here that the vector $\nu(\alpha_c(w))$ is independent of the choice of $w\in T_x M$, since $v^{TA}(0_{TM}(w))$ equals the zero of the core $0_A(0_x)=0_{TM}(0_x)$, so we obtain
\begin{align*}
  v^{TA}(\alpha_c(w))=v^{TA}\Big(0_{TM}(w)+_A\deriv\lambda 0 \lambda \alpha_x\Big)=v^{TA}\Big(\deriv\lambda 0 \lambda \alpha_x\Big).
\end{align*}
In other words,
\begin{align}
  \label{eq:v_core_section}
  v^{TA}\Big(\deriv\lambda 0\lambda\alpha_x\Big)=\deriv\lambda 0 \lambda v(\alpha_x).
\end{align}
Since any exterior cochain $\nu\in\Gamma_\ext(\MM_1,\AA_1^*)$ is determined by its values on the linear and core sections of $\AA_1$, identities \eqref{eq:C_vb_picture} and \eqref{eq:v_vb_picture} also tell us how an IM connection $(\C,v)$ induces $\nu$.
Summing up, any IM connection can alternatively be viewed as a wide $\vb$-subalgebroid $E\subset TA$:
\[\begin{tikzcd}
  E & A \\
  TM & M
  \arrow[from=1-1, to=1-2]
  \arrow[Rightarrow, from=1-1, to=2-1]
  \arrow[Rightarrow, from=1-2, to=2-2]
  \arrow[from=2-1, to=2-2]
\end{tikzcd}\]
The core of $E$ is $H=\ker v$, so we now view $v\colon A\ra \frak k$ as the induced splitting of the core $A\ra M$ of the tangent algebroid $TA\Ra TM$.\footnote{The fact that $\Gamma(A)\ra \Gamma(TM,TA), \alpha\mapsto\alpha_c$ is not a morphism of Lie algebras gives another perspective on why the induced core splitting $v\colon A\ra M$ is, in general, not a splitting of Lie algebroids, but merely of vector bundles.}

\subsubsection{Horizontal projection of exterior cochains}
\label{sec:derivation_horproj_weil}
The purpose of this subsection is to derive the formula for the horizontal projection $h^*$ of Weil cochains, induced by an IM connection (Definition \ref{def:hor_proj}). Let us first identify the notion of horizontality (Definition \ref{defn:horizontal_inf}) for the alternative model of exterior cochains.
\begin{definition}
  Suppose $\frak k\subset A$ is a bundle of ideals on a Lie algebroid $A\Ra M$. A cochain $\omega\in \Gamma_{\ext}(\MM_q,\Lambda^p\AA_q^*)$ is said to be \textit{horizontal}, if for any $\xi\in\Gamma(\frak k)$ and $i=1,\dots,q$, there holds
  \[
  \iota_{\Z_i(\xi)}\omega=0,
  \]
  where $\Z_i(\xi)$ is a core section, see \eqref{eq:linear_and_core_sections}.
  We denote the space of horizontal exterior cochains by 
  \[\Gamma_{\ext}(\MM_q,\Lambda^p\AA_q^*)^\Hor\subset \Gamma_{\ext}(\MM_q,\Lambda^p\AA_q^*).\]
\end{definition}
That this forms a subcomplex (for a fixed $q$) is a consequence of the fact that the evaluation map is an isomorphism of cochain complexes that clearly maps horizontal exterior cochains to horizontal Weil cochains, which themselves form a subcomplex. For clarity, we prove this directly.
\begin{proposition}
  $\Gamma_{\ext}(\MM_q,\Lambda^p\AA_q^*)^\Hor$ is a subcomplex of $\Gamma_{\ext}(\MM_q,\Lambda^p\AA_q^*)$.
\end{proposition}
\begin{proof}
Assuming $\iota_{\Z_i(\xi)}\omega =0$, use the definition \eqref{eq:delta_vb} of $\delta$ to get
\begin{align*}
  (\iota_{\Z_i(\xi)}\delta\omega)(X^1,\dots,X^p)=\L_{\rho_{\AA_q}(\Z_i(\xi))}\omega(X^1,\dots,X^p)+\textstyle\sum_{j=1}^p\omega([\Z_i(\xi),X^j]_{\AA_q},X^1,\dots,\widehat{X^j},\dots,X^p),
\end{align*}
for any sections $X^i\in\Gamma(\MM_q,\AA_q)$. To see the first term vanishes, note that the anchor $\smash{\rho_{\AA_q}}$ is defined componentwise, and the anchor $\rho_{TA}$ is defined as the composition of the canonical involution on $T(TM)$ with $\d\rho\colon TA\ra T(TM)$. Hence, it suffices to compute:
\[
\d\rho\big({\d 0(w)}+_A \smallderiv\lambda 0\lambda\xi\big)=\d\rho(\smallderiv\lambda 0\lambda\xi)=\smallderiv\lambda 0 \lambda\rho(\xi)=0,
\]
since $\frak k\subset\ker\rho$. For the remaining terms, using the fact that the Lie bracket on $\AA_q$ is defined componentwise, together with the definition of the Lie bracket on $TA$, we first use \eqref{eq:bracket_Z_T} to write
\[
[\Z_i(\xi),\Z_j(\alpha)]_{\AA_q}=0,\quad [\Z_i(\xi),\T(\alpha)]_{\AA_q}=\Z_i([\xi,\alpha]),
\]
for any $\alpha\in\Gamma(A)$. Notice that $[\xi,\alpha]\in\Gamma(\frak k)$ on account of $\frak k$ being a bundle of ideals, hence by multilinearity of $\omega$ with respect to $\AA_q\ra \MM_q$ and the fact that $\Gamma(\MM_q,\AA_q)$ is generated by the linear and core sections as a $C^\infty(\MM_q)$-module, the second term also vanishes.
\end{proof}
There is now a straightforward way of defining the horizontal projection of exterior cochains, given an IM connection for $\frak k$.
\begin{definition}
  \label{def:hor_proj_ext}
  Let $E\subset TA$ be an IM connection for a bundle of ideals $\frak k$ on $A\Ra M$, and let us denote by 
  $h\colon TA\ra E\hookrightarrow TA$ 
  the $\vb$-algebroid morphism over $\id_A$ and $\id_{TM}$, 
  \begin{align}
    \label{eq:horizontal_proj}
    h(X)=X-_A v^{TA}(X).
  \end{align}
  The \textit{horizontal projection} of exterior cochains is the map
\begin{align*}
  &h^*\colon \Gamma_{\ext}(\MM_q,\Lambda^p\AA_q^*)\ra \Gamma_{\ext}(\MM_q,\Lambda^p\AA_q^*)^\Hor,\\
  &(h^*\omega)_{(w_1,\dots,w_q,\zeta)}(X^1,\dots,X^p)=\omega_{(w_1,\dots,w_q,\zeta)}(hX^1,\dots,h X^p),
\end{align*}
for any vectors $X^i=(X^i_1,\dots,X^i_q,(a^i,\zeta))$ with $\d\pi(X^i_j)=w_j\in T_x M$ and $a^i\in A_x$. Here, we have denoted by $h\colon \AA_q\ra \AA_q$ the $\vb$-algebroid morphism covering $\id_A$ and $\id_{\MM_q}$, induced by \eqref{eq:horizontal_proj}. Namely, $h\colon \AA_q\ra \AA_q$ is defined on any vector $X^i$ as above by
\[hX^i=(hX_1^i,\dots,hX_q^i,(a^i,\zeta)).\] 
\end{definition}
\begin{remark}
  We will use the same letter $h$ (resp., $v$) for all three horizontal (resp., vertical) projections of vectors in $\AA_q$, $TA$, and $A$, as it will always be contextually clear which one is used.
\end{remark}
What follows is the most important property of $h^*$, the proof of which is now almost trivial. Theorem \ref{thm:derivation_hor_proj} will imply that the same holds for Weil cochains, where proving this is  tedious.
\begin{proposition}
  \label{prop:h_cochain_ext}
  Let $E\subset TA$ be an IM connection for a bundle of ideals $\frak k$ on $A\Ra M$. The horizontal projection of exterior cochains is a cochain map, that is:
  \[\delta h^*=h^*\delta.\]
\end{proposition}
\begin{proof}
  Inspecting the defining equation \eqref{eq:delta_vb} of $\delta$, we observe that this is a direct consequence of $h\colon  \AA_q\ra \AA_q$ being a Lie algebroid morphism, since this just means
  \begin{align*}
    \rho_{\AA_q}(h X)=\rho_{\AA_q}(X),\quad h[X,Y]_{TA}=[hX,hY]_{TA},
  \end{align*}
  for any sections $X,Y\in\Gamma(\MM_q,\AA_q)$.
\end{proof}

We can now start deriving the formula for $h^*$ on Weil cochains. Looking at equation \eqref{eq:def_exterior_weil_generators}, we see that to do so, we need to see how the horizontal projection acts on the generators of the module $\Gamma(\MM_q,\AA_q)$. The following lemma shows how to express the horizontal projection of any generator as a linear combination of the generators, with respect to the structure on $\AA_q\ra\MM_q$.
\begin{lemma}
  \label{lem:h_on_generators}
  Let $E\subset TA$ be an IM connection for a bundle of ideals $\frak k$ on $A$. For any section $\alpha\in\Gamma(A)$, the linear and core sections $\T\alpha,\Z_i\alpha\in\Gamma(\MM_q,\AA_q)$ satisfy:
  \begin{align*}
    h(\T\alpha)(w_1,\dots,w_q,\zeta)&=\T\alpha(w_1,\dots,w_q,\zeta)-_{\MM_q}\textstyle\sum_i^{\MM_q}\Z_i(\C\alpha(w_i))(w_1,\dots,w_q,\zeta),\\
    h(\Z_i\alpha)&=\Z_i(h\alpha),
  \end{align*}
  for any vectors $w_j\in T_x M$ and $\zeta\in V_x^*$. Here, the superscript on the sum indicates that the summation is with respect to $\AA_q\ra {\MM_q}$.
\end{lemma}
\begin{proof}
  The formula for the core sections is clear from the definition of $\Z_i\alpha$ and equation \eqref{eq:v_core_section}. For the linear sections, we first write
  \begin{align*}
    h(\T\alpha)(w_1,\dots,w_q,\zeta)&=(h\d\alpha(w_1),\dots,h\d\alpha(w_q),\chi_\alpha(\zeta)),
  \end{align*}
  so we have to compute $h\d\alpha(w_i)$ for any $w_i\in T_xM$. Observe there holds
  \begin{align}
    h\d\alpha(w_i)&=\d\alpha(w_i)-_A v\d\alpha(w_i)\nonumber\\
    &=\big({\d\alpha(w_i)}+_{TM} 0_{TM}(w_i)\big)-_A \big(0_A(\alpha_x)+_{TM}\smallderiv\lambda0 \lambda \C(\alpha)(w_i)\big)\nonumber\\
    &=\big({\d\alpha(w_i)}-_{A} 0_{A}(\alpha_x)\big)+_{TM}\big(0_{TM}(w_i)-_A \smallderiv\lambda0 \lambda \C(\alpha)(w_i)\big)\nonumber\\
    &=\d\alpha(w_i)-_{TM}\big(0_{TM}(w_i)+_A \smallderiv\lambda0 \lambda \C(\alpha)(w_i)\big)\label{eq:horizontal_proj_dalpha_TM}
  \end{align}
  where we have used identities \eqref{eq:identification_i_plus_tm} and \eqref{eq:C_vb_picture} on the second equality and compatibility of the two vector bundle structures on $TA$ on the third. On the fourth, we have used that the two bundle structures coincide on the core $A\ra M$, so in particular, the two scalar multiplications coincide on the vector $\smallderiv\lambda 0 \lambda\C(\alpha)(w_i)$. Now, since the addition (over $\MM_q$) is componentwise, we obtain
  {\begin{align*}
    &h(\T\alpha)(w_1,\dots,w_q,\zeta)\\
    &=\T\alpha(w_1,\dots,w_q,\zeta)
    -_{\MM_q}\big(0_{TM}(w_1)+_A \smallderiv\lambda0 \lambda \C(\alpha)(w_1),\dots,0_{TM}(w_q)+_A \smallderiv\lambda0 \lambda \C(\alpha)(w_q), 0_\zeta \big)\\
    &=\T\alpha(w_1,\dots,w_q,\zeta)
    -_{\MM_q}\textstyle\sum_i^{\MM_q}\big(0_{TM}(w_1),\dots, 0_{TM}(w_i)+_A \smallderiv\lambda0 \lambda \C(\alpha)(w_i),\dots, 0_{TM}(w_q),0_\zeta\big),
  \end{align*}
  }from which the claimed formula for linear sections follows.
\end{proof}
\begin{remark}
  Intuitively, equation \eqref{eq:horizontal_proj_dalpha_TM} should be seen as a formula for how to horizontally project $\d\alpha$ within the fibre of $TA\ra TM$ instead of $TA\ra A$. As a consequence, it tells us how to horizontally project $\T\alpha$ within the fibre of $\AA_q\ra \MM_q$ instead of $\AA_q\ra A$.
  Henceforth, we will denote the \textit{$\MM_q$-vertical component} of $\T\alpha$ by
  \begin{align*}
    u(\T\alpha)\in\Gamma(\MM_q,\AA_q),\quad u(\T\alpha)(w_1,\dots,w_q,\zeta)=\textstyle\sum_i^{\MM_q}\Z_i(\C\alpha(w_i))(w_1,\dots,w_q,\zeta),
  \end{align*}
  for any vectors $w_i\in T_x M$ and $\zeta\in V^*_x$.
\end{remark}
  To obtain some insight for deriving the general formula for $h^*$ on Weil cochains, let us first deal with the simplest nontrivial case: $p=1$. By the last lemma, for any $\omega\in \Gamma_\ext(\MM_q,\Lambda^1\AA_q^*)$,
  \begin{align*}
    \tilde c_0(h^*\omega)(\alpha)=\omega(h\T\alpha)=\omega(\T\alpha)-\omega(u\T\alpha),
  \end{align*}
  where the first term is just $\tilde c_0(\omega)(\alpha)$. The second term reads
  \begin{align}
    \omega(u\T\alpha)(w_1,\dots,w_q,\zeta)&=\textstyle\sum_i\omega\big(\Z_i(\C\alpha(w_i))\big)(w_1,\dots,w_q,\zeta)\nonumber\\
    &=\textstyle\sum_i (-1)^{i+1}\omega\big(\Z_1(\C\alpha(w_i))\big)(w_i,w_1,\dots,\widehat{w_i},\dots, w_q,\zeta)\label{eq:zi_expression_minus}\\
    &=\tfrac 1{(q-1)!}\textstyle\sum_{\sigma\in S_q} \sgn(\sigma)\omega\big(\Z_1(\C\alpha(w_{\sigma(1)}))\big)(w_{\sigma(1)},\dots, w_{\sigma(q)},\zeta),\label{eq:zi_to_z1}
  \end{align}
  where we have used the skew-symmetry of $\omega$ with respect to $\AA_q\ra A$ in the second equality. In the last line, we have merely rewritten the expression \eqref{eq:zi_expression_minus} with permutations because this form will be useful in the general proof. In any case, by Remark \ref{rem:any_vectors_ck}, the line \eqref{eq:zi_expression_minus} equals 
  \[
    \omega(u\T\alpha)(w_1,\dots,w_q,\zeta)=\textstyle\sum_i (-1)^{i+1}c_1(\omega)(\C\alpha(w_i))(w_1,\dots,\widehat{w_i},\dots,w_q,\zeta),
  \]
  which finally yields the formula \eqref{eq:h_p=1}, that is, \[c_0(h^*\omega)(\alpha)=c_0(\omega)(\alpha)-c_1(\omega)\wedgedot\C\alpha.\] At the level of symbols, $c_1(h^*\omega)(\beta)=c_1(\omega)(h\beta)$ is also clear from the last lemma. We now provide the desired derivation for $h^*$ on Weil cochains for the general case $p\geq 1$, preliminarily noting that exactly the same ideas will be used as above, though the procedure will be combinatorially heavier.
\begin{theorem}
  \label{thm:derivation_hor_proj}
  Let $E\subset TA$ be an IM connection for a bundle of ideals $\frak k$ on $A$. The model isomorphism $\ev\colon \Gamma_\ext(\MM_q,\Lambda^p\AA_q^*)\ra W^{p,q}(A;\frak k)$ commutes with the horizontal projections,
  \begin{align*}
    \ev\circ h^*=h^*\circ \ev,
  \end{align*}
  where $h^*$ on the left and the right side correspond to Definitions \ref{def:hor_proj_ext} and \ref{def:hor_proj}, respectively.
\end{theorem}
\begin{proof}
  Let us first deal with the leading term. For $\omega\in\Gamma_\ext(\MM_q,\Lambda^p\AA_q^*)$, we begin by computing
\begin{align}
  \label{eq:grouped_rearranged}
\begin{split}
  &\tilde c_0(h^*\omega)(\alpha_1,\dots,\alpha_p)=\omega(h\T\alpha_1,\dots,h\T\alpha_p)=\omega\big(\T\alpha_1-_{\MM_q}u\T\alpha_1,\dots,\T\alpha_p-_{\MM_q}u\T\alpha_p\big)\\
  &=\textstyle\sum_{j=0}^p(-1)^j\textstyle\sum\limits_{\mathclap{\qquad\sigma\in S_{(j,p-j)}}}\sgn(\sigma)\,\omega(u\T\alpha_{\sigma(1)},\dots, u\T\alpha_{\sigma(j)},\T\alpha_{\sigma(j+1)},\dots,\T\alpha_{\sigma(p)}),
\end{split}
\end{align}
  where we have used multilinearity of $\omega$ with respect to $\AA_q\ra\MM_q$ and grouped the terms with the same number of $\MM_q$-vertical arguments. The arguments have also been rearranged using skew-symmetry with respect to $\AA_q\ra\MM_q$, so that the $\MM_q$-vertical arguments appear first. To establish the theorem at the level of leading terms, we thus have to show
  \begin{align}
    \label{eq:omega_cj_wedgedot}
    \omega\big(u\T\alpha_1,\dots,u\T\alpha_j,\T\alpha_{j+1},\dots,\T\alpha_p\big)=c_j(\omega)(\alpha_{j+1},\dots,\alpha_p)\wedgedot (\C\alpha_1,\dots,\C\alpha_j).
  \end{align}
  The plan now is to transform each of the sections $u\T\alpha_i$ on the left-hand side into a core section of $\AA_q\Ra\MM_q$. We do this step by step, starting with $u\T\alpha_1$ and proceeding towards the right. For any vectors $w_1,\dots,w_q\in T_x M$ and $\zeta\in V^*_x$, we have
  \begin{align}
    &\omega\big(u\T\alpha_1,\dots,u\T\alpha_j,\T\alpha_{j+1},\dots,\T\alpha_p\big)(w_1,\dots,w_q,\zeta)\label{eq:omega_u_Talpha}\\
    &=\omega\big(\textstyle\sum_i\Z_i(\C\alpha_1(w_i)),u\T\alpha_2,\dots,u\T\alpha_j,\T\alpha_{j+1},\dots,\T\alpha_p\big)(w_1,\dots,w_q,\zeta)\nonumber\\
    &=\tfrac 1{(q-1)!}\textstyle\sum_{{\sigma\in S_q}}\sgn(\sigma) \omega\big(\Z_1(\C\alpha_1(w_{\sigma(1)})),\underbrace{u\T\alpha_2}_{\mathclap{\textstyle\sum_{i\geq 2}\Z_i(\C\alpha_2(w_{\sigma(i)}))}},\dots,u\T\alpha_j,\T\alpha_{j+1},\dots,\T\alpha_p\big)(w_{\sigma(1)},\dots,w_{\sigma(q)},\zeta),\nonumber
  \end{align}
  where we have used Lemma \ref{lem:h_on_generators} on the first equality and a similar identity as \eqref{eq:zi_to_z1} on the second equality (see Intermezzo \ref{intermezzo:intermediate_eq} below). Importantly, the under-brace follows from
  \begin{align}
    \label{eq:Z_Z_omega}
    \iota_{\Z_i\alpha}\iota_{\Z_i\beta}\omega=0,
  \end{align}
  for any $\alpha,\beta\in\Gamma(A)$, as a consequence of multilinearity of $\omega$ with respect to both vector bundle structures, and the general fact that linear differential forms on vector bundles vanish whenever two vertical vectors are inserted.
 \begin{intermezzo}
  \label{intermezzo:intermediate_eq}
   We observe that the following generalization of the equation \eqref{eq:zi_to_z1} holds:
   for any $\omega\in\Gamma_\ext(\MM_q,\Lambda^p\AA_q^*)$ and for any given $\ell\geq 1$, skew-symmetry of $\omega$ with respect to $\AA_q\ra A$ implies
 \begin{align*}
 \begin{split}
       &\omega\big(\textstyle\sum_{i\geq \ell}^{\MM_q}\Z_i(\C\alpha(w_i)),X^2,\dots,X^p\big)(w_1,\dots,w_q,\zeta)\\
       &=\textstyle\sum_{i\geq \ell}(-1)^{i-\ell}\omega\big(\Z_\ell(\C\alpha(w_i)),X^2,\dots,X^p\big)(w_1,\dots,w_{\ell-1},w_i,w_\ell,\dots,\widehat{w_i},\dots,w_q,\zeta)\\
       &=\tfrac 1{(q-\ell)!}\textstyle\sum_{{\sigma\in S_q^{\ell-1}}}\sgn(\sigma)\omega\big(\Z_\ell(\C\alpha(w_{\sigma(\ell)})),X^2,\dots,X^p\big)(w_{\sigma(1)},\dots,w_{\sigma(q)},\zeta),
 \end{split}
 \end{align*}
 for any vectors $X^i\in\AA_q$ over $(w_1,\dots,w_q,\zeta)$, where $S_q^{\ell-1}\subset S_q$ denotes the permutations which restrict to the identity on $\set{1,\dots, \ell-1}$. In the second line, we have used Lemma \ref{lem:h_on_generators}.
 \end{intermezzo}
\noindent By the intermezzo, the computation \eqref{eq:omega_u_Talpha} continues as
\begin{align*}
  &\eqref{eq:omega_u_Talpha}=\tfrac 1{(q-1)!(q-2)!}
    \textstyle\sum_{{\sigma\in S_q}}
    \textstyle\sum_{{\tilde\sigma\in S_q^1}}
    \sgn(\sigma)\sgn(\tilde\sigma)\cdot
    \\
    &\hspace{1em}
    \cdot\omega\big(\Z_1(\C\alpha_1(w_{\sigma(1)})),\Z_2(\C\alpha_2(w_{\sigma\tilde \sigma(2)})),u\T\alpha_3,\dots,\T\alpha_{j+1},\dots,\T\alpha_p\big)(w_{\sigma\tilde\sigma(1)},w_{\sigma\tilde\sigma(1)},\dots,w_{\sigma\tilde\sigma(q)},\zeta).
\end{align*}
By introducing $\tau=\sigma\tilde\sigma$, the condition $\tilde\sigma(1)=1$ becomes $\tau(1)=\sigma(1)$, so we can replace the second sum $\sum_{{\tilde\sigma\in S_q^1}}$ with $\sum_{\tau\in S_q, \tau(1)=\sigma(1)}$. The summand is then independent of $\sigma$, and since there are $(q-1)!$ permutations $\sigma$ in $S_q$ with a fixed $\sigma(1)$, we get
\begin{align*}
  \eqref{eq:omega_u_Talpha}=\tfrac 1{(q-2)!}\textstyle\sum\limits_{\smash[b]{\mathclap{\tau\in S_q}}}\sgn(\tau)\omega\big(\Z_1(\C\alpha_1(w_{\tau(1)})),\Z_2(\C\alpha_2(w_{\tau(2)})),u\T\alpha_3,\dots,\T\alpha_p\big)(w_{\tau(1)},\dots,w_{\tau(q)},\zeta).
\end{align*}
Repeating this procedure on each of the remaining sections $u\T\alpha_i$, we are left with
\begin{align*}
  \eqref{eq:omega_u_Talpha}&=\tfrac 1{(q-j)!}\textstyle\sum\limits_{\mathclap{\tau\in S_q}}\sgn(\tau)\omega\big(\Z_1(\C\alpha_1(w_{\tau(1)})),\dots,\Z_j(\C\alpha_j(w_{\tau(j)})),\T\alpha_{j+1}\dots,\T\alpha_p\big)(w_{\tau(1)},\dots,w_{\tau(q)},\zeta)\\
  &=\tfrac 1{(q-j)!}\textstyle\sum\limits_{\mathclap{\tau\in S_q}}\sgn(\tau)c_j(\omega)(\alpha_{j+1},\dots,\alpha_{p}\| \C\alpha_{1}(w_{\tau(1)}),\dots,\C\alpha_{j}(w_{\tau(j)}))(w_{\tau(j+1)},\dots,w_{\tau(p)}),
\end{align*}
where we have used Remark \ref{rem:any_vectors_ck} in the second equality. Comparing this expression with \eqref{eq:wedgedot_multiple} proves the wanted identity \eqref{eq:omega_cj_wedgedot}, in turn proving the theorem at the level of leading terms. 

Finally, for correction terms, first observe that a similar identity as in \eqref{eq:grouped_rearranged} holds:
\begin{align*}
  &\tilde c_k(h^*\omega)(\alpha_1,\dots,\alpha_{p-k}\|\beta_1,\dots,\beta_k)=\omega(h\Z_1\beta_1,\dots,h\Z_k\beta_k, h\T\alpha_1,\dots,h\T\alpha_{p-k})\\
  &=\textstyle\sum_{j=k}^p(-1)^{j-k}\textstyle\sum\limits_{\mathclap{\qquad\sigma\in_{(j-k,p-j)}}}\sgn(\sigma)\,\omega(h\Z_1\beta_1,\dots,h\Z_k \beta_k,u\T\alpha_{\sigma(1)},\dots, u\T\alpha_{\sigma(j)},\T\alpha_{\sigma(j+1)},\dots,\T\alpha_{\sigma(p-k)}).
\end{align*}
We now have to prove the following equality for any sections $\alpha_i,\beta_j\in\Gamma(A)$:
\begin{align*}
  \omega(h\Z_1\beta_1,\dots,h\Z_k\beta_k,\,&u\T\alpha_1,\dots,u\T\alpha_j,\T\alpha_{j+1},\dots,\T\alpha_{p-k})\\
  &=c_j(\omega)(\alpha_{j+1},\dots,\alpha_{p-k}\|h\beta_1,\dots,h\beta_k)\wedgedot (\C\alpha_1,\dots,\C\alpha_j).
\end{align*}
This holds since $h(\Z_i\beta)=\Z_i(h\beta)$ for any $\beta\in\Gamma(A)$ by Lemma \ref{lem:h_on_generators}, by using the identical procedure as for the leading term.
\end{proof}
\section{Obstruction to existence of multiplicative connections}

In this section, we develop the necessary and sufficient condition for the existence of (infinitesimal) multiplicative Ehresmann connections. This generalizes the case for groupoid extensions \cite{gerbes}*{Proposition 6.13}, and establishes the infinitesimal analogue. As we will see, the obstruction is fairly simple to obtain as soon as one considers the right cohomology---that of the horizontal subcomplexes of the Bott--Shulman--Stasheff and Weil complex introduced in the previous sections.
\subsection*{The global case}
\begin{definition}
Let $\frak k$ be a bundle of ideals on a Lie groupoid $G\rra M$. The \textit{horizontal cohomology} of $\frak k$-valued forms on $G$ is the cohomology of the horizontal subcomplex from Definition \ref{defn:horizontal}, i.e.,
\begin{align*}
H^{p,q}(G;\frak k)^\Hor\coloneqq H^p\big(\Omega^{\bullet,q}(G;\frak k)^\Hor,\delta\big).
\end{align*}
\end{definition}
We now construct the obstruction class for the existence of multiplicative connections for a bundle of ideals $\frak k$. Suppose $E\subset TG$ is any (not necessarily multiplicative) distribution on $G$ complementing $K$, and let $\omega\in \Omega^1(G;s^*\frak k)$ be the corresponding 1-form defined by equation \eqref{eq:omega}, i.e., $\omega$ is the vertical projection under the isomorphism $K\cong s^*\frak k$. Since $\delta\circ\delta=0$, we obtain a cocycle $\delta\omega$, which is horizontal, i.e., it vanishes on $K^{(2)}$:
\begin{align*}
\delta\omega(X,Y)&=\omega(Y)-\omega(\d m(X,Y))+\Ad_{h^{-1}}\omega(X)\\
&=\d(L_{h^{-1}})Y-\d{(L_{h^{-1}g^{-1}})}(\d(L_g)(Y)+\d(R_h)(X))+\Ad_{h^{-1}}\d(L_{g^{-1}})X=0,
\end{align*}
for any $X\in K_g$ and $Y\in K_h$ where $s(g)=t(h)$. Hence, we may define
\begin{align}
\label{eq:obs}
\obs_{\A(G;\frak k)}=[\delta \omega] \in H^{2,1}(G;\frak k)^\Hor.
\end{align}
The difference of two 1-forms $\omega$ and $\tilde\omega$, corresponding to different distributions, is horizontal since their restrictions to $K$ equal the Maurer--Cartan form, $\tilde\omega|_K=\omega|_K=\Theta_{MC}|_K$. Therefore, the class above is independent of the choice of the distribution $E$.
\begin{proposition}
\label{prop:existence}
A multiplicative Ehresmann connection on a Lie groupoid $G$ for a bundle of ideals $\frak k$ exists if and only if the class $\mathrm{obs}_{\A(G;\frak k)}$ vanishes.
\end{proposition}
\begin{proof}
Clearly, the existence of a multiplicative connection implies the vanishing of the obstruction class. Conversely, if the class \eqref{eq:obs} vanishes, there is a horizontal form $\alpha\in \Omega^1(G;s^*\frak k)^\Hor$ with
\begin{align*}
\delta(\omega+\alpha)=0,
\end{align*}
where $\omega$ corresponds to a fixed distribution $E$, as above. This means that the form $\alpha$ corrects the connection $\omega$, making $\omega+\alpha$ a multiplicative Ehresmann connection.
\end{proof}
\begin{remark}
  This obstruction class can be used to provide alternative proofs of the following two results from \cite{mec}; this is currently a work in progress.
  \begin{itemize}
    \item Morita invariance of existence of multiplicative Ehresmann connections \cite{mec}*{Theorem 4.1}. This follows directly from the to-be-proved Morita invariance of horizontal cohomology. For Lie groupoid extensions, this was already proved in \cite{gerbes}*{Theorem 6.9}; our work in progress concerns generalizing this result to arbitrary bundles of ideals.
    \item Existence of multiplicative Ehresmann connections on proper Lie groupoids for arbitrary bundles of ideals \cite{mec}*{Theorem 4.2}. This statement is a direct consequence of a to-be-proved version of \textit{vanishing cohomology theorem} for (horizontal!)\ cohomology on proper Lie groupoids, with values in a representation. 
  \end{itemize}
\end{remark}

\subsection*{The infinitesimal case}
\begin{definition}
Let $\frak k$ be a bundle of ideals on a Lie algebroid $A\Ra M$. The \textit{horizontal cohomology} of $\frak k$-valued Weil cochains on $A$ is the cohomology of the horizontal subcomplex from Definition \ref{defn:horizontal_inf}, that is,
\begin{align*}
H^{p,q}(A;\frak k)^\Hor\coloneqq H^p\big(W^{\bullet,q}(A;\frak k)^\Hor,\delta\big).
\end{align*}
\end{definition}
The construction of the obstruction class now proceeds as follows. Pick any triple of the following form: a splitting $v\colon A\ra \frak k$ of the short exact sequence \eqref{eq:splitting}, a linear connection on $\frak k\ra M$ and a tensor $U\in\Gamma(H^*\otimes T^*M\otimes \frak k)$, where $H=\ker v$ (for example, $U=0$). Combine the connection and the tensor to obtain a map $\C\colon \Gamma(A)\ra \Omega^1(M;\frak k)$,
\[
\C(\alpha)=\nabla(v\alpha)-U(h\alpha),
\]
which clearly satisfies $\C(f\alpha)=f\C(\alpha)+\d f\otimes v\alpha$, so it defines a Weil cochain $(\C,v)\in W^{1,1}(A;\frak k)$. Since $\delta\circ\delta=0$, we obtain a cocycle $\delta(\C,v)$, which is horizontal---indeed, by equation \eqref{eq:c2} we get
\begin{align*}
\delta(\C,v)_1(\alpha\|\xi)&=-[\alpha,v(\xi)]+v[\alpha,\xi]=0,
\end{align*}
and $\delta(\C,v)_2$ vanishes automatically since $2=k\geq q=1$. Hence, we may define
\begin{align}
\label{eq:obs_inf}
\obs_{\A(A;\frak k)}=[\delta(\C,v)]\in H^{2,1}(A;\frak k)^\Hor.
\end{align}
If $(\tilde \C,\tilde v)$ is another cochain obtained as above, the difference $(\tilde \C,\tilde v)-(\C,v)$ is clearly horizontal since both $\tilde v$ and $v$ are splittings of \eqref{eq:splitting}, hence the class above is independent of the choice of the triple $(v,\nabla,U)$. The proof of the following is analogous to that of Proposition \ref{prop:existence}.
\begin{proposition}
  \label{prop:existence_im}
An IM connection on a Lie algebroid $A$ for a bundle of ideals $\frak k$ exists if and only if the class $\obs_{\A(A;\frak k)}$ vanishes. Moreover, if $A\Ra M$ is the algebroid of $G\rra M$ and $\frak k$ is a bundle of ideals on $G$, then the van Est map relates the two obstruction classes:
\[
\ve (\obs_{\A(G;\frak k)})=\obs_{\A(A;\frak k)}.
\]
\end{proposition}
\begin{proof}
  The proof of the first part of the proposition is analogous to the groupoid case. For the second part, note that if a form $\omega\in\Omega^1(G;s^*\frak k)$ satisfies $\omega|_K=\Theta_{MC}$, then the symbol of $(\C,v)=\ve(\omega)$ is given on any $\xi\in\Gamma(\frak k)$ by $v(\xi)=\omega(\xi^L)|_M=(s^*\xi)|_M=\xi$, so $v|_\frak k=\id_\frak k$, and hence
\begin{align}
    \ve(\obs_{\A(G;\frak k)})=\ve[\delta\omega]=[\ve(\delta\omega)]=[\delta(\ve(\omega))]=[\delta(\C,v)]=\obs_{\A(A;\frak k)}.\tag*\qedhere
\end{align}
\end{proof}

\clearpage \pagestyle{plain}
\chapter[Foliated and multiplicative Yang--Mills theory]{Foliated and multiplicative\\Yang--Mills theory}{}\label{chapter:ym}
\pagestyle{fancy}
\fancyhead[CE]{Chapter \ref*{chapter:ym}} 
\fancyhead[CO]{Foliated and multiplicative Yang--Mills theory}

\section{Introduction}
As already stated in the introduction, Yang--Mills theory is fundamental to the Standard Model of elementary particle physics, describing the dynamics of the carriers of fundamental forces of nature---gauge bosons. The central mathematical notion which appears as its foundation is that of the curvature of a connection on a principal bundle; Yang--Mills theory boils down to an application of the variational principle to this geometrical framework, yielding an equivalence between the critical points of the Yang--Mills action, defined as the $L^2$-norm of the curvature, and the solutions to a certain partial differential equation, called the Yang--Mills equation. The simplest example is the case of a trivial $\mathrm U(1)$-bundle on Minkowski space, where the Yang--Mills equation corresponds  precisely to two of the four Maxwell equations, and the remaining two are captured by the Bianchi identity. They respectively read 
\begin{align*}
\d{}\star F=0\quad\text{and}\quad\d{F}=0,
\end{align*}
where $\star$ denotes the Hodge star operator on Minkowski space $M$, and the curvature $F\in\Omega^2(M)$ is interpreted as the electromagnetic field strength. As already mentioned in the introduction, classical Yang--Mills theory on principal bundles can be seen as a far-reaching generalization of the theory of electromagnetism, providing a
 coordinate-invariant way of writing the differential equations that govern the dynamics of gauge bosons in physics, on an arbitrary fixed spacetime. 

From the mathematical point of view, the interpretation of Yang--Mills theory is that by considering critical points of the Yang--Mills action, one aims to find those connections on a given principal bundle which minimize the $L^2$-norm of the curvature; this may be viewed as the next best scenario to having a flat connection. The research presented here provides a two-fold generalization of Yang--Mills theory: from principal bundles to possibly \textit{non-transitive} and \textit{non-integrable} Lie algebroids. 
To motivate this generalization, we start by interpreting connections on principal bundles within the framework of Lie groupoids and algebroids. Recall that a given principal bundle
\[
G \curvearrowright P \ra M
\]
gives rise to its \textit{gauge groupoid}, whose $\Hom$-sets are defined as the $G$-equivariant maps between the fibres of $P\ra M$. Its Lie algebroid fits into the following short exact sequence of Lie algebroids, commonly referred to as the \textit{Atiyah sequence},
\[\begin{tikzcd}
	0 & {\ad(P)} & {\frac{TP}G} & TM & 0,
	\arrow[from=1-1, to=1-2]
	\arrow[from=1-2, to=1-3]
	\arrow[from=1-3, to=1-4]
	\arrow[from=1-4, to=1-5]
\end{tikzcd}\]
where $\ad(P)=(P\times \frak g)/G$ is the adjoint bundle with respect to the adjoint representation of $G$ on its Lie algebra $\frak g$, and $G$ acts on $TP$ by the differential of the action. Since $G$-equivariance is built into the above sequence, its splittings are in a one-to-one correspondence with principal connections on $P$. This was first noted by Atiyah and Bott in \cite{atiyah-bott}, and it offers an important insight:  Yang--Mills theory on principal bundles is actually about \textit{critical splittings} of the Atiyah sequence.

This perspective provides a straightforward way of generalizing Yang--Mills theory from principal bundles to arbitrary Lie algebroids. That is, for a regular Lie algebroid $A$, we can consider the critical splittings of the short exact sequence,
\begin{align}
\label{eq:ses}
\begin{tikzcd}[ampersand replacement=\&]
	0 \& \ker\rho \& A \& T\F \& 0.
	\arrow[from=1-1, to=1-2]
	\arrow[from=1-2, to=1-3]
	\arrow["\rho", from=1-3, to=1-4]
	\arrow[from=1-4, to=1-5]
\end{tikzcd}
\end{align}
The resulting dynamics take place only along the orbital directions $T\F$, so it is instructive to think of the resulting framework as a \textit{foliated Yang--Mills theory}. Although simple and straightforward, this generalization alone already gives rise to some interesting examples (\sec\ref{sec:fym_examples}). On the other hand, the drawbacks of this framework are the following.
\begin{itemize}
	\item It cannot describe the dynamics of gauge fields in transversal directions to the orbit foliation.
	\item It works only for the case when the bundle of ideals is taken as the whole isotropy, $\frak k=\ker\rho$, hence it works only for regular algebroids.
	\item It requires additional assumptions on the orbit foliation, as well as on the Riemannian metric on the base manifold, which is part of the input data.
\end{itemize}
These drawbacks are all resolved by extending the theory to the setting of multiplicative connections. Despite the downsides, the formulation of the foliated Yang--Mills theory is an important stepping stone for developing the multiplicative Yang--Mills theory, since it becomes clearer what the ingredients should be for the multiplicative scenario. Let us thus begin with the foliated case.

\section{Foliated Yang--Mills theory}
\label{sec:foliated_ym}
We begin by considering a regular Lie algebroid $A\Ra M$, with the bundle of ideals $\frak k$ taken as the whole isotropy bundle, $\frak k=\ker\rho$. As motivated in the introduction, consider the set of splittings of the short exact sequence
\begin{align*}
\begin{tikzcd}[ampersand replacement=\&]
	0 \& \frak k \& A \& T\F \& 0,
	\arrow[from=1-1, to=1-2]
	\arrow[from=1-2, to=1-3]
	\arrow["\rho", from=1-3, to=1-4]
	\arrow[from=1-4, to=1-5]
\end{tikzcd}
\end{align*}
which will be denoted by
\[
\spl A=\set{\sigma\in\Omega^1(T\F;A)\given \rho\sigma=\id_{T\F}}.
\]
Observe that $\spl A$ is an affine space modelled on the vector space $\Omega^1(T\F;\frak k)$. Moving on, any splitting $\sigma$ induces a \textit{leafwise connection} on $\frak k$, which is a $T\F$-connection, given by
\begin{align}
\label{eq:leafwise_nabla}
\nabla^\sigma\colon \Gamma(T\F)\times\Gamma(\frak k)\ra \Gamma(\frak k),\quad\nabla^\sigma_X\xi=[\sigma(X),\xi].
\end{align}
The \textit{curvature} $F^\sigma\in\Omega^2(T\F;\frak k)$ of a splitting $\sigma$ is defined by
\begin{align}
    \label{eq:curvature_splitting_foliated}
    F^{\sigma}(X,Y)=\sigma[X,Y]-[\sigma(X),\sigma(Y)].
\end{align}
The following properties are simple and straightforward consequences of the Jacobi identity on $A$.

\begin{proposition}
\label{prop:foliated_props}
For a regular Lie algebroid $A\Ra M$, a splitting $\sigma\in\spl{A}$ satisfies the following properties.
\begin{enumerate}[label={(\roman*)}]
\item $\nabla^\sigma$ preserves the Lie bracket on $\frak k$: for any $\xi,\eta\in\Gamma(\frak k)$ there holds
\begin{align*}
\nabla^\sigma[\xi,\eta]=[\nabla^\sigma\xi,\eta]+[\xi,\nabla^\sigma\eta].
\end{align*}
\item The curvature $F^\sigma$ of $\sigma$ satisfies the leafwise Bianchi identity,
\begin{align}
\label{eq:bianchi_foliated}
\d{}^{\nabla^\sigma}F^\sigma = 0.
\end{align}
\item The curvature tensor $R^{\nabla^\sigma}\in\Omega^2(T\F;\End\frak k)$ of $\nabla^\sigma$ satisfies 
\[
R^{\nabla^\sigma}\cdot\xi=[\xi,F^\sigma],
\]
for any $\xi\in\Gamma(\frak k)$.
\end{enumerate}
\end{proposition}
\begin{remark}
Property (i) can be used to infer the well-known fact that the isotropy bundle of Lie algebras of a Lie algebroid is locally trivial when restricted to an orbit.
\end{remark}

\subsection{Yang--Mills data}
As already mentioned, similarly to the classical case of principal bundles, to construct an action functional, it is necessary to impose some data on a given Lie algebroid in order to construct a foliated (or multiplicative) Yang--Mills theory. 
\begin{definition}
    \label{def:ym_algebroid}
        A \emph{Yang--Mills data} for a bundle of ideals $\frak k\subset \ker\rho$ on a Lie algebroid $A\Ra M$ is the following additional data:
    \begin{enumerate}[label={(\roman*)}]
        \item An \emph{$\ad$-invariant} metric $\inner\cdot\cdot_{\frak k}$ on $\frak k$, that is, a positive-definite metric $\inner{\cdot}{\cdot}_{\frak k}$, such that 
        \[
        \rho(\alpha)\inner\xi\eta_{\frak k}=\inner{[\alpha,\xi]}\eta_{\frak k}+\inner\xi{[\alpha,\eta]}_{\frak k},
        \]
        for any sections $\xi,\eta\in\Gamma(\frak k)$ and $\alpha\in\Gamma(A)$.
        \item A Riemannian metric $g=\inner\cdot\cdot$ on the base manifold $M$. We also assume $M$ is orientable and choose an orientation on $M$, thus giving rise to the Riemannian volume form $\vol_M$.
    \end{enumerate} 
    \end{definition}
\noindent The importance of assuming an ad-invariant metric on the isotropy bundle is that the leafwise connection $\nabla^\sigma$, induced by \textit{any} splitting, is compatible with it. 

\begin{lemma}
\label{lemma:nabla_sigma_compatible}
Let $A\Ra M$ be a regular Lie algebroid and let $\inner\cdot\cdot_{\frak k}$ be an $\ad$-invariant metric on $\frak k$. For any splitting $\sigma\in\spl{A}$, the induced leafwise connection $\nabla^\sigma$ on $\frak k$ is compatible with $\inner\cdot\cdot_{\frak k}$. That is,
\[
X\inner\xi\eta_{\frak k}=\inner{\nabla^\sigma_X\xi}\eta_{\frak k}+\inner\xi{\nabla^\sigma_X\eta}_{\frak k}
\]
for any $X\in\Gamma(T\F)$ and $\xi,\eta\in\Gamma(\frak k)$.
\end{lemma}
\begin{remark}[Existence of ad-invariant metrics]
\label{rem:averaging}
In the case of a principal $G$-bundle $P$, an ad-invariant metric on $\ad(P)$ exists whenever the structure Lie group $G$ is compact. In fact, it is actually enough to assume there exists a compact integration of the structure Lie algebra $\frak g$ (not necessarily $G$) and that the structure Lie group $G$ is connected; see \cite{duistermaat}*{Theorem 3.6.2} and Lemma \ref{lem:pairing_A_invariant_implies_G_invariant}. For more general algebroids, to establish the existence of an $\ad$-invariant metric, it is sufficient to assume $A$ is integrable by a proper groupoid $G$. In this case, one first constructs an Ad-invariant metric, that is, \[\inner\xi\eta_{\frak k}=\inner{\Ad_g\xi}{\Ad_g\eta}_{\frak k}, \quad\text{for all $g\in G$ and $\xi,\eta\in\frak k_{s(g)}$}.\] The construction is by means of a standard averaging argument using a Haar system and a cutoff function \cite{measures_on_stacks}; this may be seen as the groupoid  version of the construction of an $\Ad$-invariant inner product on the Lie algebra of a compact Lie group. Differentiation then shows that the obtained $\Ad$-invariant metric is also $\ad$-invariant.
\end{remark}
\noindent As already mentioned in the introduction, the foliated Yang--Mills theory demands additional assumptions on the foliation and the Riemannian metric on the base. 
Namely, we will assume:
\begin{itemize}
    \item The orbit foliation $\F$ is a \emph{simple foliation}, i.e., the leaves of $\F$ are the fibres of a smooth submersion $\pi\colon M\ra N$ (with connected fibres), where $N=M/\F$ is the leaf space of $\F$.
    \item The metric $g$ on $M$ is \emph{$\pi$-invariant}, that is, it induces a (necessarily unique) Riemannian metric on the orbit space $N$, such that $\pi\colon M\ra N$ is a Riemannian submersion. More precisely, the induced metric $\pi_*g$ on $N$ reads
	\begin{align}
		\label{eq:q_invariant_metric}
		(\pi_*g)_y(v,w)=g_x((\d \pi^\perp_x)^{-1}(v), (\d \pi^\perp_x)^{-1}(w)),\quad \text{for all }v,w\in T_y N,
	\end{align}
	where $x\in \pi^{-1}(y)$, and we have denoted by $\d \pi^\perp_x\colon (\ker\d \pi_x)^\perp\xrightarrow{\cong} T_y N$ the restriction of $\d \pi$ to the orthogonal subspace. Importantly, the right-hand side is independent of the choice of $x\in \pi^{-1}(y)$.
	Such metrics always exist---we refer to \cite{riemannian_groupoids}*{\sec 2.1} for more background.
    \item The orbit foliation $\F$ is \emph{transversely orientable}, i.e., the normal bundle $TM/T\F$ is orientable, which in this case simply means the leaf space $N$ is orientable; we choose an orientation. 
\end{itemize}
\begin{definition}
	A regular Lie algebroid $A\Ra M$, together with a Yang--Mills data for $\frak k=\ker\rho$, satisfying the conditions above, will be called a \emph{foliated Yang--Mills algebroid}.
\end{definition}
Although the Lemma \ref{lemma:nabla_sigma_compatible} holds trivially, it has the following important nontrivial corollary, where the additional assumptions above come to light. First, observe that the metric on $M$ induces a symmetric $C^\infty(M)$-bilinear pairing on $\frak k$-valued foliated differential forms on the base,
\begin{align*}
  \inner\cdot\cdot_{\frak k}\colon \Omega^k(T\F;\frak k)\times \Omega^k(T\F;\frak k)\ra C^\infty(M).
\end{align*}	
In turn, this induces an inner product on compactly supported foliated forms:\footnote{If we instead assumed $(M,g)$ is pseudo-Riemannian, $\innerr\cdot\cdot_{\frak k}$ would be nondegenerate instead of positive-definite.}
\begin{align*}
  \innerr\cdot\cdot_{\frak k}\colon \Omega^k_c(T\F;\frak k)\times \Omega^k
 _c(T\F;\frak k)\ra \R,\quad \innerr\alpha\beta_{\frak k}=\int_M\inner\alpha\beta_{\frak k}\vol_M.
\end{align*}
Nondegeneracy of this pairing amounts to the fundamental lemma of calculus of variations. 
For the proof of the following result, we will assume the reader is familiar with the Hodge star operator; we direct to \cite{hamilton}*{\sec 7.2}, \cite{geometry_physics}*{\sec 14.1} for a reference.
\begin{proposition}
\label{prop:codiff}
Let $A\Ra M$ be a foliated Yang--Mills algebroid with a splitting $\sigma\in\spl{A}$. For any two compactly supported foliated forms $\alpha\in\Omega^{k-1}_c(T\F;\frak k)$, $\beta\in\Omega^{k}_c(T\F;\frak k)$, there holds
\begin{align}
  \innerr{\d{}^{\nabla^\sigma}\alpha}\beta_{\frak k}= (-1)^{k} \innerr{\alpha}{\star_\F^{-1}\d{}^{\nabla^\sigma}{\star_\F^{\mathclap{\phantom{-1}}}\beta}}_{\frak k},
\end{align}
where 
\[\star_\F\colon \Omega^\bullet(T\F;\frak k)\stackrel{\cong}\ra \Omega^{d-\bullet}(T\F;\frak k),\quad (d=\dim\F)\]
is the leafwise Hodge star operator with respect to the induced metric and orientation on $T\F$. In other words, the formal adjoint $\delta^{\nabla^\sigma}$ to the exterior differential $\d{}^{\nabla^\sigma}$ on $\Omega^k_c(T\F;\frak k)$ equals \[\delta^{\nabla^\sigma}=(-1)^{k}\star_\F^{-1}\d{}^{\nabla^\sigma}{\star}_\F^{\mathclap{\phantom{-1}}}\,.\]
\end{proposition}
\begin{proof}
The proof closely follows that of \cite{atiyah-bott}*{Equation 4.9}, with the necessary changes to account for foliated forms. We first introduce a pairing of foliated $\frak k$-valued forms,
\begin{align*}
    &\Omega^k(T\F;\frak k)\times \Omega^l(T\F;\frak k)\ra \Omega^{k+l}(T\F),\\
    &(\alpha\wedge\beta)(X_1,\dots,X_{k+l})=\sum\limits_{\smash[b]{\mathclap{\quad\sigma\in S_{(k,l)}}}}\,(\sgn\sigma)\inner{\alpha(X_{\sigma(1)},\dots, X_{\sigma(k)})}{\beta(X_{\sigma(k+1)},\dots, X_{\sigma(k+l)})}_{\frak k}.
\end{align*}
The notation for this pairing comes from \cite{atiyah-bott}, and it is non-standard. Lemma \ref{lemma:nabla_sigma_compatible} ensures  $\d{}^{\nabla^\sigma}$ behaves as a derivation under this pairing, more precisely,
\begin{align}
    \label{eq:d_nabla_sigma_derivation_wedge}
    \d(\alpha\wedge\beta)=\d{}^{\nabla^\sigma}\alpha\wedge\beta+(-1)^k\alpha\wedge\d{}^{\nabla^\sigma}\beta,
\end{align}
where $\d{}$ on the left-hand side is the de Rham differential on foliated forms (which can also be viewed as the algebroid structure on $T\F$). 

Now suppose $\alpha$ and $\beta$ are compactly supported, of degrees $k-1$ and $k$, respectively, as in the statement of the corollary. Since $M$ is oriented and $\F$ is transversely orientable, the vector bundle $T\F$ is also orientable. Pick an orientation, and note that $g$ induces a metric on $T\F$, so there is an associated Hodge star operator $\star_\F$ on $T\F$ and a volume form $\vol_{T\F}\in\Omega^d(T\F)$. We may thus take the exterior derivative of the form
$
\alpha\wedge \star_\F \beta\in\Omega^{d-1}(T\F)
$
and obtain
\[
	\d{}(\alpha\wedge \star_\F \beta)={\d{}^{\nabla^\sigma}\alpha}\wedge\star_\F\beta + (-1)^{k-1}\alpha\wedge \d{}^{\nabla^\sigma}{\star_\F\beta}\in \Omega^d(T\F)
\]
Integrating this form along the fibres of $\pi\colon M\ra N$ and using Stokes' theorem, we get
\begin{align}
	\label{eq:integral_over_TF}
    \int_{\F} {\d{}^{\nabla^\sigma}\alpha}\wedge\star_\F\beta = (-1)^k\int_{\F} \alpha\wedge \d{}^{\nabla^\sigma}{\star_\F\beta}\in C^\infty(N).
\end{align}
Note that by the definition of the Hodge star operator, we have
\begin{align*}
    {\d{}^{\nabla^\sigma}\alpha}\wedge\star_\F\beta&=\inner{\d{}^{\nabla^\sigma}\alpha}{\beta}_{\frak k} \vol_{T\F},\\
    \alpha\wedge \d{}^{\nabla^\sigma}{\star_\F\beta}&=\alpha\wedge \star_\F^{\mathclap{\phantom{-1}}} (\star^{-1}_\F\d{}^{\nabla^\sigma}{\star_\F^{\mathclap{\phantom{-1}}}\beta})=\inner{\alpha}{\star^{-1}_\F\d{}^{\nabla^\sigma}{\star_\F^{\mathclap{\phantom{-1}}}\beta}}\vol_{T\F}.
\end{align*}
Therefore, integrating \eqref{eq:integral_over_TF} with respect to the volume form $\vol_N$ associated to the metric $\pi_*g$ and the chosen orientation on $N$, we obtain
\begin{align*}
	\int_N\Big(\int_{\F}\inner{\d{}^{\nabla^\sigma}\alpha}{\beta}_{\frak k} \vol_{T\F}\Big)\vol_N = \int_N\Big(\int_{\F}\inner{\alpha}{\star^{-1}_\F\d{}^{\nabla^\sigma}{\star_\F^{\mathclap{\phantom{-1}}}\beta}}\vol_{T\F}\Big)\vol_N 
\end{align*}
The proof is finished once we recognize $\pi$-invariance of $g$ implies $\vol_M=\pm(\pr_{T\F})^*\vol_{T\F}\wedge \pi^*\vol_N$, where $\pr_{T\F}\colon TM\ra T\F$ is the splitting given by the metric $g$. We make the choice of orientation on $T\F$ in such a way that we obtain the positive sign.
\end{proof}
\begin{remark}[Relation of $\star_\F$ with $\star$]
	It is natural to ask how the foliated Hodge star operator $\star_\F$ is related to the usual $\star$. To obtain the relationship, start with the defining identity of $\star_\F$,
	\[
	\alpha\wedge \star_\F \beta =\inner\alpha\beta \vol_{T\F},
	\]
	where $\alpha, \beta\in\Omega^\bullet(T\F)$ are usual foliated forms. By $\pi$-invariance of the metric on $M$, pulling back along the orthogonal projection $\pr_{T\F}$ (induced by the metric) and wedging with $\pi^*\vol_N$ gives
	\[
	\underbrace{\pr_{T\F}{}^*(\alpha\wedge\star_\F\beta)}_{\mathclap{(\pr_{T\F}{}^*\alpha) \,\wedge\, \pr_{T\F}{}^*{}(\star_\F\beta)}}\,\wedge\, \pi^*\vol_N=\inner\alpha\beta \vol_M=(\pr_{T\F}{}^*\alpha)\wedge \star\, (\pr_{T\F}{}^*\beta).
	\]
	Since this holds for any $\alpha$, we conclude
	\begin{align*}
		\pr_{T\F}{}^*({\star_\F\beta})\wedge \pi^*\vol_N=\star\,(\pr_{T\F}{}^*\beta).
	\end{align*}
\end{remark}

\subsection{Foliated Yang--Mills action}
The definition of the following action functional is motivated by the classical Yang--Mills theory on principal bundles.
\begin{definition}
Let $A\Ra M$ be a foliated Yang--Mills algebroid. The \textit{foliated Yang--Mills action functional} is defined as the map
\begin{align}
\label{eq:foliated_ym_action}
\S\colon\splc{A}\rightarrow \R,\quad
\S(\sigma)= \int_M \inner{F^\sigma}{F^\sigma}_{\frak k}\vol_M=\innerr{F^\sigma}{F^\sigma}_{\frak k},
\end{align}
where we have denoted by
\[
\splc{A}=\set{\sigma\in\spl{A}\given F^\sigma \text{ is compactly supported}}
\]
the set of \textit{asymptotically flat} splittings. Such a splitting $\sigma\in \splc{A}$ of the sequence \eqref{eq:ses} is said to be \textit{critical} for $\S$, if there holds
\begin{align*}
\deriv\lambda0\S(\sigma+\lambda\tau)=0,
\end{align*}
for any $\tau\in\Omega^1_c(T\F;\frak k)$.
\end{definition}
\begin{remark}
The definition relies on compactly supported forms to make sure the action $\S$ is well-defined, since we are not assuming the base manifold $M$ is compact. To see that $\splc{A}$ is an affine space modelled on the vector space $\Omega_c^1(T\F;\frak k)$, observe that under an affine deformation 
\[\sigma\rightarrow \sigma+\tau\]
of a splitting $\sigma\in\spl{A}$ by a form $\tau\in\Omega^1(T\F;\frak k)$, the curvature transforms as
\begin{align}
\label{eq:expansion_curvature_foliated}
F^{\sigma+\tau}=F^\sigma-\d{}^{\nabla^\sigma}\tau-\frac{1}2[\tau,\tau].
\end{align}
Indeed, for any vector fields $X,Y\in \Gamma(T\F)$ tangent to the foliation $\F$, there holds
\begin{align*}
	F^{\sigma+\tau}(X,Y)&=(\sigma+\tau)[X,Y]-[(\sigma+\tau)(X),(\sigma+\tau)(Y)]\\
	&=F^\sigma(X,Y)-([\sigma(X),\tau(Y)]+[\tau(X),\sigma(Y)]-\tau[X,Y])-[\tau(X),\tau(Y)].
\end{align*}
\end{remark}
\begin{theorem}
\label{thm:ym}
Let $A\Ra M$ be a foliated Yang--Mills algebroid. A splitting $\sigma\in\splc{A}$ is critical if and only if its curvature $F^{\sigma}$ is a solution to the \textbf{\emph{foliated Yang--Mills equation}},
\begin{align}
\label{eq:foliated_ym_equation}
\d{}^{\nabla^\sigma}\!\star_\F F^{\sigma}=0.
\end{align}
Moreover, the Hessian of $\S$ at any critical splitting $\sigma$ is given by the quadratic form
\[
H_\sigma(\tau)=\innerr{\delta^{\nabla^\sigma}\d{}^{\nabla^\sigma}\tau-\star_\F[\star_\F F^\sigma,\tau]}{\tau}_{\frak k},
\]
for any $\tau\in\Omega^1_c(T\F;\frak k)$, under the identification of the tangent space of $\splc{A}$ with $\Omega^1_c(M;\frak k)$. 
\end{theorem}
\begin{proof}
We use the equation \eqref{eq:expansion_curvature_foliated} to compute
\begin{align}
\label{eq:computation_ym_foliated}
\begin{split}
	\S(\sigma+\lambda\tau)&=\innerr{F^{\sigma+\lambda\tau}}{F^{\sigma+\lambda\tau}}_{\frak k}\\
	&=\S(\sigma)-2\lambda\innerr{F^{\sigma}}{\d{}^{\nabla^\sigma}\tau}_{\frak k}-\lambda^2\big(\innerr{F^\sigma}{[\tau,\tau]}_{\frak k}-\innerr{\d{}^{\nabla^\sigma}\tau}{\d{}^{\nabla^\sigma}\tau}_{\frak k}\big)+\mathcal O(\lambda^3).
\end{split}
\end{align}
Differentiating at $\lambda=0$, we obtain 
\[
    \deriv\lambda 0\S(\sigma+\lambda\tau)=-2\innerr{F^{\sigma}}{\d{}^{\nabla^\sigma}\tau}_{\frak k}=-2\innerr{\delta^{\nabla^\sigma}F^{\sigma}}{\tau}_{\frak k},
\]
where we have used Proposition \ref{prop:codiff}. By the non-degeneracy of $\innerr\cdot\cdot_{\frak k}$, this expression vanishes for all $\tau\in\Omega^1_c(T\F ;\frak k)$ if and only if the curvature $F^\sigma$ satisfies the foliated Yang--Mills equation.

For the second part, first note that the Hessian at a critical point $\sigma$ is defined as
\[
H_\sigma(\tau)=\frac12 \frac{d^2}{d\lambda^2}\S(\sigma+\lambda\tau)\Big|_{\lambda=0}=\innerr{\d{}^{\nabla^\sigma}\tau}{\d{}^{\nabla^\sigma}\tau}_{\frak k}-\innerr{F^\sigma}{[\tau,\tau]}_{\frak k}
\]
where we have read out the second order coefficient in the expansion \eqref{eq:computation_ym_foliated}. To bring this into the desired form, first use Proposition \ref{prop:codiff} on the first term. For the second term, note that for any $F\in \Omega^2_c(T\F;\frak k)$ and any $\alpha,\beta\in\Omega^1_c(T\F;\frak k)$, there holds
\begin{align*}
	\innerr{{[\alpha,\beta]}}{F}_{\frak k}&=\int_M\Big([\alpha,\beta]\wedge \star_\F F\Big)\wedge q^*\vol_N=\int_M\Big(\alpha\wedge [\beta,\star_\F F]\Big)\wedge q^*\vol_N=\innerr{\alpha}{\star_\F[\star_\F F,\beta]}_{\frak k},
\end{align*}
where $\vol_N\in\Omega^{m-d}(N)$ denotes the Riemannian volume form on the orbit space, and we have used ad-invariance in the second equality, or rather, its consequence
\[
[\alpha,\beta]\wedge\gamma=\alpha\wedge[\beta,\gamma],
\]
which holds for $\frak k$-valued foliated forms of arbitrary degree. On the last equality, we have also used a property of the Hodge star operator, namely that 
\begin{align}
	\label{eq:hodge_squared}
	(\star_\F)^2=(-1)^{k(d-k)}
\end{align} 
holds on forms of degree $k$.
\end{proof}
\begin{example}[Self-dual and anti self-dual solutions]
	\label{ex:self_dual_foliated}
	If the orbit foliation is 4-dimensional, and supposing that the curvature of a splitting $\sigma$ satisfies \[F^\sigma=\pm \star_\F F^\sigma,\] then the Bianchi identity implies $F^\sigma$ is a solution to the foliated Yang--Mills equation. Note that only $\pm 1$ is allowed as the prefactor in above equality, on account of the identity \eqref{eq:hodge_squared}. This is the foliated version of \cite{atiyah_geometry_of_ym_fields}*{equations (3.7) and (3.8)}.
\end{example}
\begin{remark}
	The last part of the proof shows that any $F\in \Omega^2_c(T\F;\frak k)$ induces a map 
\begin{align*}
	&\widehat F\colon \Omega^1_c(T\F;\frak k)\ra \Omega^1_c(T\F;\frak k),\quad \widehat F(\alpha)=\star_\F[\star_\F F,\alpha],
\end{align*}
which is characterized by the equality $\innerr{\widehat F(\alpha)}{\beta}_{\frak k}=\innerr{ F}{[\alpha,\beta]}_{\frak k}$, therefore it is self-adjoint.
\end{remark}
We now inspect how underdetermined the foliated Yang--Mills equation is, i.e., in what ways we can deform a critical splitting so that it remains critical. 
\begin{proposition}
	\label{prop:underdetermined_fym}
	Let $A\Ra M$ be a foliated Yang--Mills algebroid, and suppose $\sigma\in\splc{A}$ is a critical splitting. For any $\alpha\in\Omega_c^1(T\F;\frak k)$, the splitting $\sigma+\alpha$ is critical if and only if there holds
	\begin{align*}
		\widehat{F^{\sigma+\alpha}}(\alpha)=\delta^{\nabla^\sigma}(\d{}^{\nabla^\sigma}\alpha+\tfrac 12[\alpha,\alpha]).
	\end{align*}
	In particular, when the isotropy $\frak k$ is abelian, the splitting $\sigma+\alpha$ is critical if and only if $\smash{\d{}^{\nabla^\F}}\alpha=0$, where $\nabla^\F$ denotes the canonical flat leafwise connection on $\frak k$.
\end{proposition}
\begin{proof}
	Noting that $\d{}^{\nabla^{\sigma+\alpha}}=\d{}^{\nabla^\sigma}+[\alpha,\cdot]$, we calculate
	\begin{align*}
		\delta^{\nabla^{\sigma+\alpha}} F^{\sigma+\alpha}&=\star_\F^{-1}\,(\d{}^{\nabla^{\sigma}}+[\alpha,\cdot])\star_\F F^{\sigma+\alpha}=\star_\F^{-1}\d{}^{\nabla^{\sigma}}\!\star_\F F^{\sigma+\alpha}+\widehat{F^{\sigma+\alpha}}(\alpha)\\
		&=
		-\delta^{\nabla^\sigma}(\d{}^{\nabla^\sigma}\alpha+\tfrac 12 [\alpha,\alpha])+\widehat{F^{\sigma+\alpha}}(\alpha),
	\end{align*}
	where we have used the identities \eqref{eq:expansion_curvature_foliated} and \eqref{eq:hodge_squared} together with the assumption that $\sigma$ is critical and Theorem \ref{thm:ym}. In the abelian case, criticality of $\sigma+\alpha$ is thus equivalent to $\smash{\delta^{\nabla^\F}\!\d{}^{\nabla^\F}\alpha=0}$, but then $0=\innerr{\delta^{\nabla^\F}\!\d{}^{\nabla^\F}\alpha}{\alpha}_{\frak k}=\innerr{\d{}^{\nabla^\F}\alpha}{\d{}^{\nabla^\F}\alpha}_{\frak k}$ means $\d{}^{\nabla^\F}\alpha=0$ by positive-definiteness.
\end{proof}
\begin{remark}
	\label{rem:tangent_space_of_critical_splittings}
	The computation in the proof above also provides us with the means of describing the (formal) tangent space to the space of solutions of the foliated Yang--Mills equation, at any critical splitting $\sigma\in\spl A$. Namely,  identifying the tangent space of $\splc A$ with $\Omega_c^1(T\F;\frak k)$, it consists of all $\tau\in\Omega_c^1(T\F;\frak k)$ such that
	\begin{align*}
		\smallderiv\lambda 0 \delta^{\nabla^{\sigma+\lambda\tau}}F^{\sigma+\lambda\tau}=\widehat{F^\sigma}(\tau)-\delta^{\nabla^\sigma}\!\d{}^{\nabla^\sigma}\tau=0.
	\end{align*}
	In other words, the tangent space to solutions of \eqref{eq:foliated_ym_equation} at a critical splitting $\sigma$ is precisely $\ker H_\sigma$.
\end{remark}

\subsection{Gauge invariance}
\label{sec:foliated_gauge_invariance}
In the classical setting of a principal $G$-bundle $P\ra M$, Yang--Mills theory has an important feature: the  action functional is invariant under $G$-equivariant bundle automorphisms $P\ra P$ (called gauge transformations). A natural question is how to extend this notion to the foliated Yang--Mills setting. To obtain an idea for this, one needs to first interpret gauge transformations in the groupoid language: they are closely related to inner automorphisms of the gauge groupoid (see Example \ref{ex:gauge_transformations_inner_automorphisms}). We recall an inner automorphism is defined as the conjugation with a  bisection $b\colon M\ra G$ of a Lie groupoid $G\rra M$. This notion makes sense on an arbitrary Lie groupoid $G$, hence, to extend gauge invariance to foliated Yang--Mills theory, the first idea is to assume that a given Lie algebroid $A$ is integrable, and show the foliated Yang--Mills action is invariant under inner automorphisms of an integrating groupoid. 

Hence, let $G$ be a Lie groupoid integrating $A$. We begin by showing that if $G$ is $s$-connected,  $\ad$-invariance of $\inner\cdot\cdot_{\frak k}$ is equivalent to its $\Ad$-invariance. Preliminarily, we must note that any bundle of ideals on a Lie algebroid $A$ is necessarily also a bundle of ideals on $G$, provided $G$ has connected source fibres, so we automatically get a representation $\Ad\colon G\curvearrowright\frak k$ from $\ad\colon A\curvearrowright \frak k$. This is shown in \cite{rigidity_poisson_submanifolds}*{Appendix B}, and the desired result relies on the method of proof therein.

\begin{lemma}
	\label{lem:pairing_A_invariant_implies_G_invariant}
	Let $\inner\cdot\cdot_{\frak k}$ be an $\ad$-invariant metric on a bundle of ideals $\frak k$ of a Lie algebroid $A$, and suppose a Lie groupoid $G$ integrates $A$. If $G$ has connected $s$-fibres, then $\inner{\cdot}{\cdot}_{\frak k}$ is also $\Ad$-invariant, that is, for any $g\in G$ and $\xi,\eta\in\frak k_{s(g)}$, there holds
	\begin{align}
		\label{eq:Ad_invariance}
		\inner{\Ad_g\xi}{\Ad_g\eta}_{\frak k}=\inner{\xi}{\eta}_{\frak k}.
	\end{align}
	Conversely, $\Ad$-invariance implies $\ad$-invariance regardless of connectivity of $s$-fibres of $G$.
\end{lemma}
\begin{proof}
	This is shown by a standard trick, namely, we first show that the equality above holds for all $g$ in a neighborhood of the units $u(M)\subset G$. Since any neighborhood of the units generates $G$ by $s$-connectedness (for instance, see \cite{mackenzie}*{Proposition 1.5.8}), the equality \eqref{eq:Ad_invariance} then holds for all $g\in G$. Hence, it is enough to show that for any $\xi,\eta\in\Gamma(\frak k)$, there holds
	\begin{align}
		\label{eq:ad_implies_Ad_intermediate}
		\inner{\Ad_{g_\lambda}(\xi_{s(g_\lambda)})}{\Ad_{g_\lambda}(\eta_{s(g_\lambda)})}_{\frak k}=\inner{\xi_{s(g_\lambda)}}{\eta_{s(g_\lambda)}}_{\frak k},\ \text{where}\ g_\lambda\coloneqq\phi^{\alpha^L}_\lambda(1_x),
	\end{align}
	 for any  $\alpha\in\Gamma(A)$ and all times $\lambda$ for which the integral path of $\alpha^L$ through $1_x$ is defined. 

	 To show this, we utilize the notion of derivations (Example \ref{ex:algebroid_examples} (iv)). We recall that on any vector bundle $V\ra M$, the \emph{flow} of a derivation $(D,X)\in\Gamma(\frak{gl}(V))$ is defined as the vector bundle map $\Phi^D_\lambda\colon V\ra V$ covering the flow $\phi^{X}_\lambda$ of $X\in\vf(M)$, denoted $\Phi^D_\lambda(x)\colon V_x\ra V_{\phi^X_\lambda(x)}$, which is the solution to the following differential equation.
	 \begin{equation}
		\label{eq:flow_of_derivation}
		\begin{aligned}
			\frac{d}{d\lambda}(\Phi^D_\lambda)^*(\xi)&=(\Phi^D_\lambda)^*(D\xi),\\ 
			\Phi^D_0&=\id_V.
		\end{aligned}
		\qquad
		\vcenter{\hbox{
			\begin{tikzcd}
				V & V \\
				M & M
				\arrow["{\Phi^{D}_\lambda}", from=1-1, to=1-2]
				\arrow[from=1-1, to=2-1]
				\arrow[from=1-2, to=2-2]
				\arrow["{\phi^{X}_\lambda}"', from=2-1, to=2-2]
			\end{tikzcd}
		}}
	\end{equation}
	Here, we have denoted $(\Phi^D_\lambda)^*(\xi)_x=\Phi^D_{-\lambda}(\phi^X_\lambda(x))\xi_{\phi^X_\lambda(x)}$. Given a metric $\inner\cdot\cdot_V$ on $V$, a derivation $(D,X)$ is said to be \emph{compatible} with the metric, if there holds
\begin{align}
	\label{eq:metric_derivation}
		X\inner{\xi}{\eta}_V=\inner{D\xi}{\eta}_V+\inner{\xi}{D\eta}_V\, \text{ for all }\,\xi,\eta\in\Gamma(V).
\end{align}
	Now, this is equivalent to saying that the flow $\Phi^D_\lambda$ acts on $V$ by isometries.\footnote{Proving this is tantamount to proving  that the parallel transport of a metric connection acts on $V$ by isometries.}
	In our case, we take the vector bundle $V=\frak k$. Letting  $\alpha\in\Gamma(A)$, we may assume without loss of generality that $\rho(\alpha)\in\vf(M)$ is complete, so that $\alpha^L$ is also complete (see \cite{mackenzie}*{Theorem 3.6.4}), and hence defines a derivation 
	$([\alpha,\cdot],\rho(\alpha))\in\Gamma(\frak{gl}(\frak k))$. It is shown in \cite{rigidity_poisson_submanifolds}*{Equation (43)} that its flow equals 
	\[
	\Phi_\lambda^{[\alpha,\cdot]}=\Ad_{\exp(\lambda\alpha)},
	\]
	that is, $\smash{\Phi^{[\alpha,\cdot]}_\lambda}$ equals the induced Lie algebroid map associated to the inner automorphism $I_{\exp(\lambda\alpha)}$ of $G$, induced by the bisection $\exp(\lambda\alpha)$---see the discussion after the proof and Remark \ref{rem:exp_bisections} for details. Using ad-invariance of $\inner\cdot\cdot_{\frak k}$, we conclude $\Ad_{\exp(\lambda\alpha)}$ acts on $\frak k$ by isometries. Since for any $\xi\in\frak k_x$, there holds $\Ad_{\exp(\lambda\alpha)}(\xi)=\Ad_{g_\lambda^{-1}}(\xi)$ where $g_\lambda$ is as in equation \eqref{eq:ad_implies_Ad_intermediate}, we are done.
\end{proof}
\noindent Finally, fix a bisection $b\in\mathrm{Bis}(G)$ of the integrating groupoid $G\rra M$ of $A\Ra M$, that is,
\begin{align}
\label{eq:bisection_gauge}
	b\colon M\ra G,\quad t\circ b=\id_M,\quad s\circ b=\varphi,
\end{align}
where $\varphi\colon M\ra M$ is a diffeomorphism. Consider the associated inner automorphism, defined as
\begin{equation}
\label{eq:inner_automorphism}
\begin{aligned}
	&I_b\colon G\ra G,\\ 
	&I_b(g)=b(t(g))^{-1}g b(s(g)).
\end{aligned}
\qquad
\vcenter{\hbox{
	\begin{tikzcd}
		G & G \\
		M & M
		\arrow["{I_b}", from=1-1, to=1-2]
		\arrow[shift left, from=1-1, to=2-1]
		\arrow[shift right, from=1-1, to=2-1]
		\arrow[shift left, from=1-2, to=2-2]
		\arrow[shift right, from=1-2, to=2-2]
		\arrow["\varphi", from=2-1, to=2-2]
	\end{tikzcd}
}}
\end{equation}
As portrayed in the diagram, $I_b$ is a Lie groupoid automorphism covering $\varphi$ on the base. Let $\Ad_b\coloneq(I_b)_*\colon A\ra A$ denote the induced Lie algebroid morphism over $\varphi$. It induces the following isomorphism of short exact sequences of Lie algebroids covering $\varphi$ on the base.
\[\begin{tikzcd}[column sep=large, row sep=large]
	0 & {\frak k} & A & {T\F} & 0 \\
	0 & {\frak k} & A & {T\F} & 0
	\arrow[from=1-1, to=1-2]
	\arrow[hook, from=1-2, to=1-3]
	\arrow["{\Ad_b|_{\frak k}}", from=1-2, to=2-2]
	\arrow["\rho", from=1-3, to=1-4]
	\arrow["{\Ad_b}", from=1-3, to=2-3]
	\arrow[from=1-4, to=1-5]
	\arrow["{\varphi_*}", from=1-4, to=2-4]
	\arrow[from=2-1, to=2-2]
	\arrow[hook, from=2-2, to=2-3]
	\arrow["\rho"', from=2-3, to=2-4]
	\arrow[from=2-4, to=2-5]
\end{tikzcd}\]
Since we are working with target bisections, the restriction of $\Ad_b$ to $\frak k$ is given on any $\xi\in\frak k_x$ by
\begin{align}
	\label{eq:counterintuitive_Ad_b}
	\Ad_{b}(\xi)=\Ad_{b(x)^{-1}}(\xi)\in \frak k_{\varphi(x)}.
\end{align}
Now, any splitting $\sigma\in\spl{A}$ can be pulled back along the Lie algebroid automorphism $\Ad_b$. The obtained splitting, denoted $\sigma_b=(\Ad_b)^*\sigma$, is defined by the identity
\[
	\Ad_b|_{\frak k}\circ \sigma_b =\sigma\circ \varphi_*.
\]
We are now ready to demonstrate the desired gauge invariance of the foliated Yang--Mills action.
\begin{theorem}
	\label{thm:fym_gauge_invariance}
	Let $A$ be a foliated Yang--Mills algebroid, and suppose a Lie groupoid $G$ with connected $s$-fibres integrates $A$. The foliated Yang--Mills action is invariant under inner automorphisms of $G$ covering an orientation-preserving isometry $\varphi$ of the base. That is, for any bisection $b\in\mathrm{Bis}(G)$ covering such a base map $\varphi$, there holds
	\[
	\S(\sigma_b)=\S(\sigma),
	\]
	for any splitting $\sigma$. In particular, the splitting $\sigma_b$ is critical for $\S$ if and only if $\sigma$ is.
\end{theorem}
\begin{proof}
	A short calculation shows that the curvature $F^{\sigma_b}\in\Omega^2(T\F;\frak k)$ of the splitting $\sigma_b$ reads
	\[
	F^{\sigma_b}|_x=\Ad_{b(x)}\circ (\varphi^* F^\sigma)_x,
	\]
	for all $x\in M$, where the pullback $\varphi^*\colon \Omega^\bullet(T\F;\frak k)\ra \Omega^\bullet(T\F;\varphi^*\frak k)$ is defined on simple tensors as 
\begin{align}
	\label{eq:varphi_star_foliated_ym_invariance}
	\varphi^*(\gamma\otimes\xi)= \varphi^*\gamma\otimes\xi.
\end{align}
	 Using Lemma \ref{lem:pairing_A_invariant_implies_G_invariant} together with the assumption on the map $\varphi$ concludes the proof.
\end{proof}
\begin{remark}
	\label{rem:exp_bisections}
	As seen in Example \ref{ex:exp_bisection}, an important class of bisections is defined by the exponential map, $\exp(\lambda\alpha)$, for a given section $\alpha\in\Gamma(A)$ and $\lambda\in\R$. For instance, if $\alpha_x\in\ker({\rho_x})$ for all $x\in M$, then for any $\lambda\in\R$ the global bisection $\exp(\lambda\alpha)$ covers the identity on the base, hence it automatically satisfies the conditions of the theorem above. 
\end{remark}
At this point, it is important to recognize that there is a more general fact here at play: the map $\Ad_b$ is a Lie algebroid automorphism, covering an orientation-preserving isometry $\varphi$ on the base (by assumption), which restricts on $\frak k$ to a Lie algebra bundle automorphism which is moreover an isometry (with respect to $\inner\cdot\cdot_{\frak k}$). In general, since any such Lie algebroid automorphism preserves all structure at hand, it is clear that  the action functional is invariant under its pullbacks. In particular, we have the following two special cases:
\begin{itemize}
	\item In the case when $G$ does not have connected $s$-fibres, the theorem still holds under the additional assumptions that $\frak k$ is a bundle of ideals for $G$ and $\inner\cdot\cdot_{\frak k}$ is $\Ad$-invariant.
	\item \textit{Infinitesimal gauge invariance.} As observed in the proof of Lemma \ref{lem:pairing_A_invariant_implies_G_invariant}, $\ad$-invariance ensures that the flow of the derivation $([\alpha,\cdot],\rho(\alpha))\in\Gamma(\frak{gl}(A))$ acts on $\frak k$ by isometries, hence we may pull back a splitting along the flow of such a derivation. Of course, one has to assume that $\rho(\alpha)$ is complete, so that the flow is globally defined. The action functional will remain invariant under such a pullback whenever the flow of  $\rho(\alpha)$ is an orientation-preserving isometry of the base. As mentioned, this is trivially fulfilled when $\alpha\in \Gamma(\frak k)$.
\end{itemize}

\subsection{Examples}
\label{sec:fym_examples}
In the transitive and integrable case, i.e., when $A$ is the Atiyah algebroid of a principal bundle, the framework above recovers the classical Yang--Mills theory. We now consider some more examples.

\subsubsection{Abelian isotropy}
\label{ex:abelian_ym}
Before diving into specific examples, we note that when the isotropy bundle of Lie algebras $\frak k$ is abelian, the $T\F$-connection $\nabla^\sigma$ on $\frak k$ is independent of the choice of splitting $\sigma\in\spl{A}$, and its curvature tensor $R^{\nabla^\sigma}$ vanishes by  Proposition \ref{prop:foliated_props} (iii). Therefore, $\frak k$ admits a canonical flat $T\F$-connection, which we will denote by $\nabla^\F$. However, we emphasize that the curvature $F^\sigma$ of a splitting $\sigma$ can still be nonzero, so the foliated Yang--Mills equation can in general still admit nontrivial solutions. Moreover, note that a metric on $\frak k$ is ad-invariant if and only if it is compatible with the canonical flat $T\F$-connection. In this case, the foliated Yang--Mills equation reads
\[
\d{}^{\nabla^\F}\star_\F F^\sigma=0.
\]

\subsubsection{Presymplectic forms}
\label{ex:presymplectic}
As a first specific example, we consider the well-known transitive algebroid by Almeida and Molino \cites{molino, integration_transitive}. Historically, this was the first known example of a non-integrable algebroid. To construct it, fix a closed 2-form $\omega\in\Omega^2(M)$ on a manifold $M$, and consider the Whitney sum \[A_\omega=TM\oplus \R_M\] 
where $\R_M=M\times \R$ denotes the trivial line bundle over $M$. The anchor of $A_\omega$ is given by $\rho=\pr_{TM}$ and the bracket is defined on the sections as
\begin{align*}
[(X,f),(Y,g)]_\omega=([X,Y],Xg-Yf+\omega(X,Y)),
\end{align*}
for any $X,Y\in\vf(M)$ and $f,g\in \Gamma(\R_M)=C^\infty(M)$. The associated short exact sequence reads
\[\begin{tikzcd}
	0 & {\R_M} & A_\omega & TM & 0,
	\arrow[from=1-1, to=1-2]
	\arrow[from=1-2, to=1-3]
	\arrow[from=1-3, to=1-4]
	\arrow[from=1-4, to=1-5]
\end{tikzcd}\]
and clearly, any splitting $\sigma \in\spl{A_\omega}$ of this sequence is of the form
\begin{align*}
\sigma(X)=(X,\theta(X)),
\end{align*}
for some $\theta\in\Omega^1(M)$, so we can identify $\spl{A_\omega}=\Omega^1(M)$. The canonical flat connection $\nabla$ on $\R_M$ expectedly reads
\begin{align*}
\nabla_X f=[(X,\theta(X)),(0,f)]_\omega=Xf,
\end{align*}
and the curvature of a splitting $\theta\in\Omega^1(M)$ is identified with a form $F^\theta\in\Omega^2(M)$ by computing:
\begin{align*}
F^\sigma(X,Y)&=([X,Y],\theta[X,Y])-([X,Y],X(\theta(Y))-Y(\theta(X))+\omega(X,Y))\\
&=(0,-\d\theta(X,Y)+\omega(X,Y)).
\end{align*}
Hence, $F^\theta=\omega-\d\theta$, so we observe that the curvature of a splitting $\theta$ vanishes if and only if $\omega$ is exact with the form $\theta$ as its primitive.

In order to talk about Yang--Mills theory on $A_\omega$, we also need to fix a metric on $M$ and a fibrewise inner product on $\R_M$. It is easy to see that the usual Euclidean inner product on $\R_M$ (i.e., the product of functions) is ad-invariant, which in this case reads out as the product rule $X(fg)=(Xf)g+f(Xg)$. Thus, the action reads
\begin{align*}
\S\colon\Omega^1(M)\rightarrow \R,\quad \S(\theta)=\int_M\inner{\omega-\d\theta}{\omega-\d\theta}\vol_M,
\end{align*}
and by Theorem \ref{thm:ym}, a splitting $\theta\in\Omega^1(M)$ is critical for this action if and only if
\begin{align}
\label{eq:ym_example_abelian}
\d{}\star(\omega-\d\theta)=0,
\end{align}
where $\star$ is the Hodge-$\star$ operator associated to the given metric on $M$. Assuming $M$ is compact, we now observe that by Hodge theorem \cite{foundations_manifolds}*{Theorem 6.1}, critical splittings exist. Indeed, since $\omega$ is closed, it defines a cohomology class $[\omega]\in H_{dR}^2(M)$, which in turn admits a unique harmonic representative $\eta\in\Omega^2(M)$, i.e.,  $\eta=\omega-\d\theta$ for some $\theta\in\Omega^1(M)$. Furthermore, uniqueness of $\eta$ implies that $\theta$ is only unique up to a closed 1-form; this is just a rephrasing of the fact that the PDE \eqref{eq:ym_example_abelian} is underdetermined from Proposition \ref{prop:underdetermined_fym}. To reiterate: the harmonic representative of $\omega$ is given precisely by the curvature $F^\theta$ of a critical splitting $\theta\in\spl{A_\omega}$.

Additionally, we note that a splitting $\theta$ is critical for the action $\S$ if and only it is critical for the rescaled action
\[
\theta\mapsto\int_M\bigg(\frac 12\inner{\d\theta}{\d\theta}-\inner{\omega}{\d\theta}\bigg)\vol_M,
\]
where the integrand now has a clearer interpretation as a difference of a kinetic and a potential term (i.e., a Lagrangian). 
Finally, observe that one can exchange $\R$ for an arbitrary abelian Lie algebra, and the essential features of the resulting example will remain the same as above.

\subsubsection{Riemannian manifolds with harmonic curvature}
In what follows, we will show that the theory of harmonic curvature studied in \cites{harmonic_riemannian_curvature, harmonic_riemannian_dim_4, jost_riemannian_geometry} can be viewed as a special example of the developed framework. With the benefit of hindsight, we can remark it roughly corresponds to Yang--Mills theory for the general linear algebroid. 

Suppose $V\ra M$ is a vector bundle with a Riemannian metric $\kappa=\inner\cdot\cdot_V$, and consider the general linear algebroid $\frak{gl}(V)$ from Example \ref{ex:algebroid_examples} (iii). Let $\frak o(V)$ denote its subalgebroid consisting of all derivations compatible with the metric $\kappa$, as defined in \eqref{eq:metric_derivation}.
The associated short exact sequence reads
\[\begin{tikzcd}
	0 & {\End_\kappa(V)} & {\frak{o}(V)} & TM & 0,
	\arrow[from=1-1, to=1-2]
	\arrow[from=1-2, to=1-3]
	\arrow[from=1-3, to=1-4]
	\arrow[from=1-4, to=1-5]
\end{tikzcd}\]
where $\End_\kappa(V)$ denotes the Lie algebra bundle of endomorphisms of the fibres of $V$ which are skew-symmetric with respect to the metric $\kappa$. A splitting of this sequence is just a metric linear connection on $V$, and the space of metric connections is an affine space over skew-symmetric endomorphism-valued 1-forms on $M$. For any splitting $\sigma\colon TM\ra \frak{o}(V)$, let us accordingly denote $\nabla_X=\sigma(X)$ for any $X\in TM$. The induced linear connection $\nabla^\sigma$ on $\End_\kappa(V)$ is then simply the restriction of the induced connection $\nabla^{\End V}$ on $\End(V)$, which we will denote by the same symbol, $\nabla^\sigma=\nabla^{\End V}$. The curvature of a splitting $\sigma$ is just the curvature tensor of $\nabla$:
\[
	F^\sigma\in\Omega^2(M;\End V),\quad F^\sigma(X,Y)\xi=\nabla_{[X,Y]}\xi-[\nabla_X,\nabla_Y]\xi=-R^\nabla(X,Y)\xi,
\]
which also has values in skew-symmetric endomorphisms of $V$. Now, the metric $\kappa$ induces a metric $\inner\cdot\cdot_{\End V}$ on $\End V$, whose restriction to the skew-symmetric endomorphisms $\End_{\kappa}(V)$ is ad-invariant. 
To construct a Yang--Mills theory, we need to also fix a metric and an orientation on the base $M$.
The action functional \eqref{eq:foliated_ym_action} then reads
\[
\S(\nabla)=\int_M\innersmall{R^\nabla}{R^\nabla}_{\End V}\vol_M.
\]
Given any metric connection $\nabla$, the induced connection $\nabla^{\End V}$ is compatible with $\inner\cdot\cdot_{\End V}$, hence $\nabla$ is critical if and only if its curvature satisfies
\[
\d{}^{\nabla^{\End V}}\!\star R^\nabla=0.
\]
Since the Bianchi identity $\d{}^{\nabla^{\End V}}R^\nabla=0$ always holds, this is equivalent to saying that $R^\nabla$ is harmonic. If we take $V=TM$, this amounts to the vanishing divergence of $R^\nabla$, and examples of such Riemannian manifolds include 4-dimensional Einstein manifolds and conformally flat 4-manifolds of constant scalar curvature. We refer the reader to \cites{harmonic_riemannian_curvature, harmonic_riemannian_dim_4} for more details and interesting properties of Riemannian manifolds with harmonic curvature.

\section{Multiplicative Yang--Mills theory}
\label{sec:multiplicative_ym}
We now apply the research from \sec\ref{chapter:mec} to construct a Yang--Mills theory for infinitesimal multiplicative Ehresmann connections for an arbitrary bundle of ideals $\frak k\subset\ker\rho$ on a given Lie algebroid $A\Ra M$. For the reader's convenience, let us state how we have organized this section. 
\begin{itemize}
    \item We begin by considering the multiplicative version of equation \eqref{eq:expansion_curvature_foliated}---that is, we inspect how the curvature changes as we change a given multiplicative connection. For completeness, we do so in both the global and the infinitesimal realm; the latter will be essential for developing the multiplicative Yang--Mills theory, since we will have to vary IM connections.
    \item Next, we direct our attention to a particular class of \emph{primitive} multiplicative connections, i.e., those with cohomologically trivial curvature.  Restricting to this class has a two-fold purpose: it provides a clear way of defining the action functional for multiplicative Yang--Mills theory, and shows a clear relationship thereof with the foliated Yang--Mills theory.
    \item We construct an action functional on primitive IM connections, introduce two novel notions of criticality for this action functional, and obtain the desired Yang--Mills theory in the multiplicative setting, by recognizing the additional requirements which need to be met.
\end{itemize}

\subsection{Affine deformations of multiplicative connections}
\subsubsection{Global case}
The set $\A(G;\frak k)$ of multiplicative Ehresmann connections is an affine space modelled on the vector space $\Omega^1_m(G;s^*\frak k)^\Hor$ of horizontal multiplicative 1-forms, so we would like to know how the curvature changes as we make an affine deformation \[\omega\ra \omega+\lambda \alpha\]
of a multiplicative  connection $\omega \in\A(G;\frak k)$ by a horizontal multiplicative 1-form $\alpha\in \Omega_m^1(G;\frak k)^\Hor$, scaled by $\lambda\in \R$. That is, we aim to obtain an expansion of $\Omega^{\omega+\lambda\alpha}$ in terms of the scalar $\lambda$. 
\begin{theorem}
\label{thm:expansion}
	For any multiplicative connection $\omega\in \A(G;\frak k)$ on a Lie groupoid $G\rra M$, its curvature changes with an affine deformation as
	\begin{align}
	\label{eq:expansion}
		\Omega^{\omega+\lambda\alpha}=\Omega^\omega+\lambda \D{}^\omega\alpha+\lambda^2\mathbf c_2(\alpha),
	\end{align}
	where $\alpha\in \Omega^1_m(G;\frak k)^\Hor$ is any horizontal multiplicative form and $\lambda\in\R$. Here, the map \[\mathbf c_2\colon \Omega^1_m(G;\frak k)^\Hor\rightarrow \Omega^2_m(G;\frak k)^\Hor\] is homogeneous of degree two, and 
independent of the connection $\omega\in\A(G;\frak k)$.
\end{theorem}
In the proof, an explicit formula for the coefficient $\mathbf c_2$ will also be given. The theorem is proved in several steps, starting with the following observation.
\begin{lemma}
\label{lem:tau_alpha}
Let $\frak k$ be a bundle of ideals on a Lie groupoid $G\rra M$. Any horizontal 1-form $\alpha\in\Omega^1(G;K)^\Hor$ determines a 1-form $\tau^\alpha\in\Omega^1(G;\End(K))$,
\[
\tau^\alpha(X)Y=\alpha[X,Y]+\nabla^s_Y\alpha(X),
\]
for any $X\in\vf(G)$ and $Y\in\Gamma(K)$. In turn, the form $\tau^\alpha$ defines an operator $\Omega^\bullet (G;K)\ra \Omega^{\bullet+1}(G;K)$, $\beta\mapsto \tau^\alpha\wedge\beta$, given on any $q$-form $\beta$ as
\[
(\tau^\alpha\wedge \beta)(X_i)_{i=0}^q=\sum_i(-1)^i\tau^\alpha(X_i)\cdot\beta(X_0,\dots,\hat{X_i},\dots,X_q).
\]
\end{lemma}
\begin{proof}
First note that $\nabla^s$ denotes the intrinsic $K$-connection on $K$ (see Remark \ref{rem:nabla_s}), so we are not fixing any multiplicative Ehresmann connection. The only thing to show is $C^\infty(G)$-linearity in both arguments of $\tau^\alpha$, which is straightforward. 
\end{proof}
What follows is a lemma relating the exterior covariant derivatives induced by two different multiplicative Ehresmann connections.
\begin{lemma}
\label{lem:diff_ext_cov}
	Let $G$ be a Lie groupoid with a multiplicative Ehresmann connection $\omega\in \A(G;\frak k)$. For any horizontal multiplicative 1-form $\alpha\in \Omega^1_m(G,\frak k)^\Hor$ and any form $\beta\in \Omega^\bullet(G;s^*\frak k)$, there holds
	\begin{align}
	\begin{split}
		\d{}^{\omega+\alpha}\beta-\d{}^\omega\beta=\tau^\alpha\wedge\beta\label{eq:diff_ext_cov}
		- [\alpha,\beta]_{s^*\frak k},
	\end{split}
	\end{align}
	with $\tau^\alpha\in \Omega^1(G;s^*\End(\frak k))$ under the identification $s^*\frak k\cong K$ of vector bundles.
\end{lemma}
\begin{remark}
	We are denoting $\d{}^\omega\coloneqq\d{}^{\nabla^s}$ where $\nabla$ is the linear connection on $\frak k$ from Proposition \ref{prop:conn} induced by $\omega$. In fact, in the proof below, the connection $\nabla$ induced by $\omega$ will be denoted $\nabla^\omega$. Moreover, we will denote the corresponding multiplicative distribution by $E^\omega=\ker \omega$, and the corresponding horizontal and vertical projections by $h_\omega\colon TG\ra E^\omega$ and $v_\omega\colon TG\ra K$, respectively. Throughout, we will be using the isomorphism $s^*\frak k\cong K$ of vector bundles as defined by \eqref{eq:iso_pullback_k}.
\end{remark}
\begin{proof}
	By the definitions of the exterior covariant derivatives, we have
	\begin{align}
	\begin{split}
	\label{eq:diff_ext_der_1}
	(\d{}^{\omega+\alpha}\beta&-\d{}^\omega\beta)(X_i)_{i=0}^q=\textstyle\sum_{i}(-1)^i(s^*\nabla^{\omega+\alpha}-s^*\nabla^\omega)_{X_i}\beta(X_0,\dots,\hat{X_i},\dots,X_q)
	\end{split}
	\end{align}
	for any vector fields $X_i\in \vf(G)$. First note that the difference of linear connections is an endomorphism-valued 1-form, in our case:
	\[
	s^*\nabla^{\omega+\alpha}-s^*\nabla^\omega\in \Omega^1(G;\End(s^*\frak k)).
	\]
	So, let us evaluate the expression 
	\begin{align*}
		(s^*\nabla^{\omega+\alpha}-s^*\nabla^\omega)_YX,
	\end{align*}
	for $Y\in \vf(G)$ an $s$-projectable vector field to $s_*Y=U\in \vf(M)$, and $X=s^*\xi$ the pullback of a section $\xi\in \Gamma(\frak k)$. In this case, by definition of the pullback of a linear connection, we have
	\[
	(s^*\nabla^{\omega+\alpha}-s^*\nabla^\omega)_YX=s^*(\nabla^{\omega+\alpha}_U\xi - \nabla^\omega_U\xi),
	\]
	so we inspect the difference of the covariant derivatives on the right-hand side. By Proposition \ref{prop:conn} (ii), there holds
	\begin{align*}
	(&\nabla^{\omega+\alpha}_U\xi - \nabla^\omega_U\xi)^L\\
  &=v_{\omega+\alpha}[h_{\omega+\alpha}(Y),\xi^L]-v_{\omega}[h_{\omega}(Y),\xi^L]=v_{\omega+\alpha}[Y-v_{\omega+\alpha}(Y),\xi^L]-v_{\omega}[Y-v_{\omega}(Y),\xi^L]\\
	&=(v_{\omega+\alpha}-v_\omega)[Y,\xi^L]-[(v_{\omega+\alpha}-v_\omega)Y,\xi^L]=\bar\alpha[Y,\xi^L]-[\bar\alpha(Y),\xi^L],
	\end{align*}
	where we have denoted by $\bar\alpha\in\Omega^1(G;K)$ the form $\alpha$ under the identification $s^*\frak k\cong K$, i.e., $\bar\alpha=v_{\omega+\alpha}-v_\omega$. We now want to write this as a section of $s^*\frak k$, and rewrite the second term using the bracket $[\cdot,\cdot]_{s^*\frak k}$. To do so, we first note that the canonical isomorphism $s^*\frak k\cong K$ of vector bundles, given in \eqref{eq:iso_pullback_k}, is not an isomorphism of Lie algebroids---in fact, the anchor of $K$ is the inclusion and the anchor of $s^*\frak k$ is zero. The relation between their brackets is the following: $[\cdot,\cdot]_{s^*\frak k}$ is the torsion of the canonical $K$-connection $\nabla^s$ on $K$ (up to a sign). Indeed, note that on left-invariant sections, $[\cdot,\cdot]$ on $K$ agrees with $[\cdot,\cdot]_{s^*\frak k}$ on $s^*\frak k$ (up to the identification $s^*\frak k\cong K$ as vector bundles), hence it is an application of the Leibniz rule of $[\cdot,\cdot]$ to see that
	\[
	\left[\xi_1,\xi_2\right]=[\xi_1,\xi_2]_{s^*\frak k}+\nabla^s_{\xi_1} \xi_2-\nabla^s_{\xi_2} \xi_1,
	\]
	for any $\xi_1,\xi_2\in \Gamma(K)$.  Using this relation, we can rewrite $(\nabla^{\omega+\alpha}_U\xi - \nabla^\omega_U\xi)^L$ as:
	\begin{align*}
  s^*(\nabla^{\omega+\alpha}_U\xi - \nabla^\omega_U\xi)&=\alpha[Y,\xi^L]-[\alpha(Y),s^*\xi]_{s^*\frak k}+\nabla^s_{\xi^L}\alpha(Y)\\
  &=\tau^\alpha(Y)s^*\xi-[\alpha(Y),s^*\xi]_{s^*\frak k},
\end{align*}
where we have used that $\nabla^s$ vanishes on pullback sections in the first line, and Lemma \ref{lem:tau_alpha} in the second. Due to $C^\infty(G)$-linearity of the difference $s^*\nabla^{\omega+\alpha}-s^*\nabla^\omega$, this allows us to rewrite \eqref{eq:diff_ext_der_1} as
	\begin{align*}
		(\d{}^{\omega+\alpha}\beta-\d{}^\omega\beta)(X_i)_{i=0}^q&=\textstyle\sum_{i}(-1)^i\tau^\alpha(X_i)\beta(X_0,\dots,\hat{X_i},\dots,X_q)\\
		&- \textstyle\sum_{i}(-1)^i[\alpha(X_i),\beta(X_0,\dots,\hat{X_i},\dots,X_q)]_{s^*\frak k}.
	\end{align*}
The first term is just $\tau^\alpha\wedge \beta$ in equation \eqref{eq:diff_ext_cov}, and the second term is $[\alpha,\beta]_{s^*\frak k}$
with a bit of combinatorics: since $\alpha$ is a 1-form, the bracket of forms $[\alpha,\beta]_{s^*\frak k}$ reads
	\begin{align*}
  [\alpha,\beta]_{s^*\frak k}(X_0,\dots,X_k)&=\frac1{k!}\textstyle\sum_{{\sigma\in S_{k+1}}}\sgn(\sigma)[\alpha(X_{\sigma(0)}),\beta(X_{\sigma(1)},\dots,X_{\sigma(k)})]_{s^*\frak k}\\
	&=\frac 1{k!}\textstyle\sum_{i}\textstyle\sum_{{\substack{\sigma\in S_{k+1} \\ \sigma(0)=i}}}\sgn(\sigma)[\alpha(X_i),\beta(X_{\sigma(1)},\dots,X_{\sigma(k)})]_{s^*\frak k}\\
  &=\textstyle\sum_{i}(-1)^i[\alpha(X_i),\beta(X_0,\dots,\hat{X_i},\dots,X_k)]_{s^*\frak k},
	\end{align*}
	where we have used that the number of permutations from $S_{k+1}$ with $\sigma(0)=i$ equals $k!$, and observed that given any such $\sigma$, ordering the arguments of $\beta(X_{\sigma(1)},\dots,X_{\sigma(k)})$ yields an additional factor of $(-1)^i\sgn\sigma$, concluding our proof.
\end{proof}
\begin{proof}[Proof of Theorem \ref{thm:expansion}]
	We first consider the case $\lambda=1$. The structure equation gives us
	\begin{align}
	\begin{split}
	\Omega^{\omega+\alpha}&=\d{}^{\omega+\alpha}(\omega+\alpha)+\frac12[\omega+\alpha,\omega+\alpha]_{s^*\frak k}
	\label{eq:structure_perturbed}\\
	&=\d{}^{\omega+\alpha}(\omega+\alpha)+\frac12[\omega,\omega]_{s^*\frak k}+[\omega,\alpha]_{s^*\frak k}+\frac {1}2[\alpha,\alpha]_{s^*\frak k},
	\end{split}
	\end{align}
	so we need to compute $\d{}^{\omega+\alpha}\omega$ and $\d{}^{\omega+\alpha}\alpha$. By Lemma \ref{lem:diff_ext_cov}, we have
	\begin{align*}
	\d{}^{\omega+\alpha}\omega-\d{}^\omega\omega&=\tau^\alpha\wedge \omega-[\alpha,\omega]_{s^*\frak k}=\D{}^\omega\alpha-\d{}^\omega\alpha-[\alpha,\omega]_{s^*\frak k},
	\end{align*}
	where we have used the equality
	\begin{align}
	\label{eq:diff_D_d}
		\D{}^\omega\alpha=\d{}^\omega\alpha+\tau^\alpha\wedge\omega
	\end{align}
	which holds due to horizontality of $\alpha$ and is a straightforward computation (see Remark \ref{rem:diff_D_d} after the proof). 
  On the other hand, using the lemma on $\alpha$ yields
		\begin{align*}
	\d{}^{\omega+\alpha}\alpha-\d{}^\omega\alpha=\tau^\alpha\wedge\alpha-[\alpha,\alpha]_{s^*\frak k}.
	\end{align*}
	When both equalities are plugged into the right-hand side of \eqref{eq:structure_perturbed}, the two terms with $\d{}^\omega\alpha$ cancel out, and likewise the terms containing $[\omega,\alpha]_{s^*\frak k}$, so we are left with
	\begin{align}
	\Omega^{\omega+\alpha}=\Omega^\omega+ \D{}^\omega\alpha+\big(\tau^\alpha\wedge\alpha-\frac 12[\alpha,\alpha]_{s^*\frak k}\big),
	\end{align}
	which proves the case $\lambda=1$. It is clear that the second-order coefficient 
\begin{align}
\label{eq:secondorder}
  \mathbf c_2(\alpha)=\tau^\alpha\wedge\alpha-\frac 12[\alpha,\alpha]_{s^*\frak k}
\end{align}
is independent of $\omega\in \A(G;\frak k)$, and since the forms $\Omega^{\omega+\alpha}, \Omega^\omega$ and $\D{}^\omega\alpha$ are horizontal and multiplicative by Theorem \ref{thm:deltaD}, so is $\mathbf c_2(\alpha)$. Since $\D{}^\omega$ is linear and $\mathbf c_2$ is clearly homogeneous of degree two, the general case $\lambda\in\R$ now also follows.
\end{proof}
\begin{remark}
\label{rem:diff_D_d}
More generally than in equation \eqref{eq:diff_D_d} above, it is straightforward to show (using the definition of the exterior covariant derivative $\d{}^\omega$) that for any horizontal form $\alpha\in\Omega^q(G;s^*\frak k)^\Hor$, there holds
	\begin{align*}
	(\D{}^\omega\alpha&-\d{}^\omega\alpha)(X_0,\dots,X_k)=-\textstyle\sum_i(-1)^i\nabla^s_{v_\omega(X_i)}\alpha(X_0,\dots,\hat{X_i},\dots,X_q)\\
	&-\textstyle\sum_{i<j}(-1)^{i+j}\alpha([X_i,v_\omega(X_j)]+[v_\omega(X_i),X_j],X_0,\dots,\hat{X_i},\dots,\hat{X_j},\dots,X_q),
	\end{align*}
	for any $X_i\in \vf (G)$. If $\alpha$ is a horizontal 1-form, this is just
\begin{align*}
  (\D{}^\omega\alpha-\d{}^\omega\alpha)(X,Y)&=\nabla^s_{v_\omega(Y)}\alpha(X)-\nabla^s_{v_\omega(X)}\alpha(Y)+\alpha[X,v_\omega (Y)]-\alpha[Y,v_\omega (X)]\\
  &=(\tau^\alpha\wedge\omega)(X,Y),
\end{align*}
for any $X,Y\in \vf(G)$.
\end{remark}
\begin{remark}
Roughly speaking (since the spaces are infinite dimensional), we can read from the first order coefficient that the derivative of the map
\[
\kappa\colon\A(G;\frak k)\ra \Omega^2_m(G;\frak k)^\Hor,\quad \kappa(\omega)=\Omega^\omega
\]
is given for any $\alpha\in\Omega^1_m(G;\frak k)^\Hor$ by
\[
\d \kappa_\omega(\alpha)=\D{}^\omega\alpha.
\]
\end{remark}

\subsubsection{Infinitesimal case}
Similarly to the global case, the set $\A(A;\frak k)$ of all IM connections is an affine space modelled on the vector space $\Omega^1_{im}(A;\frak k)^\Hor$ of horizontal IM 1-forms. We now discuss the infinitesimal analogue of Theorem 
\ref{thm:expansion}, i.e., we are going to inspect how the curvature of an IM connection changes as we make an affine deformation 
\begin{align*}
  (\C,v)\rightarrow (\C,v)+\lambda (L,l)
\end{align*}
of an IM connection $(\C,v)$ by a horizontal IM form $(L,l)\in\Omega^1_{im}(A,\frak k)^\Hor$ scaled by $\lambda\in\R$.

\begin{theorem}
\label{thm:expansion_inf}
On a Lie algebroid $A\Ra M$ with an IM connection $(\C,v)\in\A(A;\frak k)$, there holds
\begin{align}
\label{eq:expansion_inf}
  \Omega^{(\C,v)+\lambda(L,l)}=\Omega^{(\C,v)}+\lambda \D{}^{(\C,v)}(L,l)+\lambda^2\mathbf c_2(L,l),
\end{align}
for any horizontal IM form $(L,l)\in\Omega^1_{im}(A,\frak k)^\Hor$ and $\lambda\in \R$, where the second-order coefficient in the expansion is given by the map
\begin{align}
\begin{split}
\label{eq:secondorder_inf}
    &\mathbf c_2\colon\Omega^1_{im}(A,\frak k)^\Hor\rightarrow \Omega^2_{im}(A,\frak k)^\Hor,\\
  &\mathbf c_2(L,l)(\alpha)=-(L|_{\frak k}\wedgedot L\alpha, L|_{\frak k}\cdot l\alpha).
\end{split}
\end{align}
\end{theorem}
\begin{proof}
It is enough to show the identity 
\eqref{eq:expansion_inf} for the case $\lambda=1$, since $\mathbf c_2$ is homogeneous of degree two. Let us first fix some notation. Denote \[(\tilde\C,\tilde v)=(\C,v)+(L,l),\] so the relation between the respective horizontal projections is given by $\tilde h=h-l$, and the relation between induced linear connections is just 
$
\tilde \nabla = \nabla+ L|_{\frak k}.
$
The induced exterior derivatives are thus related by 
\[
\d{}^{\tilde\nabla}=\d{}^\nabla+L|_{\frak k}\wedge\cdot
\] 
where the wedge product is as in Remark \ref{rem:wedges}. For any $(J,j)\in\Omega^{1}_{im}(A;\frak k)$, let us simplify the notation for $(\d{}^\nabla(J,j))_1$ and simply denote it by $\ul J^\nabla$. It is not hard to see there holds
\[
  \ul J^{\tilde\nabla}\alpha = \ul J^\nabla\alpha - L|_{\frak k}\wedge j\alpha, 
\] 
for any $\alpha \in \Gamma(A)$. The task at hand now is to expand the expression \[
\D{}^{(\tilde \C,\tilde v)}(\tilde \C,\tilde v)=\D{}^{(\tilde \C,\tilde v)}(\C,v)+\D{}^{(\tilde \C,\tilde v)}(L,l),
\]
and we do this by means of a straightforward computation. The first term reads
\begin{align*}
  \D{}^{(\tilde \C,\tilde v)}(\C,v)\alpha &=(\d{}^{\tilde\nabla} \C\alpha-\ul\C^{\tilde\nabla}\wedgedot\tilde\C\alpha,\ul\C^{\tilde\nabla}\tilde h\alpha)\\
  &=(\d{}^\nabla\C\alpha+\bcancel{L|_{\frak k}\wedge\C\alpha} -\ul\C^{\nabla}\wedgedot \C\alpha + \bcancel{L|_{\frak k}\wedgedot\C\alpha}-\cancel{\ul\C^{\nabla}\wedgedot L\alpha}+L|_{\frak k}\wedgedot L\alpha,\\
  & \begingroup \color{white} =( \endgroup \ul\C^{\nabla}h\alpha-\cancel{\ul\C^{\nabla}l\alpha}-L|_{\frak k}\cdot \cancel{v(h\alpha)}+L|_{\frak k}\cdot l\alpha)\\
  &=\Omega^{(\C,v)}\alpha+ (L|_{\frak k}\wedgedot L\alpha,L|_{\frak k}\cdot l\alpha),
\end{align*}
where we have used the relation \eqref{eq:wedge_wedgedot} and observed that the restriction of $\ul\C^{\nabla}$ to $\frak k$ vanishes. On the other hand, the second term becomes
\begin{align*}
\D{}^{(\tilde \C,\tilde v)}(L,l)(\alpha)
&=(\d{}^{\tilde\nabla}L\alpha- \ul L^{\tilde\nabla}\wedgedot \tilde \C\alpha, \ul L^{\tilde\nabla}\tilde h\alpha)\\
&=(\d{}^\nabla L\alpha + L|_{\frak k}\wedge L\alpha - \ul L^{\nabla}\wedgedot \C\alpha - L|_{\frak k}\wedgedot L\alpha,
\ul L^\nabla h\alpha -L|_{\frak k}\cdot l\alpha - \ul L^{\nabla} l\alpha)\\
&=\D{}^{(\C,v)}(L,l)(\alpha)-2(L|_{\frak k}\wedgedot L\alpha,L|_{\frak k}\cdot l\alpha),
\end{align*}
where we have observed that $\ul L^{\tilde\nabla}$ coincides with $\ul L^{\nabla}$ when restricted to $\frak k$. Adding the two expressions together, we obtain the wanted identity
\begin{align*}
  \Omega^{(\tilde \C,\tilde v)}\alpha=\Omega^{(\C,v)}\alpha + \D{}^{(\C,v)}(L,l)\alpha-(L|_{\frak k}\wedgedot L\alpha, L|_{\frak k}\cdot l\alpha).\tag*\qedhere
\end{align*}
\end{proof}
The theorem may also be restated using the components $R^\nabla$ and $U$, appearing in the explicit expression \eqref{eq:im_curv} for the curvature $\Omega^{(\C,v)}$ of an IM connection $(\C,v)$. They transform with an affine deformation $(\tilde \C,\tilde v)=(\C,v)+\lambda(L,l)$ as
\begin{align}
R^{\tilde\nabla}\cdot \xi &=R^{\nabla}\cdot \xi+\lambda\d{}^{\nabla^{\End\frak k}}(L|_{\frak k})\cdot \xi+\lambda^2 L|_{\frak k}\wedge L\xi,\\
\tilde U\tilde h\alpha &=Uh\alpha+\lambda \ul L^\nabla h\alpha-\lambda^2 L|_{\frak k}\cdot l\alpha,
\end{align}
for any $\alpha \in \Gamma(A)$ and $\xi\in\Gamma(\frak k)$.

\begin{remark}
The second-order coefficient \eqref{eq:secondorder_inf} is just the infinitesimal counterpart of the one from Theorem \ref{thm:expansion}, defined in equation
\eqref{eq:secondorder}. More precisely, if a Lie groupoid $G$ integrates $A$, then the following diagram commutes. 
\[\begin{tikzcd}[row sep=large]
	{\Omega^1_m(G;\frak k)^\Hor} & {\Omega^2_m(G;\frak k)^\Hor} \\
	{\Omega^1_{im}(A,\frak k)^\Hor} & {\Omega^2_{im}(A,\frak k)^\Hor}
	\arrow["{\mathbf c_2}", from=1-1, to=1-2]
	\arrow["\ve"', from=1-1, to=2-1]
	\arrow["\ve", from=1-2, to=2-2]
	\arrow["{\mathbf c_2}", from=2-1, to=2-2]
\end{tikzcd}\]
This follows from both theorems on affine deformations and the diagram \eqref{eq:square_D}. 
\end{remark}

\subsection{Primitive IM connections and their curvings}
\label{sec:primitive}
We now focus on multiplicative connections with cohomologically trivial curvature. 
This class of connections has been very briefly studied in \cite{gerbes} for the particular case of Lie groupoid extensions; it has also implicitly appeared in \cite{mec}*{\sec 4.2} (see Definitions 4.10 and 4.11 therein). We extend this study to the infinitesimal realm, to arbitrary bundles of ideals, and develop the theory further extensively. The class of primitive IM connections turns out to be crucial for our developments of multiplicative Yang--Mills theory.

\begin{definition}
Let $\frak k\subset A$ be a bundle of ideals of a Lie algebroid $A\Rightarrow M$. An IM connection $(\C,v)\in\A(A;\frak k)$ is said to be \textit{primitive} if $\Omega^{(\C,v)}$ is cohomologically trivial, i.e., if the
cohomological class $[\Omega^{(\C,v)}]\in H^{1,2}(A;\frak k)^\Hor$ vanishes.
We denote the set of primitive connections by \[\arc(A;\frak k)\subset \A(A;\frak k).\]
Given a primitive connection $(\C,v)$, a form $F\in\Omega^2(M;\frak k)$ satisfying
\begin{align}
\label{eq:curving}
\Omega^{(\C,v)}=\delta^0 F
\end{align}
is then called a \emph{curving} of the connection $(\C,v)$. The differential form \[G\coloneqq \d{}^\nabla F\] is called the \textit{curvature 3-form} of $F$, where $\nabla=\C|_{\frak k}$ is the induced linear connection on $\frak k$.
\end{definition}
\begin{example}
If $V\ra M$ is a vector bundle, seen as a Lie algebroid with the trivial structure, and $\frak k=V$, then primitive connections are in a 1-to-1 correspondence with flat linear connections on $V$. This will follow as a particular case of Proposition 
\ref{prop:abelian_primitive}; alternatively (and more easily), this follows straightforwardly from equation \eqref{eq:R_ad_F}.
\end{example}
\begin{example}
\label{ex:transitive_primitive}
If $A\Ra M$ is a transitive algebroid and $\frak k=\ker\rho$, IM connections are in a bijective correspondence with splittings of the abstract Atiyah sequence by condition \eqref{eq:c2}, and every IM connection admits a unique curving---the curvature $F^v$ of the splitting $v\colon A\ra\frak k$. This follows from the fact that $\delta^0\colon \Omega^\bullet(M;\frak k)\ra \Omega^\bullet_{im}(A;\frak k)^\Hor$ is canonically an isomorphism, see  Proposition \ref{prop:orb_proj2}. 
We observe that the curvature 3-form of the curving $F^v$ vanishes by the Bianchi identity for the curvature of any splitting, $\d{}^\nabla F^v=0$.
\end{example}
Notice that a curving $F$ of a fixed IM connection $(\C,v)$ is in general not unique---it is only determined up to an invariant 2-form. In other words, the space of curvings $(\delta^0)^{-1}(\Omega^{(\C,v)})$ of a connection $(\C,v)$ is an affine space modelled on the vector space $\Omega^2_\inv(M;\frak k)=\ker\delta^0$ of invariant 2-forms on $M$. This means that any two curvings $\tilde F$ and $F$ must differ by a 2-form $\beta\in\Omega^2(M;\frak k)$ satisfying the following conditions:
\begin{align}
\label{eq:invariant}
\L^A_\alpha\beta=0,\quad\iota_{\rho(\alpha)}\beta=0, 
\end{align}
for any $\alpha\in \Gamma(A)$. The associated curvature 3-forms are then clearly related by
\begin{align}
\label{eq:3curv_diff}
\tilde G-G =\d{}^\nabla\beta.
\end{align}
\begin{remark}
\label{rem:invariant}
By the first condition in \eqref{eq:invariant}, any invariant form $\beta\in\Omega^\bullet(M;\frak k)$ is center-valued (take $\alpha\in\Gamma(\frak k)$). Furthermore, the second condition means that $\beta$ is \textit{transversal}, i.e.,\ it is a section of $\Lambda^\bullet(T\F)^\circ\otimes z(\frak k)$, so it must vanish wherever its degree is greater than the codimension of $\F$: for any $x\in M$, $\deg\beta>\operatorname{corank}\rho_x$ implies $\beta_x=0$. In particular, if $\codim\F=1$ then the curving of any multiplicative connection is unique, if it exists.
\end{remark}
The following is the infinitesimal analogue of \cite{gerbes}*{Theorem 6.33}.
\begin{lemma}
\label{lem:curving}
Let $(\C,v)\in \arc(A;\frak k)$ be a primitive IM connection on a Lie algebroid $A\Ra M$, and let $F$ be a curving. The following statements hold.
\begin{enumerate}[label={(\roman*)}]
\item The curvature tensors $U$ and $R^\nabla$ are determined by the curving $F$ as
\begin{align}
\label{eq:R_ad_F}
R^\nabla\cdot\xi&=[\xi,F],\\
U(\alpha)&=-\iota_{\rho(\alpha)}F \label{eq:U_iota_F}
\end{align}
for any $\xi\in\Gamma(\frak k)$ and $\alpha\in H$. In particular, $R^\nabla =0$ if and only if $F$ is centre-valued.
\item Bianchi identities: the curvature 3-form $G$ satisfies
\begin{align}
\label{eq:3curv}
\delta^0 G=0 \quad\text{and}\quad \d{}^\nabla G=0.
\end{align}
Hence, $G$ is center-valued and transversal, so it vanishes wherever $\codim\F\leq 2$.
\end{enumerate}
\end{lemma}
\begin{proof}
The point (i) follows directly from using \eqref{eq:curving} with the explicit expression \eqref{eq:im_curv} for the curvature. For the second point, note that
\[
\delta^0 G=\delta^0 \d{}^\nabla F=\D{}^{(\C,v)}\delta^0F=\D{}^{(\C,v)}\Omega^{(\C,v)}=0,
\]
where we have used Theorem \ref{thm:deltaD_inf} and the Bianchi identity. Finally,
\[
\d{}^\nabla G=(\d{}^\nabla)^2 F=R^\nabla\wedge F=[F,F]=0,
\]
since $F$ is a 2-form, where we have used the point (i).
\end{proof}
Observe that equation \eqref{eq:U_iota_F} of item (i) of the last lemma, used with equation \eqref{eq:U_along_orbits}, implies
\begin{align}
    \label{eq:Fv_pullback_F}
    F^v=(\rho_B)^*F
\end{align}
where $F^v\in\Omega^2(B;\frak k)$ is the curvature of the splitting $v\colon A\ra \frak k$, and $B=A/\frak k$. On the other hand, equation \eqref{eq:R_ad_F} will henceforth simply be written as 
\[
R^\nabla=-\operatorname{ad}F,
\]
and it is stronger than the condition \eqref{eq:s2} for the coupling data $(\nabla,U)$ of a primitive connection. In a concealed way, the last lemma also tells us what \eqref{eq:s3} should be replaced with for the primitive case, as we will now see. In light of the coupling construction from \cite{mec}, we shall reverse the process, and start with a linear connection and a curving to construct a primitive IM connection.
\begin{proposition}
\label{prop:splitting_coh_triv}
Let $B\Ra M$ be a Lie algebroid and let $(\frak k,[\cdot,\cdot]_{\frak k})$ be a bundle of Lie algebras over $M$. Suppose that a connection $\nabla$ on $\frak k$ and a form $F\in \Omega^2(M;\frak k)$ are given, such that the following conditions are satisfied:
\begin{enumerate}[label={(\roman*)}]
\item The connection $\nabla$ preserves the Lie bracket on $\frak k$.
\item The tensors $R^\nabla$ and $F$ are related by $R^\nabla=-\ad F$.
\item The 3-form $\d{}^\nabla F$ is transversal, i.e., $\iota_{\rho_B(\alpha)}\d{}^\nabla F=0$ for any $\alpha\in\Gamma(B)$.
\end{enumerate}
Then the direct sum $A=B\oplus \frak k$ has a structure of a Lie algebroid, which admits a primitive IM connection $(\C,v)$, given for any $\alpha\in\Gamma(B)$, $\xi\in\Gamma(\frak k)$ by
\[
v(\alpha,\xi)=\xi,\quad\C(\alpha,\xi)=\nabla\xi+\iota_{\rho_B(\alpha)}F,
\]
with $F$ as its curving. Conversely, any Lie algebroid $A$ with a bundle of ideals $\frak k\subset A$ that admits a primitive IM connection, is isomorphic to one of this type.
\end{proposition}
\begin{remark}
  This proposition simplifies considerably when $\frak k$ is a semisimple Lie algebra bundle; see \sec\ref{sec:semisimple} and Corollary \ref{cor:semisimple_extension}.
\end{remark}
\begin{proof}
Assumptions (i) and (ii) imply that conditions \eqref{eq:s1} and \eqref{eq:s2} are satisfied for the pair $(\nabla,U)$, where $U\in \Gamma (B^*\otimes T^*M\otimes \frak k)$ is given by $U(\alpha)=-\iota_{\rho_B(\alpha)}F$. Moreover, it is straightforward to see that \eqref{eq:s3} for our case  amounts to saying that $\d{}^\nabla F$ vanishes when evaluated on two vectors tangent to the orbit foliation of $B$. Hence, (iii) implies \eqref{eq:s3}, and we can apply \cite{mec}*{Proposition 5.13} to conclude $A$ is a Lie algebroid, with the anchor given by the composition $A\ra B\xrightarrow{\rho_B}TM$, and the bracket by
\begin{align}
\label{eq:bracket_construction}
[(\alpha,\xi),(\beta,\eta)]=\big([\alpha,\beta]_B, \nabla_{\rho_B(\alpha)}\eta-\nabla_{\rho_B(\beta)}\xi+[\xi,\eta]_{\frak k}-F(\rho_B(\alpha),\rho_B(\beta))\big). 
\end{align}
The obtained IM connection $(\C,v)$ is indeed primitive since condition (iii) used with Cartan's magic formula yields 
\[
-\d{}^\nabla U(\alpha)=\L^\nabla_{\rho_B(\alpha)}F,
\]
for any $\alpha\in\Gamma(B)$, hence also 
\[
\Omega^{(\C,v)}(\alpha,\xi)=(R^\nabla\cdot\xi-\d{}^\nabla U(\alpha),-U(\alpha))=(\L^A_{(\alpha,\xi)} F,\iota_{\rho(\alpha,\xi)}F)=\delta^0 F.\qedhere
\]
\end{proof}
\begin{corollary}
  \label{cor:primitive_bijective_correspondence}
  Let $\frak k$ be a bundle of ideals of a Lie algebroid $A$. There is a bijective correspondence between primitive IM connections for $\frak k$ together with a choice of curving, and triples $(v,\nabla,F)$, where $v\colon A\ra \frak k$ is a splitting, the pair $(\nabla,F)$ satisfies conditions (i)--(iii) of Proposition \ref{prop:splitting_coh_triv}, and the splitting is compatible with $(\nabla,F)$, that is:
  \begin{align}
    \label{eq:primitive_additional_conditions}
    \nabla^A_{h(\alpha)}=\nabla^{}_{\rho(\alpha)},\quad F^v=(\rho_B)^* F
  \end{align}
  for any $\alpha\in A$, where $B=A/\frak k$ and $F^v\in\Omega^2(B;\frak k)$ is the curvature of the splitting.
\end{corollary}
\begin{proof}
  One direction is clear from the  identities \eqref{eq:U_along_orbits}, \eqref{eq:conn_orb2} and Lemma \ref{lem:curving}. For the other direction we use the previous proposition, where the conditions \eqref{eq:primitive_additional_conditions} ensure that the obtained algebroid structure on $B\oplus \frak k$ is isomorphic to $A$. 
\end{proof}
\begin{remark}
  If $F'\in \Omega^2(M;\frak k)$ is another 2-form satisfying the same properties as $F$ in the corollary above, then $(v,\nabla,F')$ defines the same IM connection as $(v,\nabla, F)$ if and only if the difference $F-F'$ is transversal, that is, $\iota_{\rho(\alpha)}(F-F')=0$ for any $\alpha\in\Gamma(A)$. The equivalence follows by computing 
  \[
  \L^A_\alpha(F-F')=\L^\nabla_{\rho\alpha}(F-F')+[v\alpha,F-F']=0,
  \]
  where we have used Cartan's formula, and properties (ii) and (iii) from the proposition. 
\end{remark}

We now examine the affinity of the space of primitive multiplicative connections. We emphasize that this should not be confused with the aforementioned affinity of $(\delta^0)^{-1}(\C,v)$ for a fixed IM connection $(\C,v)$.
\begin{proposition}
\label{prop:affine_primitive}
Let $(\C,v)\in \arc(A;\frak k)$ be a primitive connection on a Lie algebroid $A$. If $(L,l)\in\Omega^1_{im}(A;\frak k)^\Hor$ is a cohomologically trivial IM form, then $(\C,v)+(L,l)$ is also primitive. 
In particular, if $H^{1,1}(A;\frak k)^\Hor=0$ then $\arc(A;\frak k)$ is an affine subspace of $\A(A;\frak k)$.
\end{proposition}
\begin{proof}
We have to show that if $(L,l)$ and $\Omega^{(\C,v)}$ are cohomologically trivial, then so is $\Omega^{(\C,v)+(L,l)}$. By Theorems \ref{thm:expansion_inf} and \ref{thm:deltaD_inf}, we only need to check that $\mathbf c_2(L,l)$ is cohomologically trivial. We show this by proving 
\begin{align}
\label{eq:c2_cohtriv}
\mathbf c_2(\delta^0\gamma)=-\frac 12\delta^0[\gamma,\gamma],
\end{align}
i.e., that the following diagram commutes.
\[
\begin{tikzcd}[row sep=large]
	{\Omega^1_{im}(A;\frak k)^\Hor} & {\Omega^2_{im}(A;\frak k)^\Hor} \\
	{\Omega^1(M;\frak k)} & {\Omega^2(M;\frak k)}
	\arrow["{\mathbf c_2}", from=1-1, to=1-2]
	\arrow["{\delta^0}", from=2-1, to=1-1]
	\arrow["{\gamma\mapsto-\frac 12[\gamma,\gamma]}"', from=2-1, to=2-2]
	\arrow["{\delta^0}"', from=2-2, to=1-2]
\end{tikzcd}
\]
Let us denote $(L,l)=\delta^0\gamma$ and $(J,j)= \mathbf c_2(\delta^0\gamma)$, that is, 
\[
(J,j)\alpha=-(L|_{\frak k}\wedgedot L\alpha,L|_{\frak k}\cdot l\alpha)=-(L|_{\frak k}\wedgedot\L^A_\alpha\gamma, [\gamma(\rho\alpha),\gamma]),
\]
for any $\alpha\in\Gamma(A)$, where $L|_{\frak k}=[-,\gamma]$. To check that the symbols coincide, just note that $[\gamma,\gamma](X,Y)=2[\gamma(X),\gamma(Y)]$ holds for any $X,Y\in\vf(M)$, hence
\[
j(\alpha)=-\frac 12\big(\iota_{\rho(\alpha)}[\gamma,\gamma]\big).
\]
For the leading term, we compute 
\begin{align*}
-J\alpha(X,Y)&=[(\L^A_\alpha\gamma)(X),\gamma(Y)]-[(\L^A_\alpha\gamma)(Y),\gamma(X)]\\
&=[[\alpha,\gamma(X)],\gamma(Y)]-[\gamma[\rho\alpha,X],\gamma(Y)]-[[\alpha,\gamma(Y)],\gamma(X)]+[\gamma[\rho\alpha,Y],\gamma(X)]\\
&=[\alpha,[\gamma(X),\gamma(Y)]]-[\gamma[\rho\alpha,X],\gamma(Y)]-[\gamma(X),\gamma[\rho\alpha,Y]],
\end{align*}
where  the first and third term in the second line were combined using the Jacobi identity. Now using the identity for $[\gamma,\gamma]$ on all three terms shows that this equals
\[
-J\alpha(X,Y)=\frac 12\L^A_\alpha[\gamma,\gamma](X,Y),
\]
which concludes our proof.
\end{proof}
\begin{remark}
\label{rem:deformation_curving}
The proof shows that by choosing a curving $F$ of an IM connection $(\C,v)$, we also choose a curving for all connections of the form $(\C,v)+\delta^0\gamma$:
\begin{align}
\label{eq:curving_deformation}
F^\gamma&=F+\d{}^\nabla\gamma-\frac 12[\gamma,\gamma].
\end{align}
Moreover, the linear connection on $\frak k$ induced by $(\C,v)+\delta^0\gamma$ (denoted $\nabla^\gamma$) reads
\begin{align}
\label{eq:nabla_gamma}
\nabla^\gamma\xi=\nabla\xi+[\xi,\gamma],
\end{align}
for all $\xi\in\Gamma(\frak k)$. Remarkably, the 3-curvature does not change with this deformation:
\begin{align}
\begin{split}
  \label{eq:3_curvature_invariant}
  G^\gamma&=\d{}^{\nabla^\gamma}F^\gamma=(\d{}^\nabla+[\cdot,\gamma])\big(F+\d{}^\nabla\gamma-\frac 12[\gamma,\gamma]\big)\\
  &=G+R^\nabla\wedge\gamma-[\d{}^\nabla\gamma,\gamma]+[F,\gamma]+[\d{}^\nabla\gamma,\gamma]-\frac 12[[\gamma,\gamma],\gamma]=G,
\end{split}
\end{align}
where we have used \eqref{eq:R_ad_F} and observed that the Jacobi identity for $[\cdot,\cdot]_{\frak k}$ implies $[[\gamma,\gamma],\gamma]=0$. This observation will be of great importance in \sec\ref{sec:multiplicative_ym_action}.
\end{remark}
We now inspect primitive multiplicative connections in various important special cases.
\subsubsection{The transitive case}
In this subsection, we will prove that on a transitive algebroid $A$ with bundle of ideals $\frak k=\ker\rho$, any IM connection is uniquely primitive. We first consider a slightly more general context.

\begin{definition}
Let $V\ra M$ be a representation of a regular Lie algebroid $A\Ra M$. The \textit{orbital projection} of horizontal IM forms (where $\frak k=\ker\rho$) to foliated forms on the base is defined as
\begin{align}
\begin{split}
  &T\colon \Omega_{im}^k(A;V)^\Hor\rightarrow \Omega^k(T\F;V)\\
  &T(L,l)(v_1,\dots, v_k)=l(\tilde v_1)(v_2,\dots,v_k),
\end{split}
\end{align}
where $\tilde v_1$ is any lift of $v_1$ along the anchor. 
\end{definition}
Independence of the choice of lift follows from horizontality. Moreover, condition \eqref{eq:c3} ensures that the result is indeed an antisymmetric tensor field.  We observe that the following diagram commutes.
\[\begin{tikzcd}[row sep=large]
	{\Omega_{im}^k(A;V)^\Hor} \\
	{\Omega^k(M;V)} & {\Omega^k(T\F;V)}
	\arrow["T", from=1-1, to=2-2]
	\arrow["{\delta^0}", from=2-1, to=1-1]
	\arrow[from=2-1, to=2-2]
\end{tikzcd}\]
Here, the arrow at the bottom denotes the restriction map, so in particular, $T$ is surjective. Using the map $T$, we can now prove the statement from Example \ref{ex:transitive_primitive}.
\begin{proposition}
\label{prop:orb_proj2}
Let $V$ be a representation of a transitive Lie algebroid $A\Ra M$. The map
\begin{align*}
  \delta^0\colon \Omega^k(M;V)\rightarrow \Omega^k_{im}(A;V)^\Hor
\end{align*}
is an isomorphism with inverse $T$. Hence, $H^{0,\bullet}(A;\frak k)=\ker\delta^0=0$ and $H^{1,\bullet}(A;\frak k)^\Hor=0$.
\end{proposition}
\begin{remark}
  Since all IM forms of degree $q\geq 2$ on a transitive algebroid are horizontal (see Example \ref{ex:horizontal_p=1}), this gives a canonical isomorphism $\Omega^q_{im}(A;\frak k)\cong \Omega^q(M;\frak k)$ for all $q\geq 2$. 
\end{remark}

\begin{proof}
That the map $T$ is a left inverse of $\delta^0$ follows from the diagram above. We need to show that in the transitive case, it is also its right inverse. We will denote $\gamma=T(L,l)$, so we need to show
$
\delta^0\gamma=(L,l).
$
First we show that the symbols coincide:
\begin{align*}
(\delta^0\gamma)_1(\alpha)(v_1,\dots,v_{k-1})=\gamma(\rho(\alpha),v_1,\dots,v_{k-1})=l(\alpha)(v_1,\dots,v_{k-1}),
\end{align*}
for any $\alpha\in \Gamma(A)$ and vector fields $v_i\in \vf(M)$. For the leading term, we compute
\begin{align}
(\L_\alpha^A\gamma)(v_1,\dots,v_k)&=\nabla^A_\alpha\gamma(v_1,\dots,v_k) - \textstyle\sum_{i=1}^k\gamma(v_1,\dots,[\rho(\alpha),v_i],\dots,v_k)\nonumber\\
&=\nabla^A_\alpha l(\tilde v_1)(v_2,\dots,v_k)+\textstyle\sum_{i=1}^k(-1)^i l[\alpha,\tilde v_i](v_1,\dots,\hat v_i,\dots,v_k)\label{eq:prf_coh_triv_T},
\end{align}
where we have used the definition of $\gamma$ in the last line. Using the definition of $\L^A$ on the first term, this becomes
\begin{align*}
\nabla^A_\alpha l(\tilde v_1)(v_2,\dots,v_k)&=\L_\alpha^A l(\tilde v_1)(v_2,\dots,v_k)+\textstyle\sum_{i=2}^k l(\tilde v_1)(v_2,\dots,[\rho(\alpha),v_i],\dots,v_k)\\
&=\L_\alpha^Al(\tilde v_1)(v_2,\dots,v_k)-\textstyle\sum_{i=2}^k (-1)^i l[\alpha,\tilde v_i](v_1,\dots,\hat v_i,\dots, v_k),
\end{align*}
and now notice that the second term cancels out with all but the first term of the sum in the expression \eqref{eq:prf_coh_triv_T}, so we obtain
\begin{align*}
  (\L_\alpha^A\gamma)(v_i)_i&=\L_\alpha^Al(\tilde v_1)(v_2,\dots,v_k)-l[\alpha,\tilde v_1](v_2,\dots,v_k),
\end{align*}
which equals $L(\alpha)(v_1,\dots,v_k)$ by condition \eqref{eq:c2}.
\end{proof}
\begin{example}
Applying $T$ to the curvature $\Omega^{(\C,v)}$ of an IM connection $(\C,v)$, we get 
\begin{align*}
T\Omega^{(\C,v)}(X,Y)=-U(\sigma(X))Y=-v[\sigma(X),\sigma(Y)]=\sigma[X,Y]-[\sigma(X),\sigma(Y)],
\end{align*}
for any $X,Y\in\Gamma(T\F)$, where $\sigma$ is the horizontal lift pertaining to the splitting $v\colon A\ra \frak k$ of the short exact sequence \eqref{eq:ses}, and we have used the identity \eqref{eq:U_along_orbits}. As expected, the orbital projection maps the curvature of an IM connection $(\C,v)$ to the curvature $F^\sigma$ of the splitting $\sigma$.
\end{example}

\subsubsection{Lie algebroids of principal type}
  Let $A\Ra M$ be an arbitrary algebroid with a bundle of ideals $\frak k$ and suppose we are given a primitive connection $(\C,v)\in\A(A;\frak k)$ with a curving $F$. Observe that the curvature 3-form $G=\d{}^\nabla F$ vanishes if and only if the direct sum $A'\coloneqq TM\oplus \frak k$ is a Lie algebroid, with anchor $\pr_1\colon A'\ra TM$ and the Lie bracket given by
  \begin{align}
    \label{eq:principal_bracket}
    [(X,\xi),(Y,\eta)]=([X,Y],\nabla_X\eta-\nabla_Y\xi+[\xi,\eta]_{\frak k}-F(X,Y)).
  \end{align}
  The vanishing of $G$ amounts precisely to the vanishing Jacobiator for this bracket. In this case, $A'$ is clearly a transitive algebroid, and furthermore, the algebroid $A\Ra M$ is isomorphic to a Lie algebroid \textit{of principal type} \cite{mec}*{\sec 6.6},
  \[
  A'\times_{TM} B \Ra M,\quad A'\times_{TM} B \coloneqq\set{(\alpha,\beta)\in A'\times B\given \rho_{A'}(\alpha)=\rho_B(\beta)}.
  \]
  Its algebroid structure is determined by the condition that the inclusion $A'\times_{TM}B\hookrightarrow A'\times B$ is a Lie algebroid monomorphism into the direct product algebroid. The isomorphism reads \[\alpha\mapsto (\rho(\alpha),v(\alpha), \phi(\alpha)),\] where $\phi\colon A\ra B=A/\frak k$ is the canonical projection. Hence, the existence of a primitive connection for $\frak k$, with a curving whose curvature 3-form vanishes, forces the algebroid to be of principal type. 
  
  This observation is closely related to the construction of IM connections on algebroids of principal type from \cite{mec}*{\sec 6.6}. To elaborate, let $A\coloneqq A'\times_{TM}B$ be an algebroid of principal type for some pair of Lie algebroids $A'$ and $B$, with $A'$ transitive. Denote by $\phi\colon A\ra B$ the projection and consider the bundle of ideals $\frak k=\ker\phi$, which can be identified with $\ker\rho_{A'}$ under the projection $\pr_{A'}\colon A\ra A'$. Recall that a splitting of the Atiyah sequence of $A'$,
  \begin{align}
    \label{eq:splitting_principal_type}
      \begin{tikzcd}[ampersand replacement=\&, column sep=large]
      0 \& {\frak k} \& {A'} \& {TM} \& 0,
      \arrow[from=1-1, to=1-2]
      \arrow[from=1-2, to=1-3]
      \arrow["v'"{pos=0.45}, bend left=30, from=1-3, to=1-2]
      \arrow["\rho_{A'}", from=1-3, to=1-4]
      \arrow[from=1-4, to=1-5]
    \end{tikzcd}
    \end{align}
  is the same thing as an IM connection $(\C',v')\in\A(A';\frak k)$ by condition \eqref{eq:c2}, which can then be pulled back along $\pr_{A'}$ to an IM connection $(\C,v)\in\A(A;\frak k)$:
\begin{align}
  \label{eq:principal_type_IM_connection_induced_by_splitting}
      \C(\alpha,\beta)=\C'(\alpha),\quad v(\alpha,\beta)=v'(\alpha),
\end{align}
  for any $\alpha\in\Gamma(A')$ and $\beta\in\Gamma(B)$ with $\rho_{A'}(\alpha)=\rho_B(\beta)$. In this way, the authors of \cite{mec} conclude that Lie algebroids of principal type admit IM connections. As the following shows, the obtained IM connection is actually primitive.
  \begin{proposition}
    Let $A=A'\times_{TM} B$ be a Lie algebroid of principal type, with the bundle of ideals $\frak k=\ker\phi$, where $\phi\colon A\ra B$ is the canonical projection. The IM connection \eqref{eq:principal_type_IM_connection_induced_by_splitting} induced by a splitting $v'\colon A'\ra \frak k$ of the transitive algebroid $A'$ is primitive with curving $F^{v'}$, whose curvature 3-form vanishes.
  \end{proposition}
  \begin{proof}
    First observe that the curving of the IM connection \eqref{eq:principal_type_IM_connection_induced_by_splitting} is given by the curvature $F^{v'}$ of the splitting $v'\colon A'\ra \frak k$. Indeed, we have
  \[
  \Omega^{(\C,v)}(\alpha,\beta)=\Omega^{(\C',v')}(\alpha)=\delta^0_{A'}(F^{v'})(\alpha)=(\L^{A'}_{\alpha}F^{v'},\iota_{\rho_{A'}(\alpha)}F^{v'})
  \]
  where the first equality holds since $(\C,v)=(\pr_{A'})^*(\C',v')$. Now observe that $\rho(\alpha,\beta)=\rho_{A'}(\alpha)$ and the representation of $A$ on $\frak k$ is given for any $\xi\in\Gamma(\frak k)$ as
  \[
  [(\alpha,\beta),(\xi,0)]=([\alpha,\xi],0),
  \]
  so we conclude $\smash{\L^A_{(\alpha,\beta)}=\L^{A'}_\alpha}$ and hence $\Omega^{(\C,v)}=\delta^0_A(F^{v'})$. In conclusion, any IM connection on a Lie algebroid of principal type $A'\times_{TM}B$, induced by a splitting $v'$ of \eqref{eq:splitting_principal_type}, is primitive with a canonical curving $F^{v'}$ satisfying $\d{}^\nabla F^{v'}=0$ by the Bianchi identity.
  \end{proof}


\subsubsection{The semisimple case}
\label{sec:semisimple}
A particularly well-behaved scenario for the theory of multiplicative connections is when the typical fibre of $\frak k$ is a semisimple Lie algebra, in which case an IM connection on $A$ for $\frak k$ always exists. This was established in \cite{mec}*{Corollary 6.6} by showing that $A$ is then isomorphic to an algebroid of principal type, with $\frak k$ the isotropy of a transitive algebroid $A'$, whose sections are bracket-preserving derivations of $\frak k$. We now establish some stronger properties.

In what follows, we will refer to a locally trivial bundle of Lie algebras $\frak k$ with a semisimple typical fibre as a \textit{semisimple Lie algebra bundle}.\footnote{When the base $M$ is connected and all the fibres are semisimple, local triviality is automatic by the rigidity results for bundles of semisimple Lie algebras.} In the case when $\frak k\subset A$ is also a bundle of ideals, it will be called a \textit{semisimple bundle of ideals} of $A$. As concerns cohomological triviality, we have the following.
\begin{proposition}
\label{prop:primitive_semisimple}
Let $\frak k$ be a semisimple bundle of ideals of a Lie algebroid $A\Ra M$. The following holds:
\begin{enumerate}[label={(\roman*)}]
  \item The simplicial differential $\delta^0\colon \Omega^\bullet(M;\frak k)\ra \Omega^\bullet_{im}(A;\frak k)^\Hor$ is an isomorphism.
  \item Any multiplicative connection for $\frak k$ is uniquely primitive, with vanishing  3-curvature.
\item There is a bijective correspondence:\\
\begin{minipage}{\linewidth}
\begin{align*}
\left\{
\parbox{3cm}{\centering IM connections $\A(A;\frak k)$}
\right\}
\longleftrightarrow
\left\{
\parbox{8.3cm}{\centering Pairs $(v,
\nabla)$, where $v\colon A\ra \frak k$ is a splitting and\\$\nabla$ is a bracket-preserving connection on $\frak k$,\\such that $\nabla^A_{h(\alpha)}=\nabla^{}_{\rho(\alpha)}$ for all $\alpha\in A$}
\right\}
\end{align*}
\end{minipage}
\end{enumerate}
\end{proposition}
\begin{remark}
As will be clear from the proof, it is actually enough to assume that \[\ad\colon \frak g\rightarrow \Der(\frak g)\] is an isomorphism, where $\frak g$ is the typical fibre of $\frak k$. In other words, (i) states that the vanishing of the zeroth and first Chevalley--Eilenberg cohomology groups of the typical fibre $\frak g$, with values in the adjoint representation of $\frak g$, implies the vanishing of the zeroth and first horizontal simplicial cohomology groups of $\frak k$-valued IM forms on $A\Ra M$,
\begin{align*}
  H^{0,\bullet}(A;\frak k)^\Hor=H^{1,\bullet}(A;\frak k)^\Hor=0.
\end{align*}
\end{remark}
\begin{proof}
The center $z(\frak k)$ of a semisimple Lie algebra bundle $\frak k$ is trivial, so $\delta^0$ is injective by Remark \ref{rem:invariant}. For surjectivity, we take any $(L,l)\in\Omega^\bullet_{im}(A;\frak k)^\Hor$ and observe that $L|_{\frak k}$ is tensorial by horizontality. By condition \eqref{eq:c1}, it is moreover a derivation-valued form $L|_{\frak k}\in\Omega^\bullet(M;\Der\frak k)$, i.e., its values are derivations of the bracket $[\cdot,\cdot]_{\frak k}$, hence semisimplicity implies there is a unique $\gamma\in\Omega^\bullet(M;\frak k)$ which satisfies $L|_{\frak k}=-\ad \gamma$. That is,
\begin{align}
\label{eq:semisimple_inverse}
[\xi,\gamma]=L|_{\frak k}\cdot \xi,
\end{align}
for any $\xi\in\Gamma(\frak k)$. To see that $\delta^0\gamma=(L,l)$, we let $\alpha\in\Gamma(A)$ and apply $\iota_{\rho(\alpha)}$ to both sides of \eqref{eq:semisimple_inverse}. By condition \eqref{eq:c2}, there holds
\[
[\xi,\iota_{\rho(\alpha)}\gamma]=\iota_{\rho(\alpha)}L\xi=\L_\xi^A l\alpha-l[\xi,\alpha]=[\xi,l\alpha],
\]
where we have used horizontality of $(L,l)$. This shows that the symbols coincide. Finally, for the leading term we use condition \eqref{eq:c1} to get
\begin{align*}
[\xi,L\alpha]=\L_\xi^A L\alpha=L[\xi,\alpha]+\L_\alpha^A L\xi=[[\xi,\alpha],\gamma]+\L^A_\alpha[\xi,\gamma].
\end{align*}
It is now easy to see that the right-hand side equals $[\xi,\L_\alpha^A\gamma]$ by the Jacobi identity, concluding part (i). Part (ii) is clearly implied by (i): applying equation \eqref{eq:semisimple_inverse} to the curvature $\Omega^{(\C,v)}$ of an IM connection $(\C,v)$ yields the unique curving $F$, implicitly defined by 
\begin{align}
  \label{eq:semisimple_def_F}
  R^\nabla=-\ad F.
\end{align} 
For part (iii), if we are given a pair $(v,\nabla)$, there is a unique form $F$ satisfying \eqref{eq:semisimple_def_F}, so we would now like to apply Corollary \ref{cor:primitive_bijective_correspondence}. The condition (iii) there is automatically fulfilled since $\d{}^\nabla F$ is $z(\frak k)$-valued:
\begin{align*}
[\xi,\d{}^\nabla F]=\d{}^\nabla[\xi,F]-[\nabla\xi,F]=\d{}^\nabla(R^\nabla\cdot \xi)-R^\nabla\wedge\nabla\xi=(\d{}^{\nabla^{\End\frak k}}R^\nabla)\cdot \xi=0
\end{align*}
where we used the Bianchi identity for $R^\nabla$ in the last step. Moreover, the condition $\nabla^A_{h\alpha}=\nabla_{\rho\alpha}$ implies $(\rho_B)^*R^\nabla=-\ad F^v$ already in the non-semisimple case, and now inserting \eqref{eq:semisimple_def_F} into this equation yields $(\rho_B)^*F=F^v$, which is what we needed to show.
\end{proof}
The last part of the proof above shows that the coupling construction from Proposition \ref{prop:splitting_coh_triv} significantly simplifies in the semisimple case: we just need to assume that $\nabla$ is bracket-preserving, and the other two conditions are automatically satisfied by defining $F$ with $R^\nabla$. However, the existence of a bracket-preserving connection $\nabla$ is equivalent to local triviality of $\frak k$, which is already assumed, so we obtain the following.
\begin{corollary}
\label{cor:semisimple_extension}
If $B\Ra M$ is a Lie algebroid and $(\frak k,[\cdot,\cdot]_{\frak k})$ is a semisimple Lie algebra bundle, then $A=B\oplus \frak k$ has a structure of a Lie algebroid defined by \eqref{eq:bracket_construction}, where $\nabla$ is any bracket-preserving linear connection on $\frak k$ and the form $F$ is implicitly defined by $R^\nabla=-\ad F$.  Conversely, any Lie algebroid $A$ with a semisimple bundle of ideals $\frak k$ is isomorphic to one of this type.
\end{corollary}

\subsubsection{The abelian case}
We now establish a criterion for the existence of primitive IM connections for an abelian bundle of ideals $\frak k\subset A$. Preliminarily, observe that when $\frak k$ is abelian, the representation $\nabla^A$ descends to a  representation $\nabla^B$ of $B$ on $\frak k$, given by $\nabla^B_\alpha\xi=[\tilde\alpha,\xi]$, where $\tilde\alpha\in\Gamma(A)$ denotes any lift of $\alpha\in\Gamma(B)$ along the projection $A\ra B$, and $\xi\in\Gamma(\frak k)$.

Now suppose $(\C,v)\in \arc(A;\frak k)$ is a primitive connection and $F$ is its curving. By identity \eqref{eq:R_ad_F}, $\nabla$ is now necessarily flat, and moreover, equation \eqref{eq:conn_orb2} now just says that $\nabla$ induces $\nabla^B$, that is, 
\begin{align}
\label{eq:nabla_induces}
\nabla^B_\alpha=\nabla_{\rho(\alpha)}.
\end{align}
 Let us consider the set of 2-forms $F$ with \textit{transverse coboundary} $\d{}^\nabla F$, that is,
\[
\Omega^2_{tc,\nabla}(M;\frak k)\coloneqq \set{F\in \Omega^2(M;\frak k)\given \iota_X\d{}^\nabla F=0\text{ for any $X\in T\F$}}.
\]
Let $(\rho_B)^*\colon \Omega^2(M;\frak k)\ra \Omega^2(B;\frak k)$ denote the pullback along the anchor of $B$, and observe that there holds $(\rho_B)^*\d{}^\nabla=\d{}^{\nabla^B}(\rho_B)^*$ by equation \eqref{eq:nabla_induces}, so that the pullback restricts to
\[
(\rho_B)^*_{tc,\nabla}\colon \Omega_{tc,\nabla}^2(M;\frak k)\ra Z^2(B;\frak k)
\]
where $Z^2(B;\frak k)$ denotes $\frak k$-valued 2-cocycles on $B$ with respect to the differential $\d{}^{\nabla^B}$. As stated in equation \eqref{eq:Fv_pullback_F}, the map above takes $F$ to the curvature of the splitting $F^v$. Moreover, in the usual Lie algebroid cohomology  of $B$ with values in $\frak k$, the class 
\[c(A)\coloneqq [F^v]\in H^2(B;\frak k)\] 
is independent of the splitting $v\colon A\ra \frak k$ since $\frak k$ is abelian. It is thus in the image of the map $\Omega_{tc,\nabla}^2(M;\frak k)\ra H^2(B;\frak k)$ induced by $(\rho_B)^*_{tc,\nabla}$, which we denote by the same symbol. With this in mind, we formulate the following.
\begin{proposition}
\label{prop:abelian_primitive}
An abelian bundle of ideals $\frak k\subset A$ admits a primitive connection if and only if there exists a flat connection $\nabla$ on $\frak k$, which induces $\nabla^B$ and satisfies \[c(A)\in\im(\rho_B)^*_{tc,\nabla}.\] Moreover, we have the following bijective correspondence.
\begin{align*}
\left\{
\parbox{4.9cm}{\centering Primitive connections for $\frak k$\\with a choice of curving}
\right\}
\longleftrightarrow
\left\{
\parbox{8.3cm}{\centering Triples $(v,
\nabla,F)$, where $v\colon A\ra \frak k$ is a splitting,\\$\nabla$ is a  flat connection on $\frak k$ inducing $\nabla^B$, and\\$F\in\Omega^2_{tc,\nabla}(M;\frak k)$ satisfies $F^v=(\rho_B)^*F$.}
\right\}
\end{align*}
\end{proposition}
\begin{proof}
One direction for the statement about existence is already proved above. For the other direction, suppose there is a flat connection $\nabla$ that induces $\nabla^B$ and satisfies $c(A)\in\im(\rho_B)^*_{tc,\nabla}$. Since $\frak k$ is abelian, the latter means precisely that there exists a splitting $v\colon A\ra \frak k$ with curvature 
\begin{align}
\label{eq:temp_abelian_prop}
F^v=(\rho_B)^* F,
\end{align}
for some $F\in\Omega^2(M;\frak k)$ with $\iota_{\rho_B(\alpha)}\d{}^\nabla F=0$ for any $\alpha\in B$. The pair $(\nabla, F)$ satisfies the assumptions of Proposition \eqref{prop:splitting_coh_triv}, and equation \eqref{eq:temp_abelian_prop} ensures that the Lie algebroid $A$ is isomorphic to the obtained algebroid $B\oplus \frak k$, which thus has an IM connection with curving $F$. The statement regarding the bijective correspondence is clear from this argument. 
\end{proof}
\begin{remark}
The last proposition is related to the criterion for existence of kernel flat connections on an abelian bundle of ideals $\frak k$ from \cite{mec}*{Proposition 5.21}. In fact, Proposition \ref{prop:abelian_primitive} is its refinement, and their relation can be precisely summarized with the following diagram.
\[\begin{tikzcd}[row sep=large]
	{H_{im}^2(B;\frak k)} & {H^2(B;\frak k)} \\
	{\Omega_{tc,\nabla}^2(M;\frak k)}
	\arrow["\operatorname{sym}", from=1-1, to=1-2]
	\arrow["{\delta^0}", from=2-1, to=1-1]
	\arrow["{(\rho_B)^*_{tc,\nabla}}"', from=2-1, to=1-2]
\end{tikzcd}\]
Here, $H_{im}^\bullet(B;\frak k)$ denotes the cohomology of the cochain complex $(\Omega^\bullet_{im}(B;\frak k),\d{}^\nabla_{im})$, where the differential $\d{}^\nabla_{im}$ is the restriction to IM forms of $\d{}^\nabla$ as in equation \eqref{eq:ext_cov_der_p=1}, and \[\operatorname{sym}(c_0,c_1)(\alpha,\beta)=c_1(\alpha)(\rho \beta).\]
\end{remark}

\subsection{Multiplicative Yang--Mills action}
\label{sec:multiplicative_ym_action}
We now arrive to the last main result of this thesis: the desired generalization of Yang--Mills theory to the multiplicative setting. Equipped with the theory developed in the previous sections, the idea now is to define the desired action functional on triples $((\C,v),F)$, where $(\C,v)\in\arc(A;\frak k)$ is a (primitive) IM connection and $F\in\Omega^2(M;\frak k)$ its curving. We recall from Corollary \ref{cor:primitive_bijective_correspondence} that the information stored in such a triple may alternatively be described by $(v,\nabla,F)$, where $v\colon A\ra \frak k$ is a splitting $A=B\oplus \frak k$, $\nabla$ is a connection on $\frak k$, and $F\in\Omega^2(M;\frak k)$. The compatibility conditions for such a triple read:
\begin{multicols}{2}
\begin{enumerate}
  \item $\nabla$ preserves the Lie bracket $[\cdot,\cdot]_{\frak k}$ on $\frak k$.
  \item $R^\nabla=-\ad F$.
  \item The 3-curvature $G=\d{}^\nabla F$ is transversal:\\$\iota_{X}G=0$ for all $X\in T\F$.
  \item $\nabla_{\rho(\alpha)}=[h(\alpha),\cdot]$, for any $\alpha\in A$.
  \item $F^v=(\rho_B)^*F$.
  \item[\vspace{\fill}] \phantom{}\\\phantom{}
\end{enumerate}
\end{multicols}
\noindent The conditions (i)--(iii) mean $\nabla$ and $F$ are compatible, and they generalize the foliated case (compare with Proposition \ref{prop:foliated_props}), while the conditions (iv) and (v) mean $\nabla$ and $F$ are mutually compatible with the splitting $v$. To construct a Yang--Mills theory for such triples, we must assume certain data. From Definition \ref{def:ym_algebroid}, recall the \textit{Yang--Mills data} for a bundle of ideals $\frak k$ consists of:
\begin{enumerate}[label={(\roman*)}]
  \item An $\ad$-invariant metric $\kappa=\inner\cdot\cdot_{\frak k}$ on the bundle of ideals $\frak k$. Assuming IM connections exist, this implies the typical fibre of $\frak k$ is a compact Lie algebra \cite{duistermaat}*{Theorem 3.6.2}.
  \item A Riemannian metric $g=\inner\cdot\cdot$ on the oriented base manifold $M$.
\end{enumerate}
As we will see, the space of primitive IM connections needs to be further restricted:
\begin{itemize}[leftmargin=9mm]
    \item As with foliated Yang--Mills theory, if we do not assume the base manifold $M$ is compact, then it is necessary to focus on curvings $F$ that are compactly supported, $F\in\Omega_c^2(M;\frak k)$. For simplicity, we will from now on just assume $M$ is compact.
    \item We will restrict to the IM connections $(\C,v)$ whose induced linear connection $\nabla=\C|_{\frak k}$ is compatible with $\inner\cdot\cdot_{\frak k}$.
    This is to ensure that the formal adjoint of the exterior covariant derivative can be expressed simply with the Hodge star operator, see equation \eqref{eq:delta_nabla} below.
\end{itemize}
\begin{remark}[Transversal metric compatibility]
  In the last item, it is enough to assume $\nabla$ is \emph{transversally compatible} with the given metric $\inner{\cdot}{\cdot}_{\frak k}$ on $\frak k$, i.e., that every point $x\in M$ admits tangent vectors $(X_i)_i$ which induce a basis of $N_x\F=T_xM/\im\rho_x$, satisfying
    \begin{align}
      \label{eq:transversal_compatibility}
        X_i\inner{\xi}{\eta}_{\frak k}=\inner{\nabla_{X_i}\xi}{\eta}_{\frak k} + \inner{\xi}{\nabla_{X_i}\eta}_{\frak k}
    \end{align}
    for any $\xi,\eta\in\Gamma(\frak k)$ and $i=1,\dots,\codim_x\F$. This is due to Proposition \ref{prop:conn2} and $\ad$-invariance of $\kappa$, which together imply that $\nabla=\C|_{\frak k}$ is already compatible with $\inner\cdot\cdot_{\frak k}$ in the orbital directions.
\end{remark}

Similarly to the foliated case, assuming a Yang--Mills data is given, the metric and orientation on $M$ ensure we can introduce an inner product on  $\frak k$-valued differential forms on $M$:
\begin{align*}
  \innerr\cdot\cdot_{\frak k}\colon \Omega^k(M;\frak k)\times \Omega^k
 (M;\frak k)\ra \R,\quad \innerr\alpha\beta_{\frak k}=\int_M\inner\alpha\beta_{\frak k}\vol_M.
\end{align*}
It is now a standard result that compatibility of $\nabla$ with $\kappa$ implies that the formal adjoint to $\d{}^\nabla$ on $\frak k$-valued forms of degree $k$ reads
\begin{align}
  \label{eq:delta_nabla}
  \delta^\nabla=(-1)^k\star^{-1}\!{\d{}^\nabla{\star}},
\end{align}
where $\star$ is the Hodge star operator with respect to the chosen metric and orientation on $M$. This is proved similarly (and in fact, more easily) as in Proposition \ref{prop:codiff}.

\begin{definition}
  \label{def:multiplicative_ym}
  Let $\frak k$ be a bundle of ideals of a Lie algebroid $A\Ra M$, and assume a Yang--Mills data for $\frak k$ is given. Let us denote
  \begin{align*}
    \DD(A;\frak k)=\set*{\big((\C,v),F\big)\in\A(A;\frak k)\times \Omega^2(M;\frak k)\given 
      \delta^0F=\Omega^{(\C,v)}
    }.
  \end{align*}
  The \emph{multiplicative Yang--Mills action functional} is defined as the map 
    \begin{align}
      \label{eq:multiplicative_ym_action}
      &\S\colon \DD(A;\frak k)\ra \R,\quad \S((\C,v),F)=\int_M\innersmall{F}{F}_{\frak k}\,\vol_M+\mu\int_M\innersmall{G}{G}_{\frak k}\,\vol_M,
    \end{align}
    where $G=\d{}^\nabla F$ is the 3-curvature of $F$ and $\mu\in\R$ is an a priori chosen constant, called the \emph{structure constant} of the theory. We now introduce two notions of criticality for the action functional $\S$, together with another relevant notion. A triple $((\C,v),F)$ is said to be:
    \begin{enumerate}[label={(\roman*)}]
      \item \emph{Longitudinally critical}, if 
      \begin{align}
        \label{eq:longitudinally_critical_im_connections}
        \deriv\lambda0\S\big((\C,v)+\lambda \delta^0\gamma, F^{\lambda\gamma}\big)=0,\quad \text{for all $\gamma\in\Omega^1(M;\frak k)$,}
      \end{align}
      where we have denoted $F^{\lambda\gamma}=F+\lambda\d{}^\nabla \gamma-\frac {\lambda^2}2 [\gamma,\gamma]\in\Omega^2(M;\frak k)$.
      \item \emph{Transversally critical}, if 
      \begin{align}
        \label{eq:transversally_critical_im_connections}
        \deriv\lambda0\S\big((\C,v), F+\lambda\beta\big)=0,\quad \text{for all $\beta\in\Omega_{\inv}^2(M;\frak k)$.}
      \end{align}
      \item \emph{Adapted} to the given Riemannian metric and orientation on $M$, if the 2-form $-\mu\delta^\nabla G$ is also a curving of $(\C,v)$, that is,
      \begin{align}
        \label{eq:adapted}
        \delta^0(\delta^\nabla G)=-\frac 1\mu\Omega^{(\C,v)}.
      \end{align}
      This notion is motivated by the upcoming Theorem \ref{thm:ym_multiplicative}. 
    \end{enumerate}

\end{definition}
\begin{remark}[Well-definedness]
  The expression on the left-hand side of equation \eqref{eq:longitudinally_critical_im_connections} is well-defined: by Proposition \ref{prop:affine_primitive} and equation \eqref{eq:curving_deformation}, the deformed IM connection $(\C,v)+\lambda\delta^0\gamma$ is again primitive with curving $F^{\lambda\gamma}$. Well-definedness of the left-hand side of \eqref{eq:transversally_critical_im_connections} is clear.
\end{remark}
\begin{remark}[Decompositions into affine spaces]
  \label{rem:decompositions_affine}
  For intuition regarding the definitions of criticality \eqref{eq:longitudinally_critical_im_connections} and \eqref{eq:transversally_critical_im_connections}, we note that the set $\DD(A;\frak k)$ carries two relevant equivalence relations:
  \begin{align*}
    ((\tilde\C,\tilde v),\tilde F)\sim_1 ((\C,v),F)&\iff (\tilde \C,\tilde v)-(\C,v)=\delta^0\gamma\text{ and }\tilde F=F^{\gamma}\text{ for some }\gamma \in\Omega^1(M;\frak k),\\
    ((\tilde\C,\tilde v),\tilde F)\sim_2 ((\C,v),F)&\iff (\tilde\C,\tilde v)=(\C,v)\text{ and }\delta^0(\tilde F- F)=0.
  \end{align*}
  The equivalence classes of $\sim_1$ and $\sim_2$ are affine spaces, modelled on cohomologically trivial IM 1-forms, and invariant 2-forms, respectively. Each of the two notions of criticality takes into account only the directions which are tangential to the equivalence classes defined by respective relation. We note that the closure of the union of $\sim_1$ and $\sim_2$ yields another equivalence relation:
  \begin{align}
    ((\tilde\C,\tilde v),\tilde F)\sim_3 ((\C,v),F)&{\ \iff\ } [(\tilde \C,\tilde v)]=[(\C,v)]\in H^{1,1}(A;\frak k)\label{eq:tilde_3}
    \\
    &{\ \iff\ } (\tilde \C,\tilde v)-(\C,v)=\delta^0\gamma\text{ and }\tilde F=F^{\gamma}+\beta, \nonumber\\
    &\phantom{\ \iff\ }\text{for some }\gamma \in\Omega^1(M;\frak k), \beta\in\Omega^2_{\inv}(M;\frak k),\nonumber
  \end{align}
  Intuitively, the equivalence classes of $\sim_3$ define the tangential directions which are altogether accounted for by both notions of criticality. We note the intersection of $\sim_1$ and $\sim_2$ is not trivial; the role of this fact will be clarified in \sec\ref{sec:constrained_variational_problem}.
\end{remark}
\begin{remark}[Simplification of the action]
  Let us suppose that invariant 3-forms on the base vanish, $H^{0,3}(A;\frak k)=0$. For instance, this is automatically fulfilled if either: 
   \begin{itemize}
     \item the leaves of the orbit foliation $\F$ are at most of codimension $2$, or
     \item the typical fibre $\frak g$ of $\frak k$ has a vanishing center; this is equivalent to semisimplicity of $\frak g$, since as already noted, the existence of an $\ad$-invariant metric implies  Lie algebra $\frak g$ is compact.
   \end{itemize}
   In this case, the definition above simplifies substantially. Indeed, since the 3-curvature $G$ is an invariant form, the second term of the action \eqref{eq:multiplicative_ym_action} vanishes. Moreover,  the proof of the following theorem shows that in this case, transversal criticality is equivalent to $F\perp \Omega_{\inv}^2(M;\frak k)$. On the other hand, the notion of adaptedness for this case is very restrictive: it implies the IM connection $(\C,v)$ is flat.
 \end{remark}
\begin{theorem}
  \label{thm:ym_multiplicative}
  Let $\frak k$ be a bundle of ideals of a Lie algebroid $A\Ra M$, suppose a Yang--Mills data for $\frak k$ is given, and let $\mu\in\R$ denote the structure constant for the multiplicative Yang--Mills action functional \eqref{eq:multiplicative_ym_action}. Let $((\C,v),F)\in \DD(A;\frak k)$ be a triple whose induced connection $\nabla=\C|_{\frak k}$ is compatible with $\inner\cdot\cdot_{\frak k}$. The following equivalences hold. 
  \begin{enumerate}[label={(\roman*)}]
    \item The triple $((\C,v),F)$ is longitudinally critical if and only if the curving $F$ is a solution to the \textbf{\emph{first Yang--Mills equation}},
    \begin{align}
      \label{eq:ym_1}
    \d{}^{\nabla}{\star\,F}=0.
    \end{align}
    \item The triple $((\C,v),F)$ is both transversally critical and adapted if and only if the pair $(F,G)$ is a solution of the \textbf{\emph{second Yang--Mills equation}},
    \begin{align}
      \label{eq:ym_2}
      \d{}^\nabla{\star\,G}={\tfrac 1\mu}\star F.
    \end{align}
  \end{enumerate}
  \end{theorem}
\begin{remark}
  Collecting the Yang--Mills equations (written with the covariant codifferential $\delta^\nabla$) together with the Bianchi identities and  the defining identities, we have
  \begin{align}
    \label{eq:ym_bianchi_gathered}
    \begin{aligned}
      \delta^\nabla F&=0,&\qquad \d{}^\nabla F&=G,&\qquad \delta^0 F&=\Omega^{(\C,v)},\\
      \delta^\nabla G&=-\tfrac 1\mu F,& \d{}^\nabla G&=0,&\qquad \delta^0 G&=0,
    \end{aligned}
  \end{align}
  for a curving $F$ and its curvature 3-form $G$ of an IM connection $(\C,v)$. We observe that applying $\d{}^\nabla$ to the second Yang--Mills equation yields \[\d{}^\nabla{\star\,F}=\mu[\star\, G, F],\] and hence if $\nabla$ is flat, the second Yang--Mills equation implies the first since flatness of $\nabla$ means the curving $F$ is centre-valued.
  \end{remark}
  \begin{proof}
  For the statement in (i), first observe that by equation \eqref{eq:3_curvature_invariant}, the 3-curvature is invariant under a longitudinal variation of $((\C,v),F)$, hence the second term of the action \eqref{eq:multiplicative_ym_action} will not play a role. For the first term, we use the equation \eqref{eq:curving_deformation} to compute
  \begin{align*}
    \S((\C,v)&+\lambda \delta^0\gamma,F^{\lambda\gamma})=\innerr{F^{\lambda\gamma}}{F^{\lambda\gamma}}_{\frak k}+\mu\innerr{G}{G}_{\frak k}\\
  &=\S(\C,v)+2\lambda\innerr{F}{\d{}^{\nabla}\gamma}_{\frak k}+\lambda^2(\innerr{\d{}^{\nabla}\gamma}{\d{}^{\nabla}\gamma}_{\frak k}-\innerr{F}{[\gamma,\gamma]}_{\frak k})+\mathcal O(\lambda^3),
  \end{align*}
  for any $\gamma\in\Omega^1(M;\frak k)$ and $\lambda\in\R$. Differentiating at $\lambda=0$, we obtain 
  \[
      \deriv\lambda 0\S((\C,v)+\lambda \delta^0\gamma,F^{\lambda\gamma})=2\innerr{F}{\d{}^{\nabla}\gamma}_{\frak k}=2\innerr{\delta^{\nabla}F}{\gamma}_{\frak k}.
  \]
  By the non-degeneracy of $\innerr\cdot\cdot_{\frak k}$, this expression vanishes for all $\gamma\in\Omega^1(M ;\frak k)$ if and only if the curving $F$ satisfies $\delta^\nabla F=0$, which is equivalent to \eqref{eq:ym_1} by identity \eqref{eq:delta_nabla}.

  For the proof of the statement in (ii), we begin by computing
  \begin{align*}
    \S((\C,v),F+\lambda\beta)&=\innerr{F+\lambda\beta}{F+\lambda\beta}_{\frak k}+\mu\innerr{G+\lambda\d{}^\nabla \beta}{G+\lambda\d{}^\nabla \beta}_{\frak k}\\
    &=\S((\C,v),F)+2\lambda(\innerr{F}{\beta}_{\frak k}+\mu\innerr{G}{\d{}^\nabla \beta}_{\frak k})+\lambda^2(\innerr{\beta}{\beta}_{\frak k}+\mu\innerr{\d{}^\nabla \beta}{\d{}^\nabla \beta}_{\frak k})\\
    &=\S((\C,v),F)+2\lambda\innerr{F+\mu\delta^\nabla G}{\beta}_{\frak k}+\lambda^2\innerr{\beta+\mu\delta^\nabla\d{}^\nabla\beta}{\beta}_{\frak k},
  \end{align*}
  for any $\beta\in\Omega^2_{\inv}(M;\frak k)$ and $\lambda\in\R$. Differentiating, we obtain
  \[
    \deriv\lambda 0\S((\C,v),F+\lambda\beta)=2\innerr{F+\mu\delta^\nabla G}{\beta}_{\frak k},
\]
from which we see that transversal criticality of $((\C,v),F)$ is equivalent to $F+\mu\delta^\nabla G$ being orthogonal to the subspace of invariant 2-forms $\Omega_{\inv}^2(M;\frak k)$ with respect to $\innerr\cdot\cdot_{\frak k}$, 
    \[
      (F+\mu\delta^\nabla G)\perp \Omega_{\inv}^2(M;\frak k).
    \]
    By definition, the condition of adaptedness means that the 2-form $-\mu\delta^\nabla G$ is a curving of $(\C,v)$, which means precisely that the form $F+\mu\delta^\nabla G$ is invariant. Since $\innerr{\cdot}{\cdot}_{\frak k}$ is positive-definite, it restricts to an inner product on $\Omega^\bullet_\inv(M;\frak k)$, therefore, adaptedness and transversal criticality together imply $\delta^\nabla G=-\frac 1\mu F$. The converse implication clearly holds, and using the identity \eqref{eq:delta_nabla} concludes the proof.
\end{proof}
\begin{remark}
  \label{rem:laplacian_eigenvalues}
  There is a simple necessary condition on the structure constant $\mu$ for the pair of Yang--Mills equations to admit a solution. Namely, observe that applying $\d{}^\nabla$ to the first equation and adding it to the second, we obtain
  \[
  \Delta F=(\d{}^\nabla\delta^\nabla+\delta^\nabla\d{}^\nabla)F=-\tfrac 1\mu F.
  \]
  In other words, $F$ must be an eigenvector of the Laplacian, with eigenvalue $-\frac 1\mu$; this is in striking contrast with the classical Yang--Mills theory, where the solutions of the Yang--Mills equation are harmonic. On the other hand, it is easy to see that the eigenvalues of $\Delta$ must be non-negative by positive-definiteness of $\innerr\cdot\cdot_{\frak k}$: if $\psi\in\Omega^\bullet(M;\frak k)$ is an eigenvector with eigenvalue $\lambda$,
  \[
    \lambda\innerr{\psi}{\psi}_{\frak k}=\innerr{\Delta \psi}{\psi}_{\frak k}=\innerr{\d{}^\nabla\psi}{\d{}^\nabla\psi}_{\frak k}+\innerr{\delta^\nabla\psi}{\delta^\nabla\psi}_{\frak k}\geq 0.
  \]
  Therefore, the structure constant must be negative, $\mu< 0$, as a necessary condition for the second Yang--Mills equation to admit a solution. Recall that since the Laplacian is a positive, self-adjoint and elliptic operator of second order, it has arbitrarily large eigenvalues (see \cite{foundations_manifolds}*{page 254} and \cite{laplacian_spectre}).
\end{remark}

\begin{remark}[Harmonicity of curvature tensor $R^\nabla$]
  For a longitudinally critical triple $((\C,v),F)$, the curvature tensor $R^\nabla$ is harmonic, that is,
  \[
    (\d{}^{\nabla^{\End\frak k}}\delta^{\nabla^{\End\frak k}}+\delta^{\nabla^{\End\frak k}}\d{}^{\nabla^{\End\frak k}})R^\nabla=0.
  \]
  Indeed, first note that by the Bianchi identity $\d{}^{\nabla^{\End\frak k}}R^\nabla=0$, harmonicity is equivalent to
  \begin{align}
    \label{eq:harmonic_curvature_tensor}
    \d{}^{\nabla^{\End\frak k}}\! \star R^\nabla =0.
  \end{align}
  To establish this identity, we apply the operators $\star$ and $\d{}^\nabla$ consecutively to both sides of the identity $R^\nabla\cdot\xi=[\xi,F]$, where $\xi\in\Gamma(\frak k)$  is an arbitrary section. We obtain
\begin{align*}
  [\nabla\xi,\star\, F]+[\xi,\d{}^\nabla{\star\, F}]&=\d{}^\nabla((\star\, R^\nabla)\cdot\xi)=(\d{}^{\nabla^{\End\frak k}}\!\star R^\nabla)\cdot\xi+(\star\, R^\nabla)\cdot\nabla\xi.
\end{align*}
Notice that the leftmost term and the rightmost term cancel out, hence using the assumption of criticality on the second term on the left-hand side yields the identity \eqref{eq:harmonic_curvature_tensor}.
\end{remark}

\subsubsection{Self-dual and anti self-dual solutions}
  As in Example \ref{ex:self_dual_foliated}, we now discuss a simple class of solutions to the Yang--Mills equations, this time by utilizing the symmetry of the set of equations \eqref{eq:ym_bianchi_gathered}. Let the base manifold $M$ be pseudo-Riemannian and 5-dimensional, and suppose a curving $F$ of an IM connection satisfies $G=c
   \star F$ for some constant $c\in\R-\set 0$. We must then have $F=(-1)^s \frac 1 c \star G$ since $\star\star=(-1)^{k(m-k)+s}$ on forms of degree $k$, where $m=\dim M$ and $s$ denotes the index of the metric on $M$, i.e., the number of negative components in the signature of $g$. Hence,
  \begin{align*}
    \d{}^\nabla{\star\, F} &=\tfrac 1 c\d{}^\nabla G=0,\\
    \d{}^\nabla{\star\,G}&=(-1)^s c\d{}^\nabla F = (-1)^s c\, G= (-1)^s c^2 \star F.
  \end{align*}
  Hence, we see that for $(F,G)$ to be a solution to the seond Yang--Mills equation, we would need to have $(-1)^s c^2=\frac 1\mu$. We immediately observe that in the Riemannian case ($s=0$), this implies $\mu>0$, which is in contradiction with Remark \ref{rem:laplacian_eigenvalues}. The situation is rectified by allowing the Riemannian metric to be pseudo-Riemannian, which removes the restriction on $\Delta$ having non-negative eigenvalues, hence removing the constraint $\mu< 0$. However, this comes with a small price: we need to impose that the restriction of $\innerr\cdot\cdot_{\frak k}$ to invariant forms $\Omega^\bullet_\inv(M;\frak k)$ is again a nondegenerate pairing (for the proof of Theorem \ref{thm:ym_multiplicative} (ii) to work), which was automatic in the Riemannian case. For example, in the case when $A$ is a regular algebroid and the orbit foliation $\F$ is simple, this requirement can be attained with a $\pi$-invariant metric $g$ on $M$, defined in equation \eqref{eq:q_invariant_metric}, where $\pi\colon M\ra M/\F$ is the projection to the orbit space.
  
  Summing up, we conclude the Riemannian case does not allow for (anti) self-dual solutions, and the pseudo-Riemannian case in general allows them: the (anti) self-dual solutions are defined as those of the form
  \begin{align}
    G=\pm \frac{1}{\sqrt {(-1)^s\mu}}\star F,
  \end{align}
  for a given structure constant $\mu$ satisfying $\sgn\mu=(-1)^s$.

\subsection{Multiplicative Yang--Mills theory as a constrained variational problem}
\label{sec:constrained_variational_problem}
In this section, we interpret the obtained theory as a constrained variational problem. As a consequence, this will enable us to discuss the formal tangent space to the space of triples which are both longitudinally and transversally critical. 

We have observed in Remark \ref{rem:decompositions_affine} that the directions in which we vary the triples (consisting of IM connections and their curvings) are described by a fixed equivalence class of a given primitive IM connection in $H^{1,1}(A;\frak k)$---see equation \eqref{eq:tilde_3}. 
This already suggests we are dealing with a variational problem constrained to a constant cohomological class $\chi\in H^{1,1}(A;\frak k)$. More precisely, the class $\chi$ is defined by a primitive IM connection, i.e.,
\[
\chi\in \arc(A;\frak k)/\im\delta^0\subset H^{1,1}(A;\frak k).
\]
To elaborate on this point of view, let us denote the constrained space of triples by
\[
\DD_\chi(A;\frak k)=\set{((\C,v),F)\in \arc(A;\frak k)\times \Omega^2(M;\frak k)\given \delta^0 F=\Omega^{(\C,v)}, [(\C,v)]=\chi}.
\]
The following proposition shows that this is an affine space, whose affine structure combines the two distinct notions of affinity from Remark \ref{rem:decompositions_affine} into a single one.
\begin{proposition}
  For $\chi$ as above, $\DD_\chi(A;\frak k)$ is an affine space, modelled on the vector space $Q_\chi$ defined as the quotient
\[\begin{tikzcd}
	0 & {\Omega^1_\inv(A;\frak k)} & {\Omega^1(M;\frak k)\oplus \Omega^2_\inv(M;\frak k)} & {Q_\chi} & 0,
	\arrow[from=1-1, to=1-2]
	\arrow["j_\chi", from=1-2, to=1-3]
	\arrow[from=1-3, to=1-4]
	\arrow[from=1-4, to=1-5]
\end{tikzcd}\]
where $j_\chi(\gamma)=(\gamma,-\d{}^\nabla\gamma)$ and
$\nabla$ is the connection induced by an arbitrary representative of $\chi$.
\end{proposition}
\begin{proof}
  First, note that when restricted to the centre $z(\frak k)$, any linear connection $\nabla$ as above does not depend on the choice of a representative of $\chi$, by identity \eqref{eq:nabla_gamma}. This is important since invariant forms are centre-valued, so $\d{}^\nabla$ appearing in the statement is also independent of such a choice; moreover, $\d{}^\nabla$ preserves invariance of forms by Theorem \ref{thm:deltaD_inf}, which altogether shows that the map $j_\chi$ is well-defined and only depends on the class $\chi$. The affine structure on $\DD_\chi(A;\frak k)$ is defined by
\begin{align}
  \label{eq:affine_DD_chi}
    ((\C,v),F)+\llbracket \gamma,\beta\rrbracket=((\C,v)+\delta^0 \gamma, F+\d{}^\nabla\gamma-\tfrac 12[\gamma,\gamma]+\beta)
\end{align}
  for any $\gamma\in\Omega^1(M;\frak k)$ and $\beta\in\Omega^2_\inv(M;\frak k)$, where the double bracket $\llbracket \cdot,\cdot \rrbracket$ denotes the equivalence class of $(\gamma,\beta)$ in $Q_\chi$. That this is a well-defined affine structure on $\DD_\chi(A;\frak k)$ is immediate from the definition of the quotient $Q_\chi$.
\end{proof}
Now, it is clear that a triple $((\C,v),F)$ is both longitudinally and transversally critical if and only if it is \textit{totally critical}, that is, 
  \[
    \deriv\lambda0\S\big(((\C,v),F)+\lambda\llbracket \gamma,\beta\rrbracket\big)=\deriv\lambda0\S\big((\C,v)+\lambda \delta^0\gamma, F^{\lambda\gamma}+\lambda\beta\big)=0,
  \]
  for all $\gamma\in\Omega^1(M;\frak k)$ and $\beta\in\Omega^2_\inv(M;\frak k)$. In other words, total criticality of $((\C,v),F)$ means precisely the criticality in the usual sense, with respect to the affine structure \eqref{eq:affine_DD_chi} on $\DD_{\chi}(A;\frak k)$, where $\chi=[(\C,v)]$, thus finding totally critical triples amounts precisely to a constrained variational problem for the action $\S\colon \DD(A;\frak k)\ra \R$. Note also that by equation \eqref{eq:nabla_gamma}, as soon as one representative of $\chi$ is compatible with $\inner{\cdot}{\cdot}_\frak k$, the same holds for all of them, so metric compatibility with $\inner\cdot\cdot_{\frak k}$ can in fact be viewed as a property of the class $\chi$. To sum up, Theorem \ref{thm:ym_multiplicative} tells us that a triple $((\C,v),F)$ is a solution of the constrained variational problem if and only if it satisfies
\begin{align}
  \label{eq:totally_critical_eqs}
    \d{}^\nabla\! \star F=0\quad \text{and}\quad (F+\mu \delta^\nabla G)\perp \Omega^2_\inv(M;\frak k).
\end{align}
In this picture, the adaptedness condition \eqref{eq:adapted}  simply becomes a second additional constraint of the variational problem for the multiplicative Yang--Mills action $\S\colon\DD(A;\frak k)\ra M$.
  
\subsubsection{The tangent space to the space of totally critical triples}
We now follow a similar procedure as in Remark \ref{rem:tangent_space_of_critical_splittings} to obtain the (formal) tangent space to the space of solutions of the constrained variational problem above. As usual, we identify the tangent space of $\DD_\chi(A;\frak k)$ at any triple $((\C,v),F)$ with $Q_\chi$. 

We start by observing that the Hessian at a totally critical triple $((\C,v),F)$ can now easily be computed: it is the quadratic form $H_{((\C,v),F)}\colon Q_\chi\ra \R$, given by
\begin{align*}
  H_{((\C,v),F)}\llbracket \gamma,\beta\rrbracket&=\frac12 \frac{d^2}{d\lambda^2}\S\big(((\C,v),F)+\lambda\llbracket \gamma,\beta\rrbracket\big)\Big|_{\lambda=0}\\
  &=\innerr{\delta^\nabla\d{}^\nabla\gamma-\widehat F(\gamma)}{\gamma}_{\frak k}+\innerr{\beta+2\d{}^\nabla\gamma+\mu\delta^\nabla\d{}^\nabla\beta}{\beta}_{\frak k},
\end{align*}
where $\widehat F\colon \Omega^1(M;\frak k)\ra \Omega^1(M;\frak k)$ is given by $\widehat F(\gamma)=\star\,[\star\, F,\gamma]$. The same argument as in the proof of Theorem \ref{thm:ym} shows that $\widehat F$ is characterized by the equality $\innerr{\widehat F(\gamma)}{\tilde\gamma}_{\frak k}=\innerr{F}{[\gamma,\tilde\gamma]}_{\frak k}$ for all $\gamma,\tilde\gamma\in\Omega^1(M;\frak k)$, therefore it is self-adjoint.
\begin{proposition}
  \label{prop:formal_tangent_space}
  Let $((\C,v),F)\in\DD_\chi(A;\frak k)$ be a totally critical triple of the multiplicative Yang--Mills action functional. The formal tangent space at $((\C,v),F)$ to the space  of all totally critical triples in $\DD_\chi(A;\frak k)$ is given by all $\llbracket \gamma,\beta \rrbracket\in Q_\chi$ which satisfy 
  \begin{gather}
    \delta^\nabla(\d{}^\nabla\gamma+\beta)=(-1)^{\dim M}\widehat F(\gamma),\label{eq:fts_1}\\
    (\d{}^\nabla\gamma+\beta)+\mu\delta^\nabla\d{}^\nabla\beta\perp \Omega^2_\inv(M;\frak k).\label{eq:fts_2}
  \end{gather}
  The formal tangent space at $((\C,v),F)$ to adapted triples in $\DD_\chi(A;\frak k)$ is spanned by $\llbracket \gamma,\beta \rrbracket$  with
  \begin{align}
    \d{}^\nabla\gamma+\mu\delta^\nabla\d{}^\nabla\beta\in \Omega^2_\inv(M;\frak k).\label{eq:fts_3}
  \end{align}
  In particular, the formal tangent space at $((\C,v),F)$ to triples in $\DD_\chi(A;\frak k)$ which are simultaneously transversally critical and adapted, consists of $\llbracket \gamma,\beta \rrbracket$  satisfying
  \begin{align}
    \label{eq:fts_4}
    (\d{}^\nabla\gamma+\beta)+\mu\delta^\nabla\d{}^\nabla\beta=0.
  \end{align}
\end{proposition}
\begin{proof}
  The task at hand is to differentiate the equations \eqref{eq:totally_critical_eqs}. For the first equation, we compute
\begin{align*}
  \delta^{\nabla^{\lambda\gamma}}(F^{\lambda\gamma}+\lambda\beta)&= \star^{-1}(\d{}^\nabla+\lambda[\cdot,\gamma])\star(F^{\lambda\gamma}+\lambda\beta)\\
  &=\delta^{\nabla}F+\lambda\big({\star^{-1}}[\star\, F,\gamma]+\delta^\nabla \d{}^\nabla\gamma+\delta^\nabla\beta\big)+\mathcal O(\lambda^2).
\end{align*}
Setting this expression to zero and differentiating at $\lambda=0$, we obtain the equation \eqref{eq:fts_1}. 
To differentiate the second identity in \eqref{eq:totally_critical_eqs}, we similarly compute
\begin{align*}
  (F^{\lambda\gamma}+\lambda\beta)+\mu\delta^{\nabla^{\lambda\gamma}}(G^{\lambda\gamma}+\lambda\d{}^{\nabla^{\lambda\gamma}}\beta)&=(F^{\lambda\gamma}+\lambda\beta)-\mu\star^{-1}(\d{}^\nabla+
\lambda\cancel{[\cdot,\gamma]})\star(G+\lambda\d{}^\nabla\beta),
\end{align*}
where we used equation \eqref{eq:3_curvature_invariant} and the fact that invariant forms are centre-valued. Setting this expression as orthogonal to $\Omega^2_\inv(M;\frak k)$ and differentiating at $\lambda=0$, we obtain the equation \eqref{eq:fts_2}.
At last, we must differentiate the adaptedness condition; we observe
\begin{align*}
  \mu\delta^0\delta^{\nabla^{\lambda\gamma}}(G^{\lambda\gamma}+\lambda\d{}^{\nabla^{\lambda\gamma}}\beta)=\mu\delta^0\delta^\nabla(G+\lambda\d{}^\nabla\beta).
\end{align*}
Setting this expression to $-\Omega^{(\C,v)+\lambda\delta^0\gamma}=-\delta^0 F^{\lambda\gamma}$ and differentiating, we obtain \eqref{eq:fts_3}.
\end{proof}
Combining equations \eqref{eq:fts_1} and \eqref{eq:fts_4}, we also obtain the following.
\begin{corollary}
  The formal tangent space at $((\C,v),F)\in\DD_\chi(A;\frak k)$ to the space of solutions in $\DD_\chi(A;\frak k)$ of the Yang--Mills equations \eqref{eq:ym_1} and \eqref{eq:ym_2} consists of all vectors $\llbracket \gamma,\beta \rrbracket\in Q_\chi$  satisfying  the identities $\widehat F(\gamma)=0$ and \eqref{eq:fts_4}.
\end{corollary}

\subsection{Gauge invariance}
\label{sec:multiplicative_gauge_invariance}
As in \sec\ref{sec:foliated_gauge_invariance}, we now show that the multiplicative Yang--Mills action is gauge invariant. That is, the action functional is invariant under the pullback by any Lie algebroid automorphism $\phi$, covering an orientation-preserving isometry $\varphi$ on the base, restricting on $\frak k$ to a Lie algebra bundle automorphism which is moreover an isometry. Such a morphism $(\phi,\varphi)$, preserving all structure, will simply be called a \textit{gauge transformation}.
\[\begin{tikzcd}
	A & A \\
	M & M
	\arrow["\phi", from=1-1, to=1-2]
	\arrow[Rightarrow, from=1-1, to=2-1]
	\arrow[Rightarrow, from=1-2, to=2-2]
	\arrow["\varphi", from=2-1, to=2-2]
\end{tikzcd}\]
As in the foliated case, examples of gauge transformations include the following.
\begin{enumerate}[label={(\roman*)}]
  \item Algebroid automorphisms of the form $\Ad_b=(I_b)_*$, where $I_b$ is the inner automorphism by a (target) bisection $b\in\mathrm{Bis}(\G)$ of an integrating $s$-connected Lie groupoid $\G$ of $A$, whose base map $\varphi=s\circ b$ is an orientation-preserving isometry. This relies on Lemma \ref{lem:pairing_A_invariant_implies_G_invariant}.
  \item Flows of inner derivations $[\alpha,\cdot]$ of the Lie algebroid $A$, for sections $\alpha\in\Gamma(A)$ such that the flow of the (complete) vector field $\rho(\alpha)$ is an orientation-preserving isometry of $M$. Specifically, any $\alpha\in\Gamma(\ker\rho)$ is of this kind.
\end{enumerate}
To demonstrate gauge invariance, first note that any Weil cochain $c\in W^{p,q}(A;\frak k)$ can be pulled back along $\phi$. Indeed, observe that on $\Omega^\bullet(M;\frak k)$, the pullback along $\phi$ is defined on simple tensors as
$\phi^*(\gamma\otimes\xi)=\varphi^*\gamma\otimes (\phi^{-1})_*\xi$. 
  With this in mind, the pullback $\phi^* c\in W^{p,q}(A;\frak k)$ reads
\begin{align*}
  (\phi^* c)_k(\alpha_1,\dots,\alpha_{p-k}\|\beta_1,\dots,\beta_k)=\phi^*c_k(\phi_*(\alpha_1),\dots,\phi_*(\alpha_{p-k})\|\phi_*(\beta_1),\dots,\phi_*(\beta_k)).
\end{align*}
Since $\phi$ is a Lie algebroid automorphism, $\phi^*$ defines an automorphism of the cochain complex $W^{\bullet,q}(A;\frak k)$, for any fixed $q\geq 0$. In particular, we can pull back an IM connection $(\C,v)$ and again obtain an IM connection $\phi^*(\C,v)\eqcolon(\C_\phi,v_\phi)$, since $\phi$ restricts to a Lie algebra bundle automorphism $\frak k\ra \frak k$ covering $\varphi$.
The induced linear connection $\nabla^\phi\coloneq \C_\phi|_{\frak k}$ is simply the pullback connection, given by
\begin{align*}
  \nabla^\phi_X\xi=(\phi^{-1})_*\nabla_{\varphi_*X}(\phi_*\xi),
\end{align*}
for any $X\in\vf(M)$ and $\xi\in\Gamma(\frak k)$. With this in mind, it is easy to see the following.
\begin{lemma}
  For a given Lie algebroid $A\Ra M$ with an $\ad$-invariant metric $\inner\cdot\cdot_{\frak k}$ on a bundle of ideals $\frak k$, metric compatibility of IM connections is stable under gauge transformations.
\end{lemma}
\begin{proof}
  Suppose $\nabla=\C|_{\frak k}$ is compatible with $\inner\cdot\cdot_{\frak k}$. For any $X\in\vf(M)$ and $\xi,\eta\in\Gamma(\frak k)$, we compute
  \begin{align*}
    \innersmall{\nabla^\phi_X\xi}{\eta}_{\frak k}+\innersmall{\xi}{\nabla^\phi_X\eta}_{\frak k}&=\innersmall{\nabla_{\varphi_*X}(\phi_*\xi)}{\phi_*\eta}_{\frak k}\circ\varphi+\innersmall{\phi_*\xi}{\nabla_{\varphi_*X}(\phi_*\eta)}_{\frak k}\circ\varphi\\
    &=(\varphi_*X)\innersmall{\phi_*\xi}{\phi_*\eta}_{\frak k}\circ\varphi\\
    &=(\varphi_*X)(\inner{\xi}{\eta}_{\frak k}\circ\varphi^{-1})\circ\varphi\\
    &=X\inner{\xi}{\eta}_{\frak k}
  \end{align*}
  where we have used $\ad$-invariance in the first and third equality, and the assumption that $\nabla$ is compatible with the given metric $\inner\cdot\cdot_{\frak k}$ on the second equality. In the last line, we used \cite{lee}*{Corollary 8.21}. Hence, compatibility with an $\ad$-invariant metric $\inner\cdot\cdot_{\frak k}$ is stable under gauge transformations. 
\end{proof} 

\begin{theorem}
  \label{thm:multiplicative_ym_gauge_invariance}
	Let $A$ be a Lie algebroid with Yang--Mills data. The multiplicative Yang--Mills action is invariant under the pullback of a gauge transformation $(\phi, \varphi)$, that is, 
	\[
	\S(\phi^*(\C,v),\phi^* F)=\S((\C,v),F),
	\]
	for any triple $((\C,v),F)\in \DD(A;\frak k)$. In particular, longitudinal and transversal criticality are invariant under gauge transformations, and moreover, so is adaptedness.
\end{theorem}
\begin{proof}
  We first note that the curvature of the pullback IM connection is the pullback of the curvature, that is,
  \[
  \D{}^{\phi^*(\C,v)}=\phi^*\Omega^{(\C,v)}.
  \]
  This is an immediate consequence of naturality of the horizontal exterior covariant derivative, i.e., it commutes with pullbacks along Lie algebroid isomorphisms,
\begin{align*}
    \D{}^{\phi^*(\C,v)}\phi^*=\phi^*\D{}^{(\C,v)},
\end{align*}
  which is implied by the identities 
  \begin{align*}
    \d{}^{\nabla^\phi}\phi^*=\phi^*\d{}^\nabla\,\text{ and }\,h^*_\phi\phi^*=\phi^*h^*,
  \end{align*} 
  where $h_\phi^*$ is the horizontal projection of Weil cochains with respect to the IM connection $\phi^*(\C,v)$. The first identity is clear, and the second is most easily seen in the $\vb$-algebroid picture---it follows from $\d\phi\circ \smash{h^{TA}_\phi}=h^{TA}\circ\d\phi$, where $h^{TA}_\phi\colon TA\ra E_\phi=\d\phi^{-1}(E)$ is the horizontal projection with respect to the IM connection $\phi^*(\C,v)$. 
  
  Now, if $F$ is a curving of $(\C,v)$ with 3-curvature $G$, then $\phi^*F$ is a curving of $\phi^*(\C,v)$ with 3-curvature $\phi^*G$. This follows from the already observed fact that $\phi^*$ is an automorphism of the Weil cochain complex (at any fixed degree $q$). Thus, both the curving and its 3-curvature satisfy
  \begin{align}
    \phi^*F=(\phi^{-1})_*\circ \varphi^*F,\quad \phi^*G=(\phi^{-1})_*\circ \varphi^*G,
  \end{align}
  where  the pullback $\varphi^*\colon \Omega^\bullet(M;\frak k)\ra \Omega^\bullet(M;\varphi^*\frak k)$ is defined on simple tensors by pulling back the form and leaving the coefficients alone, as in \eqref{eq:varphi_star_foliated_ym_invariance}. This concludes the proof of the first part of the theorem. For the last part, invariance of adaptedness follows from the fact that the pullback along any orientation-preserving isometry commutes with the Hodge star operator.
\end{proof}

\subsection{Example: central $S^1$-extensions}
\label{sec:central_extensions_ym}
We now present an example whose purpose is to stress the following point: although the multiplicative Yang--Mills action functional can be constructed purely in terms of infinitesimal data, restricting to primitive connections of a specific integrating groupoid might sometimes instead be desirable. This is also the simplest non-transitive example that we know of.

Let us consider an \emph{$S^1$-extension of a submersion groupoid} is a Lie groupoid $\G\rra M$, together with a surjective submersive Lie groupoid morphism $f\colon \G\ra M\times_\pi M$ onto the submersion groupoid of a surjective submersion $\pi\colon M\ra N$, whose kernel is the trivial abelian Lie group bundle $S^1_M=M\times S^1$. The situation is portrayed with the short exact sequence of Lie groupoids,
\[\begin{tikzcd}
	1 & S^1_M & \G & {M\times_\pi M} & 1.
	\arrow[from=1-1, to=1-2]
	\arrow[from=1-2, to=1-3]
	\arrow["f",from=1-3, to=1-4]
	\arrow[from=1-4, to=1-5]
\end{tikzcd}\]
An $S^1$-extension is said to be \emph{central}, if for any $g\in \G$ and $\theta\in S^1$ there holds 
\[C_g(s(g),\theta)=(t(g),\theta),\]
where $C_g$ is the conjugation by $g$. At the level of Lie algebroids, the short exact sequence reads
\[\begin{tikzcd}
	0 & \R_M & A & {T\F} & 0,
	\arrow[from=1-1, to=1-2]
	\arrow[from=1-2, to=1-3]
	\arrow["{f_*}", from=1-3, to=1-4]
	\arrow[from=1-4, to=1-5]
\end{tikzcd}\]
and centrality translates to the restricted adjoint representations $\Ad\colon \G\curvearrowright \R_M$, $\ad\colon A\curvearrowright \R_M$ being the trivial ones. It is instructive to keep in mind the special case $f=(t,s)$, in which case $f_*=\rho$, the kernel $M\times S^1$ is the isotropy bundle of $\G$, and $\F$ is its orbit foliation, which is simple. 

Let us now consider a multiplicative Ehresmann connection $\omega\in\A(\G;\R_M)$ for the groupoid morphism $f$. By \cite{mec}*{Proposition 4.9}, such a connection exists (since $\G$ must be proper) and moreover, the induced linear connection $\nabla$ on $\R_M$ from Proposition \eqref{eq:nabla}, must be the canonical flat connection. This follows from the fact that the differential of $\exp\colon \R_M\ra S^1_M$ maps the induced linear connection, viewed as a distribution  $E^\nabla\subset T(\R_M)$, to the multiplicative Ehresmann connection $E^{S^1_M}=T(S^1_M)\cap E$ where $E=\ker\omega$, see \cite{mec}*{equation (2.3)}. It implies the constant section of $\R_M$ given by $x\mapsto (x,1)$ must be flat, hence all constant sections are flat and thus $\nabla$ is the canonical flat connection. Importantly, this holds due to our choice of integration---an IM connection on $A$ for $\R_M$ might give rise to a different connection $\nabla$ on $\R_M$.

As a consequence of this observation, any multiplicative connection $\omega$ on $\G$ for $\R_M$ must in fact be primitive. The proof of this claim rests upon the following procedure from \cite{mec}. First, note that the structure equation implies the curvature of $\omega$ equals
\[
\Omega^\omega= \d\omega \in\Omega^2_m(\G).
\]
Since $\Omega^\omega$ is horizontal and because there exists a multiplicative connection with $R^\nabla=0$, there is a unique 2-form on the submersion groupoid $\ul\Omega^\omega\in\Omega^2_m(M\times_\pi M)$ which pulls back to $\Omega^\omega$ via the morphism $f$. However, every multiplicative form on the submersion groupoid is cohomologically trivial by \cite{mec}*{Lemma 4.5}---this should actually be seen as a consequence of the deeper fact that $H^{p,q}(M\times_\pi M)=0$ for all $p\geq 1,q\geq 0$ (see \cite{bundle_gerbes_original}*{\sec 8}). We remind the reader of the definition of the cohomology groups \eqref{eq:bss_cohomology}. So, there exists a 2-form $F\in\Omega^2(M)$ such that
\begin{align*}
  \ul\Omega^\omega=\pr_2^*F-\pr_1^*F.
\end{align*}
This is our desired curving of $\omega$, that is, $\delta^0 F=\Omega^\omega$. Note that due to the centrality of extension, invariant forms $\Omega_\inv^\bullet(M)$ are precisely the ones in the image of the pullback $\pi^*\colon \Omega^\bullet (N)\ra \Omega^\bullet (M)$. Hence, for instance, $F$ is only unique up to a form $\beta=\pi^*\tilde\beta$ for some $\tilde\beta\in\Omega^2(N)$, and the curvature 3-form $G=\d F$ must equal $G=\pi^* \tilde G$ for some unique $\tilde G\in\Omega^3(N)$.

Equipped with this insight, we are now ready to formulate and study the desired Yang--Mills theory, restricted to multiplicative connections on $\G$. To begin with, we take $\kappa$ as the canonical metric on $\R_M$ and fix any Riemannian metric $g$ on $M$. We define the action functional by “restricting” the functional \eqref{eq:multiplicative_ym_action} to the (primitive) multiplicative Ehresmann connections on $\G$:
\begin{align}
  \S_\G\colon \DD(\G;\R_M)\ra \R,\quad \S_\G(\omega, F)=\S(\ve(\omega),F)=\innerr{F}{F}+\mu\innerr{\d F}{\d F},
\end{align}
where the domain now consists of multiplicative Ehresmann connections and their splittings,
\[
\DD(\G;\R_M)=\set{(\omega,F)\in\A(\G;\R_M)\times \Omega^2(M)\given \delta^0F=\Omega^\omega}.
\]
The notions of criticality and adaptedness for the global case are now defined analogously to the infinitesimal case, see Definition 
\ref{def:global_multiplicative_ym}; moreover, the respective global and infinitesimal notions of criticality and adaptedness are equivalent when $\G$ has connected $s$-fibres (see Proposition \ref{prop:relation_multiplicative_ym_groupoids_algebroids}). Now, the analogous equivalences to those from Theorem \ref{thm:ym_multiplicative} read:
\begin{align*}
  \text{$(\omega,F)$ is longitudinally critical}\quad&\iff\quad \d{}\star F=0,\\[0.3em]
  \text{$(\omega,F)$ is transversally critical and adapted}\quad&\iff\quad \d{} \star G=\tfrac 1\mu\star F.\qquad\qquad
\end{align*}
Let us first discuss the case $\dim N\leq 2$, when the focus is on solving the first Yang--Mills equation. To show it admits a solution, first take any pair $(\omega,F)$. Note that since the curvature 3-form $G=\d F$ vanishes, Hodge theorem ensures $[F]\in H^2_{dR}(M)$ has a (unique) harmonic representative, i.e., there exists a form $\gamma\in\Omega^1(M)$, unique up to a closed 1-form, such that $\delta(F+\d\gamma)=0$. We conclude that longitudinally critical points exist: since $\R_M$ is abelian, $F+\d\gamma$ is the curving of $\omega+\delta^0\gamma$ by equation \eqref{eq:curving_deformation}, hence the pair
\[
(\omega+\delta^0\gamma,F+\d\gamma)\in\DD(\G;\R_M)
\]
is longitudinally critical. 

The case $\dim N\geq 3$ is more intricate. Firstly, observe that since $\nabla$ is flat, the second Yang--Mills equation implies the first. But solvability of the second equation amounts to the operator $\delta{\d{}}$ on 2-forms having an eigenvalue $-\frac 1\mu$, which is a statement about the Riemannian manifold $M$. To dissect the second equation further, we note that transversal criticality is now equivalent to orthogonality of $F+\mu\delta G$ to the subspace of pullback forms, i.e., 
\[(F+\mu\delta G)\perp \im(\pi^*)\subset \Omega^2(M).\]
If we suppose that the metric on $M$ is $\pi$-invariant, this means precisely that $F+\mu\delta G$ vanishes whenever a pair of vectors from $\ker \d \pi^\perp$ is inserted. On the other hand, adaptedness simply means that $\iota_{X}(F+\mu\delta G)=0$ whenever $X\in\ker\d \pi$. Hence, we see the two conditions together mean precisely $\delta G=-\frac 1\mu F$.

\subsection{The integrable case}
\label{sec:multiplicative_ym_groupoids}
As witnessed in \sec\ref{sec:central_extensions_ym}, the global case is of independent interest. Albeit with some apparent repetition, we shall now provide the global analogues of the notions we have introduced for algebroids, and provide a simple relation between the global and infinitesimal Yang--Mills theory.

To begin with, given a bundle of ideals $\frak k$ of a groupoid $\G$, since we are not imposing any a priori connectedness assumptions on $\G$, the metric on $\frak k$ is now assumed $\Ad$-invariant---see \eqref{eq:Ad_invariance}.
\begin{definition}
  \label{def:global_multiplicative_ym}
  Let $\frak k$ be a bundle of ideals of a Lie groupoid $\G\rra M$, together with some Yang--Mills data. Let us denote
  \begin{align*}
    \DD(\G;\frak k)=\set*{(\omega,F)\in\A(\G;\frak k)\times \Omega^2(M;\frak k)\given 
      \Omega^\omega=\delta^0F
    }.
  \end{align*}
  The (global) \emph{multiplicative Yang--Mills action functional} is now defined as the map 
    \begin{align}
      \label{eq:global_multiplicative_ym_action}
      &\S_\G\colon \DD(\G;\frak k)\ra \R,\quad \S_\G(\omega,F)=\S(\ve(\omega),F)=\innerr FF_{\frak k}+\mu \innerr GG_{\frak k},
    \end{align}
    The relevant notions pertaining to criticality are now defined for $(\omega,F)$ as:
    \begin{enumerate}[label={(\roman*)}]
      \item \emph{Longitudinal criticality}: $\smallderiv\lambda0\S_\G(\omega+\lambda \delta^0\gamma, F^{\lambda\gamma})=0$, for all $\gamma\in\Omega^1(M;\frak k)$.
      \item \emph{Transversal criticality}:
        $\smallderiv\lambda0\S_\G(\omega, F+\lambda\beta)=0,$ 
        for all $\beta\in\Omega_{\inv,\G}^2(M;\frak k)$.
      \item \emph{Adaptedness}: $\delta^0(\delta^\nabla G)=-\frac 1\mu\Omega^\omega.$
    \end{enumerate}
    We note that  $\delta^0$ is now the simplicial differential on forms on $\G$ at level zero, $\Omega^\bullet(M;\frak k)\ra \Omega^\bullet_m(\G;\frak k)$, and invariant forms $\smash{\Omega_{\inv,\G}^2(M;\frak k)}$ are those in its kernel. The form $F^{\lambda\gamma}$ is the same as in the infinitesimal case, and one analogously shows it is a curving of $\omega+\lambda\delta^0\gamma$ if $F$ is a curving of $\omega$.
\end{definition}
\noindent Inspecting the proof of Theorem \ref{thm:ym_multiplicative}, we see we again have the following equivalences:
\begin{align*}
  \text{$(\omega,F)$ is longitudinally critical}\quad&\iff\quad \d{}^{\nabla}{\star\,F}=0,\\[0.3em]
  \text{$(\omega,F)$ is transversally critical}\quad&\iff\quad (F+\mu\delta^\nabla G)\perp \Omega_{\inv,\G}^2(M;\frak k),\qquad\\[0.3em]
  \text{$(\omega,F)$ is transversally critical and adapted}\quad&\iff\quad \d{}^\nabla{\star\,G}={\tfrac 1\mu}\star F.
\end{align*}
What follows is the promised relation between the global and infinitesimal theory. 

\begin{samepage}
  \begin{proposition}
    \label{prop:relation_multiplicative_ym_groupoids_algebroids}
    Let $\frak k$ be a bundle of ideals of a Lie groupoid $\G$ with Lie algebroid $A$, together with a fixed Yang--Mills data, and suppose a pair $(\omega,F)\in\DD(\G;\frak k)$ is given. The van Est map relates the notions of criticality and adaptedness as follows.
    \begin{enumerate}[label={(\roman*)}]
      \item $(\omega,F)$ is longitudinally critical if and only if $(\ve(\omega),F)$ is longitudinally critical.
      \item If $(\ve(\omega),F)$ is transversally critical, then $(\omega,F)$ is transversally critical.
      \item If $(\omega,F)$ is adapted, then $(\ve(\omega),F)$ is adapted.
      \item $(\omega,F)$ is transversally critical and adapted if and only if $(\ve(\omega),F)$ is transversally critical and adapted.
    \end{enumerate}
    The converse implications to (ii) and (iii) hold if $\G$ has connected $s$-fibres. Furthermore, if $\G$ has simply connected $s$-fibres, the bijective map 
    \begin{align*}
      &\DD(\G;\frak k)\ra \DD(A;\frak k),\quad (\omega,F)\mapsto (\ve(\omega),F)
    \end{align*}
    restricts to bijections between pairs and triples which are longitudinally critical, transversally critical, or adapted, respectively.
  \end{proposition}
\end{samepage}
\begin{proof}
  The points (i) and (iv) are clear; (ii) and (iii) follow simply from the fact that $\ve\circ \delta^0_\G=\delta^0_A$, implying $\Omega_{\inv,\G}^\bullet(M;\frak k)\subset \Omega_{\inv,A}^\bullet(M;\frak k)$. The last part is a consequence of Corollary \ref{corollary:van_est_multiplicative}.
\end{proof}

\begin{subappendices}
\section{Related generalizations}
This section is devoted to a relaxed discussion about the relationship of the obtained framework of multiplicative Yang--Mills theory with other generalizations of Yang--Mills theory. At the very end of the section, we mention a potential umbrella framework where all the different generalizations meet as particular cases.

\subsubsection{Yang--Mills theory for bundle gerbes}
  Consider once again the example of central $S^1$-extensions, presented in  \sec\ref{sec:central_extensions_ym}. The notion of a central $S^1$-extension is very closely related to the notion of an $S^1$-bundle gerbe \cites{gerbes, stacks_gerbes, bundle_gerbes, bundle_gerbes_original}---the latter is a Morita equivalence class of the former. An attempt to obtain a Yang--Mills theory on $S^1$-bundle gerbes was already  made by Mathai and Roberts in \cite{ym_gerbes}, however, the paper was recently retracted due to the discovery of an unfixable flaw by the authors themselves. To our understanding, the flaw was in their definition of gauge transformations for bundle gerbes; it erroneously assumed the existence of a certain linear connection, in turn rendering their result on gauge invariance and several results of the article uncertain. The notion of gauge invariance we have provided in \sec\ref{sec:multiplicative_gauge_invariance} is certainly independent of any such external data, and the existence of gauge transformations is also clear by Remark \ref{rem:exp_bisections}. This rectifies the issue of \cite{ym_gerbes}, albeit with a different approach to the subject, using Lie groupoids and algebroids. Moreover, we would like to point out that the action functional  \cite{ym_gerbes}*{equation (3)} only contains the term with the 3-curvature, $\innerr {G}{G}$, hence it misses the phenomena accounted for by the curving $F$. In turn, the Yang--Mills equation reads $\d{}\star G=0$. Since $G$ is obtained from $F$, and since $F$ links the whole theory back to principal bundles, we consider the curving to be a fundamental part of a (primitive) multiplicative connection, and hence believe that the term $\innerr{F}{F}$ should also be considered for the study the moduli space of solutions $(F,G)$ of the pair of equations \eqref{eq:ym_1} and \eqref{eq:ym_2}, as attempted in \cite{ym_gerbes}.

\subsubsection{Higher Yang--Mills theory}
  In \cite{higher_ym}, Baez develops a Yang--Mills theory for (trivial) principal 2-bundles. Roughly, a principal 2-bundle is a higher-geometric analogue of a principal bundle, where the structure Lie group is replaced with a structure Lie 2-group \cite{higher_gauge}. This produces a setting for higher-geometric connections, where parallel transport is feasible over surfaces instead of only paths. The development of Yang--Mills theory for such connections results in a strikingly similar pair of Yang--Mills equations (\cite{higher_ym}, cf.\ Theorem 21) compared to the equations \eqref{eq:ym_1} and  \eqref{eq:ym_2} we have obtained in Theorem \ref{thm:ym_multiplicative}. The similarity is already apparent at the level of action functionals---compare the definition of \cite{higher_ym}*{equation (9)} with our action functional \eqref{eq:multiplicative_ym_action}. The difference between the two pairs of Yang--Mills equations is only in that the first Yang--Mills equation in \cite{higher_ym} has an additional term, which appears as a consequence of the very definition of connections on principal 2-bundles. Despite these similarities, the frameworks are conceptually very different; as mentioned below, we believe the similarities come from the fact that the frameworks are particular cases within a more general theory.

\subsubsection{Yang--Mills theory on principal LGB-bundles}
In \cite{fischer}, Fisher develops a Yang--Mills theory for principal LGB-bundles. Roughly, a principal LGB-bundle is a generalization of a principal bundle, where the structure Lie group is replaced with a Lie group bundle. The theory of connections on such principal bundles is thoroughly developed there, cf.\ \sec 6. The action functional the author considers in \cite{fischer}*{Definition 7.2} now appears to only contain the term $\innerr FF$. Albeit the corresponding Yang--Mills PDE is not obtained there, we can imagine it has the well-known form $\d{}^\nabla\star F=0$. In any case, we would like to point out that there is a similarity of the formulas in \cite{fischer}*{Theorem 7.5}  with the ones we have obtained when discussing primitive connections (see Remark \ref{rem:deformation_curving}). The so-called \textit{field redefinitions} capture the idea that the curving of a connection is determined only up to an invariant 2-form.

\vspace{1em}

We view the similarities (and differences) between these various settings as a manifestation of the fact that they are in fact just particular cases of a more general framework. Our speculation is that such a general framework has to do with \textit{$\mathcal{PB}$-groupoids}, which simultaneously generalize principal 2-bundles, LGB-bundles, and comprise the natural setting for the frame bundle of a $\vb$-groupoid---a generalization of the frame bundle construction for the usual vector bundles. This has recently been researched in \cites{pb_groupoids, pb_vb_groupoids}. We suspect this more general framework for Yang--Mills theories should be feasible by researching the theory of connections on $\mathcal{PB}$-groupoids, and employing similar techniques that have appeared in this thesis.

\end{subappendices}

\clearpage \pagestyle{plain}

\clearpage 

\pagestyle{plain}

\printindex






\end{document}